\documentclass{amsart}
\usepackage{amsthm}
\usepackage{amstext}
\usepackage{enumitem}
\usepackage{amssymb}
\usepackage{amsmath,calligra,mathrsfs}
\usepackage[all,cmtip]{xy}
\usepackage{dsfont}
\usepackage{hyperref}
\hypersetup{
	colorlinks   = true,          %Colors links instead of boxes
	urlcolor     = blue,          %Color for external hyperlinks
	linkcolor    = blue,          %Color of internal links
	citecolor   = blue             %Color of citations
}
\usepackage{graphicx}

\usepackage{comment}
\usepackage{xcolor}
\usepackage{tikz-cd}
\usepackage{mathtools}
\usepackage{cancel}

\newenvironment{detail}{\color{magenta}}{} %If you want to see details

\theoremstyle{plain}
\newtheorem{thm}{Theorem}[section]
\newtheorem{lem}[thm]{Lemma}
\newtheorem{prop}[thm]{Proposition}
\newtheorem{cor}[thm]{Corollary}

\theoremstyle{definition}
\newtheorem{defn}[thm]{Definition}
\newtheorem{exmp}[thm]{Example}

\theoremstyle{remark}
\newtheorem{rem}[thm]{Remark}

\theoremstyle{plain}

\DeclareMathOperator{\id}{id}
\DeclareMathOperator{\rank}{rank}
\DeclareMathOperator{\Spec}{Spec}
\DeclareMathOperator{\Proj}{Proj}

\DeclareMathOperator{\sheafhom}{\mathscr{H}\text{\kern -3pt {\calligra\large om}}\,}

\DeclareMathOperator{\Jac}{Jac}

\DeclareMathOperator{\CH}{CH}

\DeclareMathOperator{\Sym}{Sym}

\DeclareMathOperator{\rplane}{\Lambda}

\DeclareMathOperator{\bl}{bl}

\newcommand{\defi}[1]{\textsf{#1}} % for defined terms

\newcommand{\kbar}{{\overline{k}}}

\def\Ker{\text{Ker}}
\def\Im{\text{Im}}

\let\phi\varphi
\def\P{\mathbb{P}}
\def\O{\mathcal{O}}
\def\Z{\mathbb{Z}}
\def\C{\mathbb{C}}
\def\Q{\mathbb{Q}}

\def\F{\mathbb{F}}
\def\L{\mathbb{L}}
\def\A{\mathbb{A}}

\DeclareMathOperator{\Gr}{Gr}

\DeclareMathOperator{\Sing}{Sing}
\def\OG{\mathrm{OG}}

\makeatletter
\let\@wraptoccontribs\wraptoccontribs
\makeatother

\excludecomment{detail}

\title{Motivic classes of Fano schemes of lines on intersections of two quadrics}
\author{Lena Ji}
\address{Department of Mathematics, University of Illinois Urbana-Champaign, 214 Harker Hall, 1305 W. Green Street, Urbana, IL 61801}
\email{lenaji.math@gmail.com}
\urladdr{https://lji.web.illinois.edu/}
\author{Fumiaki Suzuki}
\address{BICMR\\
Peking University\\
5 Yiheyuan Road, Haidian District, Beijing 100871\\
China
}
\email{suzuki@bicmr.pku.edu.cn}
\urladdr{https://fumiaki-suzuki.github.io/}
\contrib[with an appendix by]{Pieter Belmans}
\address{Mathematical Institute, Utrecht University, Budapestlaan 6, 3584CD Utrecht, Netherlands}
\email{p.belmans@uu.nl}

\subjclass[2020]{14D20, 14D06, 14C25, 14F08}

\thanks{During the preparation of this article, L.J. was supported in part by NSF grant DMS-2501990 and by gift SFI-MPS-TSM-00013959 from the Simons Foundation.
F.S. was supported by the ERC grant “RationAlgic” (grant no. 948066).
}

\begin{document}

\begin{abstract}
Let $X\subset\mathbb P^N$ be a smooth complete intersection of two quadrics. We find formulas in the Grothendieck ring of varieties for the class of $X$ and for the class of its Fano scheme of lines, thereby proving the first two cases of a conjecture of Belmans et al.
We also find a formula for the class of the relative Fano schemes of linear subspaces (of any dimension) in the fibers of quadric fibrations over curves, under a simple degeneration assumption. As consequences, we compute the rational Chow motive of the relative Fano scheme, and we provide evidence for a conjecture of Shah on a residual category in a semiorthogonal decomposition of the relative Fano scheme of lines.
\end{abstract}

\maketitle

\section{Introduction}

Let \(X\) be a smooth complete intersection of two quadrics in \(\mathbb P^N\) over an algebraically closed field of characteristic \(\neq 2\). There is an associated curve that governs much of the geometry of \(X\):
if \(N\) is odd then the discriminant of the associated pencil defines a hyperelliptic curve, and if \(N\) is even there is a ``stacky \(\mathbb P^1\)" analogue.
This curve is present in the geometry of the linear spaces on \(X\), and, in characteristic zero, it determines the residual component of the derived category of \(X\) \cite{BondalOrlov,Kuznetsov08}.

The Fano schemes \(F_r(X)\) of \(r\)-dimensional linear subspaces on \(X\), for \(0\leq r\leq\lfloor\frac{N}{2}\rfloor-1\), are smooth projective varieties that interpolate between \(X\) and geometric objects associated to the curve. For \(r=0\) we have \(F_0(X)=X\); for \(r\) maximal, \(F_r(X)\) is the Jacobian of the hyperelliptic curve if \(N\) is odd and is a finite scheme of length \(2^N\) if \(N\) is even \cite{Gauthier54,Reid-thesis}, and for the second maximal value of \(r\) the Fano scheme is isomorphic to the moduli space of certain rank 2 vector bundles on the associated curve \cite{DesaleRamanan,Casagrande15}.
Recently, in characteristic zero, Belmans--Bose--Frei--Gould--Hotchkiss--Lamarche--Petok--Rodriguez Avila--Shah
proposed conjectural semiorthogonal decompositions for the derived categories of the Fano schemes \(F_r(X)\) \cite{BBFGHLPRAS}, which would generalize the results of Bondal--Orlov and Kuznetsov on the derived category of \(X\). Namely, Belmans et al. conjectured that the derived category of \(F_r(X)\) decomposes into exceptional objects and, if \(N\) is odd, copies of the derived category of symmetric powers \(\Sym^i C\) of the associated curve for \(1\leq i\leq r+1\) \cite[Conjectures A and D]{BBFGHLPRAS}. They also proposed a parallel identity for the class of \(F_r(X)\) in the Grothendieck ring of varieties, which involves the Lefschetz motive and, if \(N\) is odd, symmetric powers of the curve \cite[Conjecture B and page 5]{BBFGHLPRAS}.
As evidence for this latter conjecture, Belmans et al. determined the Hodge structures of the cohomology of \(F_r(X)\), extending results of Chen--Vilonen--Xue \cite{CVX17,CVX20}. This provides evidence for the conjectural identities since the Grothendieck ring of varieties admits a realization map to Hodge structures, but it does not determine the motivic class of \(F_r(X)\).

In this paper, we address the first uniform cases of the conjectures of Belmans et al. on the class of \(F_r(X)\) in the Grothendieck ring of varieties. Namely, we prove their conjectural identities for \(X\) (Proposition~\ref{prop:r=0-case}) and the Fano scheme of lines \(F_1(X)\), and, as consequences, recover results about their Hodge structures.

\begin{thm}\label{thm:main-odd}
    Over an algebraically closed field \(k\) of characteristic \(\neq 2\), let \(g\geq 3\), let \(X\subset\P^{2g+1}\) be a smooth complete intersection of two quadrics, and let \(C\) be the associated genus \(g\) hyperelliptic curve. Then the following equality holds in \(\widetilde{K}_0(\mathrm{Var}/k)\):
    \begin{equation*}
        \begin{split}
            [F_1(X)] &= [\Sym^2 C] \L^{2g-4} + [C] \L^{g-2} \left( \sum_{i=0}^{2g-3} \L^i - \L^{g-2} - \L^{g-1} \right) \\
            & \phantom{=}  + \left(\sum_{i=0}^{2g-2} \L^i\right) \left(\sum_{i=0}^{g-2}\L^{2i}\right)- (\L^{g-2} + \L^{g+2}) \sum_{i=0}^{2g-6}\L^i - (\L^{g-1} + \L^g + \L^{g+1} + \L^{3g-7} + \L^{3g-6} + \L^{3g-5}) .
        \end{split}
    \end{equation*}
    Furthermore, the coefficients of \([\Sym^2 C]\), \([C]\), and the constant term are all effective, i.e., polynomials in \(\mathbb Z_{\geq 0}[\mathbb L]\), and they are respectively equal to \(M_{g,1,2}\), \(M_{g,1,1}\), and \(M_{g,1,0}\) defined in \cite[(4)]{BBFGHLPRAS}.
    In particular, \cite[Conjecture B]{BBFGHLPRAS} is true for \(F_1(X)\).
\end{thm}

\begin{thm}\label{thm:main-even}
    Over an algebraically closed field \(k\) of characteristic \(\neq 2\), let \(g\geq 3\), and let \(X\subset\P^{2g}\) be a smooth complete intersection of two quadrics. Then the following equality holds in \(\widetilde{K}_0(\mathrm{Var}/k)\):
    \[[F_1(X)] = 2g(g+2)\L^{2g-4} + (\L+1) \sum_{i=0}^{g-3}(i+1)(\L^{2i} + \L^{4g-9-2i})+(2g+1)(\L^{g-2}+\L^{2g-3})\sum_{i=0}^{g-3}\L^i .\]
    In particular, the class of \(F_1(X)\) is of Tate type and only depends on the dimension of \(X\), and the analogue of \cite[Conjecture B]{BBFGHLPRAS} given on \cite[page 5]{BBFGHLPRAS} is true for \(F_1(X)\).
\end{thm}

Here \(\widetilde{K}_0(\mathrm{Var}/k)\) is the quotient of the Grothendieck ring of varieties over \(k\) by the ideal generated by surjective radicial morphisms (Definition~\ref{defn:K0-Ktilde0}). In characteristic zero, \(\widetilde{K}_0(\mathrm{Var}/k)\) is isomorphic to the Grothendieck ring of varieties over \(k\).

The motivic classes of \(F_1(X)\) computed in
Theorems~\ref{thm:main-odd} and~\ref{thm:main-even} have several consequences via the realization maps.
When $k=\C$, these identities compute the cohomology of \(F_1(X)\) and moreover the Hodge structures of $H^*(F_1(X),\Q)$ (Corollary~\ref{cor:complex}\eqref{item:cor-odd-HS},\eqref{item:cor-even-HS}), giving a new and geometric proof of the \(r=1\) case of \cite[Theorem 1.1]{CVX17}, \cite[Theorem 1.1]{CVX20}, and \cite[Theorem C, Proposition E]{BBFGHLPRAS}, and they show that $H^*(F_1(X),\Z)$ is torsion-free (Corollary \ref{cor:complex}\eqref{item:cor-cohomology-tf}). When $k=\F_q$, they compute the number of lines on $X$ defined over $\F_{q^e}$ for sufficiently divisible $e\geq 1$ (Corollary~\ref{cor:finite}).
Over \(k=\mathbb C\), the identities for the class of \(F_1(X)\) also give a decomposition of the rational numerical Chow motive of \(F_1(X)\) (see proof of Corollary~\ref{cor:motive-F_r}), and they also verify the image  in \(K_0(\mathrm{Cat}/k)\) of Belmans et al.'s conjectural semiorthogonal decompositions of the derived category of \(F_1(X)\) (Corollary~\ref{cor:K_0-cat}).

Our second main result, which we also use as an input to the proofs of Theorems~\ref{thm:main-odd} and~\ref{thm:main-even}, is a motivic description for the relative Fano scheme of lines in a quadric fibration with simple degeneration over a curve.
More generally, we compute the class of the relative Fano scheme of \(r\)-planes for any relative dimension \(r\):

\begin{thm}\label{thm:relative-F_r}
    Over an algebraically closed field \(k\) of characteristic \(\neq 2\), let \(S\) be a smooth connected curve, let \(n\geq 1\), and let \(\pi\colon\mathcal Q\to S\) be a fibration of \(n\)-dimensional quadrics over \(k\). Assume that \(\pi\) has simple degeneration and that the degeneracy locus of \(\pi\) is a smooth divisor (see Section~\ref{sec:prelim-hyperbolic-red}).
    Then for \(0\leq r \leq \lfloor\frac{n+1}{2}\rfloor\), the following equalities hold in the Grothendieck ring of varieties over \(k\):
    \[ [F_r(\mathcal Q/S)] = \begin{cases}
        [S][\OG(r+1,2g+2)] + ([\widetilde S]-2[S])[\OG(r,2g+1)]\L^{g-r} & \text{if \(n=2g\) is even, \(r\leq g-1\)}, \\
        [\widetilde{S}][\OG(g+1,2g+2)] & \text{if \(n=2g\) is even, }r=g, \\ 
        [S][\OG(r+1,2g+1)]+\delta[\OG(r,2g)]\L^{g-r} & \text{if \(n=2g-1\) is odd, \(r\leq g-1\)}, \\
        2\delta [\OG(g,2g)] & \text{if \(n=2g-1\) is odd, }r=g.
    \end{cases} \]
    Here \(\delta\) is the number of singular fibers of \(\pi\), and, if \(n=2g\) is even, then \(\widetilde{S}\to S\) is the discriminant double cover of \(\pi\) (Section~\ref{sec:prelim-F_r}).
\end{thm}

We also prove an equivalent version of Theorem~\ref{thm:relative-F_r} with expanded coefficients (Theorem~\ref{thm:relative-F_r-expanded-coeffs}).
Over \(\mathbb C\), we use Theorem~\ref{thm:relative-F_r-expanded-coeffs} to give an explicit decomposition of the rational Chow motive of \(F_r(\mathcal Q/S)\) that involves only the base curve, the discriminant data, and the Lefschetz motive;
in particular, the rational Chow motive is of abelian type (Corollary~\ref{cor:motive-F_r}). Furthermore, we show that \(F_r(\mathcal Q/S)\) satisfies the Kimura finite-dimensionality conjecture and the Murre, standard, and Hodge conjectures.

Shah has recently studied these relative Fano schemes over \(\mathbb C\) from the categorical viewpoint \cite{Shah-rel-Fr}. Under a rank hypothesis (which is satisfied in the simple degeneration case for \(1\leq r\leq\lfloor\frac{n}{2}\rfloor\)), he constructed Clifford components in their derived categories, and, in the case of lines, a semiorthogonal decomposition up to a residual category \cite[Theorems 1.1 and 1.2]{Shah-rel-Fr}.
Thus, Theorem~\ref{thm:relative-F_r} may be viewed as a motivic counterpart of Shah's results. In the case of the relative Fano scheme of lines, Theorem~\ref{thm:relative-F_r} verifies the image in \(K_0(\mathrm{Cat}/\mathbb C)\) of Shah's conjectural description of this residual category in the case of quadric fibrations over a smooth curve with simple degeneration (Corollary~\ref{cor:rel-F_r-K_0-cat}).

In a separate work, over \(\mathbb C\), Shah constructed a standard flip between the Hilbert square \(X^{[2]}\) and the relative Fano scheme of lines \(F_1(\mathcal Q/\P^1)\) in the fibers of the associated pencil \(\mathcal Q\to\P^1\), whose centers are a \(\P^2\)-bundle over \(F_1(X)\) and a \(\P^1\)-bundle over \(F_1(X)\) \cite[Theorem 23]{Shah-flips} (see also \cite[Remark A.8]{CT-intersection-quadrics}).
The relations in the Grothendieck ring of varieties then yield the equality \([X^{[2]}]-[F_1(\mathcal Q/\P^1)] = [F_1(X)]\L^2\).
Although our proof of Theorems~\ref{thm:main-odd} and~\ref{thm:main-even} does not directly involve the class of \(F_1(\mathcal Q/\P^1)\) (instead, we use the relative Fano schemes of lines of some other related quadric fibrations), this flipping relation suggests a conceptual explanation for why the class of the relative Fano scheme of lines is related to the class of \(F_1(X)\). This also gives an alternative possible approach to computing \([F_1(X)]\L^2\); however, since \(\L\) is a zero divisor in the Grothendieck ring of varieties \cite{Borisov18}, this strategy does not give a proof of Theorem~\ref{thm:main-odd} or~\ref{thm:main-even}.

We instead prove Theorems~\ref{thm:main-odd} and~\ref{thm:main-even} by studying an explicit birational map, constructed in our previous work \cite{JS24}, from \(F_r(X)\) to a symmetric power of a hyperbolic reduction of the associated pencil:
\[F_r(X) \dashrightarrow \Sym^{r+1} \mathcal Q^{(r)}_\Lambda .\]
The motivic class of the target is known by work of Kuznetsov--Shinder \cite{KuznetsovShinder18} and multiplicativity of the Kapranov zeta function.
Our main geometric input for \(r=0,1\) is to describe the loci on which this birational map fails to be an isomorphism, and to compute the cut-and-paste contributions of these loci.
For higher values of \(r\), the same strategy should in theory work to obtain a formula for the class of \(F_r(X)\) (see Section~\ref{sec:artibrary-r}), but the computations involved may be more complicated.
For higher \(r\), in the \(N\) odd case, effectivity of the coefficients in \cite[Conjecture B]{BBFGHLPRAS} is proven by Pieter Belmans in the appendix to this paper.

There have been several previous works on the conjectures proposed by \cite{BBFGHLPRAS}.
For \(r=0\), the derived category statement is the previously mentioned results of \cite{BondalOrlov,Kuznetsov08}, and the corresponding motivic identity was previously known only after multiplication by \(\L\) \cite{BBFGHLPRAS}.
At the opposite end of the range, for maximal \(r\), the Fano scheme \(F_r(X)\) is either the Jacobian of \(C\) or a reduced finite scheme, and the derived category and motivic statements are understood in these cases.
For second maximal \(r\), if \(N\) is odd then using the isomorphism of \(F_r(X)\) with the moduli space of rank 2 stable vector bundles on \(C\) with fixed odd degree determinant, the motivic conjecture was proven up to a class \(T\) such that \(T(1+\L)=0\) in \cite{BGM2023decomposition}, and the derived category conjecture was proven by \cite{TevelevTorres,Tevelev23}.
(For \(g=3\), Theorem~\ref{thm:main-odd} proves that \(T=0\).)
Some cases of the conjectures were also previously stated in the literature (see \cite[Section 1]{BBFGHLPRAS}).

There have been several previous works on the motivic classes for Fano schemes on other classes of varieties. Namely, \cite{GalkinShinder} gave a relation between the classes of a cubic hypersurface and its Fano scheme of lines. In another direction, for certain hypersurfaces in positive characteristic, \cite{Cheng25} computed the classes of certain Fano schemes of linear spaces using an explicit stratification.

\subsection{Outline}
We begin in Section~\ref{sec:preliminaries} by recalling preliminary results on the Grothendieck ring of varieties, linear subspaces on quadrics, quadric fibrations, and hyperbolic reductions. In Section~\ref{sec:relative-Fano-scheme}, we prove Theorem~\ref{thm:relative-F_r} on the class of the relative Fano scheme and deduce consequences for its algebraic cycles. In Section~\ref{sec:hyperbolic-reductions-pencil}, we specialize to pencils of quadrics and further study hyperbolic reductions and certain distinguished loci in these hyperbolic reductions, and we introduce notation that will be heavily used in later sections. In Section~\ref{sec:r=0}, we prove Proposition~\ref{prop:r=0-case} on the class of \(X\).
Section~\ref{sec:computations} contains several computational results that will be used in the proofs of Theorems~\ref{thm:main-odd} and~\ref{thm:main-even}, including explicit descriptions of the exceptional loci of \(F_1(X)\dashrightarrow\Sym^2\mathcal Q^{(1)}\);
a reader who is willing to accept the results of Section~\ref{sec:computations} may safely proceed to the next section on a first reading.
In Section~\ref{sec:r=1}, we compute the classes of the exceptional loci of \(F_1(X)\dashrightarrow\Sym^2\mathcal Q^{(1)}\), and we combine these results in Section~\ref{sec:summing-together} to prove Theorems~\ref{thm:main-odd} and~\ref{thm:main-even} and deduce several corollaries. In Section~\ref{sec:artibrary-r}, we indicate a general strategy that gives a possible approach to proving \cite[Conjecture B]{BBFGHLPRAS} and its even analogue for arbitrary \(r\).
Finally, in Appendix~\ref{appendix}, Pieter Belmans proves effectivity of the coefficients in \cite[Conjecture B]{BBFGHLPRAS} for arbitrary \(r\).

\subsection*{Acknowledgements}
We thank James Hotchkiss, Kimoi Kemboi, Emanuele Macr\`i, and Saket Shah for helpful conversations.
We are grateful to Pieter Belmans for his comments and for writing the appendix.

\subsection*{Notation}
Throughout \(k\) is a field (frequently of characteristic \(\neq 2\)).
A variety over a scheme \(S\) is a separated \(S\)-scheme of finite presentation; in particular, we do not assume varieties are reduced or irreducible.
For a scheme-theoretic point \(x\in X\), the residue field is denoted by \(\kappa(x)\).
For a quasi-projective variety \(X\) over \(k\) and an integer \(r\geq 0\), \(\Sym^r X\) denotes the \(r\)-th symmetric power of \(X\) over \(k\).
If \(Y_1,\ldots,Y_l\subset\mathbb P^N\) are closed subvarieties, we denote their linear span by \(\langle Y_1,\ldots,Y_l\rangle\).
If $Y\subset  \P^N$ is a closed subvariety and $y\in Y$ is a $k$-point, we write $T_yY\subset \P^N$ for the projective tangent space to $Y$ at $y$.
More precisely, if $Y=\{f_1=\cdots =f_r=0\}$ and $\P^N=\P^N_{[x_0:\cdots:x_N]}$, then
\[
T_yY=\left\{\sum_{j=0}^N\frac{\partial f_1}{\partial x_j}(y) x_j= \cdots=\sum_{j=0}^N\frac{\partial f_r}{\partial x_j}(y) x_j=0 \right\}\subset \P^N.
\]
For a finite-dimensional vector space \(V\) over \(k\), 
\(\mathbb P(V)\coloneqq \Proj\Sym^\bullet (V^\vee)\) parametrizes the one-dimensional subspaces of \(V\).
For a vector bundle $\mathcal E$ over a scheme $S$, we denote $\P_S(\mathcal E)\coloneqq \underline{\Proj}_{S}\Sym^\bullet(\mathcal E^\vee)$.
By convention we take \(\mathbb P^{-m}=\emptyset\) for \(m\in\mathbb Z_{>0}\), and \(\Gr(a,b)=\emptyset\) for \(a>b\).

\subsection*{AI disclosure}
OpenAI ChatGPT (GPT-5.6 Sol and GPT-6 Astra) was used to aid in computations (namely, to find closed formulas for the recursive formulas in the proofs of Theorems~\ref{thm:main-even} and~\ref{thm:relative-F_r}, and for Lemma~\ref{lem:telescoping-sum-computation} which is used in the computations in the proof of Theorem~\ref{thm:relative-F_r}), to find Example~\ref{exmp:smooth-Y-example}, for literature searches, and for proofreading. The paper was written by its authors, who take full responsibility for its correctness.

\section{Preliminaries}\label{sec:preliminaries}

\subsection{The Grothendieck ring of varieties}\label{sec:prelim-K0}

\begin{defn}\label{defn:K0-Ktilde0}
Let $k$ be a field.
\defi{The Grothendieck ring $K_0(\mathrm{Var}/k)$ of varieties over $k$} is defined to be the quotient of the free abelian group on the set of isomorphism classes of varieties over $k$, by relations of the form
\[
[X]=[Y]+[X\setminus Y],
\]
where $Y$ is a closed subvariety of a variety $X$, and the product is defined by
\[
[X]\cdot [Z]= [X\times Z]
\]
for varieties $X,Z$ over $k$.
We define $\widetilde{K}_0(\mathrm{Var}/k)$ to be the quotient ring of $K_0(\mathrm{Var}/k)$ by the ideal generated by the relations $[X]-[W]$, where there is a radicial surjective morphism $X\to W$ of varieties over $k$.
\end{defn}

If \(k\) has characteristic zero, then the quotient map \(K_0(\mathrm{Var}/k) \to \widetilde{K}_0(\mathrm{Var}/k)\) is an isomorphism \cite[Proposition 7.25]{mustata-notes}.

By definition, $[\emptyset]=0\in K_0(\mathrm{Var}/k)$, and for every variety $X$ over $k$, $[X]=[X_{\text{red}}]\in K_0(\mathrm{Var}/k)$. 
In the definitions of $K_0(\mathrm{Var}/k)$ and $\widetilde{K}_0(\mathrm{Var}/k)$, one can instead consider the free group on the set of isomorphism classes of quasiprojective varieties over $k$ in order to get the same groups \cite[Proposition 7.27]{mustata-notes}.

The Grothendieck ring of varieties admits many realization homomorphisms, i.e., homomorphisms from \(K_0(\mathrm{Var}/k)\) to other rings. See, e.g., \cite[Section 2.2]{GalkinShinder} and \cite[Section 2.3]{BBFGHLPRAS} for some examples. We will use some of these realization homomorphisms to draw corollaries from the identities we prove in the Grothendieck ring of varieties.

We collect below useful lemmas on $K_0(\mathrm{Var}/k)$ and $\widetilde{K}_0(\mathrm{Var}/k)$.

\begin{defn}\label{defn:Z-bundle}
Over a field $k$, let $f\colon X\to Y$ be a morphism of varieties and $Z$ be a variety.
We say that $f$ is a \defi{$Z$-bundle} if for every scheme-theoretic point $y\in Y$, $X\times_Y \Spec(\kappa(y))$ is isomorphic to $ Z\times_{\Spec(k)}\Spec(\kappa(y))$ over $\Spec(\kappa(y))$.
\end{defn}

\begin{lem}\label{lem:fibration-general}
Over a field $k$, let $f\colon X\to Y$ be a $Z$-bundle.
Then $[X]=[Y][Z]\in K_0(\mathrm{Var}/k)$.
\end{lem}
\begin{proof}
By \cite[Proposition 7.4]{mustata-notes}, it is enough to show that $f\colon X\to Y$ is a piecewise trivial fibration with fiber $Z$, i.e., there exists a decomposition $Y=Y_1\sqcup \dots \sqcup Y_r$ with all $Y_i$ locally closed in $Y$ such that each $f^{-1}(Y_i)$ is isomorphic to $Y_i\times Z$ over $Y_i$.
This follows from \cite[Th\'eor\`eme 4.2.3]{Sebag04}.
(Note that there are some typos in \cite[Th\'eor\`eme 4.2.3]{Sebag04}: $\pi$ is $h$, and the condition $x\in B$ in the theorem statement should be replaced by the condition $x\in A$.)
\end{proof}

\begin{lem}\label{lem:bijection-same-class}
Over a field $k$, let $f\colon X\to Y$ be a morphism of varieties such that for every algebraically closed field extension $K/k$, the induced map $f_{K}\colon X(K)\to Y(K)$ is bijective. Then $[X]=[Y]\in \widetilde{K}_0(\mathrm{Var}/k)$.
\end{lem}
\begin{proof}
By \cite[Remark A.22]{mustata-notes}, $f$ is radicial and surjective, and hence, by the definition of $\widetilde{K}_0(\mathrm{Var}/k)$, the assertion follows.
\end{proof}

For a quasi-projective variety $X$ over a perfect field $k$, the Kapranov zeta function of $X$ is defined as
\[
Z_{\text{mot}}(X,t)\coloneqq \sum_{n\geq 0}[\Sym^n X]t^n \in 1+ t\cdot \widetilde{K}_0(\mathrm{Var}/k)[[t]].
\]
%One can extend the definition of $Z_{\text{mot}}(X,t)$ to a variety $X$ over $k$ that is not necessarily quasiprojective.

\begin{lem}[{\cite[Proposition 7.28 and Proposition 7.32]{mustata-notes}}]\label{lem:K_0-sym-formulas}
Let $k$ be a perfect field.
\begin{enumerate}
\item\label{item:formula-sym} The map 
\[K_0(\mathrm{Var}/k)\to 1+t\cdot \widetilde{K}_0(\mathrm{Var}/k)[[t]],\qquad [X]\mapsto Z_{\text{mot}}(X,t)\]
is a group homomorphism, and it factors through $\widetilde{K}_0(\mathrm{Var}/k)$.
In particular, for every quasiprojective variety $X$, every closed subvariety $Y\subset X$ with complement $U\coloneq X\setminus Y$, and every $n\geq 0$, we have
$[\Sym^n X]=\sum_{i+j=n}[\Sym^i Y][\Sym^j U]\in \widetilde{K}_0(\mathrm{Var}/k)$.
\item\label{item:formula-sym-times-L}  (Totaro)  Let $X$ be a quasiprojective variety over $k$. 
Then $Z_{\text{mot}}(X\times \A^n_k,t)=Z_{\text{mot}}(X,\L^n t)$.
In particular, for every $n\geq 0$, we have $[\Sym^n(X\times \A^1_k)]=[\Sym^n X]\L^n\in \widetilde{K}_0(\mathrm{Var}/k)$.
\end{enumerate}
\end{lem}

We give explicit formulas for Grassmannians, smooth quadrics, $\Sym^2\P^n$, and $\Sym^2(\P^1\times \P^n)$ in terms of $\L$.

\begin{lem}\label{lem:formulas-Gr-quadric-Sym-Pn}
    Let \(k\) be a field, and let \(n\) be an integer.
    \begin{enumerate}
        \item\label{item:class-of-Gr} Let \(1\leq r\leq n\). The following equality holds in \(K_0(\mathrm{Var}/k)\):
        \[[\Gr(r,n)]= \binom{n}{r}_{\L} \coloneqq \prod_{i=1}^r \frac{(1-\L^{n-r+i})}{(1-\L^i)}.\]
        \item\label{item:class-of-quadric} Assume \(n\geq 1\). If \(k\) has characteristic \(\neq 2\) and \(Q^n\subset\P^{n+1}\) is a smooth quadric containing a \(\lfloor\frac{n}{2}\rfloor\)-dimensional linear subspace defined over \(k\), then in \(K_0(\mathrm{Var}/k)\):
        \[[Q^n] = \begin{cases}
        [\mathbb P^n]+\L^{n/2} & \text{if \(n\) is even}, \\
        [\mathbb P^n] & \text{if \(n\) is odd}.
        \end{cases}\]
        \item\label{item:class-of-Sym-Pn} Assume \(n\geq 0\). The following equalities hold in \(\widetilde{K}_0(\mathrm{Var}/k)\):
            \begin{equation*}
            \begin{split}
                [\Sym^2 \P^n] &= \frac{(1-\L^{n+2})(1-\L^{n+1})}{(1-\L)(1-\L^2)} , \\
                [\Sym^2(\P^1\times\P^n)] &= \frac{(1-\L^{n+1})(1+\L+2\L^2-2\L^{n+2}-\L^{n+3}-\L^{n+4})}{(1-\L)(1-\L^2)} .
            \end{split}
        \end{equation*}
    \end{enumerate}
\end{lem}
In~\eqref{item:class-of-Gr}, the Gaussian binomial coefficient \(\binom{n}{r}_{\L}\) a priori appears to be a rational function, but in fact it is a polynomial in \(\L\) with nonnegative integer coefficients \cite[Section 1.7]{Stanley12}.

\begin{proof}
    For~\eqref{item:class-of-Gr}, induct on \(r\). If \(r=1\) then \([\Gr(1,n)]=[\P^{n-1}]=\sum_{i=0}^{n-1}\L^i\). Now assume \(r\geq 2\) and the formula holds for \(\Gr(r-1,n')\) for any \(r-1\leq n'\). To show the formula for \(\Gr(r,n)\), we induct on \(n\). The base case is \([\Gr(r,r)]=1\), and the induction step follows from the relation
    \([\Gr(r,n)] = [\Gr(r,n-1)] + \L^{n-r} [ \Gr(r-1,n-1) ]\)
    \cite[Proposition 2.1]{Martin16} (this result is stated over \(\mathbb C\) but the argument works over any field).
    
    \eqref{item:class-of-quadric} is \cite[Example 2.8]{KuznetsovShinder18}.
    For~\eqref{item:class-of-Sym-Pn}, induct on \(n\), using Lemma~\ref{lem:K_0-sym-formulas} for the relations
    \([\Sym^2 \P^n] = \L^{2n} + [\P^{n-1}]\L^n + [\Sym^2 \P^{n-1}]\) and \([\Sym^2(\P^1\times\P^n)] = [\P^2]\L^{2n} + [\P^1]^2[\P^{n-1}]\L^n + [\Sym^2(\P^1\times\P^{n-1})]\) for \(n\geq 1\).
\end{proof}

\subsection{Linear subspaces on quadrics and quadric fibrations}\label{sec:prelim-F_r}

Throughout this section, we work over a field \(k\) of characteristic \(\neq 2\). We recall some definitions and results about quadric fibrations and about the linear subspaces contained in quadrics and in smooth complete intersections of two quadrics.

Recall that if \(Q^n\subset\P^{n+1}\) is a smooth quadric of dimension \(n\), then \(Q^n_\kbar\) contains \(r\)-dimensional linear subspaces if and only if \(0\leq r\leq \lfloor\frac{n}{2}\rfloor\). We denote by \(F_r(Q^n)\) the Fano scheme of \(r\)-dimensional linear subspaces on \(Q^n\).
If \(r\neq\frac{n}{2}\), then \(F_r(Q^n)\) is smooth and geometrically irreducible, and if \(k=\kbar\) then we denote it by \(\OG(r+1,n+2)\). If \(n=2g\) is even and \(r=g\), then \(F_g(Q^{2g}_{\kbar})\) is disconnected with two irreducible components, each of which is isomorphic to \(F_{g-1}(Q^{2g-1}_{\kbar})\), and we let \(\OG(g+1,2g+2)\) denote one of these components (see, e.g., \cite[Theorems 22.13 and 22.14]{Harris92}).
Since \(\OG(r+1,n+2)\) is a projective homogeneous variety under \(SO(n+2)\), it is a rational variety (see \cite[\S 14.12]{Borel91}).

Let \(S\) be a variety over \(k\), and let \(\pi\colon\mathcal Q\to S\) be a fibration in \(n\)-dimensional quadrics; that is, \(\pi\) is flat, there is a rank \(n+2\) vector bundle \(\mathcal E\) on \(S\), and \(\mathcal Q\hookrightarrow\mathbb P_S(\mathcal E)\) embeds as a divisor that has relative degree 2 over \(S\).
Then \(\pi\) is defined by a quadratic form \(q\colon \Sym^2\mathcal E\to\mathcal L^\vee\) with values in a line bundle \(\mathcal L^\vee\), or, equivalently, a self-dual morphism \(q\colon\mathcal E\otimes\mathcal L\to\mathcal E^\vee\).
We say that \(\pi\) has \defi{simple degeneration} if the fibers of \(\pi\) have corank at most 1.
If \(S\) and the generic fiber of \(\pi\) are smooth, then the degeneracy locus of \(\pi\) is smooth if and only if \(\mathcal Q\) is smooth and \(\pi\) has simple degeneration \cite[Proposition 1.2.5]{ABB14}.

For \(0 \leq r \leq \lfloor\frac{n+1}{2}\rfloor\), let \(F_r(\mathcal Q/S)\) denote the relative Fano scheme of \(r\)-planes in the fibers of \(\pi\); there is a natural map \(F_r(\mathcal Q/S)\to S\). If \(\mathcal Q\) and \(S\) are smooth and \(\pi\) has simple degeneration, then \(F_r(\mathcal Q/S)\) is smooth by \cite[Proposition 2.1 and Remark 2.2]{Kuznetsov14}.
If \(n=2g\) is even, \(S\) is smooth, and \(\pi\) has simple degeneration along a smooth divisor, then the Stein factorization of \(F_g(\mathcal Q/S)\to S\) is a double cover \(\widetilde{S}\to S\) that is isomorphic to the discriminant double cover \(\underline{\Spec}_S(\mathcal O_S\oplus(\det(\mathcal E)\otimes\mathcal L^{g+1}))\to S\) \cite[Proposition B.5]{ABB14}.

Now let \(X\subset\P^N\) be a smooth complete intersection of two quadrics over \(k\). Then \(X\) is the base locus of a pencil \(\phi\colon \mathcal Q\to\P^1\) of quadric \((N-1)\)-folds, and \(\phi\) has simple degeneration along a smooth divisor by \cite[Proposition 2.1]{Reid-thesis}. If \(N=2g+1\) is odd, then the discriminant double cover is a genus \(g\) hyperelliptic curve \(C\to\mathbb P^1\) branched over the $2g+2$ points parametrizing the singular members of the pencil.
If \(N=2g\) is even, then the associated ``stacky curve" is the root stack \(\widehat{C}\to\P^1\) with \(\mathbb Z/2\mathbb Z\)-stabilizers at the \(2g+1\) points in \(\P^1\) parametrizing the singular fibers (see \cite[Example 2.2]{Kuznetsov08}).
We record the following result on the Fano schemes of linear subspaces contained in \(X\).

\begin{lem}[{\cite[Theorems 2.6, 3.8, and 4.8]{Reid-thesis}, \cite[Lemmas A.3 and A.4]{CT-intersection-quadrics}}]\label{lem:F_r(X)}
    Let \(X\subset\P^N\) be a smooth complete intersection of two quadrics. The following hold:
    \begin{enumerate}
        \item\label{item:lem-F_r(X)-dim} \(F_r(X)\neq\emptyset\) if and only if \(0\leq r\leq\lfloor\frac{N}{2}\rfloor-1\). If \(F_r(X)\) is nonempty, then it is smooth and projective of dimension \((r+1)(N-2r-2)\).
        \item \(F_r(X)\) is nonempty and geometrically connected if and only if \(0\leq r<\frac{N}{2}-1\).
        \item\label{item:lem-F_r(X)-maximal-even} If \(\dim F_r(X)=0\), then \(N=2g\), \(r=g-1\), and \(F_{g-1}(X)\) is a reduced finite scheme of length \(2^{2g}\).
        \item\label{item:lem-F_r(X)-maximal-odd} If \(N=2g+1\) and \(r=g-1\), then \(F_{g-1}(X)\) is a torsor under the Jacobian of \(C\). 
    \end{enumerate}
\end{lem}

\subsection{Hyperbolic reductions of quadric fibrations}\label{sec:prelim-hyperbolic-red}

In this section, following \cite[Section 2]{KuznetsovShinder18}, we recall definitions and properties of quadric fibrations and their hyperbolic reductions. Throughout this section, we work over a field \(k\) of characteristic \(\neq 2\).

Let \(S\) be a variety over \(k\), and let \(\pi\colon\mathcal Q \subset \mathbb P_S(\mathcal E)\to S\) be a fibration in \(n\)-dimensional quadrics as defined in Section~\ref{sec:prelim-F_r}. Let \(r\geq 0\) be an integer.
An \defi{\(r\)-section} of \(\pi\) is a section of the morphism \(F_r(\mathcal Q/S)\to S\).
Equivalently, an \(r\)-section is defined by a rank \(r+1\) vector subbundle \(\mathcal F_{r+1}\hookrightarrow\mathcal E\) that is isotropic with respect to the quadratic form defining \(\mathcal Q\), meaning that \(\mathcal F_{r+1}\) is contained in the subsheaf \(\mathcal F_{r+1}^\perp\coloneqq\Ker(\mathcal E\xrightarrow{q}\mathcal E^\vee\otimes\mathcal L^\vee\twoheadrightarrow\mathcal F_{r+1}^\vee\otimes\mathcal L^\vee)\) of \(\mathcal E\).
An \(r\)-section \(s\) is \defi{nondegenerate} if for every geometric point \(t\in S\), the \(r\)-plane \(s(t)\) does not intersect the singular locus of \(\mathcal Q_t\); this is equivalent to the condition that the composition \(\mathcal E\xrightarrow{q} \mathcal E^\vee \otimes\mathcal L^\vee \twoheadrightarrow \mathcal F_{r+1}^\vee \otimes\mathcal L^\vee\) is surjective.
If \(s\) is not nondegenerate, we say that it is \defi{degenerate}.

\begin{defn}\label{defn:hyperbolic-red}
    Let \(s\) be a nondegenerate \(r\)-section of \(\pi\) corresponding to \(\mathcal F_{r+1}\hookrightarrow\mathcal E\).
    The induced quadratic form \(q^{(r)}_s\colon\Sym^2(\mathcal F_{r+1}^\perp/\mathcal F_{r+1})\to\mathcal L^{\vee}\) is the \defi{hyperbolic reduction} of \(q\colon\Sym^2 \mathcal E\to\mathcal L^\vee\) with respect to the isotropic subbundle \(\mathcal F_{r+1}\). The family \(\mathcal Q^{(r)}_s\subset\mathbb P_S(\mathcal F_{r+1}^\perp/\mathcal F_{r+1})\) of quadrics defined by \(q^{(r)}_s\) is the \defi{hyperbolic reduction} of \(\pi\) with respect to \(s\), and we denote the projection to \(S\) by \(\pi^{(r)}_s\colon\mathcal Q^{(r)}_s\to S\).
\end{defn}

If \(s\) is nondegenerate, \(r\leq\lfloor\frac{n}{2}\rfloor-1\), and the fibers of $\pi$ have corank at most $n-2r-1$, then \(\pi^{(r)}_s\colon\mathcal Q^{(r)}_s\to S\) is a flat family of \((n-2r-2)\)-dimensional quadrics (see~\eqref{item:KS-hyperbolic-red-degeneracy} below).

\begin{prop}[{\cite[Remark 2.3, Lemma 2.4, Proposition 2.5, Corollary 2.7, Lemma 2.10]{KuznetsovShinder18}}]\label{prop:KS-hyperbolic-reduction}
    Let \(\pi\colon\mathcal Q\to S\) be a flat family of \(n\)-dimensional quadrics as above. The following hold:
    \begin{enumerate}
        \item\label{item:KS-smooth-nondegenerate-section} If \(\mathcal Q\) is smooth, then every \(r\)-section of \(\pi\) is nondegenerate.
        \item\label{item:KS-hyperbolic-red-degeneracy} If \(s\) is a nondegenerate \(r\)-section of \(\pi\), then hyperbolic reduction along \(s\) preserves the coranks of fibers.
        If moreover \(r\leq\lfloor\frac{n}{2}\rfloor-1\) and the fibers of \(\pi\) have corank at most \(n-2r-1\) (e.g., this holds if \(\pi\) has simple degeneration), then \(\pi^{(r)}_s\colon\mathcal Q^{(r)}_s\to S\) is a flat family of \((n-2r-2)\)-dimensional quadrics.
        \item\label{item:KS-hyperbolic-red-blow-up-diagram} Let \(s\) be a nondegenerate \(r\)-section of \(\pi\) corresponding to \(\mathcal F_{r+1}\hookrightarrow \mathcal E\), and let \(\widetilde{\mathcal{Q}}\) be the blow-up of \(\mathcal Q\) in \(\mathbb P_S(\mathcal F_{r+1})\). Then we have the following diagram over \(S\)
        \[
        \begin{tikzcd}
        & \widetilde{\mathcal{Q}} \arrow[ld, "\bl_{\mathbb P_{S}(\mathcal F_{r+1})}"'] \arrow[rd, "h"] & \\
\mathcal{Q} \arrow[rr, "\pi_{\mathbb P_{S}(\mathcal F_{r+1})}|_{\mathcal Q}", dashed, swap]& & \mathbb P_{S}(\mathcal E/\mathcal F_{r+1}) & \mathcal Q^{(r)}_s \arrow[l,hook']
        \end{tikzcd}
        \]
        where \(h\) is a (Zariski locally trivial) \(\mathbb P^{r+1}\)-bundle over \(\mathcal Q^{(r)}_s\) and is a (Zariski locally trivial) \(\mathbb P^r\)-bundle over \(\mathbb P_{S}(\mathcal E/\mathcal F_{r+1})\setminus\mathcal Q^{(r)}_s\).
        \item\label{item:KS-hyperbolic-reduction-formula} If \(s\) is a nondegenerate \(r\)-section of \(\pi\), then the following equality holds in \(K_0(\mathrm{Var}/k)\): \[[\mathcal Q] = [S][\P^r](1+\L^{n-r})+[\mathcal Q^{(r)}_s]\L^{r+1}.\]
        Furthermore, the proof of \cite[Corollary 2.7]{KuznetsovShinder18} shows this equality holds in the Grothendieck ring \(K_0(\mathrm{Var}/S)\) of varieties over \(S\) \cite[Definition 1.2.2 and Proposition 2.2.1]{MotivicIntegrationBook}.
        \item\label{item:KS-hyperbolic-reduction-indep} If \(\pi\) has a nondegenerate \(r\)-section, then the class of \(\mathcal Q^{(r)}_s\) in \(K_0(\mathrm{Var}/k)\) is independent of the choice of a nondegenerate \(r\)-section \(s\).
    \end{enumerate}
\end{prop}

\begin{cor}
    If \(\mathcal Q\) and \(S\) are smooth over \(k\) and \(\pi\) has simple degeneration, then for any \(r\)-section \(s\), the hyperbolic reduction \(\mathcal Q^{(r)}_s\) is smooth over \(k\).
\end{cor}

\begin{proof}
    The degeneracy locus of \(\pi\) is smooth by \cite[Proposition 1.2.5]{ABB14}. If \(r\leq\lfloor\frac{n}{2}\rfloor-1\), then by Proposition~\ref{prop:KS-hyperbolic-reduction}\eqref{item:KS-hyperbolic-red-degeneracy}, \(\pi^{(r)}_s\colon\mathcal Q^{(r)}_s\to S\) has simple degeneration and smooth generic fiber. So \cite[Proposition 1.2.5]{ABB14} implies \(\mathcal Q^{(r)}_s\) is smooth.
    Now assume \(r=\lfloor\frac{n}{2}\rfloor\). If \(n=2g-1\) is odd then \(r=g-1\), and Lemma~\ref{lem:hyperbolic-reduction-properties-not-pencil}\eqref{item:hyperbolic-reduction-finite-set-maximal} below implies \(\mathcal Q^{(g-1)}_s\) is smooth. Finally, if \(n=2g\) is even, then \(r=g\) and the hyperbolic reduction is empty because \(\mathcal F_{r+1}^\perp/\mathcal F_{r+1}\) has rank \(0\) and hence its projectivization is empty.
\end{proof}

\begin{cor}\label{cor:formula-class-of-hyperbolic-reduction}
    Let \(1\leq r\) and \(0\leq i\leq r-1\) be integers such that \(i\leq\lfloor\frac{n}{2}\rfloor-1\) and the fibers of \(\pi\) have corank at most \(n-2i-1\), and let \(s_r\) (resp. \(s_i\)) be a nondegenerate \(r\)-section (resp. nondegenerate \(i\)-section) of \(\pi\). Then the following equality holds in \(K_0(\mathrm{Var}/k)\) and in \(K_0(\mathrm{Var}/S)\):
    \[[\mathcal Q^{(i)}_{s_i}] = [S][\P^{r-i-1}](1+\L^{n-r-i-1})+[\mathcal Q^{(r)}_{s_r}]\L^{r-i}.\]
\end{cor}

\begin{proof}
    Let \(\mathcal F_{r+1}\subset\mathcal E\) be the subbundle corresponding to \(s_r\), and let \(S=S_1\sqcup\cdots\sqcup S_m\) be a finite stratification by locally closed subvarieties such that \(\mathcal F_{r+1}|_{S_j}\cong\O_{S_j}^{\oplus r+1}\). By \cite[Lemma 1.3.3]{MotivicIntegrationBook} (see also \cite[Proposition 7.1]{mustata-notes}), it suffices to prove the equality after restriction over each \(S_j\). Therefore, we may assume \(S=S_j\) and \(\mathcal F_{r+1}\cong\mathcal O_S^{\oplus r+1}\). Let \(\mathcal F'_{i+1} \subset \mathcal F_{r+1}\cong\mathcal O_S^{\oplus r+1}\) be a subbundle, and let \(s_i'\) be the corresponding \(i\)-section of \(\pi\colon\mathcal Q\to S\).
    Then \(\pi^{(i)}_{s_i'}\colon \mathcal Q^{(i)}_{s_i'}\to S\) is a flat family of quadric \((n-2i-2)\)-folds by Proposition~\ref{prop:KS-hyperbolic-reduction}\eqref{item:KS-hyperbolic-red-degeneracy}, and
    \(\mathcal Q^{(r)}_{s_r}\) is the hyperbolic reduction of \(\pi^{(i)}_{s_i'}\) with respect to the \((r-i-1)\)-section given by \(\mathcal F_{r+1}/\mathcal F_{i+1}'\subset\mathcal F_{i+1}'^\perp/\mathcal F_{i+1}'\). We claim that this is a nondegenerate \((r-i-1)\)-section of \(\pi^{(i)}_{s_i'}\). Indeed, by the assumption that \(s_r\) is nondegenerate, the composition \(\mathcal E\to\mathcal E^\vee\otimes\mathcal L^\vee\twoheadrightarrow\mathcal F_{r+1}^\vee\otimes\mathcal L^\vee\) is surjective.
    The preimage of \((\mathcal F_{r+1}/\mathcal F'_{i+1})^\vee\otimes\mathcal L^\vee = \Ker(\mathcal F_{r+1}^\vee \to (\mathcal F_{i+1}')^\vee)\otimes\mathcal L^\vee \subset \mathcal F_{r+1}^\vee\otimes\mathcal L^\vee\) under this map is \((\mathcal F_{i+1}')^\perp\). So the induced map
    \({\mathcal F'}_{i+1}^\perp/\mathcal F'_{i+1} \to ({\mathcal F'}_{i+1}^\perp/\mathcal F'_{i+1})^\vee\otimes\mathcal L^\vee \twoheadrightarrow (\mathcal F_{r+1}/\mathcal F'_{i+1})^\vee\otimes\mathcal L^\vee\)
    is a surjection, i.e., this \((r-i-1)\)-section is nondegenerate. Therefore
    \([\mathcal Q^{(i)}_{s_i}] = [\mathcal Q^{(i)}_{s_i'}] = [S][\P^{r-i-1}](1+\L^{n-2i-2-(r-i-1)})+[\mathcal Q^{(r)}_{s_r}]\L^{r-i}\)
    by Proposition~\ref{prop:KS-hyperbolic-reduction}\eqref{item:KS-hyperbolic-reduction-formula} and~\eqref{item:KS-hyperbolic-reduction-indep}.
\end{proof}

Finally, we recall the following geometric interpretations of hyperbolic reductions (cf. \cite[Lemma 3.1]{JS24} for the case when \(\pi\) is a pencil).
\begin{lem}\label{lem:hyperbolic-reduction-properties-not-pencil}
Let \(s\) be a nondegenerate \(r\)-section of \(\pi\colon\mathcal Q\to S\).
\begin{enumerate}
    \item\label{item:hyperbolic-reduction-not-pencil-embedding}
    For \(0\leq e\leq \lfloor\frac{n+1}{2}\rfloor-r-1\), the \(\P^{r+1}\)-bundle \(E^{(r)}_s \coloneqq h^{-1}(\mathcal Q^{(r)}_s) \to \mathcal Q^{(r)}_s\) in Proposition~\ref{prop:KS-hyperbolic-reduction}\eqref{item:KS-hyperbolic-red-blow-up-diagram} induces an embedding \(F_e(\mathcal Q^{(r)}_s/S)\hookrightarrow F_{r+e+1}(\mathcal Q/S)\), whose image is the locus \(F_{r+e+1}(\mathcal Q/S)_s\) of isotropic \((r+e+1)\)-planes of \(\pi\) containing the \(r\)-planes given by \(s\). 

    In particular, for \(e=0\), after identifying \(\mathcal Q^{(r)}_s\) with its image \(F_{r+1}(\mathcal Q/S)_s\), the morphism \(h^{-1}(\mathcal Q^{(r)}_s) \to \mathcal Q^{(r)}_s\) gives the universal family.
    \item\label{item:hyperbolic-reduction-even-maximal} \cite[Lemma 2.12]{KuznetsovShinder18} If \(n=2g\) is even, \(r=g-1\) for some \(g\geq 1\), and \(\pi\) has simple degeneration, then \(\mathcal Q^{(g-1)}_s \to S\) is \(k\)-isomorphic to the discriminant double cover \(\widetilde{S}\coloneqq\underline{\Spec}_S(\mathcal O_S\oplus(\det(\mathcal E)\otimes\mathcal L^{g+1}))\to S\).
    \item\label{item:hyperbolic-reduction-finite-set-maximal} If \(n=2g-1\) is odd, \(r=g-1\), and \(\pi\) has simple degeneration, then \(\mathcal Q^{(g-1)}_s \to S\) is \(k\)-isomorphic to the inclusion of the degeneracy locus of \(\pi\).
\end{enumerate}
\end{lem}

\begin{proof}
    For~\eqref{item:hyperbolic-reduction-not-pencil-embedding}, fix \(t\in S\). Choose coordinates on \(\P_S(\mathcal E)_t\cong\P^{n+1}_{\kappa(t)}\) so that \(s(t)\) is given by \(\{x_{r+1}=\cdots=x_{n+1}=0\}\). Then \(\mathcal Q_t\) is defined by \(x_0l_0+x_1l_1+\cdots+x_rl_r + q\) where \(l_i\in \kappa(t)[x_{r+1},\ldots,x_{n+1}]\) are linear forms and \(q\in \kappa(t)[x_{r+1},\ldots,x_{n+1}]\) is quadratic. Then the fiber of \(\mathcal Q^{(r)}_s\) over \(t\) is defined by \(\{l_0=\cdots=l_r=q=0\} \subset \P^{n-r}_{\kappa(t)}\) with coordinates \([x_{r+1}:\cdots:x_{n+1}]\).
    Let \(n\) be an \(e\)-plane contained in the fiber of \(\mathcal Q^{(r)}_s\) over \(t\in S\). Then \(n=\{f_1=\cdots=f_{n-r-e}=0\} \subset \P^{n-r}_{\kappa(t)}\) for some linear forms \(f_i\) such that the containment \((l_0,\ldots,l_r,q) \subset (f_1,\ldots,f_{n-r-e})\) of ideals holds.
    Then \(\bl_{s(t)}(h^{-1}(n)) = \{f_1=\cdots=f_{n-r-e}=0\} \subset \P^{n+1}_{\kappa(t)}\) is an \((r+e+1)\)-plane in \(\P^{n+1}_{\kappa(t)}\) containing \(s(t)\), and it is contained in \(\mathcal Q_t\).
    Conversely, let \(L\) be an \((r+e+1)\)-plane in \(\mathcal Q_t\) containing \(s(t)\).
    Then \(L=\{f_1=\cdots=f_{n-r-e}=0\}\) for some linear forms \(f_i\), and the containment of ideals \((f_1,\ldots,f_{n-r-e})\subset(x_{r+1},\ldots,x_{n+1})\)  holds since \(s(t)\subset L\).
    Then \(x_0l_0+x_1l_1+\cdots+x_rl_r + q \in (f_1,\ldots,f_{n-r-e})\) implies \((l_0,\ldots,l_r,q) \subset (f_1,\ldots,f_{n-r-e})\), so the $e$-plane \(\pi_{s(t)}(L)=\{f_1=\cdots=f_{n-r-e}=0\} \subset \P^{n-r}_{\kappa(t)}\) is contained in \((\mathcal Q^{(r)}_s)_t\).
    
    For~\eqref{item:hyperbolic-reduction-finite-set-maximal}, since \(\pi\) is assumed to have simple degeneration, a fiber \(\mathcal Q_t\) of \(\pi\) contains \(g\)-planes if and only if \(\mathcal Q_t\) is singular of corank 1. Furthermore, if this holds, then for each \((g-1)\)-plane on \(\mathcal Q_t\) not containing the vertex, there is a unique \(g\)-plane on \(\mathcal Q_t\) containing this \((g-1)\)-plane. So \eqref{item:hyperbolic-reduction-finite-set-maximal} follows from \cite[Lemma 2.4]{KuznetsovShinder18}, part~\eqref{item:hyperbolic-reduction-not-pencil-embedding}, and Proposition~\ref{prop:KS-hyperbolic-reduction}\eqref{item:KS-hyperbolic-red-degeneracy}.
\end{proof}

Now we define hyperbolic reductions with respect to degenerate sections.

\begin{defn}\label{defn:hyperbolic-red-degen}
    Let \(s\) be a degenerate \(r\)-section of \(\pi\) corresponding to \(\mathcal F_{r+1}\hookrightarrow\mathcal E\). The \defi{hyperbolic reduction} of \(\pi\) with respect to \(s\) is the closed subscheme \(F_{r+1}(\mathcal Q/S)_s\subset F_{r+1}(\mathcal Q/S)\) parametrizing isotropic \((r+1)\)-planes of \(\pi\) containing the \(r\)-planes given by \(s\), with projection \(\pi^{(r)}_s\colon\mathcal Q^{(r)}_s\to S\). There is a natural embedding \(\mathcal Q^{(r)}_s \hookrightarrow \mathbb P_S(\mathcal E/\mathcal F_{r+1})\) over \(S\).
\end{defn}
Then for every \(t\in S\), the fiber of \(\pi^{(r)}_s\) over \(t\) satisfies \((\mathcal Q^{(r)}_s)_t\cong F_{r+1}(\mathcal Q_t)_{s(t)}\) as \(\kappa(t)\)-schemes.
Furthermore, by Lemma~\ref{lem:hyperbolic-reduction-properties-not-pencil}\eqref{item:hyperbolic-reduction-not-pencil-embedding}, Definition~\ref{defn:hyperbolic-red-degen} agrees with Definition~\ref{defn:hyperbolic-red} over the locus where \(s\) is nondegenerate.

If \(s\) is a degenerate section, then the fibers of \(\pi^{(r)}_s\colon\mathcal Q^{(r)}_s\to S\) may have different dimensions, even if \(r\leq\lfloor\frac{n}{2}\rfloor-1\), as the following example indicates.

\begin{exmp}\label{exmp:hyperbolic-red-corank1-quadric}
	Let \(Q\subset\mathbb P^{n+1}\) be a corank 1 quadric \(n\)-fold with vertex \(v\), let \(r\leq \lfloor\frac{n-1}{2}\rfloor\), and let \(\Lambda\subset Q\) be an \(r\)-plane that contains \(v\). Then the hyperbolic reduction of \(Q\) with respect to \(\Lambda\) is a smooth quadric \((n-2r-1)\)-fold. In particular, if \(r=0\) and \(\Lambda=v\), then \(Q^{(0)}_v\) is the base of the cone.
	
	To see this, let \(W\subset V\) be the vector spaces of dimension \(r+1\) and \(n+2\), respectively, so that \(\Lambda=\mathbb P(W)\) and \(\mathbb P^{n+1}=\mathbb P(V)\), and let \(q\) be the quadratic form on \(V\) defining \(Q\). By assumption that \(v\in \Lambda\), the subspace \(W\) contains the radical \(V^\perp\) of \(q\).
	On the quotient \(V/V^\perp\), the induced quadratic form defines a smooth quadric \((n-1)\)-fold \(\overline{Q}\subset\mathbb P(V/V^\perp)\cong\mathbb P^n\), and \(\overline{\Lambda}\coloneqq \mathbb P(W/V^\perp)\) is an \((r-1)\)-plane on \(\overline{Q}\). Since \(Q^{(r)}_\Lambda = F_{r+1}(Q)_\Lambda \cong F_r(\overline{Q})_{\overline{\Lambda}}\), we have that \(Q^{(r)}_\Lambda\) is a smooth quadric \((n-2r-1)\)-fold by Proposition~\ref{prop:KS-hyperbolic-reduction}\eqref{item:KS-hyperbolic-red-degeneracy}.
\end{exmp}

\section{The class of the relative Fano scheme of lines}\label{sec:relative-Fano-scheme}

In this section, we prove Theorem~\ref{thm:relative-F_r} on the class of the relative Fano scheme \(F_r(\mathcal Q/S)\) of a quadric fibration over a curve in the Grothendieck ring of varieties over \(k\). First, we record the following formula for the class of the orthogonal Grassmannian, for which we could not find an explicit reference in the literature.
Recall from Section~\ref{sec:prelim-F_r} that, by convention, we take \(\OG(m,2m) \cong \OG(m-1,2m-1)\) to be one of the two isomorphic components of the Fano scheme of \((m-1)\)-planes on a smooth quadric \(2(m-1)\)-fold.
\begin{lem}\label{lem:class-of-OG}
    Let \(k\) be an algebraically closed field of characteristic \(\neq 2\), and let \(1\leq r\leq \lfloor\frac{n}{2}\rfloor\). Then the class of the orthogonal Grassmannian \(\OG(r,n)\) in \(K_0(\mathrm{Var}/k)\) is equal to:
    \begin{equation*}
    \begin{split}
        [\OG(r,2m+1)] &= \binom{m}{r}_{\L} \prod_{i=m-r+1}^m (\L^i+1), \\
        [\OG(r,2m)] &= \begin{cases}
            \binom{m}{r}_{\L} \prod_{i=m-r}^{m-1} (\L^i+1) & \text{if }r<m, \\
            \prod_{i=1}^{r-1} (\L^i+1) & \text{if }r=m.
        \end{cases}
    \end{split}
\end{equation*}
\end{lem}

Recall that the Gaussian binomial coefficient \(\binom{m}{r}_{\L} \coloneqq \prod_{i=1}^r \frac{(1-\L^{m-r+i})}{(1-\L^i)}\) is a polynomial in \(\L\) with nonnegative integer coefficients \cite[Section 1.7]{Stanley12}.

\begin{proof}[Proof of Lemma~\ref{lem:class-of-OG}]
We first show the following recursive formulas:
if $n\neq 2r, 2r+1$,
\[
[\OG(r,n)] - [\OG(r,n-1)]= 
\left\{
\begin{array}{ll}
(\L^{n-2r}+\L^{(n-2r)/2})[\OG(r-1,n-1)] & \text{if $n$ is even},\\
(\L^{n-2r}-\L^{(n-2r-1)/2})[\OG(r-1,n-1)] & \text{if $n$ is odd}.
\end{array}
\right.
\]
If $n=2r$ (resp. $n=2r+1$), then
\[
[\OG(r,2r)]=[\OG(r-1,2r-1)],\quad [\OG(r,2r+1)]-2[\OG(r,2r)]=(\L-1)[\OG(r-1,2r)].
\]

For this, recall that if \(n\neq 2r\) then $\OG(r,n)=F_{r-1}(Q)$, where $Q=Q^{n-2}$ is any smooth quadric hypersurface in $\P^{n-1}$ (and if \(n=2r\) then \(\OG(r,2r)\) is one of the connected components).
Let $H$ be a general hyperplane in $\P^{n-1}$ and $Q_H\coloneqq Q\cap H$.
In what follows, regard $\OG(r,n-1)=F_{r-1}(Q_H)$ (taking one of the two connected components if \(n=2r+1\)) and $\OG(r-1,n-1)=F_{r-2}(Q_H)$.

If $n\neq 2r, 2r+1$, consider the morphism
\[
\OG(r,n) \setminus \OG(r,n-1) \rightarrow \OG(r-1,n-1),\quad [L] \mapsto [L\cap H];
\]
the fibration is $SO(n-1)$-equivariant, where the action on $\OG(r-1,n-1)$ is transitive.
Furthermore, for any field extension \(K/k\), the group \(SO(n-1)(K)\) acts transitively on \(\OG(r-1,n-1)(K)\), since the quadric \((Q_H)_K\) is split over \(K\).
Over each scheme-theoretic point $\ell \in \OG(r-1, n-1)$, the fiber equals $(Q_{\kappa(\ell)})^{(r-2)}_\ell\setminus ((Q_H)_{\kappa(\ell)})_\ell^{(r-2)}$,
where $(Q_{\kappa(\ell)})^{(r-2)}_\ell, ((Q_H)_{\kappa(\ell)})^{(r-2)}_\ell$ are respectively the hyperbolic reductions of $Q_{\kappa(\ell)}, (Q_H)_{\kappa(\ell)}$ with respect to the $(r-2)$-plane $\ell$.
Hence, by Lemma~\ref{lem:fibration-general} and Proposition~\ref{prop:KS-hyperbolic-reduction}\eqref{item:KS-hyperbolic-reduction-indep},
\[
[\OG(r,n)]-[\OG(r,n-1)]=([Q^{(r-2)}] - [(Q_H)^{(r-2)}])[\OG(r-1,n-1)].
\]
Since $Q^{(r-2)},(Q_H)^{(r-2)}$ are smooth quadric hypersurfaces of dimension $n-2r, n-2r-1$ respectively,
the rest is straightforward by Lemma~\ref{lem:formulas-Gr-quadric-Sym-Pn}\eqref{item:class-of-quadric}.

If $n=2r$, then $\OG(r,2r-1)=\emptyset$ and the above morphism gives an isomorphism $\OG(r,2r)\xrightarrow{\sim}\OG(r-1,2r-1)$, hence $[\OG(r,2r)]=[\OG(r-1,2r-1)]$.
If $n=2r+1$, consider the morphism
\[
\OG(r,2r+1)\setminus (\OG(r,2r)\sqcup \OG(r,2r)')\rightarrow \OG(r-1,2r),\quad[L]\mapsto [L\cap H],
\]
where \(\OG(r,2r)\) and \(\OG(r,2r)'\) are the two connected components of \(F_{r-1}(Q_H)\);
by a similar argument as above, over each scheme-theoretic point $\ell\in \OG(r-1,2r)$, the fiber equals a smooth split conic minus $2$ $\kappa(\ell)$-points.
By Lemma~\ref{lem:fibration-general}, we get
\[
[\OG(r,2r+1)] - 2[\OG(r,2r)]=([\P^1]-2)[\OG(r-1,2r)]=(\L-1)[\OG(r-1,2r)].
\]
This finishes the proof of the recursive formulas.

The formulas for $\OG(r,n)$ may then be deduced by verifying that they satisfy the above recursions. Indeed, we have \([\OG(1,2)]=1\); for \(n=2m+1\geq 3\) odd, \([\OG(1,2m+1)]=[\P^{2m-1}]=(\sum_{i=0}^{m-1} \L^i)(1+\L^m)=\binom{m}{1}_{\L} (\L^m+1)\); and for \(n=2m\geq 4\) even, \([\OG(1,2m)]=[\P^{2m-2}]+\L^{m-1}=\L^{m-1}+\sum_{i=0}^{2m-2}\L^i=\binom{m}{1}_{\L} (\L^{m-1}+1)\) by Lemma~\ref{lem:formulas-Gr-quadric-Sym-Pn}\eqref{item:class-of-quadric}.
Now assume \(r\geq 2\). Then for any \(m > r\), we have
\begin{equation*}
    \begin{split}
        % [\OG(r,2m)]-[\OG(r,2(m-1)+1)] &= \\
        \binom{m}{r}_{\L} \prod_{i=m-r}^{m-1} (\L^i+1) - \binom{m-1}{r}_{\L} \prod_{i=m-r}^{m-1} (\L^i+1) &= \L^{m-r} \binom{m-1}{r-1}_{\L} \prod_{i=m-r}^{m-1} (\L^i+1) \\
        &=(\L^{2m-2r}+\L^{(2m-2r)/2})\binom{m-1}{r-1}_{\L} \prod_{i=m-r+1}^{m-1} (\L^i+1)
    \end{split}
\end{equation*}
using the identity \(\binom{m}{r}_{\L} - \binom{m-1}{r}_{\L} = \L^{m-r} \binom{m-1}{r-1}_{\L}\) \cite[(1.67)]{Stanley12}, and
\begin{equation*}
    \begin{split}
        % [\OG(r,2m+1)]-[\OG(r,2m)] &= \\
        & \binom{m}{r}_{\L} \prod_{i=m-r+1}^m (\L^i+1) - \binom{m}{r}_{\L} \prod_{i=m-r}^{m-1} (\L^i+1) = \binom{m}{r}_{\L} \L^{m-r}(\L^r - 1)\prod_{i=m-r+1}^{m-1} (\L^i+1) \\
        & \qquad = (\L^{m-r+1}-1)\L^{m-r} \binom{m}{r-1}_{\L}\prod_{i=m-r+1}^{m-1} (\L^i+1)
        = (\L^{2m+1-2r}-\L^{(2m-2r)/2})\binom{m}{r-1}_{\L} \prod_{i=m-r+1}^{m-1} (\L^i+1) \\
    \end{split}
\end{equation*}
where we have used that \(\binom{m}{r}_{\L}=\frac{\L^{m-r+1}-1}{\L^r-1} \binom{m}{r-1}_{\L}\). This shows that the formulas for \([\OG(r,2m+1)]\) and \([\OG(r,2m)]\) with \(r<m\) satisfy the recursion.
Finally, for \(r\geq 2\) and \(m=r\), we have
\begin{equation*}
    \begin{split}
        [\OG(r-1,2r-1)] &= [\OG(r-1,2(r-1)+1)] = \binom{r-1}{r-1}_{\L} \prod_{i=(r-1)-(r-1)+1}^{r-1} (\L^i+1) = \prod_{i=1}^{r-1} (\L^i+1),
    \end{split}
\end{equation*}
showing that the formula for \([\OG(r,2r)]\) is correct, and
\begin{equation*}
    \begin{split}
        % [\OG(r,2r+1)] - 2 \prod_{i=1}^{r-1} (\L^i+1) &= \\
        \binom{r}{r}_{\L} \prod_{i=r-r+1}^r (\L^i+1) - 2 [\OG(r,2r)] &= \prod_{i=1}^r (\L^i+1) - 2 \prod_{i=1}^{r-1} (\L^i+1) \\
        &= (\L^r-1)\prod_{i=1}^{r-1} (\L^i+1)
        = (\L-1)[\OG(r-1,2r)],
    \end{split}
\end{equation*}
showing that the formula for \([\OG(r,2r+1)]\) is correct.
\end{proof}

Next, we prove Theorem~\ref{thm:relative-F_r}.
We will first give a simple proof of Theorem~\ref{thm:relative-F_r} when the relative dimension \(n\) is odd. Then, we will give a proof that is more complicated and works for any \(n\). The second proof uses induction and involves many similar geometric ideas to those used in the proof of Theorems~\ref{thm:main-odd} and~\ref{thm:main-even}.

\begin{proof}[Proof of Theorem~\ref{thm:relative-F_r} if \(n=2g-1\) is odd and \(1\leq r\leq g-1\)]
    Let \(\{t_1,\ldots,t_\delta\}\subset S\) be the degeneracy locus of \(\pi\).
    Since \(r\leq g-1\) and the fibers of \(\pi\) are quadric \((2g-1)\)-folds, the morphism \(F_r(\mathcal Q/S)\to S\) is surjective. Furthermore, \(F_{g-1}(\mathcal Q/S)\to S\) has a section by \cite{GHS03,deJongStarr03} because \(S\) is a curve and \(\OG(g,2g+1)\) is rational.
    So the generic fiber of $\pi\colon \mathcal{Q}\to S$ is a split smooth quadric $n$-fold and the restriction of \(F_r(\mathcal Q/S)\to S\) to \(S\setminus \{t_1,\ldots,t_\delta\}\) is an \(\OG(r+1,2g+1)\)-bundle, and \([F_r(\mathcal Q/S)|_{S\setminus \{t_1,\ldots,t_\delta\}}]=[S][\OG(r+1,2g+1)]-\delta [\OG(r+1,2g+1)]\).
    To compute the contribution over the degeneracy locus, let \(1\leq i\leq\delta\), and let \(v_i\in\mathcal Q_{t_i}\) be the singular point. The corank 1 quadric \(\mathcal Q_{t_i}\) is the cone with vertex \(v_i\) over a smooth quadric \((2g-2)\)-fold \(Q^{2g-2}\); let \(p\colon\mathcal Q_{t_i}\setminus\{v_i\}\to Q^{2g-2}\) be the projection to the base of the cone. Then \(F_r(\mathcal Q_{t_i})=\{m\in F_r(Q_{t_i}) \mid v_i\in m\} \sqcup \{m\in F_r(Q_{t_i}) \mid v_i\notin m\}\). The first set is \(F_{r-1}(Q^{2g-2})\cong\OG(r,2g)\), and the second set is an \(\mathbb A^{r+1}\)-bundle over \(F_r(Q^{2g-2})\), whose fiber over \(\overline{m}\in F_r(Q^{2g-2})\) is the set of hyperplanes in \(\langle v_i,\overline{m}\rangle\cong\P^{r+1}\) not containing \(v_i\).
    For \(r\leq g-2\) we have \(F_r(Q^{2g-2})=\OG(r+1,2g)\), whereas \(F_{g-1}(Q^{2g-2})\) is a disjoint union of two copies of \(\OG(g,2g)\), so the class of \(F_r(\mathcal Q/S)\) in \(K_0(\mathrm{Var}/k)\) is equal to
    \[
        \begin{split}
        \begin{cases}
            [S][\OG(r+1,2g+1)]+\delta(-[\OG(r+1,2g+1)] + [\OG(r,2g)] + [\OG(r+1,2g)]\L^{r+1}) & \text{if }r\leq g-2, \\
            [S][\OG(g,2g+1)]+\delta(-[\OG(g,2g+1)] + [\OG(g-1,2g)] + 2[\OG(g,2g)]\L^g) & \text{if }r=g-1.
        \end{cases}
        \end{split}
    \]
    In each case, the coefficient of \(\delta\) is equal to \([\OG(r,2g)]\L^{g-r}\) by Lemma~\ref{lem:class-of-OG}. Indeed, if \(1\leq r\leq g-2\), then
    \begin{equation*}
        \begin{split}
            & -[\OG(r+1,2g+1)] + [\OG(r,2g)] + [\OG(r+1,2g)]\L^{r+1} \\
            & \qquad\qquad = \left( -\binom{g}{r+1}_\L (\L^g+1) + \binom{g}{r}_\L + \binom{g}{r+1}_\L (\L^{g-r-1}+1)\L^{r+1} \right) \prod_{i=g-r}^{g-1}(\L^i+1) \\
            & \qquad\qquad = \binom{g}{r}_\L\left( -\frac{1-\L^{g-r}}{1-\L^{r+1}} (\L^g+1) + 1 +  \frac{1-\L^{g-r}}{1-\L^{r+1}} (\L^{g-r-1}+1)\L^{r+1} \right) \prod_{i=g-r}^{g-1}(\L^i+1) \\
            & \qquad\qquad = \binom{g}{r}_\L \frac{\L^{g-r} - \L^{g+1}}{1-\L^{r+1}} \prod_{i=g-r}^{g-1}(\L^i+1)
            % = \binom{g}{r}_\L \L^{g-r} \prod_{i=g-r}^{g-1}(\L^i+1)
            = [\OG(r,2g)] \L^{g-r} .
        \end{split}
    \end{equation*}
    If \(r=g-1\), then
    \begin{equation*}
        \begin{split}
            & -[\OG(g,2g+1)] + [\OG(g-1,2g)] + 2[\OG(g,2g)]\L^g \\
            & \qquad\qquad = -\prod_{i=1}^g(\L^i+1) + \binom{g}{g-1}_\L \prod_{i=1}^{g-1}(\L^i+1) + 2 \L^g\prod_{i=1}^{g-1}(\L^i+1) \\
            & \qquad\qquad = \left(-\L^g-1 + \sum_{i=0}^{g-1}\L^i + 2 \L^g\right)\prod_{i=1}^{g-1}(\L^i+1)
            % = \binom{g}{g-1}_\L \L \prod_{i=1}^{g-1}(\L^i+1)
            = [\OG(g-1,2g)]\L.
        \end{split}
    \end{equation*}
\end{proof}

Before proving Theorem~\ref{thm:relative-F_r} for any \(n\) and \(r\), we will first prove a lemma that will be used in the computation of the induction step. This lemma is a computation that follows from the expressions for the orthogonal Grassmannian in Lemma~\ref{lem:class-of-OG}.
\begin{lem}\label{lem:telescoping-sum-computation}
    Let \(1\leq r\leq g-1\) be integers. Let \(i\) be an integer such that \(0\leq i\leq g-r-1\), and set \(h=g-i\). In the polynomial ring \(\mathbb Z[q]\), let \([\OG(r',n')]\) be the polynomial such that \(q\mapsto\L\) maps \([\OG(r',n')]\) to the polynomial expression for the class of the orthogonal Grassmannian \(\OG(r',n')\) in \(K_0(\mathrm{Var}/k)\) given in Lemma~\ref{lem:class-of-OG}. Define the following polynomials in \(\mathbb Z[q]\):
    \begin{align*}
        A &\coloneqq \textstyle q^{r+1}\left(1-\sum_{j=0}^r q^j \right) = -q^{r+2}\sum_{j=0}^{r-1}q^j, & P_h &\coloneqq [\OG(r+1,2h+2)], \\
        B_h^{\text{even}} &\coloneqq \textstyle \left(\sum_{j=0}^{h-r-1} q^j\right)(q^r+q^{h+1})+1-q^r, & \widetilde{P}_h &\coloneqq [\OG(r,2h+1)] q^{h-r}, \\
        B_h^{\text{odd}} &\coloneqq \textstyle \left(\sum_{j=0}^{h-r-1} q^j\right)(q^r+q^h)+1-q^r, & Q_h &\coloneqq [\OG(r+1,2h+1)], \\
        & & \widetilde{Q}_h &\coloneqq [\OG(r,2h)] q^{h-r}.
    \end{align*}
    Then the following equalities of polynomials in \(\mathbb Z[q]\) hold:
    \begin{equation*}
        \begin{split}
            [\OG(r,2h)]q^h + [\OG(r-1,2h-1)]q^{h-r}B_h^{\text{even}} &= \widetilde{P}_h- \widetilde{P}_{h-1}A, \\
            [\OG(r,2h)](B_h^{\text{even}}+2 q^h) &= \begin{cases}
                P_h - P_{h-1} A & \text{if }h \geq r+2, \\
                P_{r+1}-2 P_r A & \text{if }h=r+1,
            \end{cases} \\
            [\OG(r,2h-1)] B_h^{\text{odd}} &= Q_h - Q_{h-1} A, \\
            [\OG(r,2h-1)] q^h + [\OG(r-1,2h-2)](q^{2h-r}+q^{h-r} B_h^{\text{odd}}+q^{2h-r-1}) &= \begin{cases}
                \widetilde{Q}_h-\widetilde{Q}_{h-1}A & \text{if }h \geq r+2, \\
                \widetilde{Q}_{r+1}-2\widetilde{Q}_rA & \text{if }h = r+1.
            \end{cases}
        \end{split}
    \end{equation*}
\end{lem}

\begin{proof}
    First, note that we have the following identities:
    \begin{align}
        (1-q^r)(1-q^h-B_h^{\text{even}}) &=\frac{A}{q}(1-q^{2h-2r}), \label{eq:telescoping-lemma-id1} \\
        (1-q^{r+1})(B_h^{\text{even}}+2q^h) &= (1-q^{h+1})(1+q^h)-A(1-q^{h-r})(1+q^{h-r-1}), \label{eq:telescoping-lemma-id2} \\
        (1-q^{r+1})B_h^{\text{odd}} &= 1-q^{2h}-A(1-q^{2h-2r-2}), \label{eq:telescoping-lemma-id3} \\
        (1-q^r)(1-q^h-B_h^{\text{odd}}) &= \frac{A}{q}(1-q^{h-r})(1+q^{h-r-1}). \label{eq:telescoping-lemma-id4}
    \end{align}
    These can be verified by direct computation, using the identity \((1-q)\sum_{j=0}^l q^j = 1-q^{l+1}\). For~\eqref{eq:telescoping-lemma-id3}, one may check the equality after multiplication by \(1-q\), since \(\mathbb Z[q]\) is an integral domain.

\begin{detail}
    Indeed, for~\eqref{eq:telescoping-lemma-id1}, the left-hand side is equal to
    \begin{equation*}
        \begin{split}
            & \textstyle (1-q^r)\left(1-q^h-\left(\sum_{j=0}^{h-r-1} q^j\right)(q^r+q^{h+1})-1+q^r\right) = \textstyle (1-q^r)\left(q^r-q^h-\left(\sum_{j=0}^{h-r-1} q^j\right)(q^r+q^{h+1})\right) \\
            & \qquad\qquad = \textstyle -(1-q^r)q^{r+1}(1+q^{h-r}) \sum_{j=0}^{h-r-1} q^j \\
            & \qquad\qquad = \textstyle -(1-q)\left(\sum_{j=0}^{r-1} q^j\right)q^{r+1}(1+q^{h-r}) \sum_{j=0}^{h-r-1} q^j \\
            & \qquad\qquad = \textstyle -\left(\sum_{j=0}^{r-1} q^j\right)q^{r+1}(1+q^{h-r}) (1-q^{h-r}) \\
            & \qquad\qquad = \textstyle -q^{r+1}(1-q^{2h-2r})\sum_{j=0}^{r-1} q^j = \textstyle \frac{A}{q} (1-q^{2h-2r}).
        \end{split}
    \end{equation*}
    Similarly, for~\eqref{eq:telescoping-lemma-id4}, the left-hand side is equal to
    \begin{equation*}
        \begin{split}
            & \textstyle (1-q^r)\left(q^r-q^h-\left(\sum_{j=0}^{h-r-1} q^j\right)(q^r+q^h)\right) = \textstyle -(1-q^r)(q^{r+1}+q^h)\sum_{j=0}^{h-r-1} q^j \\
            & \qquad\qquad = \textstyle -(1-q^{h-r})q^{r+1}(1+q^{h-r-1}) \sum_{j=0}^{r-1} q^j = \frac{A}{q}(1-q^{h-r})(1+q^{h-r-1}).
        \end{split}
    \end{equation*}
    For~\eqref{eq:telescoping-lemma-id2}, first note that \(B_h^{\text{even}} = 1-q^h+q^{r+1}\sum_{j=0}^{2h-2r-1} q^j\). So the difference \((1-q^{h+1})(1+q^h)-(1-q^{r+1})(B_h^{\text{even}}+2q^h)\) is equal to
    \begin{equation*}
        \begin{split}
            & \textstyle (1-q^{h+1})(1+q^h) - (1-q^{r+1})\left(1+q^h+q^{r+1}\sum_{j=0}^{2h-2r-1} q^j\right)
            \\ & \textstyle \qquad\qquad = q^{r+1}(1-q^{h-r})(1+q^h) - q^{r+1}(1-q^{r+1})\sum_{j=0}^{2h-2r-1} q^j \\
            & \textstyle \qquad\qquad = q^{r+1}(1-q^{h-r})(1+q^h) - q^{r+1}(1-q^{r+1})(1+q^{h-r}) \sum_{j=0}^{h-r-1} q^j \\
            & \textstyle \qquad\qquad = q^{r+1}(1-q^{h-r})(1+q^h) - q^{r+1}(1-q^{h-r})(1+q^{h-r}) \sum_{j=0}^r q^j \\
            & \textstyle \qquad\qquad = q^{r+1}(1-q^{h-r})\left(1+q^h - (1+q^{h-r}) \sum_{j=0}^r q^j\right) \\
            & \textstyle \qquad\qquad = q^{r+1}(1-q^{h-r})(1+q^{h-r-1})\left(1-\sum_{j=0}^r q^j\right) = A (1-q^{h-r})(1+q^{h-r-1}).
        \end{split}
    \end{equation*}
    This shows~\eqref{eq:telescoping-lemma-id2}.
    For~\eqref{eq:telescoping-lemma-id3}, if \(h=r+1\) then \(B_{r+1}^{\text{odd}}=q^{r+1}+1\) and the desired equality holds.
    If \(h\geq r+2\), the computation is similar to that of~\eqref{eq:telescoping-lemma-id2}. Indeed, first observe that \(B_h^{\text{odd}}=1+q^{r+1}\sum_{j=0}^{2h-2r-2} q^j\). So the difference \(1-q^{2h}-(1-q^{r+1})B_h^{\text{odd}}\) is equal to
    \begin{equation*}
        \begin{split}
            & \textstyle 1-q^{2h}-(1-q^{r+1})\left(1+q^{r+1}\sum_{j=0}^{2h-2r-2} q^j\right) = q^{r+1}\left(1-q^{2h-r-1}-(1-q^{r+1})\sum_{j=0}^{2h-2r-2}q^j\right)
        \end{split}
    \end{equation*}
    Multiplying the right-hand side of the above expression by \(1-q\), we obtain
    \begin{equation*}
        \begin{split}
            & \textstyle q^{r+1} \left((1-q)(1-q^{2h-r-1})-(1-q^{r+1})(1-q^{2h-2r-1}) \right)
            = q^{r+1}(-q+q^{r+1}+q^{2h-2r-1}-q^{2h-r-1}) \\
            & \textstyle \qquad\qquad = -q(1-q^r)q^{r+1}(1-q^{2h-2r-2}) \\
            & \textstyle \qquad\qquad = (1-q)\left(1-\sum_{j=0}^r q^j\right)q^{r+1}(1-q^{2h-2r-2})
            = (1-q) A(1-q^{2h-2r-2}).
        \end{split}
    \end{equation*}
    This shows that \((1-q)(1-q^{2h}-(1-q^{r+1})B_h^{\text{odd}})=(1-q)A(1-q^{2h-2r-2})\), and dividing both sides by \(1-q\) (which is a valid operation since \(\mathbb Z[q]\) is an integral domain) we obtain the equality~\eqref{eq:telescoping-lemma-id3}. This shows the equalities \eqref{eq:telescoping-lemma-id1}--\eqref{eq:telescoping-lemma-id4}.
\end{detail}

    Now we show the six equalities in the lemma statement.
    First we consider \(\widetilde{P}_h-\widetilde{P}_{h-1}A\).
    Define \(\widetilde{D}_h^{\text{even}}=\widetilde{P}_h-\widetilde{P}_{h-1}A - [\OG(r,2h)]q^h - [\OG(r-1,2h-1)]q^{h-r}B_h^{\text{even}} \) to be the right-hand side minus the left-hand side in the desired equality. We will show that \((1-q^r)\widetilde{D}_h^{\text{even}}=0\); since \(\mathbb Z[q]\) is an integral domain, this will imply that \(\widetilde{D}_h^{\text{even}}=0\) as required. To show that \((1-q^r)\widetilde{D}_h^{\text{even}}=0\), first note that Lemma~\ref{lem:class-of-OG} and the definition of the Gaussian binomial coefficients imply that \((1-q^r)[\OG(r,2h+1)]=(1-q^{2h})[\OG(r-1,2h-1)]\), \((1-q^r)[\OG(r,2h)]=(1-q^h)(1+q^{h-r})[\OG(r-1,2h-1)]\), and \((1-q^r)[\OG(r,2h-1)]=(1-q^{2h-2r})[\OG(r-1,2h-1)]\). So \((1-q^r) \widetilde{D}_h^{\text{even}}\) is equal to
    \begin{equation*}
        \begin{split}
            & (1-q^r) \left( [\OG(r,2h+1)]q^{h-r}-[\OG(r,2h-1)]q^{h-r-1}A - [\OG(r,2h)]q^h - [\OG(r-1,2h-1)]q^{h-r}B_h^{\text{even}} \right) \\
            & \qquad\qquad = [\OG(r-1,2h-1)](1-q^{2h})q^{h-r}-[\OG(r-1,2h-1)](1-q^{2h-2r})q^{h-r-1}A  \\
            & \qquad\qquad\phantom{=} - [\OG(r-1,2h-1)](1-q^h)(1+q^{h-r})q^h - [\OG(r-1,2h-1)](1-q^r)q^{h-r}B_h^{\text{even}} \\
            & \textstyle \qquad\qquad = [\OG(r-1,2h-1)]q^{h-r}
            \left( 1-q^{2h} - \frac{A}{q}(1-q^{2h-2r}) - (1-q^h)(1+q^{h-r})q^r - (1-q^r)B_h^{\text{even}} \right) \\
            & \textstyle \qquad\qquad = [\OG(r-1,2h-1)]q^{h-r}
            \left((1-q^r)(1-q^h-B_h^{\text{even}}) - \frac{A}{q}(1-q^{2h-2r}) \right) =0
        \end{split}
    \end{equation*}
    by~\eqref{eq:telescoping-lemma-id1}.

    For \(P_h-P_{h-1}A\) and \(h\geq r+2\), define \(D_h^{\text{even}}= P_h-P_{h-1}A-[\OG(r,2h)](B_h^{\text{even}}+2q^h)\) to be the difference of the two sides of the desired equality. We will show that \((1-q^{r+1})D_h^{\text{even}}=0\). For this, first note that \((1-q^{r+1})P_h=(1-q^{h+1})(1+q^h)[\OG(r,2h)]\) and \((1-q^{r+1})P_{h-1}=(1-q^{h-r})(1+q^{h-r-1})[\OG(r,2h)]\) by Lemma~\ref{lem:class-of-OG}. So, by~\eqref{eq:telescoping-lemma-id2}, \((1-q^{r+1}) D_h^{\text{even}}\) is equal to
    \begin{equation*}
        \begin{split}
            [\OG(r,2h)] \left((1-q^{h+1})(1+q^h)-A(1-q^{h-r})(1+q^{h-r-1})-(1-q^{r+1})(B_h^{\text{even}}+2q^h) \right) = 0 .
        \end{split}
    \end{equation*}

    For \(P_h-2 P_{h-1} A\) in the case \(h=r+1\), we will show that the desired equality holds after dividing both sides by a common factor. We compute that \begin{equation*}
        \begin{split}
            & \textstyle (q^{r+1}+1)\left(\sum_{j=0}^{r+1} q^j \right) - 2 A - \left(\sum_{j=0}^r q^j\right) (1+q^{r+2}+2q^{r+1}) = -2A-2q^{r+2} \sum_{j=0}^{r-1} q^j = 0
        \end{split}
    \end{equation*}
    by definition of \(A\), so \(\left(\sum_{j=0}^r q^j\right) (1+q^{r+2}+2q^{r+1}) = (q^{r+1}+1)\left(\sum_{j=0}^{r+1} q^j \right) - 2 A\). Multiplying both sides of this equality by \(\prod_{j=1}^r(q^j+1)\) yields the desired equality by Lemma~\ref{lem:class-of-OG} and using that \(B_{r+1}^{\text{even}}=q^{r+2}+1\).

    To address \(Q_h-Q_{h-1}A\), let \(D_h^{\text{odd}}=Q_h-Q_{h-1}A - [\OG(r,2h-1)] B_h^{\text{odd}}\) be the difference of the right-hand and left-hand sides in the desired equality. We will show that \((1-q^{r+1})D_h^{\text{odd}}=0\). For this, first note that Lemma~\ref{lem:class-of-OG} implies \((1-q^{r+1})Q_h = (1-q^{2h}) [\OG(r,2h-1)]\), \((1-q^{r+1})Q_{h-1} = (1-q^{2h-2r-2})[\OG(r,2h-1)]\) for \(h\geq r+2\), and \((1-q^{r+1})Q_{h-1} = 0 = (1-q^{2h-2r-2})[\OG(r,2h-1)]\) for \(h=r+1\). So, by~\eqref{eq:telescoping-lemma-id3},
    \begin{equation*}
        \begin{split}
            (1-q^{r+1})D_h^{\text{odd}} &= [\OG(r,2h-1)] \left( 1-q^{2h}- A(1-q^{2h-2r-2}) - (1-q^{r+1}) B_h^{\text{odd}}\right) = 0.
        \end{split}
    \end{equation*}

    For \(\widetilde{Q}_h-\widetilde{Q}_{h-1}A\) and \(h\geq r+2\), let \(D_h^{\text{odd}}=\widetilde{Q}_h-\widetilde{Q}_{h-1}A - [\OG(r,2h-1)] q^h - [\OG(r-1,2h-2)](q^{2h-r}+q^{h-r} B_h^{\text{odd}}+q^{2h-r-1})\) be the difference of the right-hand and left-hand sides of the desired equality; we will show that \((1-q^r) D_h^{\text{odd}}=0\).
    For this, first note that \((1-q^r)[\OG(r,2h)]=(1-q^h)(1+q^{h-1})[\OG(r-1,2h-2)]\), \((1-q^r)[\OG(r,2h-2)]=(1-q^{h-r})(1+q^{h-r-1})[\OG(r-1,2h-2)]\), and \((1-q^r)[\OG(r,2h-1)]=(1-q^{h-r})(1+q^{h-1})[\OG(r-1,2h-2)]\) by Lemma~\ref{lem:class-of-OG}. So \((1-q^r) D_h^{\text{odd}}\) is equal to
    \begin{equation*}
        \begin{split}
            & \textstyle [\OG(r-1,2h-2)]q^{h-r} \Bigl((1-q^h)(1+q^{h-1})-\frac{A}{q}(1-q^{h-r})(1+q^{h-r-1}) - (1-q^{h-r})(1+q^{h-1})q^r  \\ & \phantom{=}\qquad - (1-q^r)(q^h+ B_h^{\text{odd}}+q^{h-1}) \Bigr) \\
            % &\qquad = \textstyle [\OG(r-1,2h-2)]q^{h-r} \left( -\frac{A}{q}(1-q^{h-r})(1+q^{h-r-1}) + (1-q^r)(1+q^{h-1}) - (1-q^r)(q^h+ B_h^{\text{odd}}+q^{h-1}) \right) \\
            &\qquad = \textstyle [\OG(r-1,2h-2)]q^{h-r} \left( -\frac{A}{q}(1-q^{h-r})(1+q^{h-r-1}) + (1-q^r)(1-q^h- B_h^{\text{odd}}) \right) = 0 \\
        \end{split}
    \end{equation*}
    by~\eqref{eq:telescoping-lemma-id4}.

    Finally, we address \(\widetilde{Q}_h-2\widetilde{Q}_{h-1} A\) in the case \(h=r+1\). As in the argument for \(P_{r+1}-2P_r A\), we will show that the desired equality holds after dividing both sides by a common factor.
    We compute that
    \begin{equation*}
        \begin{split}
            & \textstyle q(q^r+1)\left(\sum_{j=0}^r q^j\right)-2A - q^{r+1}(q^r+1) - (q+q^{r+1}+2q^{r+2})\sum_{j=0}^{r-1} q^j
            = -2A -2q^{r+2} \sum_{j=0}^{r-1} q^j = 0.
        \end{split}
    \end{equation*}
    This shows that \(q^{r+1}(q^r+1) + (q+q^{r+1}+2q^{r+2})\sum_{j=0}^{r-1} q^j=q(q^r+1)\left(\sum_{j=0}^r q^j\right)-2A\). Multiplying both sides of this equality by \(\prod_{j=1}^{r-1}(q^j+1)\) and using Lemma~\ref{lem:class-of-OG} and the equality \(B_{r+1}^{\text{odd}}=q^{r+1}+1\) yields the desired equality.
    This completes the proof of the lemma.
\end{proof}

Before proving Theorem~\ref{thm:relative-F_r}, we recall some results on the Grothendieck rings of varieties over other bases. Let \(S\) be a variety over \(k\) with structure morphism \(f\colon S\to\Spec k\). The Grothendieck ring \(K_0(\mathrm{Var}/S)\) of varieties over \(S\) is the quotient of the free abelian group on the set of isomorphism classes of varieties over \(S\) by cut-and-paste relations, with multiplication defined by \([X]\cdot_S[Y]=[X\times_S Y]\) for varieties \(X,Y\) over \(S\) \cite[Definition 1.2.2 and Proposition 2.2.1]{MotivicIntegrationBook}. Viewing an \(S\)-variety as a \(k\)-variety defines a morphism of abelian groups \(f_!\colon K_0(\mathrm{Var}/S) \to K_0(\mathrm{Var}/k)\) \cite[(1.2.7)]{MotivicIntegrationBook}, and pullback defines a ring homomorphism \(f^*\colon K_0(\mathrm{Var}/k) \to K_0(\mathrm{Var}/S)\) that sends \([Z]\mapsto[Z\times_{\Spec k} S]\) for a variety \(Z\) over \(k\) \cite[(2.2.4)]{MotivicIntegrationBook}. In this setting, there is a projection formula \(f_!(f^* x \cdot_S y) = x \cdot f_! y\) for classes \(x\in K_0(\mathrm{Var}/k)\) and \(y\in K_0(\mathrm{Var}/S)\) \cite[(2.2.4.1)]{MotivicIntegrationBook}.

Now we are ready to prove Theorem~\ref{thm:relative-F_r} by induction on \(r\). In the beginning of the proof we will work in \(K_0(\mathrm{Var}/k)\) (and here brackets will denote classes in \(K_0(\mathrm{Var}/k)\)). The induction step will take place in \(K_0(\mathrm{Var}/S)\).

\begin{proof}[Proof of Theorem~\ref{thm:relative-F_r}]
    Throughout the proof, we denote \(g\coloneqq\lfloor\frac{n+1}{2}\rfloor\).
    Smooth quadrics, \(\OG(r,n+2)\), and \(\OG(g,n+2)\) are rational, so since the base \(S\) is a curve, the quadric fibration \(\pi\) has a section, an \((r-1)\)-section and a \((g-1)\)-section by \cite{GHS03,deJongStarr03}; furthermore, these are necessarily nondegenerate by Proposition~\ref{prop:KS-hyperbolic-reduction}\eqref{item:KS-smooth-nondegenerate-section}.
    For an integer \(-1\leq l\leq g-1\), we denote by \(\pi^{(l)}\colon\mathcal Q^{(l)}\to S\) the hyperbolic reduction of \(\pi\) with respect to some nondegenerate \(l\)-section (recall that the class of \(\mathcal Q^{(l)}\) in the Grothendieck ring is well defined by Proposition~\ref{prop:KS-hyperbolic-reduction}\eqref{item:KS-hyperbolic-reduction-indep}).

    First we show the required formulas for \(r=g\), i.e., for the Fano schemes of maximal isotropic subspaces.
    If \(n=2g\) is even, then the Stein factorization \(F_g(\mathcal Q/S)\to\widetilde{S}\to S\) factors as an \'etale locally trivial \(\OG(g+1,2g+2)\)-bundle over \(\widetilde{S}\), which is Zariski locally trivial because it has a section {by \cite{GHS03,deJongStarr03},} and hence the generic fiber of $\mathcal{Q}_{\widetilde{S}}\to \widetilde{S}$ is split. Therefore \([F_g(\mathcal Q/S)]=[\widetilde{S}][\OG(g+1,2g+2)]\).
    If \(n=2g-1\) is odd, then the image of \(F_g(\mathcal Q/S)\to S\) consists of the \(\delta\) points of the degeneracy locus of \(\pi\), and the fiber over each point is the Fano scheme of \((g-1)\)-planes on a smooth quadric \((2g-2)\)-fold and hence a disjoint union of two copies of \(\OG(g,2g)\) (since every \(g\)-plane in the corank 1 quadric \((2g-1)\)-fold \(\mathcal Q_{t_i}\) contains the singular point of \(\mathcal Q_{t_i}\)). So \([F_g(\mathcal Q/S)]=2\delta[\OG(g,2g)]\).
    
    For the remainder of the proof, we assume \(r\leq g-1\). To prove the result, we induct on \(r\).
    
    For \(r=0\), there are two base cases, depending on the parity of \(n\). By Proposition~\ref{prop:KS-hyperbolic-reduction}\eqref{item:KS-hyperbolic-reduction-formula}, the fact that \(\pi\) has a nondegenerate \((g-1)\)-section, and Lemma~\ref{lem:hyperbolic-reduction-properties-not-pencil}\eqref{item:hyperbolic-reduction-even-maximal},\eqref{item:hyperbolic-reduction-finite-set-maximal},
    \begin{equation*}
        \begin{split}
            [\mathcal Q]
            &= \begin{cases}
                [S][\P^{g-1}](1+\L^{g+1})+[\widetilde{S}]\L^g & \text{if \(n=2g\) is even,}\\
                [S][\P^{g-1}](1+\L^g)+\delta\L^g & \text{if \(n=2g-1\) is odd.}
            \end{cases}
        \end{split}
    \end{equation*}

    Now assume \(r\geq 1\). If \(n\in\{1,2\}\) then \(r=g\), so we may assume \(n\geq 3\).
    We will now work in \(K_0(\mathrm{Var}/S)\). Let \(f\colon S\to\Spec k\) be the structure morphism. By abuse of notation, for an \(S\)-variety, we will now use brackets to denote its class in \(K_0(\mathrm{Var}/S)\), and we will use \(\L\) (resp. \([\P^j]\)) to denote the class of \(f^*\L\) (resp. \(f^*[\P^j]\)) in \(K_0(\mathrm{Var}/S)\).
    We will show that in \(K_0(\mathrm{Var}/S)\), the class of \(F_r(\mathcal Q/S)\) is equal to
    \[ \begin{cases}
        [S][\OG(r+1,2g+2)] + ([\widetilde S]-2[S])[\OG(r,2g+1)]\L^{g-r} & \text{if \(n=2g\) is even, \(r\leq g-1\)}, \\
        [\widetilde{S}][\OG(g+1,2g+2)] & \text{if \(n=2g\) is even, }r=g, \\ 
        [S][\OG(r+1,2g+1)]+\delta[\OG(r,2g)] \L^{g-r} & \text{if \(n=2g-1\) is odd, \(r\leq g-1\)}, \\
        2\delta [\OG(g,2g)] & \text{if \(n=2g-1\) is odd, }r=g.
    \end{cases} \]
    In the above, by abuse of notation, \([S][\OG(r',n')]\) denotes the class of the \(S\)-variety \([S\times \OG(r',n')]\) (and similarly for \([\widetilde{S}] [\OG(g+1,2g+2)])\), and \(\delta[\OG(r,2g)]\) denotes the class of the product of \([\OG(r,2g)]\) and the degeneracy locus of \(\pi\), which is an \(S\)-variety.
    (We will use this notation for convenience in the computation of the induction step.)
    The projection formula for \(f^*\colon K_0(\mathrm{Var}/k) \to K_0(\mathrm{Var}/S)\) and \(f_!\colon K_0(\mathrm{Var}/S) \to K_0(\mathrm{Var}/k)\) will imply the desired result in \(K_0(\mathrm{Var}/k)\).
    
    Let \(s\) be a section of \(\pi\), and let \(\pi^{(0)}_s\colon\mathcal Q^{(0)}_s\to S\) be the hyperbolic reduction of \(\pi\) with respect to \(s\); the assumption that \(n\geq 3\) implies that \(\pi^{(0)}_s\) is a quadric fibration. Recall that \(F_{r-1}(\mathcal Q^{(0)}_s/S)\cong F_r(\mathcal Q/S)_s\) by Lemma~\ref{lem:hyperbolic-reduction-properties-not-pencil}\eqref{item:hyperbolic-reduction-not-pencil-embedding}; using this identification, let \(U\to F_{r-1}(\mathcal Q^{(0)}_s/S)\) be the \(\P^r\)-bundle whose fiber over \((m',t)\) is the set of hyperplanes in \(m'\cong\P^r\). Now define the correspondence
    \[\Gamma \coloneqq \{(m,m',y,t) \in F_r(\mathcal Q/S)\times_S U \mid y \subset m\}\]
    where \((m,t)\in F_r(\mathcal Q/S)\) and \((m',y,t)\in U\), so that \(m\) is an \(r\)-plane in \(\mathcal Q_t\) and \(m'\) is an \(r\)-plane in \(\mathcal Q_t\) containing the point \(s(t)\) and the \((r-1)\)-plane \(y\). We will use the first projection \(p_1\colon\Gamma\to F_r(\mathcal Q/S)\) to compute the class of \(F_r(\mathcal Q/S)\).

    First, we compute the class of \(\Gamma\) using the second projection \(\Gamma\to U\). For this, we identify \(\Gamma\) with a hyperbolic reduction, as follows. Let \(\pi_U\colon U\times_S\mathcal Q\to U\) be the quadric \(n\)-fold fibration obtained from \(\pi\) by base change, and define the \((r-1)\)-section \(\sigma\) of \(\pi_U\) by \((m',y,t) \mapsto (m',y,y,t)\). This \((r-1)\)-section may be degenerate; it is nondegenerate away from
    \[\Delta\coloneqq\{ (m',y,t_i) \in U \mid v_i \in y \subset m' \subset\mathcal Q_{t_i}, 1\leq i\leq\delta \}\]
    where \(\{t_1,\ldots,t_\delta\}\) is the degeneracy locus of \(\pi\) and \(v_i\in\mathcal Q_{t_i}\) is the singular point.
    Let \((\pi_U)^{(r-1)}_\sigma\colon (U\times_S\mathcal Q)^{(r-1)}_\sigma \to U\) be the corresponding hyperbolic reduction; the fiber of \((\pi_U)^{(r-1)}_\sigma\) over \((m',y,t)\) is the quadric \((\mathcal Q_t)^{(r-1)}_y\), i.e., the hyperbolic reduction of the quadric \(\mathcal Q_t\) with respect to the \((r-1)\)-plane \(y\). Then
    \[(U\times_S\mathcal Q)^{(r-1)}_\sigma\cong\Gamma\]
    by definition of the hyperbolic reduction (Definition~\ref{defn:hyperbolic-red-degen}).

    To compute the class of \((U\times_S\mathcal Q)^{(r-1)}_\sigma\), we compare it to the class of \(U\times_S\mathcal Q^{(r-1)}\), where \(\mathcal Q^{(r-1)}\to S\) is the hyperbolic reduction of \(\pi\) with respect to any nondegenerate \((r-1)\)-section. We have \([(U\times_S\mathcal Q)^{(r-1)}_\sigma|_{U\setminus\Delta}] = [U\times_S\mathcal Q^{(r-1)}|_{U\setminus\Delta}]\) by Proposition~\ref{prop:KS-hyperbolic-reduction}\eqref{item:KS-hyperbolic-reduction-indep}. Next we compute the class of \(\Delta\) and of the fibers of these hyperbolic reductions over \(\Delta\).
    For the class of \(\Delta\), note that it is a \(\P^{r-1}\)-bundle over
    \begin{equation}\label{eqn:U-bad-locus}
        \{(m',t_i) \in F_{r-1}(\mathcal Q^{(0)}_s/S) \mid v_i \in m' \subset \mathcal Q_{t_i}, 1\leq i\leq \delta\} = \bigsqcup_{i=1}^\delta \{m'\in F_r(\mathcal Q_{t_i}) \mid s(t_i), v_i \in m'\}.
    \end{equation}
    We claim that~\eqref{eqn:U-bad-locus} is a disjoint union of \(\delta\) copies of \(\OG(r-1,n-1)\). Indeed, for each \(1\leq i\leq\delta\), we have \(\{m'\in F_r(\mathcal Q_{t_i}) \mid s(t_i), v_i \in m'\} \cong \{\overline{m} \in F_{r-1}((\mathcal Q_{t_i})^{(0)}_{v_i}) \mid \overline{s(t_i)} \in \overline{m}\}\) by the cone description of \(\mathcal Q_{t_i}\), and this is isomorphic to the Fano scheme of \((r-2)\)-planes in the further hyperbolic reduction of \((\mathcal Q_{t_i})^{(0)}_{v_i}\) with respect to the point \(\overline{s(t_i)}\); this second hyperbolic reduction is a smooth quadric of dimension \(n-3\).
    This shows that \(\Delta\) is a \(\P^{r-1}\)-bundle over a disjoint union of \(\delta\) copies of \(\OG(r-1,n-1)\) and hence \([\Delta]=\delta[\OG(r-1,n-1)][\P^{r-1}]\).

    Next, we compare the fibers of the two different hyperbolic reductions of \(U\times_S\mathcal Q\) over \(\Delta\). For this, let \((m',y,t_i)\in\Delta\), and let \([Q^l]\) denote the class of a smooth quadric \(l\)-fold.
    \begin{enumerate}
        \item The fiber of \((\pi_{U})_{\sigma}^{(r-1)}\colon (U\times_S\mathcal Q)^{(r-1)}_\sigma\to U\) over \((m',y,t_i)\) is the hyperbolic reduction of the corank 1 quadric \(n\)-fold \(\mathcal Q_{t_i}\) with respect to an \((r-1)\)-plane containing the vertex. This is a smooth quadric of dimension \(n-2r+1\) in \(\P^{n-2r+2}\) by Example~\ref{exmp:hyperbolic-red-corank1-quadric} (view it as the hyperbolic reduction of the smooth quadric \((n-1)\)-fold \((\mathcal Q_{t_i})^{(0)}_{v_i}\) with respect to an \((r-2)\)-plane), so it has class \([Q^{n-2r+1}]\).
        More generally, for every scheme-theoretic point $x\in \Delta$ lying over $t_i$, the fiber of $(\pi_U)_\sigma^{(r-1)}$ over $x$ is a split smooth quadric of dimension $n-2r+1$ in $\P^{n-2r+2}_{\kappa(x)}$
        because it is a hyperbolic reduction of the split smooth quadric $(\mathcal{Q}_{t_i})_{v_i}^{(0)}\times_{\Spec(k)}\Spec\kappa(x)$ with respect to an $(r-2)$-plane and itself is split by the Witt cancellation theorem.
        Thus, by Lemma~\ref{lem:fibration-general}, we have that $[((\pi_U)^{(r-1)}_\sigma)^{-1}(\Delta)] = [\Delta][Q^{n-2r+1}]$ (the equality holds in \(K_0(\mathrm{Var}/S)\) because the morphisms involved are \(S\)-morphisms).
        \item The fiber of \((\pi^{(r-1)})_U\colon U\times_S(\mathcal Q^{(r-1)})\to U\) over \((m',y,t_i)\) is the hyperbolic reduction of the corank $1$ quadric \(n\)-fold \(\mathcal Q_{t_i}\) with respect to a nondegenerate \((r-1)\)-plane, so it is a quadric \((n-2r)\)-fold of corank $1$ in \(\P^{n-2r+1}\) and has class \([Q^{n-2r-1}]\L+1\).
        Furthermore, since \((\pi^{(r-1)})_U\) is the base change of \(\mathcal Q^{(r-1)}\to S\) by \(U\), and the image of \(\Delta\) in \(S\) consists of \(\delta\) distinct points, we have that \([(\pi^{(r-1)})_U^{-1}(\Delta)] = [\Delta]([Q^{n-2r-1}]\L+1)\).
    \end{enumerate}
    Since \(U\to S\) factors through the \(\P^r\)-bundle \(U\to F_{r-1}(\mathcal Q^{(0)}_s/S)\), we have that \(U\times_S(\mathcal Q^{(r-1)})\) is a \(\P^r\)-bundle over \(F_{r-1}(\mathcal Q^{(0)}_s/S)\times_S(\mathcal Q^{(r-1)})\). So
    \begin{equation}\label{eqn:class-of-Gamma}
        \begin{split}
            [\Gamma] &= [F_{r-1}(\mathcal Q^{(0)}_s/S)\times_S(\mathcal Q^{(r-1)})][\P^r] + ([Q^{n-2r+1}] - [Q^{n-2r-1}]\L - 1)[\Delta] \\
            &= [F_{r-1}(\mathcal Q^{(0)}_s/S)\times_S(\mathcal Q^{(r-1)})][\P^r] + \delta[\OG(r-1,n-1)][\P^{r-1}]\L^{n-2r+1} \quad \text{by Lemma~\ref{lem:formulas-Gr-quadric-Sym-Pn}\eqref{item:class-of-quadric}}.
        \end{split}
    \end{equation}
    
    Next, we relate the class of \(F_r(\mathcal Q/S)\) to that of \(\Gamma\) using the first projection \(p_1\colon\Gamma\to F_r(\mathcal Q/S)\).
    This projection factors as
    \[\Gamma \xrightarrow{\overline{p}} \Gamma_2 \coloneqq \{(m,m',t) \in F_r(\mathcal Q/S)\times_S F_{r-1}(\mathcal Q^{(0)}_s/S) \mid m \cap m' \text{ contains an \((r-1)\)-plane}\} \xrightarrow{\overline{q}} F_r(\mathcal Q/S).\]

    Let \(T\to S\) be the \(\P^n\)-bundle whose fiber over \(t\) is the tangent space \(T_{s(t)\mathcal Q_t}\subset \P^{n+1}\) to \(\mathcal Q_t\) at the smooth point \(s(t)\). After identifying \(F_{r-1}(\mathcal Q^{(0)}_s/S)\cong F_r(\mathcal Q/S)_s\), we have an inclusion \(F_{r-1}(\mathcal Q^{(0)}_s/S)\subset F_r(\mathcal Q\cap T/S)\). We consider the stratification \(\{F_r(\mathcal Q/S)\setminus F_r(\mathcal Q\cap T/S), F_r(\mathcal Q\cap T/S) \setminus F_{r-1}(\mathcal Q^{(0)}_s/S), F_{r-1}(\mathcal Q^{(0)}_s/S)\}\) of \(F_r(\mathcal Q/S)\).
    \begin{enumerate}
        \item Over \(F_r(\mathcal Q/S)\setminus F_r(\mathcal Q\cap T/S)\), the first projection \(p_1\) is an isomorphism. Indeed, for \((m,t)\in F_r(\mathcal Q/S)\setminus F_r(\mathcal Q\cap T/S)\), the \(r\)-plane \(m\) intersects \(T_{s(t)}\mathcal Q_t\) in an \((r-1)\)-plane that does not contain \(s(t)\), so \((m,t) \mapsto (m,\langle s(t), m \cap T_{s(t)}\mathcal Q_t\rangle, m \cap T_{s(t)}\mathcal Q_t, t)\) defines the inverse of \(p_1\) over this locus.
        \item Over \(F_r(\mathcal Q\cap T/S) \setminus F_{r-1}(\mathcal Q^{(0)}_s/S)\), the map \(\overline{p}\) is an isomorphism, and \(\overline{q}\) is a Zariski locally trivial \(\P^r\)-bundle whose fiber over \((m,t)\) is the set of hyperplanes in the \((r+1)\)-plane \(\langle s(t), m \rangle\) passing through \(s(t)\). (The \((r+1)\)-plane \(\langle s(t),m\rangle\) is contained in \(\mathcal Q_t\cap T_{s(t)}\mathcal Q_t\) because this intersection is the cone with vertex \(s(t)\) over a quadric \((n-2)\)-fold containing \(m\).) So the class of
        \(p_1^{-1}(F_r(\mathcal Q\cap T/S) \setminus F_{r-1}(\mathcal Q^{(0)}_s/S))\) is \([F_r(\mathcal Q\cap T/S)][\P^r] - [F_{r-1}(\mathcal Q^{(0)}_s/S)][\P^r]\).
        The first projection is a \(\P^r\)-bundle.
        \item For the preimage of \(F_{r-1}(\mathcal Q^{(0)}_s/S)\), we consider the cases \(s(t)\notin y\) and \(s(t)\in y\) separately.
        First, if \((m,m',y,t) \in p_1^{-1}(F_{r-1}(\mathcal Q^{(0)}_s/S))\) satisfies \(s(t)\notin y\), then we must have \(m=m'\). So the locus where \(s(t)\notin y\) is an \(\mathbb A^r\)-bundle over \(F_{r-1}(\mathcal Q^{(0)}_s/S)\) whose fiber over \((m,t)\) is the set of hyperplanes in \(m\) not containing the point \(s(t)\), and hence its class is \([F_{r-1}(\mathcal Q^{(0)}_s/S)]\L^r\).

        For the locus where \(s(t)\in y\), let \(I\to F_{r-1}(\mathcal Q^{(0)}_s/S)\) be the \(\P^{r-1}\)-bundle whose fiber over \((m,t)\) is the set of hyperplanes in \(m\cong\P^r\) containing \(s(t)\). Let \(\pi_I\colon I\times_S\mathcal Q\to I\) be the quadric \(n\)-fold fibration obtained by base change, and define the \((r-1)\)-section \(\overline{\sigma}\) by \((m,y,t) \mapsto (m,y,y,t)\). Then the corresponding hyperbolic reduction \((I\times_S\mathcal Q)^{(r-1)}_{\overline{\sigma}}\) satisfies
        \[(I\times_S\mathcal Q)^{(r-1)}_{\overline{\sigma}} \xrightarrow{\sim} \{(m,y,m',t) \in I \times F_r(\mathcal Q/S) \mid y \subset m'\}\]
        where \((m,y,t)\in I\) and \(m'\) is an \(r\)-plane in \(\mathcal Q_t\). Note that since \(s(t)\in y\), we have \(s(t)\in m'\), so this is isomorphic to the locus of \(p_1^{-1}(F_{r-1}(\mathcal Q^{(0)}_s/S))\) where $s(t)\in y$. The \((r-1)\)-section \(\overline{\sigma}\) is nondegenerate over \(I\setminus\overline{\Delta}\) where
        \[\overline{\Delta}\coloneqq \{(m,y,t_i) \in I \mid s(t_i),v_i \in y \subset \mathcal Q_{t_i} , 1\leq i\leq\delta\}.\]
        This is a \(\P^{r-2}\)-bundle over \(\bigsqcup_{i=1}^\delta \{(m,t_i) \mid s(t_i),v_i \in m \subset \mathcal Q_{t_i}\}\), which (as computed above) is a disjoint union of \(\delta\) copies of \(\OG(r-1,n-1)\). So \(\overline{\Delta}\) has class \(\delta[\OG(r-1,n-1)][\P^{r-2}]\).

        To compute the class of \((I\times_S\mathcal Q)^{(r-1)}_{\overline{\sigma}}\), we compare it to that of \(I\times_S (\mathcal Q^{(r-1)})\). The latter is the base change of \(\mathcal Q^{(r-1)}\to S\) by \(I\) so in particular it is a \(\P^{r-1}\)-bundle over \(F_{r-1}(\mathcal Q^{(0)}_s/S)\times_S(\mathcal Q^{(r-1)})\). The two hyperbolic reductions \((I\times_S\mathcal Q)^{(r-1)}_{\overline{\sigma}}\to I\) and \(I\times_S (\mathcal Q^{(r-1)}) \to I\) have the same class over \(I\setminus\overline{\Delta}\) by Proposition~\ref{prop:KS-hyperbolic-reduction}\eqref{item:KS-hyperbolic-reduction-indep}. To compare the fibers over \(\overline{\Delta}\), let \((m,y,t_i)\in\overline{\Delta}\).
        \begin{enumerate}
            \item The fiber of \((I\times_S\mathcal Q)^{(r-1)}_{\overline{\sigma}}\to I\) over \((m,y,t_i)\) is a smooth quadric \((n-2r+1)\)-fold by Example~\ref{exmp:hyperbolic-red-corank1-quadric}.
            More generally, for every scheme-theoretic point $x\in\overline{\Delta}$, the fiber of \((I\times_S\mathcal Q)^{(r-1)}_{\overline{\sigma}}\to I\) over $x$ is a split smooth quadric $(n-2r+1)$-fold over \(\kappa(x)\).
            Thus, by Lemma~\ref{lem:fibration-general}, the preimage of $\overline{\Delta}$ has class \([\overline{\Delta}][Q^{n-2r+1}]\).
            \item The fiber of \(I\times_S (\mathcal Q^{(r-1)}) \to I\) over \((m,y,t_i)\) is a quadric \((n-2r)\)-fold of corank 1. Since this hyperbolic reduction is obtained by base change from \(\mathcal Q^{(r-1)}\to S\), the preimage of \(\overline{\Delta}\) has class \([\overline{\Delta}]([Q^{n-2r-1}]\L+1)\).
        \end{enumerate}
        As computed above, the difference of the classes of the two hyperbolic reductions over \(\overline{\Delta}\) has class \([\overline{\Delta}]\L^{n-2r+1}\), so
        \begin{equation*}
            \begin{split}
                & [p_1^{-1}(F_{r-1}(\mathcal Q^{(0)}_s/S))] = [F_{r-1}(\mathcal Q^{(0)}_s/S)\times_S(\mathcal Q^{(r-1)})][\P^{r-1}] + [\overline{\Delta}]\L^{n-2r+1} + [F_{r-1}(\mathcal Q^{(0)}_s/S)]\L^r \\
                & \qquad\qquad\qquad = [F_{r-1}(\mathcal Q^{(0)}_s/S)\times_S(\mathcal Q^{(r-1)})][\P^{r-1}] + \delta[\OG(r-1,n-1)][\P^{r-2}]\L^{n-2r+1} + [F_{r-1}(\mathcal Q^{(0)}_s/S)]\L^r .
            \end{split}
        \end{equation*}
    \end{enumerate}
    Putting these together, we obtain
    \begin{equation}\begin{split}\label{eqn:rel-F_r-in-terms-of-corresp}
        [F_r(\mathcal Q/S)]
        % &= [\Gamma] + [F_r(\mathcal Q\cap T/S)] - [F_r(\mathcal Q\cap T/S)][\P^r] + [F_{r-1}(\mathcal Q^{(0)}_s/S)][\P^r] \\ & \qquad -[F_{r-1}(\mathcal Q^{(0)}_s/S)\times_S(\mathcal Q^{(r-1)})][\P^{r-1}] - \delta[\OG(r-1,n-1)][\P^{r-2}]\L^{n-2r+1} - [F_{r-1}(\mathcal Q^{(0)}_s/S)]\L^r \\
        &= [\Gamma] + [F_r(\mathcal Q\cap T/S)](1-[\P^r]) + [F_{r-1}(\mathcal Q^{(0)}_s/S)][\P^{r-1}] \\ & \qquad -[F_{r-1}(\mathcal Q^{(0)}_s/S)\times_S(\mathcal Q^{(r-1)})][\P^{r-1}] - \delta[\OG(r-1,n-1)][\P^{r-2}]\L^{n-2r+1} .
    \end{split}\end{equation}
    For the class of \(F_r(\mathcal Q\cap T/S)\), consider the correspondence
    \[\{(m,P,t) \in F_r(\mathcal Q\cap T/S) \times_S F_{r+1}(\mathcal Q/S)_s \mid m \subset P \subset \mathcal Q_t\}.\]
    The first projection is an isomorphism over \(F_r(\mathcal Q\cap T/S)\setminus F_{r-1}(\mathcal Q^{(0)}_s/S)\), with inverse \((m,t) \mapsto (m,\langle s(t), m\rangle, t)\). The preimage of \(F_r(\mathcal Q\cap T/S)\setminus F_{r-1}(\mathcal Q^{(0)}_s/S)\) under the first projection is the locus of \((m,P,t)\) with \(s(t)\notin m\), which is a Zariski locally trivial \((\P^{r+1}\setminus\P^r)\)-bundle over \(F_{r+1}(\mathcal Q/S)_s\) via the second projection. Therefore, recalling that \(F_r(\mathcal Q^{(0)}_s/S) \cong F_{r+1}(\mathcal Q/S)_s\) by Lemma~\ref{lem:hyperbolic-reduction-properties-not-pencil}\eqref{item:hyperbolic-reduction-not-pencil-embedding}, we have
    \begin{equation}\label{eqn:rel-F_r-QcapH}
        [F_r(\mathcal Q\cap T/S)] = [F_{r-1}(\mathcal Q^{(0)}_s/S)] + [F_r(\mathcal Q^{(0)}_s/S)] \L^{r+1} .
    \end{equation}

    Combining~\eqref{eqn:class-of-Gamma}, \eqref{eqn:rel-F_r-in-terms-of-corresp}, and~\eqref{eqn:rel-F_r-QcapH}, we obtain the equality
    \begin{equation*}
        \begin{split}
            [F_r(\mathcal Q/S)] &= [F_{r-1}(\mathcal Q^{(0)}_s/S)\times_S(\mathcal Q^{(r-1)})][\P^r] + \delta[\OG(r-1,n-1)][\P^{r-1}]\L^{n-2r+1} \\ & \qquad + [F_{r-1}(\mathcal Q^{(0)}_s/S)](1-[\P^r]) + [F_r(\mathcal Q^{(0)}_s/S)] \L^{r+1}(1-[\P^r]) + [F_{r-1}(\mathcal Q^{(0)}_s/S)][\P^{r-1}] \\ & \qquad -[F_{r-1}(\mathcal Q^{(0)}_s/S)\times_S(\mathcal Q^{(r-1)})][\P^{r-1}] - \delta[\OG(r-1,n-1)][\P^{r-2}]\L^{n-2r+1} \\
            % &= [F_{r-1}(\mathcal Q^{(0)}_s/S)\times_S(\mathcal Q^{(r-1)})][\P^r] + \delta[\OG(r-1,n-1)]([\P^{r-1}] - [\P^{r-2}])\L^{n-2r+1} \\ & \qquad + [F_{r-1}(\mathcal Q^{(0)}_s/S)](1-[\P^r]+[\P^{r-1}]) + [F_r(\mathcal Q^{(0)}_s/S)] \L^{r+1}(1-[\P^r]) \\ & \qquad -[F_{r-1}(\mathcal Q^{(0)}_s/S)\times_S(\mathcal Q^{(r-1)})][\P^{r-1}] \\
            % &= [F_{r-1}(\mathcal Q^{(0)}_s/S)\times_S(\mathcal Q^{(r-1)})][\P^r] + \delta[\OG(r-1,n-1)]\L^{n-r} \\ & \qquad + [F_{r-1}(\mathcal Q^{(0)}_s/S)](1-[\P^r]+[\P^{r-1}]) + [F_r(\mathcal Q^{(0)}_s/S)] \L^{r+1}(1-[\P^r]) \\ & \qquad -[F_{r-1}(\mathcal Q^{(0)}_s/S)\times_S(\mathcal Q^{(r-1)})][\P^{r-1}] \\
            &= [F_{r-1}(\mathcal Q^{(0)}_s/S)\times_S(\mathcal Q^{(r-1)})]\L^r + \delta[\OG(r-1,n-1)]\L^{n-r} \\ & \qquad + [F_{r-1}(\mathcal Q^{(0)}_s/S)](1-\L^r) + [F_r(\mathcal Q^{(0)}_s/S)] \L^{r+1}(1-[\P^r])
        \end{split}
    \end{equation*}
    in \(K_0(\mathrm{Var}/S)\).
    
    Using the equality \([\mathcal Q^{(r-1)}] = [S][\P^{g-r-1}](1+\L^{n-g-r+1})+[\mathcal Q^{(g-1)}] \L^{g-r}\) from Corollary~\ref{cor:formula-class-of-hyperbolic-reduction}, we obtain
    \begin{equation}\label{eqn:rel-F_r-expression-to-plug-in-recursion}
        \begin{split}
            [F_r(\mathcal Q/S)]
            % &= [F_{r-1}(\mathcal Q^{(0)}_s/S)\times_S(\mathcal Q^{(r-1)})]\L^r + \delta[\OG(r-1,n-1)]\L^{n-r} \\ & \qquad + [F_{r-1}(\mathcal Q^{(0)}_s/S)](1-\L^r) + [F_r(\mathcal Q^{(0)}_s/S)] \L^{r+1}(1-[\P^r]) \\
            % &= [F_{r-1}(\mathcal Q^{(0)}_s/S)\times_S S][\P^{g-r-1}](1+\L^{n-g-r+1})\L^r+[F_{r-1}(\mathcal Q^{(0)}_s/S)\times_S(\mathcal Q^{(g-1)})]\L^g + \delta[\OG(r-1,n-1)]\L^{n-r} \\ & \qquad + [F_{r-1}(\mathcal Q^{(0)}_s/S)](1-\L^r) + [F_r(\mathcal Q^{(0)}_s/S)] \L^{r+1}(1-[\P^r]) \\
            &= [F_{r-1}(\mathcal Q^{(0)}_s/S)\times_S(\mathcal Q^{(g-1)})]\L^g + \delta[\OG(r-1,n-1)]\L^{n-r} \\ & \qquad + [F_{r-1}(\mathcal Q^{(0)}_s/S)]([\P^{g-r-1}](\L^r+\L^{n-g+1}) + 1-\L^r) + [F_r(\mathcal Q^{(0)}_s/S)] \L^{r+1}(1-[\P^r])
        \end{split}
    \end{equation}

    For \(0\leq i\leq g-r-2\), the hyperbolic reduction \(\pi^{(i)}\colon \mathcal Q^{(i)}\to S\) has relative dimension \(n-2i-2\) and \(\lfloor\frac{n-2i-2+1}{2}\rfloor = g-i-1\). Since \((\mathcal Q^{(i)})^{(g-i-2)}\to S\) is the hyperbolic reduction of \(\mathcal Q\to S\) with respect to a \((g-1)\)-section, we have
    \begin{equation*}
        \begin{split}
            [F_r(\mathcal Q^{(i)}/S)]
            % &= [F_{r-1}(\mathcal Q^{(i+1)}/S)\times_S((\mathcal Q^{(i)})^{(g-i-2)})]\L^{g-i-1} + \delta[\OG(r-1,n-2i-3)]\L^{n-r-2i-2} \\ & \qquad + [F_{r-1}(\mathcal Q^{(i+1)}/S)]([\P^{g-r-i-2}](\L^r+\L^{n-g-i}) + 1-\L^r) + [F_r(\mathcal Q^{(i+1)}/S)] \L^{r+1}(1-[\P^r]) \\
            &= [F_{r-1}(\mathcal Q^{(i+1)}/S)\times_S(\mathcal Q^{(g-1)})]\L^{g-i-1} + \delta[\OG(r-1,n-2i-3)]\L^{n-r-2i-2} \\ & \qquad + [F_{r-1}(\mathcal Q^{(i+1)}/S)]([\P^{g-r-i-2}](\L^r+\L^{n-g-i}) + 1-\L^r) + [F_r(\mathcal Q^{(i+1)}/S)] \L^{r+1}(1-[\P^r]) .
        \end{split}
    \end{equation*}
    Plugging this into~\eqref{eqn:rel-F_r-expression-to-plug-in-recursion}, we get
    \begin{equation}\label{eqn:F_r-recursive-formula}
    \begin{split}
    [F_r(\mathcal{Q}/S)]&=
    \sum_{i=0}^{g-r-1}[F_{r-1}(\mathcal{Q}^{(i)}/S)\times_{S}\mathcal{Q}^{(g-1)}]\L^{g-i}(\L^{r+1}(1-[\P^r]))^i\\
    &\quad+\sum_{i=0}^{g-r-1}[F_{r-1}(\mathcal{Q}^{(i)}/S)]([\P^{g-r-i-1}](\L^r+\L^{n-g-i+1})+1-\L^r)(\L^{r+1}(1-[\P^r]))^i \\    
    &\quad + \delta\sum_{i=0}^{g-r-1}[\OG(r-1,n-2i-1)]\L^{n-r-2i}(\L^{r+1}(1-[\P^{r}]))^i\\
    &\quad +[F_r(\mathcal{Q}^{(g-r-1)}/S)](\L^{r+1}(1-[\P^r]))^{g-r}.
    \end{split}
    \end{equation}
    Since the hyperbolic reduction \(\pi^{(g-r-1)}\colon \mathcal Q^{(g-r-1)}\to S\) has relative dimension \(n-2g+2r\), we have \(\lfloor\frac{n-2g+2r+1}{2}\rfloor=r\), so the Fano scheme \(F_r(\mathcal{Q}^{(g-r-1)}/S)\) of maximal isotropic subspaces has class
    \[[F_r(\mathcal{Q}^{(g-r-1)}/S)] = \begin{cases}
        [\widetilde{S}][\OG(r+1,2r+2)] & \text{if \(n=2g\) is even}, \\ 
        2\delta [\OG(r,2r)] & \text{if \(n=2g-1\) is odd}.
    \end{cases}\]

    The induction hypothesis on \(r-1\) gives formulas for the classes of \(F_{r-1}(\mathcal{Q}^{(i)}/S)\). Now we will complete the induction. We consider the cases of \(n\) even and \(n\) odd separately.

    \underline{If \(n=2g\) is even:} Then \(\mathcal Q^{(g-1)}\cong\widetilde{S}\) by Lemma~\ref{lem:hyperbolic-reduction-properties-not-pencil}\eqref{item:hyperbolic-reduction-even-maximal}, and the induction hypothesis implies
    \begin{equation}\begin{split}\label{eqn:IH-even-case}
        [F_{r-1}(\mathcal Q^{(i)}/S)] &= [S][\OG(r,2g-2i)] + ([\widetilde S]-2[S])[\OG(r-1,2g-2i-1)] \L^{g-r-i}\\
        % &= [S]([\OG(r,2g-2i)]-2[\OG(r-1,2g-2i-1)]\L^{g-r-i}) + [\widetilde S][\OG(r-1,2g-2i-1)]\L^{g-r-i}
    \end{split}\end{equation}
    because \(\mathcal Q^{(i)}\to S\) is a quadric fibration of relative dimension \(2g-2i-2\).
    Since \([\widetilde{S}\times_S\widetilde{S}]=2[\widetilde{S}]-\delta\), we have
    \begin{equation*}
        \begin{split}
            [F_{r-1}(\mathcal{Q}^{(i)}/S)\times_{S}\widetilde{S}]\L^{g-i}(\L^{r+1}(1-[\P^r]))^i &= [\widetilde{S}][\OG(r,2g-2i)]\L^{g-i}(\L^{r+1}(1-[\P^r]))^i \\ & \quad -\delta [\OG(r-1,2g-2i-1)]\L^{2g-r-2i}(\L^{r+1}(1-[\P^r]))^i
        \end{split}
    \end{equation*}
    for each \(0\leq i\leq g-r-1\). Substituting this, the equality \([F_r(\mathcal{Q}^{(g-r-1)}/S)]=[\widetilde{S}][\OG(r+1,2r+2)]\), and~\eqref{eqn:IH-even-case} into~\eqref{eqn:F_r-recursive-formula}, we obtain
    {\footnotesize
    \begin{equation*}
    \begin{split}
    & [F_r(\mathcal{Q}/S)]=
    [\widetilde{S}]\sum_{i=0}^{g-r-1}[\OG(r,2g-2i)]\L^{g-i}(\L^{r+1}(1-[\P^r]))^i\\
    &\qquad+[S]\sum_{i=0}^{g-r-1}([\OG(r,2g-2i)]-2[\OG(r-1,2g-2i-1)]\L^{g-r-i})([\P^{g-r-i-1}](\L^r+\L^{g-i+1})+1-\L^r)(\L^{r+1}(1-[\P^r]))^i \\    
    &\qquad+[\widetilde{S}]\sum_{i=0}^{g-r-1}[\OG(r-1,2g-2i-1)]\L^{g-r-i}([\P^{g-r-i-1}](\L^r+\L^{g-i+1})+1-\L^r)(\L^{r+1}(1-[\P^r]))^i \\
    &\qquad + [\widetilde{S}][\OG(r+1,2r+2)](\L^{r+1}(1-[\P^r]))^{g-r} .
    \end{split}
    \end{equation*}}

    Now we simplify the coefficients of \([\widetilde{S}]\) and \([S]\) in the above expression. Let \(h=g-i\), and let \(A\), \(B_h^{\text{even}}\), \(\widetilde{P}_h\), and \(P_h\) be the polynomials in \(\mathbb Z[\L]\subset K_0(\mathrm{Var}/S)\) obtained by setting \(q\mapsto\L\) for the corresponding polynomials in Lemma~\ref{lem:telescoping-sum-computation}. The coefficient of \([\widetilde{S}]\) in the above expression is
    \begin{equation*}
        \begin{split}
            & \left(\sum_{h=r+1}^g \left([\OG(r,2h)]\L^h + [\OG(r-1,2h-1)]\L^{h-r}B_h^{\text{even}}\right) A^{g-h} \right) + [\OG(r+1,2r+2)]A^{g-r} \\
            &\qquad\qquad= \left(\sum_{h=r+1}^g (\widetilde{P}_h- \widetilde{P}_{h-1}A) A^{g-h} \right) + [\OG(r+1,2r+2)]A^{g-r} \\
            &\qquad\qquad = \widetilde{P}_g- \widetilde{P}_r A^{g-r} + [\OG(r+1,2r+2)] A^{g-r} \\
            &\qquad\qquad = [\OG(r,2g+1)]\L^{g-r} - [\OG(r,2r+1)] A^{g-r} + [\OG(r+1,2r+2)] A^{g-r} \\
            &\qquad\qquad = [\OG(r,2g+1)]\L^{g-r} ,
        \end{split}
    \end{equation*}
    where the second equality holds by Lemma~\ref{lem:telescoping-sum-computation}, and the final equality holds because \(\OG(r,2r+1)\cong \OG(r+1,2r+2)\).
    
    To address the coefficient of \([S]\), we will prove that the coefficient of \([S]\) plus twice the coefficient of \([\widetilde{S}]\) is equal to \([\OG(r+1,2g+2)]\). The sum of the coefficient of \([S]\) and two times the coefficient of \([\widetilde{S}]\) is
    \begin{equation*}
        \begin{split}
            & \left(\sum_{h=r+1}^{g}[\OG(r,2h)] (B_h^{\text{even}} + 2\L^h) A^{g-h} \right) + 2 P_r A^{g-r} \\
            & \qquad\qquad = P_{r+1} A^{g-r-1} -2 P_r A^{g-r} + \left(\sum_{h=r+2}^{g} (P_h - P_{h-1} A ) A^{g-h} \right) + 2 P_r A^{g-r} \\
            & \qquad\qquad = P_g-2P_r A^{g-r} + 2 P_r A^{g-r} \\
            & \qquad\qquad = [\OG(r+1,2g+2)],
        \end{split}
    \end{equation*}
    where the second equality holds by Lemma~\ref{lem:telescoping-sum-computation}.
    Since the coefficient of \([\widetilde{S}]\) is \([\OG(r,2g+1)]\L^{g-r}\),
    this shows the coefficient of \([S]\) is equal to \([\OG(r+1,2g+2)]-2[\OG(r,2g+1)]\L^{g-r}\), and hence that \([F_r(\mathcal Q/S)] = [S][\OG(r+1,2g+2)] + ([\widetilde S]-2[S]) [\OG(r,2g+1)]\L^{g-r}\), as required.

    \underline{If \(n=2g-1\) is odd:} Then \(\mathcal Q^{(g-1)}\) is the degeneracy locus of \(\pi\) by Lemma~\ref{lem:hyperbolic-reduction-properties-not-pencil}\eqref{item:hyperbolic-reduction-finite-set-maximal}, so it consists of \(\delta\) points. The induction hypothesis implies
    \begin{equation}\label{eqn:IH-odd-case}
        [F_{r-1}(\mathcal Q^{(i)}/S)] = [S][\OG(r,2g-2i-1)] + \delta [\OG(r-1,2g-2i-2)]\L^{g-r-i}
    \end{equation}
    because \(\mathcal Q^{(i)}\to S\) is a quadric fibration of relative dimension \(2g-2i-3\). Since \(\mathcal Q^{(g-1)}\times_S\mathcal Q^{(g-1)}\cong\mathcal Q^{(g-1)}\) has class \(\delta\), for each \(0\leq i\leq g-r-1\) we have
    \begin{equation*}
        \begin{split}
            & [F_{r-1}(\mathcal{Q}^{(i)}/S)\times_{S}\mathcal{Q}^{(g-1)}]\L^{g-i}(\L^{r+1}(1-[\P^r]))^i =
            [\mathcal{Q}^{(g-1)}][\OG(r,2g-2i-1)]\L^{g-i}(\L^{r+1}(1-[\P^r]))^i \\
            & \qquad\qquad \phantom{=} + [\mathcal{Q}^{(g-1)}\times_{S}\mathcal{Q}^{(g-1)}][\OG(r-1,2g-2i-2)]\L^{g-r-i}\L^{g-i}(\L^{r+1}(1-[\P^r]))^i \\
            & \qquad\qquad = \delta [\OG(r,2g-2i-1)]\L^{g-i}(\L^{r+1}(1-[\P^r]))^i \\
            & \qquad\qquad \phantom{=} + \delta [\OG(r-1,2g-2i-2)]\L^{2g-r-2i}(\L^{r+1}(1-[\P^r]))^i .
        \end{split}
    \end{equation*}
    Substituting this, the equality \([F_r(\mathcal{Q}^{(g-r-1)}/S)]=2\delta [\OG(r,2r)]\), and~\eqref{eqn:IH-odd-case} into~\eqref{eqn:F_r-recursive-formula}, we obtain
    \begin{equation*}
    \begin{split}
    [F_r(\mathcal{Q}/S)]&=
    \delta \sum_{i=0}^{g-r-1} [\OG(r,2g-2i-1)]\L^{g-i}(\L^{r+1}(1-[\P^r]))^i \\
    &\quad +\delta \sum_{i=0}^{g-r-1} [\OG(r-1,2g-2i-2)]\L^{2g-r-2i}(\L^{r+1}(1-[\P^r]))^i \\
    &\quad + [S]\sum_{i=0}^{g-r-1}[\OG(r,2g-2i-1)]([\P^{g-r-i-1}](\L^r+\L^{g-i})+1-\L^r)(\L^{r+1}(1-[\P^r]))^i \\
    &\quad + \delta\sum_{i=0}^{g-r-1} [\OG(r-1,2g-2i-2)]\L^{g-r-i}([\P^{g-r-i-1}](\L^r+\L^{g-i})+1-\L^r)(\L^{r+1}(1-[\P^r]))^i \\
    &\quad + \delta\sum_{i=0}^{g-r-1}[\OG(r-1,2g-2i-2)]\L^{2g-r-2i-1}(\L^{r+1}(1-[\P^{r}]))^i\\
    &\quad + 2\delta [\OG(r,2r)](\L^{r+1}(1-[\P^r]))^{g-r}.
    \end{split}
    \end{equation*}
    To simplify the coefficients of \([S]\) and \(\delta\) in the above expression, let \(h=g-i\), and let \(A\), \(B_h^{odd}\), \(Q_h\), and \(\widetilde{Q}_h\) be the polynomials in \(\mathbb Z[\L] \subset K_0(\mathrm{Var}/S)\) obtained by setting \(q\mapsto\L\) for the corresponding polynomials in Lemma~\ref{lem:telescoping-sum-computation}.
    The coefficient of \([S]\) is
    \begin{equation*}
        \begin{split}
            & \sum_{h=r+1}^g[\OG(r,2h-1)] B_h^{\text{odd}} A^{g-h}
            = \sum_{h=r+1}^g (Q_h-Q_{h-1}A) A^{g-h} \\
            &\qquad\qquad = [\OG(r+1,2g+1)]-[\OG(r+1,2r+1)] A^{g-r} =  [\OG(r+1,2g+1)]
        \end{split}
    \end{equation*}
    where the first equality holds by Lemma~\ref{lem:telescoping-sum-computation} and the third holds because \(\OG(r+1,2r+1)=\emptyset\).
    
    Next, the coefficient of \(\delta\) is
    \begin{equation*}
        \begin{split}
            & \left(\sum_{h=r+1}^g \left( [\OG(r,2h-1)]\L^h + [\OG(r-1,2h-2)](\L^{2h-r} + \L^{h-r} B_h^{\text{odd}} + \L^{2h-r-1}) \right) A^{g-h}\right) + 2 \widetilde{Q}_r A^{g-r} \\
            & \qquad\qquad = \widetilde{Q}_{r+1}A^{g-r-1} - 2\widetilde{Q}_r A^{g-r} + \left(\sum_{h=r+2}^g (\widetilde{Q}_h - \widetilde{Q}_{h-1} A) A^{g-h}\right) + 2 \widetilde{Q}_r A^{g-r} \\
            & \qquad\qquad = \widetilde{Q}_g - 2\widetilde{Q}_r A^{g-r} + 2 \widetilde{Q}_r A^{g-r}
            = [\OG(r,2g)]\L^{g-r},
        \end{split}
    \end{equation*}
    where the first equality holds by Lemma~\ref{lem:telescoping-sum-computation}.
    This proves that \([F_r(\mathcal Q/S)] = [S][\OG(r+1,2g+1)] + \delta [\OG(r,2g)]\L^{g-r}\), as required.
    
\end{proof}

The following equivalent formulation of Theorem~\ref{thm:relative-F_r} will be more convenient for the following Corollaries~\ref{cor:motive-F_r} and~\ref{cor:rel-F_r-K_0-cat}:
\begin{thm}\label{thm:relative-F_r-expanded-coeffs}
    Over an algebraically closed field \(k\) of characteristic \(\neq 2\), let \(S\) be a smooth connected curve, let \(n\geq 1\), and let \(\pi\colon\mathcal Q\to S\) be a fibration of \(n\)-dimensional quadrics over \(k\). Assume that \(\pi\) has simple degeneration and that the degeneracy locus of \(\pi\) is a smooth divisor.
    Let \(\delta\) be the number of singular fibers of \(\pi\), and let \(0\leq r \leq \lfloor\frac{n+1}{2}\rfloor\).

    If \(n=2g\) is even, let \(\widetilde{S}\to S\) be the discriminant double cover of \(\pi\). Then the following equalities hold in \(K_0(\mathrm{Var}/k)\):
    \[ [F_r(\mathcal Q/S)] = \begin{cases}
        [S]\binom{g}{r+1}_{\L} \prod_{i=g-r+1}^{g+1} (\L^i+1) + [\widetilde S]\binom{g}{r}_{\L} \L^{g-r} \prod_{i=g-r+1}^g (\L^i+1) & \text{if \(r\leq g-1\)}, \\
        [\widetilde{S}]\prod_{i=1}^g(\L^i+1) & \text{if }r=g.
    \end{cases} \]

    If \(n=2g-1\) is odd, then the following equalities hold in \(K_0(\mathrm{Var}/k)\):
    \[ [F_r(\mathcal Q/S)] = \begin{cases}
        [S]\binom{g}{r+1}_{\L} \prod_{i=g-r}^g (\L^i+1)+\delta\binom{g}{r}_{\L} \prod_{i=g-r}^{g-1} (\L^i+1)\L^{g-r} & \text{if \(r\leq g-1\)}, \\
        2\delta \prod_{i=1}^{g-1}(\L^i+1) & \text{if }r=g.
    \end{cases} \]
\end{thm}  
\begin{proof}
Theorem~\ref{thm:relative-F_r-expanded-coeffs} is equivalent to Theorem~\ref{thm:relative-F_r} by Lemma~\ref{lem:class-of-OG} and the computation
\begin{equation*}
    \begin{split}
        & [\OG(r+1,2g+2)]-2[\OG(r,2g+1)]\L^{g-r} = \binom{g+1}{r+1}_{\L} \prod_{i=g-r}^g (\L^i+1) - 2\binom{g}{r}_{\L} \prod_{i=g-r+1}^g (\L^i+1)\L^{g-r} \\
        &\qquad= \binom{g}{r}_{\L} \left( \frac{1-\L^{g+1}}{1-\L^{r+1}} (\L^{g-r}+1) - 2 \L^{g-r} \right) \prod_{i=g-r+1}^g (\L^i+1) \\
        % &= \binom{g}{r}_{\L} \frac{-\L^{2g-r+1}-\L^{g-r}+1+\L^{g+1}}{1-\L^{r+1}} \prod_{i=g-r+1}^g (\L^i+1) \\
        &\qquad= \binom{g}{r}_{\L} \frac{(1+\L^{g+1})(1-\L^{g-r})}{1-\L^{r+1}} \prod_{i=g-r+1}^g (\L^i+1) 
        % &= \binom{g}{r}_{\L} \frac{1-\L^{g-r}}{1-\L^{r+1}} \prod_{i=g-r+1}^{g+1} (\L^i+1) \\
        = \binom{g}{r+1}_{\L} \prod_{i=g-r+1}^{g+1} (\L^i+1)
    \end{split}
\end{equation*}
in the \(n=2g\) and \(r\leq g-1\) case.
\end{proof}

Recall that a \defi{rational Chow motive} (or a \defi{motive}) over $\C$ is a triple $(X,p,n)$, where $X$ is a smooth complex projective variety of pure dimension $d_X$, $p\in \CH^{d_X}(X\times X)_\Q$ satisfies $p\circ p=p$, and $n\in \Z$.
The motive $h(X)$ of $X$ is defined by $h(X)=(X,\Delta_X,0)$. We denote $\Q=(\Spec \mathbb C,\Delta_{\Spec \mathbb C},0)=h(\Spec \mathbb C)$ and $\Q(-1)=(\Spec \mathbb C, \Delta_{\Spec \mathbb C},-1)$; the latter is known as the Lefschetz motive.
A morphism from a motive $(X,p,n)$ to another motive $(Y,q,m)$ is an element of $q\circ \CH^{d_X+m-n}(X\times Y)_\Q\circ p$,
and in the category of motives, we have $h(\P^1)=\Q\oplus \Q(-1)$.
The category of motives is equipped with the tensor product $(X,p,n)\otimes(Y,q,m)=(X\times Y,p\times q, n+m)$,
and we denote $\Q(-i)=\Q(-1)^{\otimes i}$.
A motive is Tate if it is isomorphic to the direct sum \(\bigoplus_{j=1}^m \Q(-i_j)\) for some integers \(i_j\).
Finally, a motive is \defi{of abelian type} if it belongs to the thick and rigid subcategory of motives generated by the motives of abelian varieties.

\begin{cor}\label{cor:motive-F_r}
Over \(\C\), let \(S\) be a smooth projective connected curve, let \(n\geq 1\), and let \(\pi\colon\mathcal Q\to S\) be a fibration of \(n\)-dimensional quadrics. 
Assume that \(\pi\) has simple degeneration and that the degeneracy locus of \(\pi\) is a smooth divisor.
Let \(0\leq r \leq \lfloor\frac{n+1}{2}\rfloor\) and let $\delta$ be the number of singular fibers of $\pi$. If $n=2g$ is even, let $\widetilde{S}\to S$ be the discriminant double cover.
\begin{enumerate}
\item\label{item:motive-abelian-type} The rational Chow motive $h(F_r(\mathcal{Q}/S))$ of $F_r(\mathcal{Q}/S)$ is of abelian type. More precisely, we have the following decomposition of $h(F_r(\mathcal Q/S))$:
{\footnotesize
\[
  h(F_r(\mathcal{Q}/S)) \cong
  \begin{cases}
    \!\begin{aligned}%[b]
       h(S)\otimes & \textstyle h(\Gr(r+1,g))\otimes \left(\bigotimes_{i=g-r+1}^{g+1}(\Q(-i)\oplus \Q)\right) \\
       \oplus & \textstyle h(\widetilde{S})\otimes h(\Gr(r,g))(-g+r)\otimes \left(\bigotimes_{i=g-r+1}^g(\Q(-i)\oplus\Q)\right)
    \end{aligned}           & \text{if }n=2g, r\leq g-1, \\
    h(\widetilde{S})\otimes \left(\bigotimes_{i=1}^g(\Q(-i)\oplus \Q)\right) & \text{if }n=2g, r=g, \\
    \!\begin{aligned}%[b]
       h(S)\otimes & \textstyle h(\Gr(r+1,g))\otimes \left(\bigotimes_{i=g-r}^g (\Q(-i)\oplus \Q)\right) \\
       \oplus & \textstyle \left(h(\Gr(r,g))(-g+r)\otimes \left(\bigotimes_{i=g-r}^{g-1}(\Q(-i)\oplus \Q)\right)\right)^{\oplus \delta}
    \end{aligned}           & \text{if }n=2g-1,r\leq g-1, \\
    \left(\bigotimes_{i=1}^{g-1}(\Q(-i)\oplus \Q)\right)^{\oplus 2\delta} & \text{if }n=2g-1,r=g.
  \end{cases}
\]}
\item\label{item:motive-conjectures} $F_r(\mathcal{Q}/S)$ satisfies the Kimura finite-dimensionality conjecture, the Murre conjectures, the standard conjecture, and the Hodge conjecture (see \cite{vial13} for the statements).
\end{enumerate}
\end{cor}
\begin{proof}
For \eqref{item:motive-conjectures}, let $p\colon F_r(\mathcal Q/S)\to S$ be the projection.
The domain \(F_r(\mathcal Q/S)\) and the codomain \(S\) are smooth and projective (Section~\ref{sec:prelim-F_r}). 
We claim that for every closed point $s$, the reduced subscheme of the fiber $p^{-1}(s)$ is cellular.
Indeed, if $\pi^{-1}(s)$ is smooth, then $p^{-1}(s) \cong F_r(Q^n)$, and $F_r(Q^n)$ (which is the disjoint union of two copies of $\OG(r+1,n+2)$ if $2r+2=n+2$ and $\OG(r+1,n+2)$ otherwise) is cellular by \cite[\S 14.12]{Borel91}, so is $p^{-1}(s)$.
If $\pi^{-1}(s)$ is singular, then it is a cone with vertex \(v\) over a smooth quadric in $\P^n$ by assumption; in this case, \(p^{-1}(s)\) may be non-reduced (see, e.g., \cite[Introduction]{Kuznetsov14}).
Then $p^{-1}(s)_{\text{red}}$ has a closed stratum \(\{\rplane\in p^{-1}(s)_{\text{red}} \mid v\in \rplane\} \cong F_{r-1}(Q^{n-1})\) whose open complement is
an \(\mathbb A^{r+1}\)-bundle over \(F_r(Q^{n-1})\) that parametrizes the \(r\)-planes on \(\pi^{-1}(s)\) that do not contain \(v\). (If \(r=0\) then the first stratum is just the vertex \(v\).)
Since $F_{r-1}(Q^{n-1})$ (or a point in the \(r=0\) case) and \(F_r(Q^{n-1})\)
are both cellular, $p^{-1}(s)_{\text{red}}$ is cellular, as desired.
By the claim, \cite[Example 1.9.1]{fulton1998intersection} shows that for every closed point $s\in S$, $\CH_*(p^{-1}(s)_{\text{red}})=\CH_*(p^{-1}(s))$ is finitely generated.

Consequently, 
if \(r\neq\frac{n+1}{2}\), then by \cite[Theorem 6.12]{vial13}, $\CH_*(F_r(\mathcal{Q}/S))$ is of niveau $\leq 1$, and hence by \cite[Theorem 7.1]{vial13}, $F_r(\mathcal{Q}/S)$ satisfies the Kimura finite-dimensionality conjecture, the Murre conjectures, the standard conjecture, and the Hodge conjecture.
If \(n=2g-1\) is odd and \(r=g\), then \(F_g(\mathcal Q/S)\) is supported over the singular fibers of \(\pi\), and each nonempty fiber is the disjoint union of two copies of \(\OG(g,2g)\), so $\CH_*(F_g(\mathcal{Q}/S))$ is finitely generated, hence $F_g(\mathcal Q/S)$ again satisfies these conjectures.

For \eqref{item:motive-abelian-type}, Theorem \ref{thm:relative-F_r-expanded-coeffs} yields a decomposition of the class of $F_r(\mathcal{Q}/S)$ in the Grothendieck ring of rational numerical motives \cite[Theorem 4 and Theorem 2(v)]{GilletSoule96},
and since the category of rational numerical motives is semisimple due to Jannsen \cite{Jannsen92}, it lifts to the following decomposition of the rational numerical motive $h_{num}(F_r(\mathcal{Q}/S))$ of $F_r(\mathcal{Q}/S)$:
{\footnotesize\[
  h_{num}(F_r(\mathcal{Q}/S)) \cong
  \begin{cases}
    \!\begin{aligned}%[b]
       h_{num}(S)\otimes & \textstyle h_{num}(\Gr(r+1,g))\otimes \left(\bigotimes_{i=g-r+1}^{g+1}(\Q(-i)\oplus \Q)\right) \\
       \oplus & \textstyle h_{num}(\widetilde{S})\otimes h_{num}(\Gr(r,g))(-g+r)\otimes \left(\bigotimes_{i=g-r+1}^g(\Q(-i)\oplus\Q)\right)
    \end{aligned}           & \text{if }n=2g, r\leq g-1, \\
    h_{num}(\widetilde{S})\otimes \left(\bigotimes_{i=1}^g(\Q(-i)\oplus \Q)\right) & \text{if }n=2g, r=g, \\
    \!\begin{aligned}%[b]
       h_{num}(S)\otimes & \textstyle h_{num}(\Gr(r+1,g))\otimes \left(\bigotimes_{i=g-r}^g (\Q(-i)\oplus \Q)\right) \\
       \oplus & \textstyle \left(h_{num}(\Gr(r,g))(-g+r)\otimes \left(\bigotimes_{i=g-r}^{g-1}(\Q(-i)\oplus \Q)\right)\right)^{\oplus \delta}
    \end{aligned}           & \text{if }n=2g-1,r\leq g-1, \\
    \left(\bigotimes_{i=1}^{g-1}(\Q(-i)\oplus \Q)\right)^{\oplus 2\delta} & \text{if }n=2g-1,r=g.
  \end{cases}
\]}\ignorespaces 
Moreover, for any \(r\) and \(n\), \(h(F_r(\mathcal{Q}/S))\) and the motives in the corollary statement are Kimura finite-dimensional.
This follows from \eqref{item:motive-conjectures} and the fact that the motive of a curve is finite-dimensional \cite[Corollary 4.4]{kimura05} and that Grassmannians are cellular so their motives are Tate by \cite[Corollary 66.4]{EKM-quadratic-forms}.
Since the functor from motives to numerical motives restricts to a full and conservative functor on Kimura finite-dimensional motives \cite{andre2005}, 
 we can lift the above decomposition of $h_{num}(F_r(\mathcal{Q}/S))$ to the desired decomposition of $h(F_r(\mathcal{Q}/S))$, which in turn shows that $h(F_r(\mathcal{Q}/S))$ is of abelian type.
\end{proof}

\begin{cor}\label{cor:rel-F_r-K_0-cat}
    In the setting of Theorem~\ref{thm:relative-F_r}, assume that \(k\) is algebraically closed of characteristic zero and \(S\) is projective. Then the class of the derived category of \(F_r(\mathcal Q/S)\) in the Grothendieck ring of categories \(K_0(\mathrm{Cat}/k)\) (see \cite[Definition 2.7 and Example 2.8]{BBFGHLPRAS}) is equal to
    \[ [D^b(F_r(\mathcal Q/S))] = \begin{cases}
        2^{r+1} \binom{g}{r+1} [D^b(S)] + 2^r \binom{g}{r}[D^b(\widetilde{S})] & \text{if \(n=2g\) is even, \(r\leq g-1\)}, \\
        2^g [D^b(\widetilde{S})] & \text{if \(n=2g\) is even, }r=g, \\ 
        2^{r+1} \binom{g}{r+1} [D^b(S)]+ 2^r \binom{g}{r}\delta \cdot 1 & \text{if \(n=2g-1\) is odd, \(r\leq g-1\)}, \\
        2^g \delta \cdot 1 & \text{if \(n=2g-1\) is odd, }r=g,
    \end{cases} \]
    where \(1 = [D^b(\Spec k)]\).
    
    In particular, if \(r=1\), $n\geq 2$, and \(k=\mathbb C\), the image of \(\mathcal R_{even}\) (resp. \(\mathcal R_{odd}\)) from \cite[Theorem 1.2]{Shah-rel-Fr} in \(K_0(\mathrm{Cat}/\mathbb C)\) is \([D^b(\widetilde{S})]\) (resp. \(2\delta\cdot 1\)); this verifies the images of \cite[Conjecture 8.5]{Shah-rel-Fr} and \cite[Expectation after Proposition 8.2]{Shah-rel-Fr} in \(K_0(\mathrm{Cat}/\mathbb C)\) (in the case of quadric fibrations over curves with simple degeneration).
\end{cor}

\begin{proof}
    By \cite{BondalLarsenLunts}, \([Y] \mapsto [D^b(Y)]\) for smooth projective varieties \(Y\) over \(k\) defines a ring homomorphism \(K_0(\mathrm{Var}/k) \to K_0(\mathrm{Cat}/k)\), and the image of \(\L\) under this homomorphism is \(1\). Since \(S\), \(\widetilde{S}\), and \(F_r(\mathcal Q/S)\) are smooth and projective, the above identities hold by Theorem~\ref{thm:relative-F_r-expanded-coeffs} and the fact that the Gaussian binomial coefficient \(\binom{n'}{r'}_{\L}\) evaluated at \(\L=1\) is equal to \(\binom{n'}{r'}\).
    % \begin{equation*}
    %     \begin{split}
    %         \mu_{D^b}\left( \binom{g}{r+1}_{\L} \prod_{i=g-r+1}^{g+1} (\L^i+1) \right) &= \binom{g}{r+1} 2^{r+1} \text{ for }r\leq g-1 \\
    %         \mu_{D^b}\left( \binom{g}{r}_{\L} \L^{g-r} \prod_{i=g-r+1}^g (\L^i+1) \right) &= \binom{g}{r}2^r \text{ for }r\leq g-1 \\
    %         \mu_{D^b}\left( \prod_{i=1}^g(\L^i+1) \right) &= 2^g \\
    %     \end{split}
    % \end{equation*}
    % \begin{equation*}
    %     \begin{split}
    %         \mu_{D^b}\left( \binom{g}{r+1}_{\L} \prod_{i=g-r}^g (\L^i+1) \right) &= \binom{g}{r+1} 2^{r+1} \text{ for }r\leq g-1 \\
    %         \mu_{D^b}\left( \binom{g}{r}_{\L} \prod_{i=g-r}^{g-1} (\L^i+1)\L^{g-r} \right) &= \binom{g}{r}2^r \text{ for }r\leq g-1 \\
    %         \mu_{D^b}\left( \prod_{i=1}^{g-1}(\L^i+1) \right) &= 2^{g-1} \\
    %     \end{split}
    % \end{equation*}
    
    If \(r=1\), then the above formulas specialize to
    \begin{equation}\label{eqn:rel-F_1-K_0-cat}
        [D^b(F_1(\mathcal Q/S))] = \begin{cases}
        2g(g-1) [D^b(S)] + 2g [D^b(\widetilde{S})] & \text{if \(n=2g\) is even, \(r\leq g-1\)}, \\
        2^g [D^b(\widetilde{S})] & \text{if \(n=2g\) is even, }r=g, \\ 
        2g(g-1) [D^b(S)]+ 2g \delta\cdot 1 & \text{if \(n=2g-1\) is odd, \(r\leq g-1\)}, \\
        2^g \delta\cdot 1 & \text{if \(n=2g-1\) is odd, }r=g.
    \end{cases}\end{equation}
    
    If \(n\geq 2\) is even and \(k=\mathbb C\), then the equality
    \([D^b(F_1(\mathcal Q/S))] = [\mathcal R_{even}] + g [\mathcal A] + (g-1) [\mathcal B]\)
    holds in \(K_0(\mathrm{Cat}/\mathbb C)\) by \cite[Theorem 1.2 and Proposition 7.1]{Shah-rel-Fr}, where \(\mathcal A\) and \(\mathcal B\) are subcategories of \(D^b(F_1(\mathcal Q/S))\) such that \([\mathcal A]=(g-1)[D^b(S)]+[D^b(S,\mathcal{C}l_0)]\) and \([\mathcal B]=g[D^b(S)]+[D^b(S,\mathcal{C}l_0)]\), and \(\mathcal R_{even}\) is a residual category. Here \(\mathcal{C}l_0\) is the sheaf of even parts for the Clifford algebra on \(\pi\) \cite[(12)]{Kuznetsov08}, and \(D^b(S,\mathcal{C}l_0)\) is the derived category of coherent sheaves of \(\mathcal{C}l_0\)-modules on \(S\) \cite[Section 2.1]{Kuznetsov08}. Since \(D^b(S,\mathcal{C}l_0)\cong D^b(\widetilde{S})\) by \cite[Corollary 3.14]{Kuznetsov08}, we have
    \[[D^b(F_1(\mathcal Q/S))] = [\mathcal R_{even}] + 2g(g-1)[D^b(S)] + (2g-1) [D^b(\widetilde{S})]\]
    so by~\eqref{eqn:rel-F_1-K_0-cat}, \([\mathcal R_{even}] = [D^b(\widetilde{S})]\) in \(K_0(\mathrm{Cat}/\mathbb C)\).

    If \(n\geq 3\) is odd and \(k=\mathbb C\), then the equality \([D^b(F_1(\mathcal Q/S))] = [\mathcal R_{odd}] + (2g-2) [\mathcal B]\) holds in \(K_0(\mathrm{Cat}/\mathbb C)\) by \cite[Theorem 1.2 and Proposition 7.1]{Shah-rel-Fr}, where \([\mathcal B]=(g-1)[D^b(S)]+[D^b(S,\mathcal{C}l_0)]\) and \(\mathcal R_{odd}\) is a residual category.
    Let \(\widehat{S}\to S\) be the root stack with \(\mathbb Z/2\mathbb Z\)-stabilizers at the \(\delta\) points in \(S\) parametrizing the singular fibers (see \cite[Example 2.2]{Kuznetsov08}). Then \(D^b(S,\mathcal{C}l_0)\cong D^b(\widehat{S})\) by \cite[Corollary 3.16]{Kuznetsov08}, and the class of \(D^b(\widehat{S})\) in \(K_0(\mathrm{Cat}/\mathbb C)\) is equal to \([D^b(S)] + [D^b(\delta \text{ points})] = [D^b(S)]+\delta\cdot 1\) by \cite[Theorem 1.6]{IshiiUeda} (see also \cite[Proof of Proposition 8.2]{Shah-rel-Fr}), so
    \[[D^b(F_1(\mathcal Q/S))] = [\mathcal R_{odd}] + 2g(g-1)[D^b(S)]+(2g-2)\delta\cdot 1 .\]
    Therefore \([\mathcal R_{odd}]=2\delta\cdot 1\) in \(K_0(\mathrm{Cat}/\mathbb C)\).
\end{proof}

In the case of the relative Fano scheme of lines, Theorem~\ref{thm:relative-F_r} reduces to the following formulas:

\begin{cor}\label{cor:relative-F_1}
    Over an algebraically closed field \(k\) of characteristic \(\neq 2\), let \(S\) be a smooth connected curve, let \(n\geq 1\), and let \(\pi\colon\mathcal Q\to S\) be a fibration of \(n\)-dimensional quadrics over \(k\). Assume \(\pi\) has simple degeneration and that the degeneracy locus of \(\pi\) is a smooth divisor.
    Let \(\delta\) be the number of singular fibers of \(\pi\). If \(n\) is even, let \(\widetilde{S}\) be the curve obtained as the Stein factorization of \(F_{n/2}(\mathcal Q/S)\to S\). Then the following equalities hold in \(K_0(\mathrm{Var}/k)\):
    \[ [F_1(\mathcal Q/S)] = \begin{cases}
        [\widetilde{S}] \left(\sum_{i=0}^{2g-1} \L^i\right) \L^{g-1} + [S]\left(\sum_{i=0}^{g-2} \L^i \right)\left(\sum_{i=0}^{g-1} \L^{2i} \right)(\L^{g+1}+1) & \text{if \(n=2g\) is even,} \\
        \delta \left(\L^{g-1} + \sum_{i=0}^{2g-2} \L^i \right)\L^{g-1} + [S](\L+1)\left(\sum_{i=0}^{g-2}\L^{2i}\right)\left(\sum_{i=0}^{g-1}\L^{2i}\right) & \text{if \(n=2g-1\) is odd.}
    \end{cases}\]
    (If \(g=1\) then the sums \(\sum_{i=0}^{g-2} \L^i\) and \(\sum_{i=0}^{g-2} \L^{2i}\) are understood to be 0.)
    In particular, if \(S=\P^1\) then
    \[[F_1(\mathcal Q/\P^1)] = \begin{cases}
        [\widetilde{S}][\P^{2g-1}]\L^{g-1} + [\P^{2g-1}]([\P^{2g-1}] - \L^g - \L^{g-1}) & \text{if \(n=2g\) is even,} \\
        \delta([\P^{2g-2}]+\L^{g-1})\L^{g-1} + [\P^{2g-2}]^2-\L^{2g-2} & \text{if \(n=2g-1\) is odd.}
    \end{cases}\]
\end{cor}

\begin{proof}
    The result follows from Theorem~\ref{thm:relative-F_r} and simplifying the expressions for \([\OG(2,n+2)]\) using Lemma~\ref{lem:class-of-OG}. In the case \(S=\P^1\), the expressions in the corollary statement follow by simplifying the second term.
\end{proof}

\section{Geometry of hyperbolic reductions of a pencil}\label{sec:hyperbolic-reductions-pencil}

In this section, we specialize the results of Section~\ref{sec:prelim-hyperbolic-red} to the case of pencils and recall some results from \cite{JS24}.
Throughout this section, let \(\phi\colon\mathcal Q\to\P^1\) be a pencil of quadrics in \(\P^N\) over a field \(k\) of characteristic \(\neq 2\), and assume the base locus \(X\) is smooth.
If \(N=2g+1\) is odd, then the discriminant double cover (Lemma~\ref{lem:hyperbolic-reduction-properties-not-pencil}\eqref{item:hyperbolic-reduction-even-maximal}) is a genus \(g\) curve \(C\to\P^1\).

Over an algebraic closure of \(k\), \(X\) contains \(r\)-planes if and only if \(0\leq r\leq \lfloor\frac{N}{2}\rfloor - 1\) (Lemma~\ref{lem:F_r(X)}).
Let \(\rplane\) be an \(r\)-plane on \(X\) defined over \(k\).
After a suitable choice of coordinates on \(\P^N\), we may assume \(\rplane=\{x_{r+1}=\cdots=x_N=0\}\). Then
\begin{equation}\label{eqn:X}
    X = \left\{ \begin{pmatrix} l_{00} & \dots &l_{0r} & q_0\\ l_{10} & \dots & l_{1r} & q_1 \end{pmatrix} \begin{pmatrix} x_0\\ \vdots\\ x_r\\ 1 \end{pmatrix} =0 \right\} \subset \P^N
\end{equation}
for some linear forms \(l_{ij}\in k[x_{r+1},\ldots,x_N]\) and quadratic forms \(q_i\in k[x_{r+1},\ldots,x_N]\).
Define
\begin{equation}\label{eqns:P(r)-and-Q(r)}
    \begin{split}
        \mathcal P^{(r)}_{\rplane} &\coloneqq \left\{ \begin{pmatrix} s & t \end{pmatrix} \begin{pmatrix} l_{00} & \dots &l_{0r}\\ l_{10} & \dots & l_{1r} \end{pmatrix} =0 \right\} \subset \P^1_{[s:t]}\times \P^{N-r-1}_{[x_{r+1}:\cdots:x_N]}, \\
        \mathcal Q^{(r)}_{\rplane} &\coloneqq \left\{ \begin{pmatrix} s & t \end{pmatrix} \begin{pmatrix} l_{00} & \dots &l_{0r} & q_0\\ l_{10} & \dots & l_{1r} & q_1 \end{pmatrix} =0 \right\} \subset \P^1_{[s:t]}\times \P^{N-r-1}_{[x_{r+1}:\cdots:x_N]} .
    \end{split}
\end{equation}
Then the first projection \(\phi^{(r)}_{\rplane}\colon \mathcal Q^{(r)}_\Lambda\to\P^1\) is the hyperbolic reduction of \(\phi\) with respect to the nondegenerate \(r\)-section defined by \(\rplane\), and \(\mathcal P^{(r)}_{\rplane}\to\P^1\) is the \(\P^{N-2r-2}\)-bundle \(\mathbb P_S(\mathcal F_{r+1}^\perp/\mathcal F_{r+1})\) from Definition~\ref{defn:hyperbolic-red}.
Define
\begin{equation}\label{eqns:Z-and-Y}
    \begin{split}
        \mathcal P_{Y_{\rplane}} & \coloneqq \mathcal P^{(r)}_{\rplane} \cap \{l_{ij}=0 \mid 0\leq i\leq 1, 0\leq j\leq r\} = \{l_{ij}=0 \mid 0\leq i\leq 1, 0\leq j\leq r\} \subset \P^1_{[s:t]}\times \P^{N-r-1}_{[x_{r+1}:\cdots:x_N]} , \\
        \mathcal Q_{Y_{\rplane}} & \coloneqq \mathcal Q^{(r)}_{\rplane} \cap \{l_{ij}=0 \mid 0\leq i\leq 1, 0\leq j\leq r\} , \\
        Y_{\rplane} & \coloneqq \{l_{ij}=q_i=0 \mid 0\leq i\leq 1, 0\leq j\leq r\} \subset \P^{N-r-1}_{[x_{r+1}:\cdots:x_N]};
    \end{split}
\end{equation}
then \(\P^1\times Y_{\rplane}=\mathcal Q^{(r)}_{\rplane}\cap\{l_{ij}=q_i=0 \mid 0\leq i\leq 1, 0\leq j\leq r\}\).
(This notation will be explained in Lemma~\ref{lem:general-l-hyperbolic-red}.)

We have the following diagrams, where the right-hand diagram is from Proposition~\ref{prop:KS-hyperbolic-reduction}\eqref{item:KS-hyperbolic-red-blow-up-diagram} and \(h=\tilde{h}|_{\widetilde{\mathcal Q}}\):
\begin{equation}\label{eqn:hyperbolic-red-blow-up-diagram}
\begin{tikzcd}[column sep=tiny]
&\P^1\times \P_{\P^{N-r-1}}(\O^{\oplus r+1}\oplus \O(-1))\arrow[ld, "\bl_{\P^1\times \rplane}"'] \arrow[rd, "\tilde{h}"] & \\
\P^1\times \P^{N} \arrow[rr, "\pi_{\P^1\times \rplane}", dashed]& & \P^1\times \P^{N-r-1} ,
\end{tikzcd}
\qquad
\begin{tikzcd}[column sep=scriptsize]
& \widetilde{\mathcal{Q}} \arrow[ld, "\bl_{\P^1\times \rplane}"'] \arrow[rd, "h"] & \\
\mathcal{Q} \arrow[rr, "\pi_{\P^1\times\rplane}|_{\mathcal{Q}}", dashed]& & \P^1\times \P^{N-r-1}           .
\end{tikzcd}
\end{equation}

Write \(g=\lfloor\frac{N}{2}\rfloor\). Then the hyperbolic reduction of \(\phi\) with respect to any \((g-1)\)-plane on \(X\) is \(C\) if \(N=2g+1\) is odd, and is a reduced finite scheme of length \(2g+1\) if \(N=2g\) is even \cite[Lemma 3.1(6)]{JS24}. So in the case of pencils, the formula in Corollary~\ref{cor:formula-class-of-hyperbolic-reduction} specializes to:
\begin{cor}\label{cor:class-of-Qr-formula}
    Let \(k\) be an algebraically closed field of characteristic \(\neq 2\), let \(0\leq r\leq\lfloor\frac{N}{2}\rfloor - 2\), and let \(\rplane\) be an \(r\)-plane on \(X\). Then in \(K_0(\mathrm{Var}/k)\),
    \begin{equation*}
        \begin{split}
            [\mathcal Q^{(r)}_{\rplane}] %&= [\P^1][\P^{g-r-2}](1+\L^{N-g-r-1})+[\mathcal Q^{(g-1)}_s]\L^{g-r-1} \\
            &= \begin{cases}
                [\P^1][\P^{g-r-2}](1+\L^{g-r})+[C]\L^{g-r-1} & \text{if }N=2g+1, \\
                [\P^1][\P^{g-r-2}](1+\L^{g-r-1})+(2g+1)\L^{g-r-1} & \text{if }N=2g.
            \end{cases}
        \end{split}
    \end{equation*}
\end{cor}

In the remainder of this section, we further describe the geometry of the hyperbolic reductions of \(\phi\colon\mathcal Q\to\P^1\).
Recall from Lemma~\ref{lem:hyperbolic-reduction-properties-not-pencil}\eqref{item:hyperbolic-reduction-not-pencil-embedding} that \(\mathcal Q^{(r)}_{\rplane}\) embeds into \(F_{r+1}(\mathcal Q/\P^1)\) with image the locus of isotropic \((r+1)\)-planes of \(\phi\) containing \(\rplane\). The loci \(\mathcal Q_{Y_{\rplane}}\) and \(\P^1\times Y_{\rplane}\) have the following additional geometric descriptions:
\begin{lem}\label{lem:r-planes-in-hyperbolic-reduction-loci}
    Let \(\rplane\) be an \(r\)-plane on \(X\).
\begin{enumerate}
	\item\label{item:image-of-Z-in-rel-Fr} \cite[Lemma 3.1(4)]{JS24} The embedding in Lemma~\ref{lem:hyperbolic-reduction-properties-not-pencil}\eqref{item:hyperbolic-reduction-not-pencil-embedding} induces an isomorphism \[\mathcal Q^{(r)}_{\rplane}\setminus\mathcal Q_{Y_{\rplane}} \xrightarrow{\sim} \{ M \in F_r(X) \mid \dim(\rplane\cap M)=r-1 \text{ and }\langle\rplane,M\rangle\not\subset X\}.\]
    \item\label{item:image-of-E-in-rel-Fr} Under the embedding in Lemma~\ref{lem:hyperbolic-reduction-properties-not-pencil}\eqref{item:hyperbolic-reduction-not-pencil-embedding}, the image of \(\P^1\times Y_{\rplane}\) is the locus of relative \((r+1)\)-planes contained in \(X\). That is, \[Y_{\rplane}\cong F_{r+1}(X)_{\rplane}\] where \(F_{r+1}(X)_{\rplane}\) is the Fano scheme of \((r+1)\)-planes on \(X\) containing \(\rplane\).
    \item\label{item:relative-m-planes-in-Qr} Let \(0 \leq e \leq \lfloor\frac{N}{2}\rfloor - r - 1\). The \(\P^{r+1}\)-bundle \(h^{-1}(\mathcal Q^{(r)}_{\rplane}) \to \mathcal Q^{(r)}_{\rplane}\) in diagram~\eqref{eqn:hyperbolic-red-blow-up-diagram} induces an embedding \(F_{e}(\mathcal Q^{(r)}_{\rplane}/\P^1)\hookrightarrow F_{r+e+1}(\mathcal Q/\P^1)\) over \(\P^1\) whose image is the locus \(F_{r+e+1}(\mathcal Q/\P^1)_{\rplane}\) of isotropic \((r+e+1)\)-planes of \(\phi\) containing \(\rplane\). Furthermore, this induces an isomorphism \(F_e(Y_{\rplane})\cong F_{r+e+1}(X)_{\rplane}\).
\end{enumerate}
\end{lem}

\begin{proof}
    Part~\eqref{item:image-of-E-in-rel-Fr} is a special case of~\eqref{item:relative-m-planes-in-Qr}.
    For~\eqref{item:relative-m-planes-in-Qr},
    the first statement is contained in Lemma~\ref{lem:hyperbolic-reduction-properties-not-pencil}\eqref{item:hyperbolic-reduction-not-pencil-embedding}.
    For the description of \(F_e(Y_{\rplane})\), let \(L=\{f_1=\cdots=f_{N-r-e-1}=0\} \subset \mathcal Q_{[s:t]}\) be an \((r+e+1)\)-plane containing \(\rplane\). Then, since \(\rplane\subset L\), we have that \(L\) is contained in \([s:t]\times X\) if and only if \((x_0l_{00}+\cdots+x_rl_{0r}+q_0, x_0l_{10}+\cdots+x_rl_{1r}+q_1) \subset (f_1,\ldots,f_{N-r-e-1})\) if and only if \(l_{0j},l_{1j},q_0,q_1\in (f_1,\ldots,f_{N-r-e-1})\) if and only if \(h(\widetilde{L_{[s:t]}}) \subset [s:t]\times Y_{\rplane}\), where \(\widetilde{L_{[s:t]}}\) denotes the strict transform of \([s:t]\times L\) under \(\bl_{\P^1\times\rplane}\).
    Let \(F_{r+e+1}(\P^1\times X/\P^1)_{\rplane} \subset F_{r+e+1}(\P^1\times X/\P^1)\) be the locus of relative \((r+e+1)\)-planes that contain \(\rplane\).
    This shows that \(F_e(\P^1\times Y_{\rplane}/\P^1) \cong \P^1 \times F_e(Y_{\rplane})\) and \(F_{r+e+1}(\P^1\times X/\P^1)_{\rplane}\cong \P^1 \times F_{r+e+1}(X)_{\rplane}\) are isomorphic over \(\P^1\), so \(F_e(Y_{\rplane})\cong F_{r+e+1}(X)_{\rplane}\).
\end{proof}

The following lemma explains the choice of notation for \(\mathcal Q_{Y_{\rplane}}\) and \(Y_{\rplane}\) in~\eqref{eqns:Z-and-Y}.

\begin{lem}\label{lem:general-l-hyperbolic-red}
    Let \(k\) be an algebraically closed field of characteristic \(\neq 2\), let \(0 \leq r \leq \lfloor\frac{N}{2}\rfloor -1\), and let \(\rplane\) be a general \(r\)-plane on \(X\). Then the following hold:
    \begin{enumerate}
        \item\label{item:general-l-Z-dimension} Assume \(N-3r-3 \geq 0\), and choose coordinates on \(\P^N\) so that \(\rplane=\{x_{r+1}=\cdots=x_N=0\}\) and \(X\) has the form~\eqref{eqn:X}. Then the locus \(\{l_{ij} = 0 \mid 0\leq i\leq 1, 0\leq j\leq r\} \subset \mathbb P^{N-r-1}_{[x_{r+1}:\cdots:x_N]}\) is isomorphic to \(\mathbb P^{N-3r-3}\).
        \item\label{item:general-l-Y-smooth-ci} (cf. \cite[Proposition 2.14(f)]{CT-intersection-quadrics}) \(Y_{\rplane}\) is a smooth complete intersection of two quadrics in \(\mathbb P^{N-3r-3}\) with associated pencil \(\mathcal Q_{Y_{\rplane}}\to\P^1_{[s:t]}\) (unless \(N-3r-5<0\), in which case \(Y_{\rplane}\) is empty).
    \end{enumerate}
    In particular, if \(\rplane\) is general and \(N-3r-3 > 0\), then \(\mathcal Q_{Y_{\rplane}}\to\P^1\) is a pencil of quadrics in \(\mathbb P^{N-3r-3}\) with smooth base locus. If \(\rplane\) is general and \(N-3r-3=0\), then \(\mathcal Q_{Y_{\rplane}}\) is a (reduced) point.
\end{lem}

\begin{proof}
    For~\eqref{item:general-l-Y-smooth-ci}, for an \(r\)-plane \(\rplane'\) on \(X\) we identify \(Y_{\rplane'}\cong F_{r+1}(X)_{\rplane'}\) using Lemma~\ref{lem:r-planes-in-hyperbolic-reduction-loci}\eqref{item:image-of-E-in-rel-Fr}. If \(r=\lfloor\frac{N}{2}\rfloor-1\) then \(X\) contains no \((r+1)\)-planes by Lemma~\ref{lem:F_r(X)} so \(Y_{\rplane}=\emptyset\).
    So we may assume \(r\leq\lfloor\frac{N}{2}\rfloor-2\). First assume \(r<\frac{N}{2}-2\). Define \(I=\{(\rplane,L) \mid \rplane \subset L\} \subset F_r(X)\times F_{r+1}(X)\). By Lemma~\ref{lem:F_r(X)}, \(F_{r+1}(X)\) (resp. \(F_r(X)\)) is smooth and geometrically connected of dimension \((r+2)(N-2r-4)\) (resp. \((r+1)(N-2r-2)\)), so \(I\) is smooth of dimension \((r+2)(N-2r-4)+r+1\). So, if \(N-3r-5 < 0\), then \(\dim I < \dim F_r(X)\) so \(p_1\) is not dominant and its general fiber is empty. If \(N-3r-5\geq 0\), then a general fiber of the first projection \(p_1\colon I \to F_r(X)\) has dimension \(N - 3 r - 5\).
    For the smoothness of a general fiber in this case, the locus over which the fibers of \(p_1\) are smooth is open. This locus is nonempty by generic smoothness in characteristic zero, and by Example~\ref{exmp:smooth-Y-example} below in characteristic \(p>2\).

    Finally, assume \(N=2g\) is even and \(r=g-2 \geq 0\), and let \(I\subset F_{g-2}(X)\times F_{g-1}(X)\) be the incidence variety defined above. Then \(F_{g-1}(X)\) is \(2^{2g}\) distinct points by \cite[Theorem 3.8]{Reid-thesis} and \(F_{g-2}(X)\) is smooth and geometrically connected of dimension \(2g-2\). Then \(I\) is the disjoint union of \(2^{2g}\) copies of \(\P^{g-1}\), and each component \(\P^{g-1}\) has image of dimension at most \(g-1 < \dim F_{g-2}(X)\), so \(p_1\colon I\to F_{g-2}(X)\) is not dominant, and hence \(Y_\Lambda=\emptyset\) for general \(\Lambda\in F_{g-2}(X)\).
    
    For~\eqref{item:general-l-Z-dimension}, for $x\in X$, denote by $T_{x}X\subset \P^N$ the tangent space to $X$ at $x$.
    For each \(\rplane\in F_r(X)\), after choosing coordinates so that \(\rplane=\{x_{r+1}=\cdots=x_N=0\}\) and \(X\) has the form~\eqref{eqn:X}, the locus \(\{l_{ij}=0\} \subset \P^N\) is equal to \(\bigcap_{x\in \Lambda} T_{x}X\), and the locus \(\{l_{ij}=0\} \subset \P^{N-r-1}\) is equal to \(\bigcap_{x\in \Lambda}T_{x}X/\Lambda\). Note that if \(\Lambda\neq \bigcap_{x\in \Lambda} T_{x}X\), then \(\dim(\bigcap_{x\in \Lambda} T_{x}X)-\dim(\bigcap_{x\in \Lambda}T_{x}X/\Lambda)=r+1\). Define $I'\coloneqq \{(\Lambda,p)\mid p\in \bigcap_{x\in \Lambda}T_{x}X\}\subset F_r(X)\times \P^N$.
    Fibers of the first projection $p_1'\colon I'\to F_r(X)$ are defined by $2r+2$ linear equations in $\P^N$, hence their dimensions are at least $N-2r-2$.
    In addition, Example~\ref{exmp:smooth-Y-example} shows that if $N-3r-3\geq 0$, then there exists a fiber of $p_1'$ of dimension $N-2r-2$.
    The upper semicontinuity of the dimension of fibers of $p_1'$ then shows that if $N-3r-3\geq 0$, then a general fiber of $p_1'$ is $\P^{N-2r-2}$.
    The assertion follows.
\end{proof}
\begin{defn}\label{defn:C_Y-curve}
    If \(N=2g+1\) is odd, \(0\leq r \leq \frac{2g-5}{3}\) is odd, \(\rplane\) is a general \(r\)-plane on \(X\), then \(2g-3r-2\) is odd and the smooth complete intersection of two quadrics \(Y_{\rplane} \subset \P^{2g-3r-2}\) (Lemma~\ref{lem:general-l-hyperbolic-red}\eqref{item:general-l-Y-smooth-ci}) has dimension at least 1. Let \(C_{\rplane}\) denote the associated genus \(g-\frac{3}{2}(r+1)\) curve.
\end{defn}

\begin{lem}\label{lem:general-r-plane-contained-in-r+1-plane}
    Assume $k$ is an algebraically closed field of characteristic \(\neq 2\).
    \begin{enumerate}
        \item\label{odd:even-general-r-plane-contained-in-r+1-plane} Assume \(N=2g+1\) is odd and \(0\leq r\leq g-2\).
        A general \(r\)-plane on \(X\) is contained in an \((r+1)\)-plane on \(X\) \(\iff\)
        every \(r\)-plane on \(X\) is contained in an \((r+1)\)-plane on \(X\) \(\iff\)
        \(g\geq\frac{3}{2}r+2\).
        \item\label{it:even-general-r-plane-contained-in-r+1-plane} Assume \(N=2g\) is even and \(0 \leq r \leq g-3\). A general \(r\)-plane on \(X\) is contained in an \((r+1)\)-plane on \(X\) \(\iff\)
        every \(r\)-plane on \(X\) is contained in an \((r+1)\)-plane on \(X\) \(\iff\)
        \(g\geq\frac{3}{2}r+\frac{5}{2}\).
    \end{enumerate}
\end{lem}

\begin{proof}
    Fix \(0\leq r < \frac{N}{2}-2\), and let \(I = \{(\ell,\Lambda) \mid \ell \subset \Lambda\} \subset F_r(X)\times F_{r+1}(X)\). Then \(F_r(X)\) and \(F_{r+1}(X)\) are integral by Lemma~\ref{lem:F_r(X)}. Since the fibers of the second projection \(I\to F_{r+1}(X)\) are each isomorphic to \(\mathbb P^{r+1}\), \(I\) is irreducible. For a general \(r\)-plane \(\ell\) on \(X\), the fiber of the first projection \(I\to F_r(X)\) over \(\ell\) is isomorphic to a smooth complete intersection of two quadrics if \(N-3r-5\geq 0\) and is empty otherwise by Lemma~\ref{lem:general-l-hyperbolic-red}\eqref{item:general-l-Y-smooth-ci}. That is, the first projection is dominant if and only if \(N-3r-5\geq 0\), and by irreducibility of \(I\) and \(F_r(X)\) this is equivalent to surjectivity of the first projection.
\end{proof}

\begin{lem}\label{lem:general-l-vertices-avoid-Z}
    Assume \(k\) is an algebraically closed field of characteristic \(\neq 2\) and \(N\geq 2r+3\). If \(\rplane\) is a general \(r\)-plane on \(X\), then the set of vertices of the \(N+1\) singular fibers of \(\phi^{(r)}_{\rplane}\) is disjoint from \(\mathcal Q_{Y_{\rplane}}\).
\end{lem}

\begin{proof}
    Let \(\rplane\) be any \(r\)-plane on \(X\). We first show that if \(t\in\P^1\) is a point such that \(\mathcal Q_t\) is singular, then the singular point of \(\mathcal Q^{(r)}_t\) is contained in \(\mathcal Q_{Y_{\rplane}}\) if and only if \(\rplane\) is contained in a certain hyperplane.
    
    For this, first we show that the singular points of singular fibers of \(\phi^{(r)}_{\rplane}\colon \mathcal Q^{(r)}_{\rplane}\to\P^1\) are the projections under \(\pi_{\P^1\times \rplane}\) of the singular points of singular fibers of \(\phi\colon \mathcal Q\to\P^1\).
    Let \(W_{\rplane}\subset V\) be a vector subspace so that \(\rplane=\P(W_{\rplane}) \subset \P(V)=\P^N\), 
    and let \(U_{\rplane}\) be such that \(V=W_{\rplane}\oplus U_{\rplane}\).
    After a coordinate change to assume \(\rplane=\{x_{r+1}=\cdots=x_N=0\}\), the fiber of the pencil \(\mathcal Q\to\P^1\) over \(t\in \P^1\) is defined by
    \[M_t = \begin{pmatrix} 0_{(r+1)\times (r+1)} & L_t \\ L_t^T & A_t\end{pmatrix}\]
    where \(L_t\) is an \((r+1)\times(N-r)\) matrix and \(A_t\) is an \((N-r)\times(N-r)\) matrix.
    (If \(t=[s:t]\), then the rows of \(L_t\) are given by the coefficients of \(\frac{1}{2}(sl_{0j}+tl_{1j})\), and the matrix \(A_t\) corresponds to \(sq_0+tq_1\), where \(l_{ij},q_i\) are as defined in~\eqref{eqn:X}.)
    Let \(B_t\) be the bilinear form corresponding to the matrix \(M_t\).
    Then \(\Ker(L_t)\cong W_{\rplane}^{\perp_t}/W_{\rplane}\) (where the orthogonal complement is with respect to the quadratic form defined by \(M_t\)), and \(\mathcal Q^{(r)}_{\rplane,t} = \P(\{u \in \Ker(L_t) \mid u^T(A_t)u=0\})\).
    For \(t\in\P^1\), by Proposition~\ref{prop:KS-hyperbolic-reduction}\eqref{item:KS-hyperbolic-red-degeneracy} the fiber \(\mathcal Q_t\) of \(\phi\) over \(t\) is singular if and only if the fiber \(\mathcal Q^{(r)}_{\rplane,t}\) of \(\phi^{(r)}_{\rplane}\) over \(t\) is singular.
    Let \(t\in\P^1\) be a point such that \(\mathcal Q_t\) is singular, let \(p_t=\P(\mathrm{span}\{v_t\})\) be the singular point of \(\mathcal Q_t\), and write \(v_t=w_t+u_t\) where \(w_t\in W_{\rplane}\) and \(u_t\in U_{\rplane}\). Then, for any \(v\in V\),
    \[0=v^T M_t v_t = v^T M_t (w_t+u_t) = v^T\begin{pmatrix} L_t u_t \\ L_t^T w_t + A_t u_t\end{pmatrix},\]
    so \(L_t u_t = 0_{(r+1)\times 1}\) and \(L_t^T w_t + A_t u_t = 0_{(N-r) \times 1}\).
    This implies that for any \(u \in \Ker(L_t)\), we have \(u^T A_t u_t = 0\), so \(\P(\mathrm{span}\{u_t\})=\pi_{\P^1\times \rplane}(p_t)\) is the singular point of \(\mathcal Q^{(r)}_{\rplane,t}\).

    If \(p_t=\P(\mathrm{span}\{v_t\})\) is the singular point of \(\mathcal Q_t\), then the linear form defined by \(B_t(v_t,-)\colon V \to k\) is zero. Let \(t'\neq t\), and let \(f_{t'}\) be the linear form defined by \(B_{t'}(v_t,-)\). Then the hyperplane \(H_t\coloneqq \{f_{t'}=0\} \subset\P(V)\) is independent of the choice of \(t'\neq t\), and
    we claim that \(\pi_{\P^1\times\rplane}(p_t) \in \mathcal Q_{Y_{\rplane}}\) if and only if \(\rplane\subset H_t\). Indeed, by~\eqref{eqns:Z-and-Y}, \(\pi_{\P^1\times\rplane}(p_t) \in \mathcal Q_{Y_{\rplane}}\) if and only if \(L_{[0:1]} u_t = L_{[1:0]} u_t = 0\) if and only if \(B_{[0:1]}(u_t, W_{\rplane})=B_{[1:0]}(u_t,W_{\rplane})=0\) if and only if \(B_{[0:1]}(v_t, W_{\rplane})=B_{[1:0]}(v_t,W_{\rplane})=0\) (since \(\rplane\subset X\)), and this last condition is equivalent to \(\rplane \subset H_t\).
    
    Now we show that if \(\rplane\) is a general \(r\)-plane on \(X\), then \(\rplane\not\subset H_t\) for any \(t\) in the degeneracy locus of \(\phi\). By \cite[Proposition 2.1]{Reid-thesis}, we may choose coordinates on \(\P^N\) so that \(X\) is defined by \(\sum_{i=0}^N x_i^2\) and \(\sum_{i=0}^N \lambda_i x_i^2\) where the \(\lambda_i\) are distinct. Then the singular points of the singular fibers of \(\phi\colon\mathcal Q\to\P^1\) are given by \(([-\lambda_i:1],[0:\cdots:0:1_i:0:\cdots:0])\), so for each \(0\leq i\leq N\), we have \(H_{[-\lambda_i:1]}=\{x_i=0\}\).
    By a dimension count, a general \(r\)-plane on \(X\) is not contained in \(H_{[-\lambda_i:1]}\) for any \(i\), so this proves the result.
    (Indeed, \(\dim F_r(X)=(r+1)(N-2r-2)\) by Lemma~\ref{lem:F_r(X)}. On the other hand, for each \(0 \leq i \leq N\), the intersection \(X\cap H_{[-\lambda_i:1]}\) is a smooth complete intersection of two quadrics in \(H_{[-\lambda_i:1]} \cong \P^{N-1}\) by \cite[Proposition 2.1]{Reid-thesis}, so \(\dim F_r(X\cap H_{[-\lambda_i:1]}) = (r+1)(N-2r-3)\) by Lemma~\ref{lem:F_r(X)}.)
\end{proof}

The following explicit example of an \(r\)-plane was used in the proof of Lemma~\ref{lem:general-l-hyperbolic-red}.

\begin{exmp}[{cf. \cite[Proof of Lemma 4.9]{Reid-thesis} for \(r=0\)}]\label{exmp:smooth-Y-example}
    Let \(X\subset\P^N\) be a smooth complete intersection of two quadrics over an algebraically closed field \(k\) of characteristic \(\neq 2\), and let \(r\) be such that 
    \(N-3r-3\geq 0\).
    By \cite[Proposition 2.1]{Reid-thesis} one may choose coordinates \([y_0:\cdots:y_N]\) on \(\P^N\) such that \(X\) is defined by \(\sum_{i=0}^N y_i^2 = \sum_{i=0}^N \lambda_i y_i^2=0\) where \(\lambda_i\in k^\times\) are pairwise distinct.
    For each \(0\leq j\leq r\), choose \(a_j,b_j,c_j\in k\) such that \[a_j^2 = \lambda_{3j+1}-\lambda_{3j+2}, \quad b_j^2 = \lambda_{3j+2}-\lambda_{3j}, \quad c_j^2 = \lambda_{3j}-\lambda_{3j+1}.\]
    Then \(a_j,b_j,c_j\neq 0\) and \(a_j^2+b_j^2+c_j^2=\lambda_{3j}a_j^2+\lambda_{3j+1}b_j^2+\lambda_{3j+2}c_j^2 = 0\).
    Define the \(r\)-plane \[\rplane = \{[a_0z_0: b_0z_0: c_0z_0: a_1z_1: b_1 z_1: c_1 z_1: \cdots : a_rz_r: b_rz_r: c_rz_r: 0: \cdots : 0 ] \mid [z_0:\cdots:z_r]\in\P^r\}.\]
    Note that \(\rplane \subset X\cap\{y_{3r+3}=\cdots=y_N=0\}\).

    After the coordinate change
    \[y_{3j}=a_j x_j + x_{r+2j+1}, \quad y_{3j+1} = b_j x_j + x_{r+2j+2}, \quad y_{3j+2} = c_j x_j \quad \text{ for }0 \leq j\leq r\]
    and \(y_i=x_i\) for \(3r+3\leq i\leq N\), we have \(\rplane = \{x_{r+1}=\cdots=x_N=0\}\) and \(X\) has the form~\eqref{eqn:X} where
    \begin{equation*}
        \begin{split}
            l_{0j} &= 2 a_j x_{r+2j+1} + 2 b_j x_{r+2j+2}, \\
            l_{1j} &= 2\lambda_{3j} a_j x_{r+2j+1} + 2\lambda_{3j+1} b_j x_{r+2j+2}, \\
            q_0 &= \sum_{j=0}^r (x_{r+2j+1}^2+x_{r+2j+2}^2) + \sum_{l=3r+3}^N x_l^2, \\
            q_1 &= \sum_{j=0}^r (\lambda_{3j}x_{r+2j+1}^2+\lambda_{3j+1}x_{r+2j+2}^2) + \sum_{l=3r+3}^N \lambda_l x_l^2 .
        \end{split}
    \end{equation*}
    For each \(0\leq j\leq r\), we have \(\{l_{0j}=l_{1j}=0\} = \{x_{r+2j+1}=x_{r+2j+2}=0\}\) since the matrix
    \[\begin{pmatrix}
        a_j & b_j \\ \lambda_{3j}a_j & \lambda_{3j+1} b_j \end{pmatrix}\]
    has determinant \(a_j b_j (\lambda_{3j+1}-\lambda_{3j}) \neq 0\).
    In particular,
    \(\{l_{0j}=l_{1j}=0 \mid 0\leq j\leq r\}\subset\mathbb P^{N-r-1}\) is isomorphic to \(\P^{N-3r-3}\).
    So
    \begin{equation*}\begin{split}
        \mathcal Q_{Y_{\rplane}} &= \left\{ s \sum_{l=3r+3}^N x_l^2 + t \sum_{l=3r+3}^N \lambda_l x_l^2 =0 \right\} \subset \P^1_{[s:t]}\times \P^{N-3r-3}_{[x_{3r+3}:\cdots:x_N]} , \\
        Y_{\rplane} &= \left\{ \sum_{l=3r+3}^N x_l^2 = \sum_{l=3r+3}^N \lambda_l x_l^2 =0 \right\} \subset \P^{N-3r-3}_{[x_{3r+3}:\cdots:x_N]}.
    \end{split}\end{equation*}
    In addition, if \(N-3r-5\geq 0\), then \(Y_{\rplane}\) is a smooth complete intersection of two quadrics in \(\P^{N-3r-3}\) by \cite[Proposition 2.1]{Reid-thesis}, and otherwise  \(Y_{\rplane}=\emptyset\).
\end{exmp}

Finally, we record the following lemma on possibly singular intersections of two quadrics.
\begin{lem}\label{lem:singular-intersection}
Let $Y\subset \P^M$ be a subvariety defined by two quadratic equations over an algebraically closed field of characteristic $\neq 2$, let $y\in Y$, and let $T_{y}Y\subset \P^M$ be the tangent space to $Y$ at $y$.
\begin{enumerate}
\item\label{Y-possibilities} The variety $Y$ is either a complete intersection of two quadrics, the union of a hyperplane and an $(M-2)$-dimensional linear subspace (possibly embedded in the hyperplane), a quadric, or $\P^M$.
\item\label{Y-tangentspace-dimension} We have $M-2\leq \dim T_{y}Y\leq M$, and $\dim T_{y}Y=M$ holds if and only if $Y$ is a cone with vertex $y$.
\item\label{Y-tangentspace-intersection} The intersection $Y\cap T_{y}Y$ is a cone, and the base of the cone may be identified with the set of lines on $Y$ passing through $y$.
\end{enumerate}
\end{lem}
\begin{proof}
Write $Y=Q_0\cap Q_1$ for some quadrics $Q_0,Q_1\subset \P^M$.

For~\eqref{Y-possibilities}, if the first case does not happen, then the quadratic forms defining \(Y\) have a common factor. If at least one of \(Q_0\) or \(Q_1\) is not equal to \(\P^M\), then this factor is either linear or quadratic, corresponding to the second and third cases, respectively. The case $Q_0=Q_1=\P^M$ gives the fourth case.

For~\eqref{Y-tangentspace-dimension}, since \(T_{y}Y=T_{y}Q_0\cap T_{y}Q_1\subset \P^M\) and each \(T_{y}Q_i\) is either a hyperplane or \(\P^M\), we have $M-2\leq \dim T_{y}Y\leq M$. 
We have $\dim T_{y}Y=M$ if and only if $\dim T_{y}Q_0=\dim T_{y}Q_1=M$,  and the latter is equivalent to the statement that $Q_0$ and $Q_1$ are cones with vertex $y$, which is equivalent to $Y$ being a cone with vertex $y$. (Indeed, after a coordinate change we may assume \(y=[1:0:\cdots:0]\), so that \(Q_i\) is defined by an equation of the form \(x_0 l_i(x_1,\ldots,x_M)+q_i(x_1,\ldots,x_M)\). Then \(Q_i\) is a cone with vertex \(y\) if and only if \(l_i=0\), and \(Y\) is a cone with vertex \(y\) if and only if \(l_0=l_1=0\).)

For~\eqref{Y-tangentspace-intersection}, $Y\cap T_{y}Y\subset T_{y}Y$ is an intersection of two quadrics with $T_{y}(Y\cap T_{y}Y)=T_{y}Y$, hence by~\eqref{Y-tangentspace-dimension}, it is a cone with vertex $y$.
Let \(H\subset T_{y}Y\) be a hyperplane not containing \(y\); then \(Y\cap T_{y}Y\) is the cone with vertex \(y\) over \(Y\cap T_{y}Y\cap H\).
Since any line \(\ell\) of $Y$ passing through $y$ lies on $T_{y}Y$, it lies on $Y\cap T_{y}Y$, and furthermore it is equal to \(\langle y,\ell\cap H\rangle\) for the uniquely determined point \(\ell\cap H \in Y\cap T_{y}Y\cap H\).
This proves the assertion.
\end{proof}

\section{The class of \texorpdfstring{\(X\)}{X}}\label{sec:r=0}

In this section, we compute the class of \(X\) in the Grothendieck ring of varieties.
When \(N=2g+1\) is odd, this was proven up to multiplication by \(\L\) in \cite[Lemma 3.14]{BBFGHLPRAS}, and when \(N=5\) the result was known by a classical description of the blow-up of \(X\) along a line (see \cite[Remark 3.15]{BBFGHLPRAS}).
Our argument for general \(N\) uses the \(r=0\) case of Theorem~\ref{thm:js-bir-description-F_r(X)}, i.e., a birational map \(X\dashrightarrow\mathcal Q^{(0)}\), which was first observed in \cite{CTSSD}.

\begin{prop}\label{prop:r=0-case}
Over an algebraically closed field \(k\) of characteristic $\neq 2$, let \(N\geq 2\), and let \(\mathcal Q\to\P^1\) be a pencil of quadrics in \(\P^N\) whose base locus \(X\) is smooth of dimension \(N-2\). Let \(\mathcal Q^{(0)}\to\P^1\) be the hyperbolic reduction of \(\mathcal Q\to\P^1\) with respect to any point of \(X\).
If \(N=2g+1\) is odd, let \(C\) be the genus \(g\) curve associated to \(X\).
We have the following equalities in $K_0(\mathrm{Var}/k)$ (where by convention we take \(\P^{-1}=\emptyset\)):
\begin{equation*}\begin{split}
    [X] &= [\mathcal Q^{(0)}]-[\P^{N-3}]+1 \\
    &= \begin{cases}
          [C]\L^{g-1} + [\P^{2g-1}] - \L^{g-1}[\P^1] & \text{if $N=2g+1$ is odd}, \\
          [\P^{2g-2}] +(2g+1)\L^{g-1} & \text{if $N=2g$ is even}.
    \end{cases}
\end{split}\end{equation*}
\end{prop}

\begin{proof}
    First we address the cases \(N=2,3\). If \(N=2\), then \(g=1\), \(X\) is 4 distinct points, and \(\mathcal Q^{(0)}\) is 3 distinct points by \cite[Lemma 3.1(6)]{JS24}, so the desired equalities hold. If \(N=3\), then \(g=1\), \(X\cong C\) by \cite[Proposition 4.2]{Reid-thesis}, and \(C\cong\mathcal Q^{(0)}\) by \cite[Lemma 3.1(6)]{JS24}, so the desired equalities again hold.
    So we may assume \(N\geq 4\). Let \(p\in X\) be a general closed point. After a coordinate change, we may assume \(p=[1:0:\cdots:0]\), so that \(X = \{ x_0 l_0 + q_0 = x_0 l_1 + q_1 = 0 \}\) for some forms \(l_0, l_1, q_0, q_1 \in k[x_1,\ldots,x_N]\) with \(\deg(l_i)=1\) and \(\deg(q_i)=2\). Let \(\mathcal Q^{(0)}_p\to\P^1\) be the hyperbolic reduction of the pencil \(\mathcal Q\to\P^1\) with respect to \(p\). By \cite[Theorem 3.2]{CTSSD}, the projection away from \(p\) induces an isomorphism
    \[X\setminus\{l_0=l_1=0\} \xrightarrow{\sim} \mathcal Q^{(0)}_p \setminus \mathcal Q_{Y_p},\]
    where \(\mathcal Q_{Y_p}\) and \(Y_p\) are as defined in~\eqref{eqns:Z-and-Y}, and \(Y_p\cong F_1(X)_p\) by Lemma~\ref{lem:r-planes-in-hyperbolic-reduction-loci}\eqref{item:image-of-E-in-rel-Fr}.
    By Lemma~\ref{lem:singular-intersection}\eqref{Y-tangentspace-intersection}, the intersection \(X\cap\{l_0=l_1=0\}\) is 
    the cone over \(Y_p\) with vertex \(p\), and its class in \(K_0(\mathrm{Var}/k)\) is equal to \([Y_p]\L+1\). Since \(p\) is general, \(\mathcal Q_{Y_p}\) is the total space of a pencil of quadrics in \(\P^{N-3}\) with smooth base locus \(Y_p\) by Lemma~\ref{lem:general-l-hyperbolic-red}\eqref{item:general-l-Y-smooth-ci}, so \([\mathcal Q_{Y_p}]=[Y_p]\L+[\P^{N-3}]\). Therefore
    \[ [X] = [\mathcal Q^{(0)}_p] - [Y_p]\L - [\P^{N-3}] + [Y_p]\L+1 = [\mathcal Q^{(0)}_p] - [\P^{N-3}] + 1,\]
    and the class of \(\mathcal Q^{(0)}_p\) is independent of the choice of \(p\) by Proposition~\ref{prop:KS-hyperbolic-reduction}\eqref{item:KS-hyperbolic-reduction-indep}.
    Finally, the expressions in terms of \([C]\) and \(\L\) follow from Corollary~\ref{cor:class-of-Qr-formula} and simplifying the expression.
\end{proof}

\section{Computational lemmas for \texorpdfstring{\(F_1(X)\)}{F1X}}\label{sec:computations}

In this section, we collect several computational lemmas that will be used in Section~\ref{sec:r=1}. A reader willing to accept these results when they are invoked later can safely proceed to Section~\ref{sec:r=1} on a first reading.

\subsection{Exceptional loci of \texorpdfstring{\(F_1(X)\dashrightarrow\Sym^2\mathcal Q^{(1)}\)}{F1X to Sym2(Q1)}}\label{sec:computation-exceptional-cases}

Throughout this subsection, we work over an algebraically closed field \(\kbar\) of characteristic \(\neq 2\).

Let \(\mathcal Q\to\P^1\) be a pencil of quadrics in \(\P^N\) with smooth base locus \(X\).
First we state several lemmas that give explicit descriptions for the exceptional set of the birational map \(F_1(X)\dashrightarrow\Sym^2(\mathcal Q^{(1)}_\ell\setminus\mathcal Q_{Y_\ell})\) of \cite[Theorem 1.3]{JS24} and its inverse. First, we need the following definition, which will be stated in greater generality in Section~\ref{sec:artibrary-r} (Definition~\ref{defn:reid-type}).

\begin{defn}\label{defn:reid-type-r=1}
    A complete intersection \(X'\subset\mathbb P^3_{\kbar}\) of two quadrics is said to be \defi{singular of Reid type} if \(X'= l_1 + l_2 + l_3 + l_4\) is a union of lines such that for each \(l_i\), there are exactly two indices \(j,j'\) such that \(l_j\) and \(l_{j'}\) each meet \(l_i\) at a point, and furthermore \(X'\) has exactly 4 singular points, which are given by these intersections.
\end{defn}

\begin{lem}\label{lem:line-disjoint-from-l-cases}
    Assume \(N\geq 4\), and fix a line \(\ell\) on \(X\). Let \(m\) be a line on \(X\) such that \(\ell\cap m=\emptyset\). Then one may choose coordinates on \(\langle\ell,m\rangle\cong\mathbb P^3\) so that \(\mathcal Q\cap (\P^1\times \langle\ell,m\rangle)\) is spanned by the quadratic forms corresponding to
	\begin{equation}\label{eqn:pencil-intersect-P3-disjoint-lines} M_i = \begin{pmatrix} 0 & 0 & a_{i02} & a_{i03} \\ 0 & 0 & a_{i12} & a_{i13} \\ a_{i02} & a_{i12} & 0 & 0 \\ a_{i03} & a_{i13} & 0 & 0 \end{pmatrix} \quad \text{ for }i\in\{0,1\},
	\end{equation}
	and we have the following possibilities for the intersection \(X\cap\langle\ell,m\rangle\):
		\begin{center}
	\begin{tabular}{ |c|c|c|c| } 
	 \hline
	 & \(X\cap\langle\ell,m\rangle\) & \(\det(sM_0+tM_1)\) & Fibers of \(\mathcal Q\cap(\P^1\times\langle\ell,m\rangle) \to \mathbb P^1\) \\ 
	 \hline
	 \eqref{item:case-F_1-computation-P3} & \(\langle\ell,m\rangle\cong\mathbb P^3\) & 0 & \(\langle\ell,m\rangle \cong\mathbb P^3\) \\ 
	 \hline
	 \eqref{item:case-F_1-computation-quadric-surface} & \(\substack{P_1+P_2 \text{ union of two 2-planes} \\
	 \text{(necessarily \(P_1\cap P_2\neq\ell,m\))}}\)
	 	& 0 & one fiber \(\mathbb P^3\), others the same quadric \\ 
		\hline
	 \eqref{item:case-F_1-computation-quadric-surface} & smooth quadric surface & (degree 1)\(^4\) & one fiber \(\mathbb P^3\), others the same quadric \\ 
	 \hline
	 \eqref{item:case-F_1-computation-line-and-plane} & \(\substack{\text{\(\ell + P\) where \(\ell\cap P\) is a point, or} \\ \text{\(m + P\) where \(m\cap P\) is a point}}\) & 0 & \(\substack{\text{all rank 2 quadrics (one component of} \\ \text{each fiber is \(P\) and the other varies)}}\) \\ 
	 \hline
	 \eqref{item:case-F_1-computation-l-m-2m_1} & \(\substack{\text{\(\ell + m + 2 m'\) where} \\ \text{\(\ell\cap m'\), \(m\cap m'\) are each a point}}\) & (degree 1)\(^4\) & 1 singular fiber \(\langle\ell, m'\rangle +  \langle m, m'\rangle\) \\ 
	 \hline
	 \eqref{item:case-F_1-computation-Reid-type} & \(\substack{ \ell + m + l_1 + l_2 \\ \text{singular of Reid type}}\) & (degree 2 with distinct roots)\(^2\) & \(\substack{\text{2 singular fibers, each of rank 2} \\ \langle\ell, l_1\rangle + \langle m, l_2\rangle \text{ and } \langle\ell, l_2\rangle + \langle m, l_1\rangle}\) \\ 
	 \hline
	\end{tabular}
	\end{center}
\end{lem}

Lemma~\ref{lem:line-disjoint-from-l-cases} immediately implies:
\begin{cor}\label{cor:lines-meeting-l-span-P3-disjoint}
    Assume \(N\geq 4\), and fix a line \(\ell\) on \(X\). Let \(m_0,m_1\) be lines on \(X\) such that \(m_0\cap m_1=\emptyset\) and \(\ell\cap m_i\) is a point for each \(i\in\{0,1\}\). Then one may choose coordinates on \(\langle\ell,m_0,m_1\rangle = \langle m_0,m_1\rangle\cong\mathbb P^3\) so that \(\mathcal Q\cap(\P^1\times\langle\ell,m_0,m_1\rangle)\) is spanned by quadratic forms corresponding to matrices of the form~\eqref{eqn:pencil-intersect-P3-disjoint-lines}, and we have the following possibilities for the intersection \(X\cap\langle\ell,m_0,m_1\rangle\):
	\begin{center}
	\begin{tabular}{ |c|c|c| } 
	 \hline
	 \(X\cap\langle\ell,m_0,m_1\rangle = X\cap\langle m_0,m_1\rangle\) & \(\det(sM_0+tM_1)\) & Fibers of \(\mathcal Q\cap(\P^1\times\langle\ell,m_0,m_1\rangle) \to \mathbb P^1\) \\ 
	 \hline
	 \(\langle\ell,m_0,m_1\rangle\cong\mathbb P^3\) & 0 & \(\langle\ell,m_0,m_1\rangle \cong\mathbb P^3\) \\ 
	 \hline
	 \(\substack{\text{\(P_1+P_2\) with \(m_i \subset P_i\)} \\ \text{and \(P_1\cap P_2\neq m_i\)}}\)
	 	& 0 & one fiber \(\mathbb P^3\), others the same quadric \\ 
		\hline
	 smooth quadric surface & (degree 1)\(^4\) & one fiber \(\mathbb P^3\), others the same quadric \\ 
	 \hline
	 \(\substack{\text{\(m_0 + P\) where \(m_0\cap P\) is a point, or} \\ \text{\(m_1 + P\) where \(m_1\cap P\) is a point}}\) & 0 & \(\substack{\text{all rank 2 quadrics (one component of} \\ \text{each fiber is \(P\) and the other varies)}}\) \\ 
	 \hline
	  \(m_0 + m_1 + 2 \ell \) & (degree 1)\(^4\) & 1 singular fiber \(\langle m_0, \ell \rangle +  \langle m_1, \ell\rangle\) \\ 
	  \hline
	  \(\substack{ m_0 + m_1 + \ell + l' \\ \text{singular of Reid type}}\) & (degree 2 with distinct roots)\(^2\) & \(\substack{\text{2 singular fibers, each of rank 2} \\ \langle m_0, \ell \rangle + \langle m_1, l' \rangle \text{ and } \langle m_0, l' \rangle + \langle m_1, \ell \rangle}\) \\ 
	 \hline
	\end{tabular}
	\end{center}
\end{cor}

\begin{lem}\label{lem:lines-meeting-l-span-P2}
    Assume \(N\geq 4\), and fix a line \(\ell\) on \(X\). If \(m_0\neq m_1\) are lines on \(X\) both different from \(\ell\), and \(\langle\ell,m_0,m_1\rangle\cong\mathbb P^2\), then \(\langle\ell,m_0,m_1\rangle\subset X\).
\end{lem}

\begin{lem}\label{lem:lines-meeting-l-span-P3-common-point}
    Assume \(N\geq 4\), and fix a line \(\ell\) on \(X\). Let \(m_0,m_1\) be lines on \(X\) such that \(\ell\cap m_0 = \ell\cap m_1\) is a point and \(\langle\ell,m_0,m_1\rangle\cong\P^3\). Then there exists a choice of coordinates on \(\langle\ell,m_0,m_1\rangle\cong\mathbb P^3\) so that \(\mathcal Q\cap(\P^1\times\langle\ell,m_0,m_1\rangle)\) is spanned by the quadratic forms corresponding to
	\[ M_i = \begin{pmatrix} 0 & 0 & 0 & 0 \\ 0 & 0 & a_{i12} & a_{i13} \\ 0 & a_{i12} & 0 & a_{i23} \\ 0 & a_{i13} & a_{i23} & 0 \end{pmatrix} \quad \text{ for }i\in\{0,1\}, \]
	and, if \(C_i\) is the \(3\times 3\) symmetric matrix obtained by removing the first row and column of \(M_i\), we have the following possibilities for \(X\cap\langle\ell,m_0,m_1\rangle\):
	{
	\begin{center}
	\begin{tabular}{ |c|c|c|c| } 
	 \hline
	 & \(X\cap\langle\ell,m_0,m_1\rangle\) & \(\det(sC_0+tC_1)\) & Fibers of \(\mathcal Q\cap(\P^1\times\langle\ell,m_0,m_1\rangle) \to \mathbb P^1\) \\ 
	 \hline
	 \eqref{item:case-sym-Q1-meeting-computation-P3} & \(\langle\ell,m_0,m_1\rangle\cong\mathbb P^3\) & 0 & \(\langle\ell,m_0,m_1\rangle \cong\mathbb P^3\) \\ 
	 \hline
	 \eqref{item:case-sym-Q1-meeting-computation-quadric-surface-rank-2-or-3} & \(\substack{\text{\(P_1+P_2\) union of two 2-planes} \\ \text{at least one \(P_i\) contains exactly two of \(\{\ell,m_0,m_1\}\)} \\ \text{(\(P_1\cap P_2\) could possibly be one of \(\ell,m_0,m_1\))}}\)
	 	& 0 & \(\substack{\text{one fiber \(\mathbb P^3\),} \\ \text{others the same quadric}}\) \\ 
		\hline
	 \eqref{item:case-sym-Q1-meeting-computation-quadric-surface-rank-2-or-3} & rank 3 quadric surface & (linear)\(^3\) & \(\substack{\text{one fiber \(\mathbb P^3\),} \\ \text{others the same quadric}}\) \\ 
	 \hline
	 \eqref{item:case-sym-Q1-meeting-computation-general-member-rank-2-nonconstant} & \(\substack{\text{\(\ell + P\) or \(m_0 + P\) or \(m_1 + P\)} \\ \text{ with line \(\cap\) plane = \(\ell\cap m_0 \cap m_1\)}}\) & 0 & \(\substack{\text{all rank 2 quadrics (one component of} \\ \text{each fiber is \(P\) and the other varies)}}\) \\ 
	 \hline
	 \eqref{item:case-sym-Q1-meeting-computation-general-rank-3-multiplicity-2-root} & \(\begin{cases} \ell + m_0 + 2 m_1 \\ \ell + 2 m_0 + m_1 \\ 2\ell + m_0 + m_1 \end{cases}\) & \(\substack{\text{\(\lambda (s - a t)^2(s-b t)\)} \\ \text{ with }a\neq b, \lambda \neq 0}\) & \(\substack{\text{two fibers of rank 2,} \\ \text{others rank 3}}\) \\ 
	 \hline
	 \eqref{item:case-sym-Q1-meeting-computation-general-rank-3-distinct-roots} & \(\substack{ \text{\(\ell + m_0 + m_1 + m'\) where \(m'\) is another line} \\ \text{with \(\ell \cap m' = \ell \cap m_0 \cap m_1\)}}\) & \(\substack{\text{\(\lambda (s - a t)(s-b t)(s - c t)\)} \\ \text{ with }a, b,c \text{ distinct},\lambda\neq 0}\) & \(\substack{\text{three fibers of rank 2,} \\ \text{others rank 3}}\) \\ 
	 \hline
	\end{tabular}
	\end{center}}
In case~\eqref{item:case-sym-Q1-meeting-computation-general-rank-3-multiplicity-2-root}, we have the following description of the rank 2 fibers of \(\mathcal Q\cap(\P^1\times\langle\ell,m_0,m_1\rangle)\to\mathbb P^1\):
\begin{enumerate}[label=(\alph*)]
	\item\label{item:w-2ell+m_0+m_1-mult-2-root} Over the multiplicity 2 root of \(\det(sC_0+tC_1)\), the fiber is the rank 2 quadric surface
	\[\begin{cases}
	\langle\ell,m_1\rangle + \langle m_0, m_1\rangle & \text{ if }X\cap\langle\ell,m_0,m_1\rangle = \ell + m_0 + 2 m_1, \\
	\langle \ell,m_0\rangle + \langle m_0, m_1\rangle & \text{ if }X\cap\langle\ell,m_0,m_1\rangle = \ell + 2m_0 + m_1, \\
	\langle \ell,m_0\rangle + \langle\ell, m_1\rangle & \text{ if }X\cap\langle\ell,m_0,m_1\rangle = 2\ell + m_0 + m_1.
	\end{cases}\]
	\item Over the simple root of \(\det(sC_0+tC_1)\), the fiber is the rank 2 quadric surface
	\[\begin{cases}
	\langle\ell,m_0\rangle + P \text{ where }m_1\subset P \text{ and }\ell,m_0\not\subset P & \text{ if }X\cap\langle\ell,m_0,m_1\rangle = \ell + m_0 + 2 m_1, \\
	\langle\ell,m_1\rangle + P \text{ where }m_0 \subset P \text{ and }\ell, m_1\not\subset P & \text{ if }X\cap\langle\ell,m_0,m_1\rangle = \ell + 2m_0 + m_1, \\
	\langle m_0,m_1\rangle + P \text{ where }\ell\subset P \text{ and }m_0,m_1\not\subset P & \text{ if }X\cap\langle\ell,m_0,m_1\rangle = 2\ell + m_0 + m_1.
	\end{cases}\]
\end{enumerate}
In case~\eqref{item:case-sym-Q1-meeting-computation-general-rank-3-distinct-roots}, the three rank 2 fibers are given by
\[\langle\ell,m_0\rangle + \langle m_1,m'\rangle, \quad 
\langle\ell,m_1\rangle + \langle m_0,m'\rangle, \quad
\langle\ell,m'\rangle + \langle m_0,m_1\rangle. \]
\end{lem}

\begin{proof}[Proof of Lemma~\ref{lem:line-disjoint-from-l-cases}]
    We may choose coordinates on \(\P^N\) so that \(\langle\ell,m\rangle=\{[x_0:x_1:x_2:x_3:0:\cdots:0]\}\), \(\ell=\{[x_0:x_1:0:0:0:\cdots:0]\}\), and \(m=\{[0:0:x_2:x_3:0:\cdots:0]\}\). Then the assumption that \(\ell,m\subset X\) implies that the intersection \(\mathcal Q_{[s:t]}\cap\langle\ell,m\rangle \subset \langle\ell,m\rangle\) is defined by
    \[s\begin{pmatrix} 0 & A_0 \\ A_0^T & 0 \end{pmatrix} + t\begin{pmatrix} 0 & A_1 \\ A_1^T & 0 \end{pmatrix} \text{ for some }2\times 2 \text{ matrices }A_i=\begin{pmatrix} a_{i02} & a_{i03} \\ a_{i12} & a_{i13}\end{pmatrix}.\]
    We proceed by considering several cases.

    First assume that \(A_0\) and \(A_1\) are scalar multiples of each other. By a coordinate change on \(\P^1\), we may assume that \(A_1\) is the zero matrix. Then the determinant of \(sM_0+tM_1\) is \(\det(sA_0+tA_1)^2=(\det A_0)^2s^4\).
    \begin{enumerate}
        \item\label{item:case-F_1-computation-P3} If \(A_0\) is the zero matrix, then \(X\) contains the 3-plane \(\langle\ell,m\rangle\).
        \item\label{item:case-F_1-computation-quadric-surface} If \(A_0\) is not the zero matrix, then \(X\cap\langle\ell,m\rangle\) is the quadric surface defined by \(a_{002}x_0x_2+a_{003}x_0x_3+a_{012}x_1x_2+a_{013}x_1x_3\). This quadric surface has rank equal to \(2\rank A_0\), so it is smooth if and only if the matrix \(A_0\) has rank 2, and it has rank 2 if and only if \(A_0\) has rank 1.
    \end{enumerate}

    Now assume that \(A_0\) and \(A_1\) are not scalar multiples of each other. Then \(\rank A_0, \rank A_1 \geq 1\).
    \begin{enumerate}[resume]
        \item\label{item:case-F_1-computation-line-and-plane}
        If the polynomial \(\det(sA_0+tA_1)\) is uniformly zero, then \(\rank A_0=\rank A_1 = 1\), and by a coordinate change preserving the lines \(\ell\) and \(m\) we may assume \(A_0=\mathrm{diag}(1,0)\).
        \footnote{Indeed, for any invertible \(2\times 2\) matrix \(S\), the matrix \(L_S = \begin{pmatrix} (S^T)^{-1} & 0 \\ 0 & S\end{pmatrix}\) satisfies
        \[L_S^T\begin{pmatrix} 0 & A_0 \\ A_0^T & 0 \end{pmatrix} L_S = \begin{pmatrix} 0 & S^{-1}A_0 S \\ (S^{-1} A_0 S)^T & 0 \end{pmatrix} ,\]
        and the lines \(\ell,m\) are invariant under coordinate change by \(L_S\).
        If \(A_0\) is diagonalizable, pick such an \(S\) such that \(S^{-1} A_0 S=\mathrm{diag}(\lambda,0)\) for some \(\lambda\neq 0\).
        If \(A_0\) is not diagonalizable, pick such an \(S\) such that \(S^{-1} A_0 S = \begin{pmatrix} 0 & 1 \\ 0 & 0 \end{pmatrix}\), and then perform the further coordinate change \(x_2\leftrightarrow x_3\).}
        Then
        \[\det(sA_0+tA_1)= a_{113} st + (\det A_1) t^2\]
        is uniformly zero, i.e., \(a_{113}=a_{103}a_{112}=0\).
        Furthermore, at least one of \(a_{103}\) or \(a_{112}\) is nonzero by the assumption that \(A_1\) is not a scalar multiple of \(A_0\).
        \begin{enumerate}
            \item If \(a_{113}=a_{103}=0\), then \(a_{112}\neq 0\) and the intersection \(X\cap\langle\ell,m\rangle\) is the union of \(m\) and the plane \(P=V(x_2)\). The fiber of \(\mathcal Q\cap(\P^1\times\langle\ell,m\rangle)\to\P^1\) over \([s:t]\in \P^1\) is the rank 2 quadric whose components are \(P\) and the 2-plane defined by \((s+a_{102}t)x_0 + a_{112}tx_1\).
            \item If \(a_{113}=a_{112}=0\), then \(a_{103}\neq 0\) and the intersection \(X\cap\langle\ell,m\rangle\) is the union of \(\ell\) and the plane \(P=V(x_0)\). The fiber of \(\mathcal Q\cap(\P^1\times\langle\ell,m\rangle)\to\P^1\) over \([s:t]\in \P^1\) is the rank 2 quadric whose components are \(P\) and the 2-plane defined by \((s+a_{102}t)x_2 + a_{103}t x_3\).
        \end{enumerate}
        \item If the polynomial \(\det(sA_0+tA_1)\) is not uniformly zero, then by a coordinate change on \(\P^1\) we may assume \(\rank A_0=2\). Then, by a coordinate change preserving \(\ell\) and \(m\), we may assume \(A_0\) is the identity matrix. (Indeed, by the argument in the previous case, after a coordinate change preserving \(\ell\) and \(m\) we may assume \(A_0\) is either the identity or \(\begin{pmatrix} 1 & 1 \\ 0 & 1 \end{pmatrix}\), and in the latter case we may apply the further coordinate change \(x_2+x_3\mapsto x_2\) and \(x_3\mapsto x_3\).)
        Then
        \(\det(sA_0+tA_1) = s^2+(a_{102}+a_{113})st + (\det A_1) t^2\),
        so by a coordinate change on \(\P^1\) we may assume \(\rank A_1 = 1\). 
        Then the polynomial
        \[\det(sA_0+tA_1) = s(s+(a_{102}+a_{113})t)\]
        has two distinct roots if and only if \(a_{102}+a_{113} \neq 0\).
        \begin{enumerate}
            \item\label{item:case-F_1-computation-l-m-2m_1} Assume \(a_{102}+a_{113} = 0\), i.e., \(\det(sA_0+tA_1)\) has a unique root. Then \(A_1 = \begin{pmatrix} a_{102} & a_{103} \\ a_{112} & -a_{102}\end{pmatrix}\) and \(-a_{102}^2 = a_{103}a_{112}\).
            \begin{enumerate}
                \item If \(a_{102}=a_{103}=0\), then the intersection \(X\cap\langle\ell,m\rangle\) is \(V(x_0x_2+x_1x_3,x_1x_2)=\ell+m+2m'\) where \(m'=V(x_1,x_2)\). The unique singular fiber of \(\mathcal Q\cap(\P^1\times\langle\ell,m\rangle)\to\P^1\) is the rank 2 quadric surface \(V(x_1x_2)=\langle\ell,m'\rangle + \langle m,m'\rangle\).
                \item If \(a_{102}=a_{112}=0\), then the intersection \(X\cap\langle\ell,m\rangle\) is \(V(x_0x_2+x_1x_3,x_0x_3)=\ell+m+2m'\) where \(m'=V(x_0,x_3)\). The unique singular fiber of \(\mathcal Q\cap(\P^1\times\langle\ell,m\rangle)\to\P^1\) is the rank 2 quadric surface \(V(x_0x_3)=\langle\ell,m'\rangle+\langle m,m'\rangle\).
                \item If \(a_{102}\neq 0\) then we may assume \(a_{102}=1\). By the coordinate change \(x_2\mapsto a_{112}x_2\) and \(x_0\mapsto\frac{1}{a_{112}} x_0\) we may assume \(a_{112}=1\), so that \(A_1=\begin{pmatrix} 1 & -1 \\ 1 & -1 \end{pmatrix}\). Then \(X\cap\langle\ell,m\rangle\) is equal to \(V(x_0x_2+x_1x_3,(x_0+x_1)(x_2-x_3))=\ell+m+2m'\) where \(m'=V(x_0+x_1,x_2-x_3)\). The unique singular fiber of \(\mathcal Q\cap(\P^1\times\langle\ell,m\rangle)\to\P^1\) is the rank 2 quadric surface \(V((x_0+x_1)(x_2-x_3))=\langle\ell,m'\rangle+\langle m,m'\rangle\).
            \end{enumerate}
            \item\label{item:case-F_1-computation-Reid-type} Assume \(a_{102}+a_{113} \neq 0\), i.e., \(\det(sA_0+tA_1)\) has two distinct roots. Then the rank 1 matrix \(A_1\) is diagonalizable, so by a further coordinate change preserving \(\ell,m\), and \(A_0\), we may assume \(A_1=\mathrm{diag}(1,0)\). Then the intersection \(X\cap\langle\ell,m\rangle\) is equal to \(V(x_0x_2+x_1x_3,x_0x_2) = \ell + m  + l_1 + l_2\) where \(l_1=V(x_1,x_2)\) and \(l_2=V(x_0,x_3)\). There are two singular fibers of \(\mathcal Q\cap(\P^1\times\langle\ell,m\rangle)\to\P^1\), which occur over \([0:1]\) and \([-1:1]\). Over \([0:1]\), the fiber is the rank 2 quadric surface \(V(x_0x_2)=\langle m,l_2\rangle + \langle \ell,l_1\rangle\), whose singular locus is not contained in \(X\). Over \([-1:1]\), the fiber is the rank 2 quadric surface \(V(x_1x_3)=\langle m,l_1\rangle + \langle\ell,l_2\rangle\), whose singular locus is not contained in \(X\).
        \end{enumerate}
    \end{enumerate}
\end{proof}

\begin{proof}[Proof of Lemma~\ref{lem:lines-meeting-l-span-P2}]
    This is a direct computation with two cases: \(\ell\cap m_0\neq \ell\cap m_1\) or \(\ell\cap m_0 = \ell\cap m_1\).
    In either case, for every $t\in \P^1$, the intersection $\mathcal{Q}_t\cap \langle \ell, m_0, m_1\rangle\subset \langle\ell,m_0,m_1\rangle = \P^2$ contains the degree $3$ curve $\ell+m_0+m_1$, hence by degree reasons, $\langle\ell, m_0,m_1\rangle \subset \mathcal{Q}_t$, which shows $\langle\ell,m_0,m_1\rangle\subset X$, as desired.
\end{proof}

\begin{proof}[Proof of Lemma~\ref{lem:lines-meeting-l-span-P3-common-point}]
    The assumptions imply \(m_0\cap m_1=\ell\cap m_0=\ell\cap m_1\). We may choose coordinates on \(\P^N\) so that
    \begin{alignat*}{2}
        \langle\ell,m_0,m_1\rangle &= \{[x_0:x_1:x_2:x_3:0:\cdots:0]\}, \qquad \ell\, &= \{[x_0:x_1:0:0:0:\cdots:0]\}, \\
        m_0 &= \{[x_0:0:x_2:0:0:\cdots:0]\}, \qquad m_1 &= \{[x_0:0:0:x_3:0:\cdots:0]\}.
    \end{alignat*}
    Then \(\ell\cap m_0=\ell\cap m_1 = [1:0:\cdots:0]\). The assumption that \(\ell,m_0,m_1\subset X\) implies that the intersection \(\mathcal Q_{[s:t]}\cap\langle\ell,m_0,m_1\rangle \subset \langle\ell,m_0,m_1\rangle\) is defined by
    \[s M_0 + t M_{1} \text{ where } M_i = \begin{pmatrix} 0 & 0 & 0 & 0 \\ 0 & 0 & a_{i12} & a_{i13} \\ 0 & a_{i12} & 0 & a_{i23} \\ 0 & a_{i13} & a_{i23} & 0 \end{pmatrix},\]
    so the fibers of \(\mathcal Q\cap(\P^1\times\langle\ell,m_0,m_1\rangle)\to\P^1\) are cones over the conics in \(\P^2_{[x_1:x_2:x_3]}\) defined by \[sC_0 + t C_1 \text{ where } C_i = \begin{pmatrix} 0 & a_{i12} & a_{i13} \\ a_{i12} & 0 & a_{i23} \\ a_{i13} & a_{i23} & 0 \end{pmatrix}.\]
    Note that \(\det(sC_0+tC_1)=2(a_{012}s+a_{112}t)(a_{013}s+a_{113}t)(a_{023}s+a_{123}t)\) and the $2\times 2$ minors of $sC_0+tC_1$ include $-(sa_{012}+ta_{112})^2$, $-(sa_{013}+ta_{113})^2$, and $-(sa_{023}+ta_{123})^2$. If a matrix \(sC_0+tC_1\) has rank at most one, then all three minors vanish and hence the matrix is zero. So each of the matrices \(sC_0+tC_1\) has rank 0, 2, or 3. We now consider several cases.

    First assume that \(C_0\) and \(C_1\) are scalar multiples of each other. By a coordinate change on \(\P^1\), we may assume that \(C_1\) is the zero matrix. Then \(\det(sC_0+tC_1)=(\det C_0) s^3\).
    \begin{enumerate}
        \item\label{item:case-sym-Q1-meeting-computation-P3} If \(C_0\) is the zero matrix, then \(\langle\ell,m_0,m_1\rangle\subset X\).
        \item\label{item:case-sym-Q1-meeting-computation-quadric-surface-rank-2-or-3} If \(\rank C_0=2\), then \(X\cap\langle\ell,m_0,m_1\rangle\) is a rank 2 quadric surface. If \(\rank C_0 = 3\), then \(X\cap\langle\ell,m_0,m_1\rangle\) is a rank 3 quadric surface.
    \end{enumerate}

    Now assume that \(C_0\) and \(C_1\) are not scalar multiples of each other. We may assume that \(2\leq\rank C_1\leq\rank C_0\leq 3\). First we consider the case when \(\rank C_0=2\).
    \begin{enumerate}[resume]
        \item\label{item:case-sym-Q1-meeting-computation-general-member-rank-2-nonconstant} If \(\rank C_0=\rank C_1=2\), then \(a_{012}a_{013}a_{023}=0\) and at least one of \(a_{012},a_{013},a_{023}\) is nonzero. After a coordinate change (possibly interchanging \(\ell,m_0,m_1\)) we may assume \(a_{012}=0\) and \(a_{023}=1\).
        Then
        \(\det(sC_0+tC_1)=2a_{112}t(a_{013}s+a_{113}t)(s+a_{123}t)\), and if \(a_{112}\neq 0\) and at least one of \(a_{013}\) or \(a_{113}\) is nonzero, then this polynomial is not the zero polynomial; in this case after a coordinate change on \(\P^1\) we may assume \(\rank C_0=3\). Therefore we may assume either \(a_{112}=0\) or \(a_{013}=a_{113}=0\).
        \begin{enumerate}
            \item If \(a_{112}=0\), then the fibers of \(\mathcal Q\cap(\P^1\times\langle\ell,m_0,m_1\rangle)\to\P^1\) are the rank 2 quadric surfaces with irreducible components defined by \(x_3\) and \((a_{013}s+a_{113}t)x_1+(s+a_{123}t)x_2\). The intersection \(X\cap\langle\ell,m_0,m_1\rangle\) is \(V(a_{013}x_1x_3+x_2x_3, (a_{113}-a_{013}a_{123})x_1x_3)=V(x_2x_3,x_1x_3)\) by the assumption that \(C_1\) is not a scalar multiple of \(C_0\). That is, \(X\cap\langle\ell,m_0,m_1\rangle = m_1 + P\) where \(P=V(x_3)\).
            \item If \(a_{013}=a_{113}=0\), then \(a_{112}\neq 0\) since \(C_1\) is not a scalar multiple of \(C_0\). Then we may assume \(a_{112}=1\). The fibers of \(\mathcal Q\cap(\P^1\times\langle\ell,m_0,m_1\rangle)\to\P^1\) are the rank 2 quadric surfaces with components defined by \(x_2\) and \(tx_1+(s+a_{123}t)x_3\), and the intersection \(X\cap\langle\ell,m_0,m_1\rangle\) is \(V(x_2x_3,x_1x_2)=m_0+P\) where \(P=V(x_2)\).
        \end{enumerate}
    \end{enumerate}
    Now assume \(\rank C_0=3\), so that \(a_{012}a_{013}a_{023}\neq 0\). By rescaling the equation and applying a coordinate change, we may assume \(a_{012}=a_{013}=1\). Then \(\det(sC_0+tC_1)=2 a_{023}(s+a_{112}t)(s+a_{113}t)(s+\tfrac{a_{123}}{a_{023}}t)\) has roots \([-a_{112}:1],[-a_{113}:1],[-a_{123}:a_{023}]\). The assumption that \(C_1\) is not a multiple of \(C_0\) implies that \(\det(sC_0+tC_1)\) has at least two distinct roots. Thus, we consider separately the cases when \(\det(sC_0+tC_1)\) has two distinct roots or three distinct roots.
    \begin{enumerate}[resume]
        \item\label{item:case-sym-Q1-meeting-computation-general-rank-3-multiplicity-2-root} Assume \(\det(s C_0 + t C_1)\) has two distinct roots.
        \begin{enumerate}
            \item If \(a_{112}=a_{113}\), then \(a_{123}-a_{023}a_{112}\neq 0\), so the intersection \(X\cap\langle\ell,m_0,m_1\rangle\) is \(V(x_1x_2+x_1x_3+a_{023}x_2x_3,(a_{123}-a_{023}a_{112})x_2x_3)=V(x_1(x_2+x_3),x_2x_3) = m_0+m_1+2\ell\). The two distinct roots of \(\det(sC_0+tC_1)\) are \([-a_{112}:1]\) with multiplicity two and \([-a_{123}:a_{023}]\) with multiplicity one. Over the multiplicity two root, the fiber of \(\mathcal Q\cap(\P^1\times\langle\ell,m_0,m_1\rangle)\to\P^1\) is \(V(x_2x_3)=\langle\ell,m_0\rangle+\langle\ell,m_1\rangle\), and the fiber over the multiplicity one root is \(V(x_1(x_2+x_3))=\langle m_0,m_1\rangle+V(x_2+x_3)\).
            \item If \(a_{112}=\tfrac{a_{123}}{a_{023}}\) or \(a_{113}=\tfrac{a_{123}}{a_{023}}\), then we may assume by a coordinate change (possibly exchanging \(m_0,m_1\)) that \(a_{112}=\tfrac{a_{123}}{a_{023}}\) and \(a_{112}\neq a_{113}\). Then the intersection \(X\cap\langle\ell,m_0,m_1\rangle\) is \(V(x_1x_2+x_1x_3+a_{023}x_2x_3,(a_{113}-a_{112})x_1x_3) = V(x_1x_2+a_{023}x_2x_3,x_1x_3) = \ell+2m_0+m_1\).
            The two distinct roots of \(\det(sC_0+tC_1)\) are \([-a_{112}:1]\) with multiplicity two and \([-a_{113}:1]\) with multiplicity one. Over the multiplicity two root, the fiber of \(\mathcal Q\cap(\P^1\times\langle\ell,m_0,m_1\rangle)\to\P^1\) is \(V(x_1x_3)=\langle\ell,m_0\rangle+\langle m_0,m_1\rangle\), and the fiber over the multiplicity one root is \(V(x_2(x_1+a_{023}x_3))=\langle\ell,m_1\rangle+V(x_1+a_{023}x_3)\).
        \end{enumerate}
        \item\label{item:case-sym-Q1-meeting-computation-general-rank-3-distinct-roots} Assume \(\det(s C_0 + t C_1)\) has three distinct roots, so that \(a_{112}\), \(a_{113}\), and \(\tfrac{a_{123}}{a_{023}}\) are all distinct. After a coordinate change on \(\P^1\), we may assume \(\rank C_1=2\), so that exactly one of \(a_{112},a_{113},a_{123}\) is zero.
        \begin{enumerate}
            \item Assume \(a_{112}=0\) or \(a_{113}=0\). Then we may assume by a coordinate change (possibly exchanging \(m_0,m_1\)) that \(a_{112}=0\). We may then further assume that \(a_{113}=1\) and \(a_{123}\neq a_{023}\). Then \(X\cap\langle\ell,m_0,m_1\rangle\) is equal to \(V(x_1x_2+x_1x_3+a_{023}x_2x_3, x_1x_3+a_{123}x_2x_3)=\ell+m_0+m_1+m'\) where \(m'=V(x_1+(a_{023}-a_{123})x_3, x_1+a_{123}x_2)\).
            
            The fibers of \(\mathcal Q\cap(\P^1\times\langle\ell,m_0,m_1\rangle)\to\P^1\) over the roots of \(\det(sC_0+tC_1)\) are given as follows.
            The fiber over \([0:1]\) is the rank 2 quadric surface \(V(x_3(x_1+a_{123}x_2))=\langle\ell,m_0\rangle + \langle m_1,m'\rangle\),
            the fiber over \([-1:1]\) is the rank 2 quadric surface \(V(x_2(x_1+(a_{023}-a_{123})x_3))=\langle\ell,m_1\rangle+\langle m_0,m'\rangle\), and
            the fiber over \([-a_{123}:a_{023}]\) is the rank 2 quadric surface \(V(x_1(a_{123}x_2+(a_{123}-a_{023}) x_3)) = \langle m_0,m_1\rangle + \langle \ell,m'\rangle\).
            \item Assume \(a_{123}=0\). Then we may assume \(a_{112}=1\), so that \(a_{113}\neq 1\). Then \(X\cap\langle\ell,m_0,m_1\rangle\) is equal to \(V(x_1x_2+x_1x_3+a_{023}x_2x_3, (a_{113}-1)x_1x_3 - a_{023}x_2x_3) = V(x_1(x_2+a_{113}x_3), x_3((a_{113}-1)x_1 - a_{023} x_2)) = \ell + m_0 + m_1 + m'\) where \(m'=V(x_2+a_{113}x_3, (a_{113}-1)x_1 - a_{023} x_2)\).

            The fibers of \(\mathcal Q\cap(\P^1\times\langle\ell,m_0,m_1\rangle)\to\P^1\) over the roots of \(\det(sC_0+tC_1)\) are given as follows.
            The fiber over \([-1:1]\) is the rank 2 quadric surface \(V(x_3((a_{113}-1)x_1-a_{023}x_2)) = \langle\ell,m_0\rangle + \langle m_1,m'\rangle \),
            the fiber over \([-a_{113}:1]\) is the rank 2 quadric surface \(V(x_2((-a_{113}+1)x_1-a_{023}a_{113}x_3))=\langle \ell,m_1\rangle + \langle m_0, m'\rangle\), and
            the fiber over \([0:1]\) is the rank 2 quadric surface \(V(x_1(x_2+a_{113}x_3))=\langle m_0,m_1\rangle + \langle\ell,m'\rangle\).
        \end{enumerate}
    \end{enumerate}
\end{proof}

This completes the proofs of Lemmas~\ref{lem:line-disjoint-from-l-cases}--\ref{lem:lines-meeting-l-span-P3-common-point}.

Finally, the following lemmas will be useful in Section~\ref{section:w-concurrent}.

\begin{lem}\label{lem:intersect-3plane-containing-2plane-meeting-in-2m'-cases}
    Assume \(N\geq 4\), and fix a line \(\ell\) on \(X\). Let \(P\subset\P^N\) be a 2-plane such that \(X\cap P = 2m'\) for some line \(m'\) such that \(\ell\cap m'\) is a point. Assume additionally that no member of the pencil corresponding to \(X\) contains the 3-plane \(\langle\ell,P\rangle\). Then we have the following possibilities for the intersection \(X\cap\langle\ell,P\rangle\):
    \begin{enumerate}
        \item \(X\cap\langle\ell,P\rangle = \ell + 2m' + n\) for some line \(n\neq \ell,m'\) such that \(\ell\cap m'\cap n\) is a point,
        \item \(X\cap\langle\ell,P\rangle=2\ell+2m'\),
        \item\label{item:intersect-3-plane-cases-l+3m'} \(X\cap\langle\ell,P\rangle=\ell+3m'\), or
        \item\label{item:intersect-3-plane-2m'-cases-nonred-plane} \(X\cap\langle\ell,P\rangle\) is supported on the 2-plane \(\langle\ell,m'\rangle\) and has nonreduced structure along \(m'\).
    \end{enumerate}
    In this situation, there is a unique \(t\in\P^1\) such that \(P\subset\mathcal Q_t\). Furthermore, for this \(t\), we have \(\langle\ell,m'\rangle\subset\mathcal Q_t\) if and only if \eqref{item:intersect-3-plane-cases-l+3m'} or \eqref{item:intersect-3-plane-2m'-cases-nonred-plane} occurs.

    In case~\eqref{item:intersect-3-plane-cases-l+3m'}, the 2-plane \(P\subset\mathcal Q_t\) can be recovered from the scheme-theoretic intersection \(X\cap\langle\ell,P\rangle\) as the unique 2-plane in the 3-plane \(\langle\ell,P\rangle\) containing \(m'\) whose intersection with \(X\cap\langle\ell,P\rangle\) is nonreduced along \(m'\).
\end{lem}

\begin{proof}
    We may choose coordinates on \(\P^N\) so that \(\ell=\{[x_0:x_1:0:\cdots:0]\}\), \(m'=\{[x_0:0:x_2:0:\cdots:0]\}\), and \(P=\{[x_0:0:x_2:x_3:0:\cdots:0]\}\). The assumption that \(X\cap P=2m'\) implies that \(P\) is contained in a unique member of the pencil corresponding to \(X\); we may assume that this is the fiber over \([0:1]\). Then, the assumptions that \(\ell,m'\subset X\) and \(X\cap P=2m'\) imply that the intersection \(\mathcal Q_{[s:t]}\cap\langle\ell,P\rangle \subset \langle\ell,P\rangle\) is defined by \(sM_0 + t M_1\), where
    \[M_0 = \begin{pmatrix}
        0 & 0 & 0 & 0 \\
        0 & 0 & a_{012}/2 & a_{013}/2 \\
        0 & a_{012}/2 & 0 & 0 \\
        0 & a_{013}/2 & 0 & 1
    \end{pmatrix}, \qquad
    M_1 = \begin{pmatrix}
        0 & 0 & 0 & 0 \\
        0 & 0 & a_{112}/2 & a_{113}/2 \\
        0 & a_{112}/2 & 0 & 0 \\
        0 & a_{113}/2 & 0 & 0
    \end{pmatrix}.\]
    Since \(\langle\ell,P\rangle\not\subset\mathcal Q_{[s:t]}\) for any \([s:t]\), neither \(M_0\) nor \(M_1\) is the zero matrix.
    Denote \(P'=\{a_{112}x_2+a_{113}x_3=0\}\subset\langle\ell,P\rangle\); then \(\mathcal Q_{[0:1]}\cap \langle\ell,P\rangle\) is the rank 2 quadric surface \(P+P'\). Consider the line \(P\cap P'=\{x_1=a_{112}x_2+a_{113}x_3=0\}\).

    First we consider the cases when \(P\cap P'\neq m'\) (so in particular \(P'\neq\langle\ell,m'\rangle\)).
    In this case we have \(a_{112}\neq 0\), so we may assume \(a_{112}=1\).
    Since \(X\cap P\) does not contain the line \(P\cap P'\), this implies that \(P'\not\subset X\). Then the intersection \(\mathcal Q_{[1:0]}\cap P'\) is the singular conic
    \(V\left(((-a_{012}a_{113}+a_{013})x_1+x_3)x_3\right)\).
    \begin{enumerate}
        \item If \(-a_{012}a_{113}+a_{013}\neq 0\), then \(\mathcal Q_{[1:0]}\cap P'=\ell+n\) where \(n=V((-a_{012}a_{113}+a_{013})x_1+x_3)\) satisfies \(\ell\cap n=[1:0:\cdots:0] = \ell \cap m'\). In this case we have \(X\cap\langle\ell,P\rangle=V(x_1(a_{012}x_2+a_{013}x_3)+x_3^2, x_1(x_2+a_{113}x_3))=\ell+2m'+n\).
        \item If \(-a_{012}a_{113}+a_{013}=0\), then \(\mathcal Q_{[1:0]}\cap P'=2\ell\), and \(X\cap\langle\ell,P\rangle=V(x_3^2, x_1(x_2+a_{113}x_3))=2\ell+2 m'\).
    \end{enumerate}
    Now we consider the cases when \(P\cap P'=m'\). In this case we have \(a_{112}=0\), so \(P'=\{x_3=0\} = \langle\ell,m'\rangle\).
    \begin{enumerate}[resume]
        \item If \(P'\not\subset\mathcal Q_{[1:0]}\), then \(a_{012}\neq 0\) and \(\mathcal Q_{[1:0]}\cap P'\) is the conic \(\ell+m'\). After a coordinate change we may assume \(a_{012}=1\). Then \(X\cap\langle\ell,P\rangle=V(x_1x_2+a_{013}x_1x_3+x_3^2,x_1x_3)=V(x_1x_2+x_3^2,x_1x_3)=\ell+3m'\).
        \item If \(P'\subset\mathcal Q_{[1:0]}\), then \(a_{012}=0\) and \(X\cap\langle\ell,P\rangle=V(x_3(a_{013}x_1+x_3), x_1x_3)\) is supported on \(P'\) with nonreduced structure along \(m'=\{x_1=x_3=0\}\).
    \end{enumerate}

    Finally, assume we are in case~\eqref{item:intersect-3-plane-cases-l+3m'}, and let \(H\supset m'\) be a 2-plane contained in \(\langle\ell,P\rangle\). Then \(H=\{bx_1+cx_3=0\}\) for some \(b,c\in k\) not both zero. If \(c=0\), i.e., if \(H=P\), then \(X\cap\langle\ell,P\rangle\cap H=V(x_3^2,x_1)\) is nonreduced along \(m'\).
    If \(c\neq 0\) we may assume \(c=1\); then \(X\cap\langle\ell,P\rangle\cap H=V(x_1x_2+b^2x_1^2,bx_1^2,bx_1+x_3)\). If \(b\neq 0\) this is equal to \(V(x_1x_2,x_1^2,bx_1+x_3)\), and if \(b = 0\) this is equal to \(V(x_1x_2,x_3)\); in either case, it is reduced along \(m'\).
\end{proof}

\begin{lem}\label{lem:intersect-3m'+l-then-2plane}
    Assume \(N\geq 4\), and fix lines \(\ell,m'\subset X\) such that \(\ell\cap m'\) is a point. Let \(L\subset \P^N\) be a 3-plane, and assume that \(X\cap L = \ell+3m'\). Then there is a unique 2-plane \(P\subset L\) such that \(X\cap P = 2m'\).
\end{lem}

\begin{proof}
    Consider the 2-plane \(P'\coloneqq \langle\ell,m'\rangle\), which is not contained in \(X\), and consider its intersection \(\mathcal Q_t\cap P'\) with the fibers of the pencil corresponding to \(X\). For every \(t\in\P^1\), this contains the reducible conic \(\ell+m'\), so there is a unique \(t'\in \P^1\) such that \(P'\subset\mathcal Q_{t'}\), and for every $t\neq t'$ we have $\mathcal{Q}_t\cap P'=\ell+m'$; we may assume \(t'=[0:1]\).
    Since \(L\not\subset\mathcal Q_{[0:1]}\) (as \(X\cap L\) is not a quadric surface), this implies \(\mathcal Q_{[0:1]}\cap L\) is a quadric surface containing \(P'\), so it is either the rank 1 quadric surface \(2P'\) or a rank 2 quadric surface \(P+P'\) for some \(P\neq P'\).
    The first possibility does not occur, because if \(\mathcal Q_{[0:1]}\cap L = 2P'\) then \(X\cap L\) cannot be reduced along \(\ell\). We claim that \(P\) gives the required 2-plane in the lemma statement.
    Indeed, consider the Cartier divisors \(\mathcal Q_{[1:0]}\cap L\) and \(\mathcal Q_{[0:1]}\cap L=P+P'\) on \(L\). Then \(\ell + 3m' = X\cap L = (\mathcal Q_{[1:0]}\cap L) \cap (\mathcal Q_{[0:1]}\cap L) = (\mathcal Q_{[1:0]}\cap L) \cap P + (\mathcal Q_{[1:0]}\cap L) \cap P' = (\mathcal Q_{[1:0]}\cap L) \cap P + \ell + m'\) as cycles, so \((\mathcal Q_{[1:0]}\cap L) \cap P = 2 m'\) as cycles. Since \(\mathcal Q_{[1:0]} \cap P = X \cap P\) is a conic, this implies \(X\cap P = 2m'\).

    Then the uniqueness statement for \(P\) follows from Lemma~\ref{lem:intersect-3plane-containing-2plane-meeting-in-2m'-cases}.
\end{proof}

\subsection{Distinguished loci in \texorpdfstring{\(\mathcal Q^{(1)}\)}{Q1}, \texorpdfstring{\(\mathcal Q^{(2)}\)}{Q2}, and \texorpdfstring{\(\mathcal P^{(2)}\)}{P2}}

Next, we give further explicit descriptions of the loci from Lemma~\ref{lem:r-planes-in-hyperbolic-reduction-loci} in the cases when \(r=1,2\). Recall that \(\mathcal P_{Y_\Lambda}\), \(\mathcal Q_{Y_\Lambda}\), and \(Y_\Lambda\) were defined in~\eqref{eqns:Z-and-Y}.
\begin{lem}\label{lem:lines-in-hyperbolic-reduction-loci}
Let \(\ell\) be a line on \(X\), and let \(\bl_{\P^1\times\ell}\) and \(h\) be the morphisms from diagram~\eqref{eqn:hyperbolic-red-blow-up-diagram}.
    \begin{enumerate}
        \item\label{item:Q^1-points-interpretation} Under the embedding in Lemma~\ref{lem:hyperbolic-reduction-properties-not-pencil}\eqref{item:hyperbolic-reduction-not-pencil-embedding} and Lemma~\ref{lem:r-planes-in-hyperbolic-reduction-loci},
        \begin{equation*}
            \begin{split}
                \mathcal Q^{(1)}_{\ell} \setminus \mathcal Q_{Y_{\ell}} &\xrightarrow{\sim} \{(t,P) \mid P \subset \mathcal Q_t \text{ and }X \cap P = \ell + m \text{ for some line }m\neq \ell\}, \\
                \mathcal Q_{Y_{\ell}} \setminus (\mathbb P^1\times Y_{\ell}) & \xrightarrow{\sim} \{(t,P) \mid P \subset \mathcal Q_t \text{ and }X \cap P = 2\ell\}, \\
                \mathbb P^1\times Y_{\ell} & \to \{(t,P) \mid \ell \subset P \subset X \subset \mathcal Q_t\}.
            \end{split}
        \end{equation*} \label{item:image-of-Q^1-in-rel-F2}
        \item\label{item:rel-lines-in-Q^1-in-rel-F3} Sending \(l\in F_1(\mathcal Q^{(1)}_{\ell}/\mathbb P^1)\) to \(\bl_{\mathbb P^1\times\ell}(h^{-1}(l))\) defines an isomorphism over \(\P^1\)
        \[F_1(\mathcal Q^{(1)}_{\ell}/\P^1) \cong F_3(\mathcal Q/\P^1)_\ell \]
        with the relative Fano scheme of 3-planes in the fibers of \(\mathcal Q\to\mathbb P^1\) containing \(\ell\). This restricts to the following isomorphisms over \(\P^1\)
        \begin{equation*}
            \begin{split}
                F_1(\mathcal Q_{Y_{\ell}}/\mathbb P^1)\setminus F_1((\mathbb P^1\times Y_{\ell})/\P^1) & \xrightarrow{\sim} \left\{ (t,L) \,\middle\vert\, \substack{ L \subset \mathcal Q_t \text{ and }X\cap L \text{ is a quadric surface} \\ \text{of rank \(\leq 2\) whose singular locus contains }\ell } \right\}, \\
                F_1((\P^1\times Y_{\ell})/\P^1) & \xrightarrow{\sim} \P^1\times F_3(X)_{\ell}, \\
                F_1(\mathcal{Q}^{(1)}_\ell/\P^1)_{\mathbb P^1\times Y_{\ell}}^{\text{int}} \setminus F_1((\mathbb P^1\times Y_{\ell})/\P^1) &\xrightarrow{\sim} \left\{ (t,L) \,\middle\vert\, \substack{ L \subset \mathcal Q_t \text{ and }X\cap L \text{ is a quadric surface} \\ \text{of rank \(\leq 2\) and contains }\ell } \right\}, \\
                F_1(\mathcal{Q}_{Y_\ell}/\P^1)_{\P^1\times Y_\ell}^{\text{tan}} \setminus F_1((\mathbb P^1\times Y_{\ell})/\P^1) &\xrightarrow{\sim} \left\{ (t,L) \,\middle\vert\, \substack{ L \subset \mathcal Q_t \text{ and }X\cap L = 2P \text{ is a} \\ \text{rank \(1\) quadric surface and }\ell\subset P } \right\},
            \end{split}
        \end{equation*}
        where \(F_1(\mathcal{Q}^{(1)}_\ell/\P^1)_{\mathbb P^1\times Y_{\ell}}^{\text{int}}\) denotes the locus of relative lines meeting \(\P^1\times Y_{\ell}\), and
        \[F_1(\mathcal{Q}_{Y_\ell}/\P^1)_{\P^1\times Y_\ell}^{\text{tan}} \coloneqq \{(t, n) \in F_1(\mathcal{Q}_{Y_\ell}/\P^1) \mid n \text{ is tangent to }Y_\ell\}.\]
    \end{enumerate}
\end{lem}

\begin{proof}
    Part~\eqref{item:Q^1-points-interpretation} follows from Lemma~\ref{lem:r-planes-in-hyperbolic-reduction-loci}\eqref{item:image-of-Z-in-rel-Fr} and~\eqref{item:image-of-E-in-rel-Fr}.
    For part~\eqref{item:rel-lines-in-Q^1-in-rel-F3}, the isomorphisms \(F_1(\mathcal Q^{(1)}_\ell/\P^1)\cong F_3(\mathcal Q/\P^1)_\ell\) and \(F_1((\P^1\times Y_\ell)/\P^1)\cong \P^1\times F_3(X)_\ell\) over \(\P^1\) are in Lemma~\ref{lem:r-planes-in-hyperbolic-reduction-loci}\eqref{item:relative-m-planes-in-Qr}.
    The statements about \(F_1(\mathcal{Q}^{(1)}_\ell/\P^1)_{\mathbb P^1\times Y_{\ell}}^{\text{int}} \setminus F_1((\mathbb P^1\times Y_{\ell})/\P^1)\) and \(F_1(\mathcal{Q}_{Y_\ell}/\P^1)_{\P^1\times Y_\ell}^{\text{tan}} \setminus F_1((\mathbb P^1\times Y_{\ell})/\P^1)\) follow from part~\eqref{item:Q^1-points-interpretation}.
    
    To address \(F_1(\mathcal Q_{Y_{\ell}}/\mathbb P^1)\setminus F_1((\mathbb P^1\times Y_{\ell})/\P^1)\),
    let \(L_n=\{f_1=\cdots=f_{N-3}=0\} \subset \mathcal Q_{[s:t]}\) be a 3-plane containing \(\ell\), and let \(n=h(\widetilde{L}_n) \subset \mathcal (Q^{(1)}_\ell)_{[s:t]}\) be the corresponding line, where \(\widetilde{L}_n\) is the strict transform of \(L_n\) under \(\bl_{\P^1\times\ell}\).
    The intersection \(X \cap L_n\) is either \(L_n\) or a quadric surface in \(L_n\), and is defined by
    \[X \cap L_n = \{f_1=\cdots=f_{N-3}= x_0l_{i0}+x_1l_{i1}+q_i=0\}\]
    where \(i\in\{0,1\}\) if \(s\) and \(t\) are both nonzero, \(i=0\) if \(s=0\), and \(i=1\) if \(t=0\).
    Let \(i\) be as above.
    We have
    \(n\subset \mathcal Q_{Y_{\ell}}\) and \(n\not\subset\P^1\times Y_\ell\) if and only if \(X\cap L_n = \{f_1=\cdots=f_{N-3}= q_i=0\} \neq L_n\) if and only if \(X\cap L_n\) is a quadric surface realized as the cone with vertex \(\ell\) over \(n \cap Y_\ell=\{f_1=\cdots=f_{N-3}= q_i=0\} \subset \P^{N-2}_{[x_2:\cdots:x_N]}\). In this case, \(X\cap L_n\) is either a rank 2 quadric surface with singular locus \(\ell\) (if \(n\) meets \(Y_\ell\) in two distinct points) or a rank 1 quadric surface (if \(n\) is tangent to \(Y_\ell\)). This shows the statement about \(F_1(\mathcal Q_{Y_{\ell}}/\mathbb P^1)\setminus F_1((\mathbb P^1\times Y_{\ell})/\P^1)\).
\end{proof}

\begin{lem}\label{lem:planes-in-hyperbolic-reduction-loci}
Let \(P\) be a 2-plane on \(X\), and let \(\tilde{h}\) be the \(\P^{3}\)-bundle from diagram~\eqref{eqn:hyperbolic-red-blow-up-diagram}. Choose coordinates on \(\P^N\) so that \(P=\{x_3=\cdots=x_N=0\}\), and let \(l_{ij},q_i \in k[x_3,\ldots,x_N]\) be defined as in~\eqref{eqn:X}.
    \begin{enumerate}
        \item Sending \(x\in\mathcal P^{(2)}_P\setminus\mathcal Q^{(2)}_P\) to \(\bl_{\P^1\times P}(\tilde{h}^{-1}(x))\) defines an isomorphism over \(\mathbb P^1\)
        \begin{equation*}
            \begin{split}
                \mathcal P^{(2)}_P\setminus\mathcal Q^{(2)}_P \xrightarrow{\sim}& \{ (t,L)\in \mathbb P^1\times \Gr(4,N+1) \mid P \subset L \text{ and }\mathcal Q_t \cap L = 2P\}.
            \end{split}
        \end{equation*}
        Furthermore, for \(x\in\mathcal P^{(2)}_P\setminus\mathcal Q^{(2)}_P\), the intersection \(X\cap \bl_{\P^1\times P}(\tilde{h}^{-1}(x))\) is
        \[\begin{cases}
            2P & \text{ if }x\in\mathcal P_{Y_{P}}, \\
            \text{supported on \(P\) with non-reduced structure along a line} & \text{otherwise.}
        \end{cases}\]
        The first case above (i.e., \(x\in\mathcal P_{Y_{P}}\)) happens exactly when \(L\subset\mathcal Q_{t'}\) for some \(t'\neq t\).
        The second case above happens if and only if there is a unique \(t\in\P^1\) such that \(\mathcal Q_t\cap L = 2P\).
        
        Additionally, we have an isomorphism
        \[ \mathcal P^{(2)}_P\setminus(\mathcal Q^{(2)}_P \cup \mathcal P_{Y_P}) \xrightarrow{\sim} \{ L \in \Gr(4,N+1) \mid X \cap L \text{ is supported on \(P\) with a non-reduced line} \}. \]
        \label{item:image-of-P^2-in-rel-F3}
        \item\label{item:image-of-Q^2-loci-in-rel-F3} Sending \(x\in\mathcal Q^{(2)}_P\) to \(\bl_{\P^1\times P}(\tilde{h}^{-1}(x))\) defines isomorphisms over \(\mathbb P^1\)
        \begin{equation*}
            \begin{split}
                \mathcal Q^{(2)}_P\setminus \mathcal Q_{Y_P} &\xrightarrow{\sim} \{(t,L)\in F_3(\mathcal Q/\mathbb P^1) \mid P\subset L \text{ and }X \cap L \text{ is a rank 2 quadric surface} \} , \\
                \mathcal Q_{Y_P}\setminus (\mathbb P^1\times Y_P) &\xrightarrow{\sim} \{(t,L)\in F_3(\mathcal Q/\mathbb P^1) \mid P\subset L \text{ and } X \cap L = 2P \} , \\
                \mathbb P^1\times Y_P &\xrightarrow{\sim}\mathbb P^1\times F_3(X)_P .
            \end{split}
        \end{equation*}
        \item\label{item:image-of-P^2-2P-locus} \(\mathcal P_{Y_P}\setminus(\P^1\times Y_P) = (\mathcal P^{(2)}_P\cap\{l_{ij}=0\})\setminus(\P^1\times Y_P)\) is isomorphic over \(\P^1\) to
        \[ \P^1\times \{ L \in \Gr(4,N+1) \mid P \subset L \text{ and }X \cap L = 2P \}. \]
\end{enumerate}
\end{lem}

\begin{proof}
    Let \(x = ([s:t],\underline{x}) \in \mathcal P^{(2)}_P\). Then \(\bl_{\P^1\times P}(\tilde{h}^{-1}(x))\) is the 3-plane \(\{[s:t]\}\times L_{\underline{x}}\) where \(L_{\underline{x}}\coloneqq \{[z_0:z_1:z_2:z_3x_3:\cdots:z_3 x_N] \mid [z_0:z_1:z_2:z_3]\in\mathbb P^3\}\) contains \(P\). The intersection \(\mathcal Q_{[s:t]}\cap \bl_{\P^1\times P}(\tilde{h}^{-1}(x))\) is defined by \(z_3^2(sq_0(x_3,\ldots,x_N)+tq_1(x_3,\ldots,x_N))\), so it is either the 3-plane \(\bl_{\P^1\times P}(\tilde{h}^{-1}(x))\) or it is the rank 1 quadric surface \(2P\), and the former case occurs if and only if \(x\in\mathcal Q^{(2)}_P\).

    If \(x\in\mathcal P^{(2)}_P\setminus\mathcal Q^{(2)}_P\), then \(sq_0(x_3,\ldots,x_N)+tq_1(x_3,\ldots,x_N) \neq 0\). The intersection \(X\cap\bl_{\P^1\times P}(\tilde{h}^{-1}(x))\) is defined by \(z_3(l_{i0}(x_3,\ldots,x_N) z_0+l_{i1}(x_3,\ldots,x_N) z_1 +l_{i2}(x_3,\ldots,x_N)z_2+q_i(x_3,\ldots,x_N)z_3)\) for \(i\in\{0,1\}\), so it is given by
    \[X\cap\bl_{\P^1\times P}(\tilde{h}^{-1}(x)) = \begin{cases}
        V\left(z_3(l_{10}(\underline{x}) z_0+l_{11}(\underline{x}) z_1 +l_{12}(\underline{x})z_2+ q_1(\underline{x})z_3) , (s q_0(\underline{x}) + t q_1(\underline{x}))z_3^2\right) & \text{if }s\neq 0, \\
        V\left( z_3(l_{00}(\underline{x}) z_0+l_{01}(\underline{x}) z_1 +l_{02}(\underline{x})z_2+ q_0(\underline{x})z_3) , (s q_0(\underline{x}) + t q_1(\underline{x})) z_3^2 \right) & \text{if }t\neq 0.
    \end{cases}\]
    Thus, this intersection is equal to \(2P\) if and only if \(l_{ij}(\underline{x})=0\) for all \(i,j\); otherwise, the intersection is supported on \(P\) and has non-reduced structure along the line \(l_{i0}(\underline{x}) z_0+l_{i1}(\underline{x}) z_1 +l_{i2}(\underline{x})z_2 = z_3 = 0\) for the appropriate \(i\).
    The assumption that \(x\notin\mathcal Q^{(2)}_P\) implies in particular that at least one of \(q_0(\underline{x})\) or \(q_1(\underline{x})\) is nonzero. If furthermore \(l_{ij}(\underline{x})=0\) for all \(i,j\), then the point \([s':t']\coloneqq [-q_1(\underline{x}):q_0(\underline{x})]\) satisfies \(L_{\underline{x}} \subset\mathcal Q_{[s':t']}\).

    On the other hand, let \(L\subset\P^N\) be a 3-plane containing \(P\). Then \(L\) is of the form \(\{[z_0:z_1:z_2:z_3x_3:\cdots:z_3x_N] \mid [z_0:z_1:z_2:z_3]\in\P^3\}\) for some \([x_3:\cdots:x_N]\in\P^{N-3}\).
    To prove~\eqref{item:image-of-P^2-in-rel-F3}, it remains to show the following:
    \begin{enumerate}[label=(\alph*)]
        \item\label{item:pf-image-of-P^2-1} If \(\mathcal Q_{[s:t]}\cap L = 2P\), then \(([s:t],L) = \bl_{\P^1\times P}(\tilde{h}^{-1}(x))\) for a unique \(x=([s:t],\underline{x}) \in \mathcal P^{(2)}_P\setminus\mathcal Q^{(2)}_P\).
        \item\label{item:pf-image-of-P^2-2} If \(X\cap L\) is supported on \(P\) with non-reduced structure along a line, then there exists a unique \(x=([s:t],\underline{x}) \in \mathcal P^{(2)}_P\setminus\{l_{ij}=0\}\) such that \(([s:t],L) = \bl_{\P^1\times P}(\tilde{h}^{-1}(x))\).
    \end{enumerate}

    For~\ref{item:pf-image-of-P^2-1}, the intersection \(\mathcal Q_{[s:t]}\cap L\) is defined by \(z_3((sl_{00}(\underline{x})+tl_{10}(\underline{x}))z_0 + (sl_{01}(\underline{x})+tl_{11}(\underline{x}))z_1 + (sl_{02}(\underline{x})+tl_{12}(\underline{x}))z_2 + (sq_0(\underline{x})+tq_1(\underline{x}))z_3)\). This is equal to \(2P\) if and only if \(sl_{0j}(\underline{x})+tl_{1j}(\underline{x})=0\) for all \(j\in\{0,1,2\}\) and \(sq_0(\underline{x})+tq_1(\underline{x}) \neq 0\), which is precisely the condition that \(([s:t],\underline{x})\in\mathcal P^{(2)}_P\setminus\mathcal Q^{(2)}_P\).
    For~\ref{item:pf-image-of-P^2-2}, if \(X\cap L\) is supported on \(P\) with non-reduced structure along a line, then there is a unique member \(\mathcal Q_{[s:t]}\) of the pencil such that \(\mathcal Q_{[s:t]} \cap L = 2P\), so~\ref{item:pf-image-of-P^2-1} and the previous paragraphs imply the desired result.
    This shows~\eqref{item:image-of-P^2-in-rel-F3}.

    For~\eqref{item:image-of-Q^2-loci-in-rel-F3}, the first statement is Lemma~\ref{lem:r-planes-in-hyperbolic-reduction-loci}\eqref{item:image-of-Z-in-rel-Fr}, and the third statement is Lemma~\ref{lem:r-planes-in-hyperbolic-reduction-loci}\eqref{item:image-of-E-in-rel-Fr}. Thus, by Lemma~\ref{lem:hyperbolic-reduction-properties-not-pencil}\eqref{item:hyperbolic-reduction-not-pencil-embedding}, it remains to show that if \(x = ([s:t],\underline{x}) \in \mathcal Q_{Y_P}\setminus (\mathbb P^1\times Y_P)\), then \(X\cap\bl_{\P^1\times P}(\tilde{h}^{-1}(x)) = 2P\). This follows from a direct computation. Indeed, if \(x\in\mathcal Q_{Y_P}\setminus (\mathbb P^1\times Y_P)\) then \(l_{ij}(\underline{x})=0\) for all \(i,j\), \(sq_0(\underline{x})+tq_1(\underline{x})=0\), and at least one of \(q_0(\underline{x})\) or \(q_1(\underline{x})\) is nonzero. The intersection \(X\cap\bl_{\P^1\times P}(\tilde{h}^{-1}(x))\) is defined by \(q_1(\underline{x})z_3^2\) if \(s\neq 0\) and by \(q_0(\underline{x})z_3^2\) if \(t\neq 0\); in either case this is \(2P\). This shows~\eqref{item:image-of-Q^2-loci-in-rel-F3}.

    Part~\eqref{item:image-of-P^2-2P-locus} follows from the statements about \((\mathcal P^{(2)}_P\setminus\mathcal Q^{(2)}_P)\cap\{l_{ij}=0\}\) and \(\mathcal Q_{Y_P}=\mathcal Q^{(2)}_P\cap\{l_{ij}=0\}\) in~\eqref{item:image-of-P^2-in-rel-F3} and~\eqref{item:image-of-Q^2-loci-in-rel-F3}.
\end{proof}

\begin{lem}\label{lem:plane-line-tangent-isom}
Let $\ell$ be a line on $X$, let $P$ be a $2$-plane on $X$ containing $\ell$, and let $y\in Y_\ell$ be the point corresponding to $P$ under the isomorphism $Y_\ell\cong F_2(X)_\ell$ in Lemma~\ref{lem:r-planes-in-hyperbolic-reduction-loci}\eqref{item:image-of-E-in-rel-Fr}.
Then $\P^1\times (T_yY_\ell/y)$ is isomorphic to $\mathcal P_{Y_P}$ over $\P^1$.
\end{lem}
\begin{proof}
Choose coordinates on $\P^N$ so that $\ell=\{x_2=\cdots=x_N=0\}$ and $P=\{x_3=\cdots=x_N=0\}$.
Then $y=\{x_3=\cdots = x_N=0\}\in \P^{N-2}_{[x_2:\cdots:x_N]}$.
For $0\leq i\leq 1$ and $0\leq j\leq 1$, let $l_{ij}, q_i\in k[x_2,\ldots, x_N]$ be defined as in~\eqref{eqn:X} for $\Lambda =\ell$.
Then $l_{ij}\in k[x_3,\ldots, x_N]$ and we can write $q_i=l_{i2}x_{2}+q_i'$ for some $l_{i2}, q_i'\in k[x_3,\ldots, x_N]$.
It is direct to check that
\[
T_yY_\ell = \{l_{ij}=0\mid 0\leq i\leq 1, 0\leq j\leq 2\}\subset \P^{N-2}_{[x_2:\cdots:x_N]}.
\]
The projection $\P^{N-2}_{[x_2:\cdots:x_N]}\dashrightarrow \P^{N-3}_{[x_3:\cdots:x_N]}$ away from $y$ therefore induces an isomorphism 
\[T_yY_\ell/y\xrightarrow{\sim} \{l_{ij}=0\mid 0\leq i\leq 1, 0\leq j\leq 2\}\subset \P^{N-3}_{[x_3:\cdots: x_N]},\] 
and hence an isomorphism $\P^1\times(T_yY_\ell/y)\xrightarrow{\sim}\mathcal P_{Y_P}$ over $\P^1$, as desired.
\end{proof}

The following lemma gives the remaining possible intersection of \(X\) with a 3-plane in \(\P^N\) containing a 2-plane on \(X\).
\begin{lem}\label{lem:intersect-3plane-containing-2plane-cases}
Let \(N\geq 6\), and let \(X\subset\mathbb P^N\) be a smooth complete intersection of two quadrics.
Let \(P\) be a 2-plane in \(X\), and let \(L\) be a 3-plane in \(\mathbb P^N\) such that
\(P\subset L\).
Then \(X\cap L\) is one of the following:
\begin{enumerate}
\item\label{item:case-intersection-with-3plane-general-case} The union of \(P\) and a line \(m\subset L\) meeting \(P\) in a point,
\item\label{item:case-intersection-with-3plane-rk2-quadric-surface} A rank 2 quadric surface \(P + P'\) where \(P'\subset L\) is a different 2-plane,
\item\label{item:case-intersection-with-3plane-plane-nonred-along-line} A scheme supported on \(P\) that has non-reduced structure along a line in \(P\) and is reduced elsewhere,
\item\label{item:case-intersection-with-3plane-rk1-quadric-surface} A rank 1 quadric surface \(2P\), or
\item\label{item:case-intersection-with-3plane-contains-3plane} \(L\).
\end{enumerate}
The 3-plane \(L\) is contained in exactly one member of the pencil if and only if~\eqref{item:case-intersection-with-3plane-rk2-quadric-surface} or~\eqref{item:case-intersection-with-3plane-rk1-quadric-surface} occurs.
Furthermore, for a fixed 2-plane \(P\subset X\), case~\eqref{item:case-intersection-with-3plane-general-case} defines a (Zariski) open subset of \(\{L\in\Gr(4,N+1) \mid L \supset P\}\), and the other cases each define a (Zariski) locally closed subset.
\end{lem}

\begin{proof}
    We may choose coordinates on \(\P^N\) so that
    \[L = \{[x_0:x_1:x_2:x_3:0:\cdots:0]\}, \quad P = \{[x_0:x_1:x_2:0:0:\cdots:0]\}.\]
    The assumption that \(P\subset X\) implies that \(\mathcal Q_{[s:t]}\cap L \subset L\) is defined by \(s M_0 + t M_1\) where
    \[M_i = \begin{pmatrix} 0 & 0 & 0 & a_{i03}/2 \\ 0 & 0 & 0 & a_{i13}/2 \\ 0 & 0 & 0 & a_{i23}/2 \\ a_{i03}/2 & a_{i13}/2 & a_{i23}/2 & a_{i33} \end{pmatrix} .\]
    We may assume \(\#(\text{nonzero }a_{1j3}, 0\leq j \leq 3)\leq\#(\text{nonzero }a_{0j3}, 0\leq j \leq 3)\). Then the result follows by a direct computation of all possible cases for the entries of \(M_0\) and \(M_1\). For illustration, we compute a few of these cases here.

    For example, suppose that \(a_{003}=a_{013}=a_{023}=0\).
    If \(a_{033}=0\), then \(M_0=M_1\) is the zero matrix and \(L\subset X\) (case~\eqref{item:case-intersection-with-3plane-contains-3plane}). If \(a_{033}\neq 0\) and \(M_1\) is a scalar multiple of \(M_0\), then there is exactly one fiber of \(\mathcal Q\cap (\P^1\times L)\to\P^1\) that is equal to \(L\), and the intersection \(X\cap L\) is the rank 1 quadric surface \(2P\) (case~\eqref{item:case-intersection-with-3plane-rk1-quadric-surface}).
    If \(a_{033}\neq 0\) and \(M_1\) is not a scalar multiple of \(M_0\), then \(a_{1j3}\neq 0\) for exactly one \(j\in\{0,1,2\}\). In this case the fiber of \(\mathcal Q\cap (\P^1\times L)\to\P^1\) over \([s:t]\) is defined by \(x_3(a_{033}sx_3+a_{1j3}tx_j)\), so none of the fibers are \(L\). The intersection \(X\cap L\) is \(V(x_3^2,x_jx_3)\), which is supported on \(P\) and has non-reduced structure along the line \(V(x_j,x_3)\) (case~\eqref{item:case-intersection-with-3plane-plane-nonred-along-line}).

    Now suppose instead that \(a_{003}=a_{013}=a_{033}=0\) and \(a_{023}\neq 0\). If \(M_1\) is a scalar multiple of \(M_0\), then the intersection \(X\cap L\) is the rank 2 quadric surface \(P + V(x_2)\) (case~\eqref{item:case-intersection-with-3plane-rk2-quadric-surface}) and there is exactly one fiber of \(\mathcal Q\cap (\P^1\times L)\to\P^1\) that is equal to \(L\). If \(M_1\) is not a scalar multiple of \(M_0\), then \(a_{1j3}\neq 0\) for exactly one \(j\in\{0,1,3\}\).
    If \(j=3\), then the fiber of \(\mathcal Q\cap (\P^1\times L)\to\P^1\) over \([s:t]\) is the rank \(\leq 2\) quadric surface defined by \(x_3(a_{023}sx_2+a_{133}tx_3)\) and the intersection \(X\cap L\) is \(V(x_2x_3,x_3^2)\), which is supported on \(P\) and is non-reduced along the line \(V(x_2,x_3)\) (case~\eqref{item:case-intersection-with-3plane-plane-nonred-along-line}).
    If \(j\in\{0,1\}\), then the fiber of \(\mathcal Q\cap (\P^1\times L)\to\P^1\) over \([s:t]\) is the rank 2 quadric surface defined by \(x_3(a_{023}sx_2+a_{1j3}tx_j)\) and the intersection \(X\cap L\) is \(V(x_2x_3,x_jx_3)\), which is the union of \(P\) and the line \(V(x_2,x_j)\) (case~\eqref{item:case-intersection-with-3plane-general-case}).

    The statement about each of the cases defining a locally closed subset of \(\{L\in\Gr(4,N+1) \mid L \supset P\}\) follows from Lemma~\ref{lem:planes-in-hyperbolic-reduction-loci}. The locus defined by~\eqref{item:case-intersection-with-3plane-general-case} is open because it is the complement of the locus of 3-planes obtained as \(\bl_{\P^1\times P}(\tilde{h}^{-1}(x))\) for some \(x\in\mathcal P^{(2)}_P\).
\end{proof}

\section{The classes of exceptional loci of \texorpdfstring{\(F_1(X)\dashrightarrow\Sym^2\mathcal Q^{(1)}\)}{F1X-birational-map-toSym2Q1}}\label{sec:r=1}

Over an algebraically closed field $k$ of characteristic $\neq 2$, let $\varphi\colon \mathcal{Q}\to \P^1$ be a pencil of quadrics in $\P^N$ such that the base locus $X$ is smooth, and assume $N\geq 6$. Fix a general line $\ell\in F_1(X)$.
To prove Theorems~\ref{thm:main-odd} and~\ref{thm:main-even}, we will compute the classes of the exceptional loci of the birational map \(F_1(X)\dashrightarrow\Sym^2\mathcal Q^{(1)}_\ell\) from \cite[Theorem 1.3]{JS24}.
For this, we define the following loci in \(F_1(X)\) and \(\Sym^2\mathcal Q^{(1)}_\ell\).

Let $V\coloneqq F_1(X)$ and define its subsets
\begin{equation*}
\begin{split}
V_{\text{plane}}&\coloneqq\{m\mid \langle\ell, m\rangle=\P^2\text{ and }\langle\ell,m\rangle \in F_2(X)\},\\
V_{\ell+m}&\coloneqq\{m\mid \langle \ell, m\rangle =\P^2\text{ and }X\cap \langle\ell,m\rangle =\ell+m\},\\
V_{\text{$3$-plane}}&\coloneqq\{m\mid \langle\ell,m\rangle=\P^3 \text{ and }\langle\ell,m\rangle\in F_3(X)\},\\
V_{\text{quad,rk}\geq 3}&\coloneqq\{m\mid \langle\ell, m\rangle =\P^3\text{ and }X\cap \langle\ell,m\rangle\text{ is a quadric surface of rank $\geq 3$}\},\\
V_{\text{quad,rk}\leq 2}&\coloneqq\{m\mid \langle\ell, m\rangle =\P^3\text{ and }X\cap \langle\ell,m\rangle\text{ is a quadric surface of rank $\leq 2$}\},\\
V_{\ell+P}&\coloneqq\{m\mid \langle \ell,m\rangle =\P^3\text{ and }X\cap \langle \ell,m\rangle = \ell +P \text{ for some }P\in F_2(X)\},\\
V_{m+P}&\coloneqq\{m\mid \langle \ell,m\rangle =\P^3\text{ and }X\cap \langle \ell,m\rangle = m +P \text{ for some }P\in F_2(X)\},\\
V_{\ell+m+2m'}&\coloneqq\{m\mid \langle \ell,m\rangle =\P^3 \text{ and }X\cap \langle\ell,m\rangle = \ell+m+2m' \text{ for some }m'\in F_1(X)\}.
\end{split}
\end{equation*}
All of these sets are locally closed and non-empty, except that \(V_{\ell+P}=\emptyset\) for \(N=6\) by Lemma~\ref{lem:F_r(X)}, $V_{\text{plane}}=V_{\text{quad,rk}\leq 2}=V_{m+P}=\emptyset$ for $N=6,7$, and $V_{\text{$3$-plane}}=\emptyset$ for $N=6,7,8,9$ by Lemma \ref{lem:general-r-plane-contained-in-r+1-plane}. 

Lemma~\ref{lem:r-planes-in-hyperbolic-reduction-loci}\eqref{item:image-of-Z-in-rel-Fr} defines an isomorphism \(\mathcal Q^{(1)}_\ell\setminus\mathcal Q_{Y_\ell} \cong V_{\ell+m}\), so we may identify \(V_{\ell+m}\) with the open subset \(\mathcal Q^{(1)}_\ell\setminus\mathcal Q_{Y_\ell}\) of \(\mathcal Q^{(1)}_\ell\).
Now let $W\coloneqq \Sym^2(V_{\ell+m})$ and define its subsets 
\begin{equation*}
\begin{split}
W_{\ell+m_0}&\coloneqq \{(m_0,m_1)\mid \langle\ell,m_0,m_1\rangle =\P^2\text{ and }X\cap \langle \ell,m_0,m_1\rangle = \ell+m_0=\ell+m_1\}=\Delta_{W},\\
W_{\text{quad,rk}\geq 3}&\coloneqq\{(m_0,m_1)\mid \langle\ell, m_0,m_1\rangle=\P^3\text{ and }X\cap \langle \ell,m_0,m_1\rangle \text{ is a quadric surface of rank $\geq 3$}\},\\
W_{\text{quad,rk}\leq 2}&\coloneqq\{(m_0,m_1)\mid \langle \ell,m_0,m_1\rangle=\P^3 \text{ and }X\cap \langle\ell,m_0,m_1\rangle \text{ is a quadric surface of rank $\leq 2$} \},\\
W_{\ell+P}&\coloneqq\{(m_0,m_1)\mid \langle \ell,m_0,m_1\rangle =\P^3 \text{ and }X\cap \langle \ell,m_0,m_1\rangle =\ell+P\text{ for some }P\in F_2(X)\},\\
W_{2\ell+m_0+m_1}&\coloneqq \{(m_0,m_1)\mid \langle \ell,m_0,m_1\rangle =\P^3\text{ and }X\cap \langle \ell, m_0,m_1\rangle = 2\ell+m_0+m_1\},\\
W_{\ell+m_0+2m_1}&\coloneqq\{(m_0,m_1)\mid \langle \ell,m_0,m_1\rangle =\P^3 \text{ and }X\cap \langle \ell,m_0,m_1\rangle = \ell + m_0 + 2m_1\},\\
W_{\mathrlap{\times}+}&\coloneqq \left\{(m_0,m_1)\middle\vert \begin{array}{ll}
\langle \ell,m_0,m_1\rangle =\P^3, X\cap \langle \ell, m_0,m_1\rangle = \ell + m_0 + m_1 + m \\ \text{for some }m\in F_1(X) \setminus \{\ell,m_0,m_1\},\text{ and }\ell\cap m_0\cap m_1\cap m\neq\emptyset
\end{array}\right\}.
\end{split}
\end{equation*}
All of these sets are locally closed and non-empty, except that \(W_{\ell+P}=\emptyset\) for \(N=6\) by Lemma~\ref{lem:F_r(X)}, and $W_{\text{quad,rk}\leq 2}=\emptyset$ for $N=6,7$ by Lemma \ref{lem:general-r-plane-contained-in-r+1-plane}.

\begin{prop}\label{prop:v-w-isom}
In the above setting,
we get an isomorphism
\begin{equation}\label{eq:isom-V-W}
\begin{split}
&V\setminus (\{\ell\}\sqcup V_{\text{plane}}\sqcup V_{\ell+m}\sqcup V_{\text{$3$-plane}}\sqcup  V_{\text{quad,rk}\geq 3}\sqcup V_{\text{quad,rk}\leq 2}\sqcup V_{\ell+P}\sqcup V_{m+P}\sqcup V_{\ell+m+2m'})\\
&\xrightarrow{\sim}W\setminus (W_{\ell+m_0}\sqcup W_{\text{quad,rk}\geq 3}\sqcup W_{\text{quad,rk}\leq 2}\sqcup W_{\ell+P} \sqcup W_{2\ell+m_0+m_1}\sqcup W_{\ell+m_0+2m_1}\sqcup W_{\mathrlap{\times}+})
\end{split}
\end{equation}
by sending $m$ to $(m_0, m_1)$, where $X\cap \langle \ell, m\rangle = \ell+m + m_0 + m_1$.
\end{prop}
\begin{proof}
By Lemma~\ref{lem:line-disjoint-from-l-cases}, Corollary~\ref{cor:lines-meeting-l-span-P3-disjoint}, Lemma~\ref{lem:lines-meeting-l-span-P2}, and Lemma~\ref{lem:lines-meeting-l-span-P3-common-point}, the domain (resp. codomain) of \eqref{eq:isom-V-W} 
is exactly the set of $m$ (resp. $(m_0,m_1)$) such that $\langle \ell,m\rangle =\P^3$ (resp. $\langle \ell,m_0,m_1\rangle=\P^3$) and the intersection $X\cap \langle \ell, m\rangle$ (resp. $X\cap \langle \ell, m_0,m_1\rangle$) is singular of Reid type.
In particular, the domain (resp. codomain) of~\eqref{eq:isom-V-W} is open in \(V\) (resp. \(W\)), and hence is irreducible. Furthermore,
the domain and codomain of \eqref{eq:isom-V-W} are both smooth quasi-projective varieties by Lemma~\ref{lem:F_r(X)} and smoothness of \(\mathcal Q^{(1)}_\ell\) (Section~\ref{sec:prelim-F_r} and Proposition~\ref{prop:KS-hyperbolic-reduction}\eqref{item:KS-hyperbolic-red-degeneracy}), so by Zariski's main theorem, \eqref{eq:isom-V-W} is an isomorphism as desired.
\end{proof}

Proposition \ref{prop:v-w-isom} reduces computing the class of $V$ in \(\widetilde{K}_0(\mathrm{Var}/k)\) to computing the classes of $W$ and the exceptional loci of the birational map $V\dashrightarrow W$, which we will achieve in the following subsections. Here is the content of each subsection: \ref{section:v-plane}. $V_{\text{plane}}$; \ref{section:v-w-ell+m}. $V_{\ell+m}$, $W_{\ell+m_0}$, $W$; \ref{section:v-3plane}. $V_{\text{$3$-plane}}$; \ref{section:quad}. $V_{\text{quad,rk}\geq 3}$, $W_{\text{quad,rk}\geq 3}$, $V_{\text{quad,rk}\leq 2}$, $W_{\text{quad,rk}\leq 2}$; \ref{section:ell+P}. $V_{\ell+P}$, $W_{\ell+P}$; \ref{section:v-m+P}. $V_{m+P}$; \ref{section:v-ell+m+2m'}. $V_{\ell+m+2m'}$; \ref{section:w-2ell+m_0+m_1}. $W_{2\ell+m_0+m_1}$; \ref{section:w-concurrent}. $W_{\ell+m_0+2m_1}\sqcup W_{\mathrlap{\times}+}$.
See Proposition \ref{prop:excep-loci} for a summary of the results of the computation.

Throughout this section, brackets will denote classes in \(\widetilde{K}_0(\mathrm{Var}/k)\).

\subsection{\texorpdfstring{$V_{\text{plane}}$}{Vplane}}\label{section:v-plane}
Recall that
\[
F_2(X)_\ell \coloneqq\{P\in F_2(X)\mid \ell\subset P\}\cong Y_\ell
\]
by Lemma \ref{lem:r-planes-in-hyperbolic-reduction-loci}\eqref{item:image-of-E-in-rel-Fr}.
Let $U_2(X)_\ell\to F_2(X)_\ell$ be the pullback of the universal family $U_2(X)\to F_2(X)$,
let $F_1(U_2(X)_\ell/F_2(X)_\ell)\to F_2(X)_\ell$ be the natural map,
and $s_\ell\colon F_2(X)_\ell\to F_1(U_2(X)_\ell/F_2(X)_\ell)$ be the section corresponding to $\ell$.
The map
\[
F_1(U_2(X)_\ell/F_2(X)_\ell)\setminus s_\ell \to V_{\text{plane}}, \qquad (m,P)\mapsto m
\]
is a bijection with inverse $m\mapsto (m,\langle \ell,m\rangle)$, so by Lemma \ref{lem:bijection-same-class}, $[V_{\text{plane}}]=[F_1(U_2(X)_\ell/F_2(X)_\ell)\setminus s_\ell]$. 
Moreover, $F_1(U_2(X)_\ell/F_2(X)_\ell)\setminus s_\ell \to F_2(X)_\ell$ is a ($\P^2\setminus \text{point}$)-bundle, and hence Lemma~\ref{lem:fibration-general} implies that $[F_1(U_2(X)_\ell/F_2(X)_\ell)\setminus s_\ell]= [F_2(X)_\ell](\L^2+\L)=[Y_\ell](\L^2+\L)$.
We conclude that
\begin{equation}\label{eq:v-plane}
[V_{\text{plane}}]=[Y_\ell](\L^2+\L).
\end{equation}

\subsection{\texorpdfstring{$V_{\ell+m}$}{Vellplusm}, \texorpdfstring{$W_{\ell+m_0}$}{Wellplusm0}, and \texorpdfstring{$W$}{W}}\label{section:v-w-ell+m}
By Lemma \ref{lem:r-planes-in-hyperbolic-reduction-loci}\eqref{item:image-of-Z-in-rel-Fr}, we have \(\mathcal{Q}^{(1)}_\ell\setminus \mathcal{Q}_{Y_\ell}\xrightarrow{\sim} V_{\ell+m}\), so
\begin{equation}\label{eq:v-w-ell-m}
[V_{\ell+m}]=[W_{\ell+m_0}]=[\mathcal{Q}_{\ell}^{(1)}]-[\mathcal{Q}_{Y_\ell}].
\end{equation}
By Lemma~\ref{lem:K_0-sym-formulas}\eqref{item:formula-sym}, \(W=\Sym^2(V_{\ell+m})\) has class
\begin{equation}\label{eq:w}
[W]=[\Sym^2\mathcal{Q}^{(1)}_\ell]-[\Sym^2\mathcal{Q}_{Y_\ell}] + [\mathcal{Q}_{Y_\ell}]^2 - [\mathcal{Q}^{(1)}_\ell][\mathcal{Q}_{Y_\ell}].
\end{equation}

\subsection{\texorpdfstring{$V_{\text{$3$-plane}}$}{V3plane}}\label{section:v-3plane}

Recall that
\[
F_3(X)_\ell\coloneqq \{L\in F_3(X)\mid \ell\subset L\}\cong F_1(Y_\ell)
\]
by Lemma \ref{lem:r-planes-in-hyperbolic-reduction-loci}\eqref{item:relative-m-planes-in-Qr}.
Let $U_3(X)_\ell\to F_3(X)_\ell$ be the pullback of the universal family $U_3(X)\to F_3(X)$,
and define
\[\
F_1(U_3(X)_\ell/F_3(X)_\ell)_{\ell}^{\text{int}}\coloneqq \{(m,L)\in F_1(U_3(X)_\ell/F_3(X)_\ell)\mid \ell \cap m\neq \emptyset\}.
\]
The map
\[
F_1(U_3(X)_\ell/F_3(X)_\ell)\setminus F_1(U_3(X)_\ell/F_3(X)_\ell)_{\ell}^{\text{int}}\to V_{\text{$3$-plane}}, \qquad (m,L)\mapsto m
\]
is a bijection with inverse $m\mapsto (m,\langle \ell,m\rangle)$; hence, by Lemma \ref{lem:bijection-same-class}, $[V_{\text{$3$-plane}}]=[F_1(U_3(X)_\ell/F_3(X)_\ell)]- [F_1(U_3(X)_\ell/F_3(X)_\ell)_{\ell}^{\text{int}}]$.
Moreover, $F_1(U_3(X)_\ell/F_3(X)_\ell)\setminus F_1(U_3(X)_\ell/F_3(X)_\ell)_{\ell}^{\text{int}}\to F_3(X)_\ell$ is an $\A^4$-bundle whose fibers are isomorphic to standard affine covers of $\Gr(2,4)$, so by Lemma~\ref{lem:fibration-general}, $[F_1(U_3(X)_\ell/F_3(X)_\ell)]- [F_1(U_3(X)_\ell/F_3(X)_\ell)_{\ell}^{\text{int}}]=[F_3(X)_\ell]\L^4=[F_1(Y_\ell)]\L^4$.
Therefore
\begin{equation}\label{eq:v-3-plane}
[V_{\text{$3$-plane}}]=[F_1(Y_\ell)]\L^4.    
\end{equation}

\subsection{\texorpdfstring{$V_{\text{quad,rk}\geq 3}$}{Vquadrankgeq3}, \texorpdfstring{$W_{\text{quad,rk}\geq 3}$}{Wquadrankgeq3}, \texorpdfstring{$V_{\text{quad,rk}\leq 2}$}{Vquadrankleq2}, and \texorpdfstring{$W_{\text{quad,rk}\leq 2}$}{Wquadrankleq2}}\label{section:quad}
Let \(S_3(X)\) be the Hilbert scheme of quadric surfaces contained in \(X\),
and define
\begin{equation*}
\begin{split}
S_3(X)_\ell&\coloneqq \{Q\in S_3(X)\mid \ell\subset Q\}, \\
S_3(X)_{\ell,\text{$3$-plane}}&\coloneqq\{Q\in S_3(X)_\ell\mid Q\subset L\text{ for some }L\in F_3(X)\},\\
S_3(X)_{\ell,\text{rk}\geq3}&\coloneqq \{Q\in S_3(X)_\ell\setminus S_3(X)_{\ell,\text{$3$-plane}}\mid \text{rk}\, Q\geq 3\}, \\
S_3(X)_{\ell,\text{rk}\leq 2}^\circ&\coloneqq\{Q\in S_3(X)_\ell\setminus S_3(X)_{\ell,\text{$3$-plane}}\mid \text{rk}\, Q\leq 2\text{ and }\ell\not\subset \Sing Q\}.
\end{split}
\end{equation*}
All of these sets are non-empty, except that $S_3(X)_{\ell,\text{rk}\leq 2}^\circ=\emptyset$ for $N=6,7$, and $S_3(X)_{\ell,\text{$3$-plane}}=\emptyset$ for $N=6,7,8,9$ by Lemma \ref{lem:general-r-plane-contained-in-r+1-plane}.

We first aim to express the classes of $V_{\text{quad,rk}\geq 3}$, $W_{\text{quad,rk}\geq 3}$, $V_{\text{quad,rk}\leq 2}$, and $W_{\text{quad,rk}\leq 2}$ in terms of the classes of $S_3(X)_{\ell,\text{rk}\geq3}$ and $S_3(X)_{\ell,\text{rk}\leq 2}^\circ$.
We will first address $V_{\text{quad,rk}\geq 3}$ and $W_{\text{quad,rk}\geq 3}$. For \(i=3,4\), define
\[S_3(X)_{\ell,\text{rk}=i}\coloneqq \{Q\in S_3(X)_{\ell,\text{rk}\geq3} \mid \text{rk}\, Q= i\}.\]
The map \(V_{\text{quad,rk}\geq 3} \to S_3(X)_{\ell,\text{rk}\geq3}\) given by \(m\mapsto X\cap\langle\ell,m\rangle\) has image equal to \(S_3(X)_{\ell,\text{rk}=4}\), and over \(S_3(X)_{\ell,\text{rk}=4}\) it is an \(\mathbb A^1\)-bundle whose fiber over \(Q\) is the set of lines on \(Q\) in the same ruling as \(\ell\) and not equal to \(\ell\).
Next, consider the map \(W_{\text{quad,rk}\geq 3} \to S_3(X)_{\ell,\text{rk}\geq3}\) given by \((m_0,m_1)\mapsto X\cap\langle\ell,m_0,m_1\rangle\). Over \(S_3(X)_{\ell,\text{rk}=4}\) this is a $((\Sym^2 \P^1)\setminus \P^1)$-bundle, whose fiber over \(Q \in S_3(X)_{\ell,\text{rk}=4}\) is the set of pairs of distinct lines in the opposite ruling from \(\ell\), and over \(S_3(X)_{\ell,\text{rk}=3}\) this is a $((\Sym^2 \A^1)\setminus \A^1)$-bundle, whose fiber over \(Q \in S_3(X)_{\ell,\text{rk}=3}\) is the set of pairs of distinct lines not equal to \(\ell\). Therefore, by Lemma~\ref{lem:fibration-general},
\[
[V_{\text{quad,rk}\geq 3}]=[S_3(X)_{\ell,\text{rk}\geq3}]\L-[S_3(X)_{\ell,\text{rk}=3}]\L,\qquad [W_{\text{quad,rk}\geq3}]=[S_3(X)_{\ell,\text{rk}\geq 3}]\L^2-[S_3(X)_{\ell,\text{rk}=3}]\L.
\]
In addition, the maps
\begin{equation*}
\begin{split}
V_{\text{quad,rk}\leq 2}\to S_3(X)_{\ell,\text{rk}\leq 2}^\circ,&\qquad m\mapsto X\cap \langle \ell,m\rangle,\\
W_{\text{quad,rk}\leq 2}\to S_3(X)_{\ell,\text{rk}\leq 2}^\circ,&\qquad (m_0,m_1)\mapsto X\cap \langle \ell,m_0,m_1\rangle
\end{split}
\end{equation*}
are respectively $\A^2$- and $(\A^2\setminus \A^1)$-bundles, so by Lemma~\ref{lem:fibration-general} and Lemma \ref{lem:bijection-same-class}, 
\[
[V_{\text{quad,rk}\leq 2}]=[S_3(X)_{\ell,\text{rk}\leq 2}^\circ]\L^2,\qquad [W_{\text{quad,rk}\leq 2}]=[S_3(X)_{\ell,\text{rk}\leq 2}^\circ](\L^2-\L).
\]

It remains to compute the classes of $S_3(X)_{\ell,\text{rk}\geq  3}$ and $S_3(X)_{\ell,\text{rk}\leq 2}^\circ$. 
By Lemma \ref{lem:lines-in-hyperbolic-reduction-loci}\eqref{item:rel-lines-in-Q^1-in-rel-F3},
\begin{equation*}
\begin{split}
&F_1(\mathcal{Q}^{(1)}_\ell/\P^1)\setminus F_1(\mathcal{Q}^{(1)}_\ell/\P^1)_{\P^1\times Y_\ell}^{\text{int}}\xrightarrow{\sim} \{(L,t)\in F_3(\mathcal{Q}/\P^1)_\ell\mid X\cap L \text{ is a quadric surface of rank $\geq 3$}\}\\
&F_1(\mathcal{Q}^{(1)}_\ell/\P^1)_{\P^1\times Y_\ell}^{\text{int}}\setminus F_1(\mathcal{Q}_{Y_\ell}/\P^1)\xrightarrow{\sim} \left\{(L,t)\in F_3(\mathcal{Q}/\P^1)_\ell \,\middle\vert\, \substack{X\cap L\text{ is a quadric surface of rank $\leq 2$} \\ \text{and $\ell\not\subset \Sing (X\cap L)$}}\right\},
\end{split}
\end{equation*}
and the maps
\begin{equation*}
\begin{split}
&\{(L,t)\in F_3(\mathcal{Q}/\P^1)_\ell\mid X\cap L \text{ is a quadric surface of rank $\geq 3$}\}\to S_3(X)_{\ell,\text{rk}\geq 3}, \\
&\{(L,t)\in F_3(\mathcal{Q}/\P^1)_\ell\mid X\cap L \text{ is a quadric surface of rank $\leq 2$ and $\ell\not\subset \Sing (X\cap L)$}\}\to S_3(X)_{\ell,\text{rk}\leq 2}^\circ
\end{split}
\end{equation*}
given by $(L,t)\mapsto X\cap L$
are bijections with inverses $Q\mapsto (\langle Q\rangle,t)$, where $t\in \P^1$ is the unique point such that $L\subset \mathcal Q_t$.
Here the linear span of \(Q\) is a 3-plane since the assumption that \(\ell\not\subset\Sing Q\) implies \(Q\) has rank at least 2.
Thus, by Lemma \ref{lem:bijection-same-class}, 
\[
[S_3(X)_{\ell,\text{rk}\geq 3}]=[F_1(\mathcal{Q}^{(1)}_\ell/\P^1)]-[F_1(\mathcal{Q}^{(1)}_\ell/\P^1)_{\P^1\times Y_{\ell}}^{\text{int}}],\qquad [S_3(X)_{\ell,\text{rk}\leq 2}^\circ]=[F_1(\mathcal{Q}^{(1)}_\ell/\P^1)_{\P^1\times Y_\ell}^{\text{int}}]-[F_1(\mathcal{Q}_{Y_\ell}/\P^1)].
\]

Finally, we compute the class of $F_1(\mathcal{Q}^{(1)}_\ell/\P^1)_{\P^1\times Y_\ell}^{\text{int}}$.
First, if \(N=6,7\), then \(Y_\ell=\emptyset\) by Lemma~\ref{lem:general-l-hyperbolic-red}\eqref{item:general-l-Y-smooth-ci} so \(F_1(\mathcal{Q}^{(1)}_\ell/\P^1)_{\P^1\times Y_\ell}^{\text{int}}=\emptyset\). Therefore, for the remainder of the computation of \([F_1(\mathcal{Q}^{(1)}_\ell/\P^1)_{\P^1\times Y_\ell}^{\text{int}}]\) we assume \(N\geq 8\). Then
the map $\id_{Y_\ell}\times (\varphi^{(1)}_\ell|_{{\mathcal Q}_{Y_\ell}}) \colon Y_\ell\times \mathcal{Q}_{Y_\ell}\to Y_\ell\times \P^1$ has a section $s\colon Y_\ell\times \P^1\to Y_\ell\times \mathcal{Q}_{Y_\ell}$ with image $\Delta_{Y_\ell}\times \P^1\subset Y_\ell\times \mathcal{Q}_{Y_\ell}$ given by $(y,t)\mapsto (y,(y,t))$, and thus, so does $\id_{Y_\ell}\times \varphi^{(1)}_\ell\colon Y_\ell\times\mathcal{Q}^{(1)}_\ell\to Y_\ell\times \P^1$.
These sections are both nondegenerate, the first one by the fact that the base locus $Y_\ell$ of the pencil $\mathcal{Q}_{Y_\ell}\to \P^1$ is smooth by Lemma~\ref{lem:general-l-hyperbolic-red}\eqref{item:general-l-Y-smooth-ci}, and the second one by Lemma~\ref{lem:general-l-vertices-avoid-Z}.
Let \((Y_\ell\times \mathcal{Q}^{(1)}_\ell)^{(0)}_s\to Y_\ell\times\P^1\) and \((Y_\ell\times \mathcal{Q}_{Y_\ell})^{(0)}_s \to Y_\ell\times\P^1\) be the corresponding hyperbolic reductions. Then by Lemma~\ref{lem:hyperbolic-reduction-properties-not-pencil}\eqref{item:hyperbolic-reduction-not-pencil-embedding},
\[
(Y_\ell\times \mathcal{Q}^{(1)}_\ell)^{(0)}_s\setminus (Y_\ell\times \mathcal{Q}_{Y_\ell})^{(0)}_s \xrightarrow{\sim} \{(n,y,t) \in F_1(\mathcal Q^{(1)}_\ell/\P^1) \times Y_\ell \mid y \in n \subset (\mathcal Q^{(1)}_\ell)_t \text{ and }n\not\subset\mathcal Q_{Y_\ell}\}.
\]
Since \(\mathcal Q_{Y_\ell}\subset\mathcal Q^{(1)}_\ell\) is defined by linear equations~\eqref{eqns:Z-and-Y}, a line in \(F_1(\mathcal Q^{(1)}_\ell/\P^1) \setminus F_1(\mathcal Q_{Y_\ell}/\P^1)\) meets \(Y_\ell\) in at most one point. So the map
\[
(Y_\ell\times \mathcal{Q}^{(1)}_\ell)^{(0)}_s\setminus (Y_\ell\times \mathcal{Q}_{Y_\ell})^{(0)}_s\to F_1(\mathcal{Q}^{(1)}_\ell/\P^1)_{Y_\ell\times \P^1}^{\text{int}}\setminus F_1(\mathcal{Q}_{Y_\ell}/\P^1),\qquad (n,y,t)\mapsto (n,t)
\]
is a bijection with inverse $(n,t)\mapsto (n, Y_\ell\cap n, t)$, and hence by Lemma \ref{lem:bijection-same-class},
\[
[F_1(\mathcal{Q}^{(1)}_\ell/\P^1)_{Y_\ell\times \P^1}^{\text{int}}]=[F_1(\mathcal{Q}_{Y_\ell}/\P^1)]+ [(Y_\ell\times \mathcal{Q}^{(1)}_\ell)^{(0)}_s]-[(Y_\ell\times \mathcal{Q}_{Y_\ell})^{(0)}_s].
\]
By Proposition \ref{prop:KS-hyperbolic-reduction}\eqref{item:KS-hyperbolic-reduction-indep},
\[
[(Y_\ell\times \mathcal{Q}^{(1)}_\ell)^{(0)}_s]=[Y_\ell][(\mathcal{Q}^{(1)})^{(0)}]=[Y_\ell][\mathcal{Q}^{(2)}],\qquad [(Y_\ell\times \mathcal{Q}_{Y_\ell})^{(0)}_s]=[Y_\ell][(\mathcal{Q}_{Y_\ell})^{(0)}].
\]
Here the superscripts \(^{(r)}\) denote the hyperbolic reduction with respect to any nondegenerate \(r\)-section.
We then get
\[
[F_1(\mathcal{Q}^{(1)}_\ell/\P^1)_{Y_\ell\times \P^1}^{\text{int}}]
=[F_1(\mathcal{Q}_{Y_\ell}/\P^1)]+[Y_\ell]([\mathcal{Q}^{(2)}]-[(\mathcal{Q}_{Y_\ell})^{(0)}]).
\]

Combining the results in the last three paragraphs yields the following equalities for any \(N\geq 6\), where the terms involving \([Y_\ell]\) are taken to be 0 if \(N=6,7\):
\begin{equation*}
\begin{split}
[V_{\text{quad,rk}\geq 3}]&=([F_1(\mathcal{Q}^{(1)}_\ell/\P^1)]-[F_1(\mathcal{Q}_{Y_\ell}/\P^1)]-[Y_\ell]([\mathcal{Q}^{(2)}]-[(\mathcal{Q}_{Y_\ell})^{(0)}]))\L-[S_3(X)_{\ell,\text{rk}=3}]\L,\\
[W_{\text{quad,rk}\geq 3}]&=([F_1(\mathcal{Q}^{(1)}_\ell/\P^1)]-[F_1(\mathcal{Q}_{Y_\ell}/\P^1)]-[Y_\ell]([\mathcal{Q}^{(2)}]-[(\mathcal{Q}_{Y_\ell})^{(0)}]))\L^2 - [S_3(X)_{\ell,\text{rk}=3}]\L,\\
[V_{\text{quad,rk}\leq 2}]&=[Y_\ell]([\mathcal{Q}^{(2)}]-[(\mathcal{Q}_{Y_\ell})^{(0)}])\L^2,\\
[W_{\text{quad,rk}\leq 2}]&=[Y_\ell]([\mathcal{Q}^{(2)}]-[(\mathcal{Q}_{Y_\ell})^{(0)}])(\L^2-\L).
\end{split}
\end{equation*}
Since $[\mathcal{Q}^{(1)}_\ell]=[\P^1](1+\L^{N-5})+[\mathcal{Q}^{(2)}]\L$ and (if \(N\geq 8\))
$[\mathcal{Q}_{Y_\ell}]=[\P^1](1+\L^{N-7})+[(\mathcal{Q}_{Y_\ell})^{(0)}]\L$ by Corollary \ref{cor:formula-class-of-hyperbolic-reduction} and Proposition \ref{prop:KS-hyperbolic-reduction}\eqref{item:KS-hyperbolic-reduction-formula},
we get
\begin{equation}\label{eq:v-w-quad}
\begin{split}
[V_{\text{quad,rk}\geq 3}]&=([F_1(\mathcal{Q}^{(1)}_\ell/\P^1)]-[F_1(\mathcal{Q}_{Y_\ell}/\P^1)])\L-[Y_\ell]([\mathcal{Q}^{(1)}_\ell]
-[\mathcal{Q}_{Y_\ell}]+[\P^1](\L^{N-7}-\L^{N-5}))-[S_3(X)_{\ell,\text{rk}=3}]\L,\\
[W_{\text{quad,rk}\geq 3}]&=([F_1(\mathcal{Q}^{(1)}_\ell/\P^1)]-[F_1(\mathcal{Q}_{Y_\ell}/\P^1)])\L^2-[Y_\ell]([\mathcal{Q}^{(1)}_\ell]-[\mathcal{Q}_{Y_\ell}]+[\P^1](\L^{N-7}-\L^{N-5}))\L - [S_3(X)_{\ell,\text{rk}=3}]\L,\\
[V_{\text{quad,rk}\leq 2}]&=[Y_\ell]([\mathcal{Q}^{(1)}_\ell]-[\mathcal{Q}_{Y_\ell}]+[\P^1](\L^{N-7}-\L^{N-5}))\L,\\
[W_{\text{quad,rk}\leq 2}]&=[Y_\ell]([\mathcal{Q}^{(1)}_\ell]-[\mathcal{Q}_{Y_\ell}]+[\P^1](\L^{N-7}-\L^{N-5}))(\L-1).
\end{split}
\end{equation}
In~\eqref{eq:v-w-quad}, if \(N=6,7\), then \(Y_\ell=F_1(\mathcal{Q}_{Y_\ell}/\P^1)=\emptyset\) by Lemma~\ref{lem:general-l-hyperbolic-red}, and these formulas are equal to the following:
\begin{equation*}\begin{split}
[V_{\text{quad,rk}\geq 3}]&=[F_1(\mathcal{Q}^{(1)}_\ell/\P^1)]\L-[S_3(X)_{\ell,\text{rk}=3}]\L,\\
[W_{\text{quad,rk}\geq 3}]&=[F_1(\mathcal{Q}^{(1)}_\ell/\P^1)]\L^2 - [S_3(X)_{\ell,\text{rk}=3}]\L,\\
[V_{\text{quad,rk}\leq 2}]&=[W_{\text{quad,rk}\leq 2}]=0.
\end{split}\end{equation*}

\subsection{\texorpdfstring{$V_{\ell+P}$}{VellplusP} and \texorpdfstring{$W_{\ell+P}$}{WellplusP}}\label{section:ell+P}
Define
\[
F_2(X)_\ell^{\text{int}}\coloneqq \{P\in F_2(X)\mid \ell\cap P\neq \emptyset\}.
\]
Then $F_2(X)_\ell\subset F_2(X)_\ell^{\text{int}}$.
Let
\begin{equation*}
\begin{split}
F_2(X)_{\ell,\text{quad}}^{\text{int}}&\coloneqq \{P\in F_2(X)_\ell^{\text{int}}\mid \langle \ell, P\rangle = \P^3\text{ and }X\cap \langle \ell,P\rangle \text{ is a rank $2$ quadric surface}\},\\
F_2(X)_{\ell,\text{$3$-plane}}^{\text{int}}&\coloneqq\{P\in F_2(X)_\ell^{\text{int}}\mid \langle\ell, P\rangle =\P^3\text{ and }\langle \ell,P\rangle \in F_3(X)\},\\
(F_2(X)_{\ell}^{\text{int}})^\circ&\coloneqq F_2(X)_{\ell}^{\text{int}}\setminus (F_2(X)_\ell\sqcup F_2(X)_{\ell,\text{quad}}^{\text{int}}\sqcup F_2(X)_{\ell,\text{$3$-plane}}^{\text{int}}).
\end{split}
\end{equation*}
Let $(U_2(X)_\ell^{\text{int}})^\circ\to (F_2(X)_\ell^{\text{int}})^\circ$ be the pullback of the universal family $U_2(X)\to F_2(X)$,
and define
\[
F^{\text{int}}
\coloneqq \{(m,P)\in F_1((U_2(X)_\ell^{\text{int}})^\circ/(F_2(X)_\ell^{\text{int}})^\circ)\mid \ell\cap m\neq\emptyset \}.
\]
Any \((m,P)\in F_1((U_2(X)_\ell^{\text{int}})^\circ/(F_2(X)_\ell^{\text{int}})^\circ)\) satisfies \(m\subset P\) and therefore \(m\neq \ell\). So by Lemma~\ref{lem:intersect-3plane-containing-2plane-cases}, the maps
\begin{equation*}
\begin{split}
F_1((U_2(X)_{\ell}^{\text{int}})^\circ/(F_2(X)_\ell^{\text{int}})^\circ)\setminus F^{\text{int}} \to V_{\ell+P},&\qquad (m,P)\mapsto m,\\
\Sym^2(F^{\text{int}}/(F_2(X)_\ell^{\text{int}})^\circ)\setminus \Delta\to W_{\ell+P},&\qquad (m_0,m_1,P)\mapsto (m_0,m_1)
\end{split}
\end{equation*}
are bijections with inverses $m\mapsto (m,P)$ and $(m_0,m_1)\mapsto (m_0,m_1, P')$, where $X\cap \langle \ell,m\rangle = \ell + P$ and $X\cap \langle \ell, m_0,m_1\rangle =\ell + P'$, and hence by Lemma \ref{lem:bijection-same-class}, 
\begin{equation*}
\begin{split}
[V_{\ell+P}]&=[F_1((U_2(X)_{\ell}^{\text{int}})^\circ/(F_2(X)_\ell^{\text{int}})^\circ)\setminus F^{\text{int}} ],\\
[W_{\ell+P}]&=[\Sym^2(F^{\text{int}}/(F_2(X)_\ell^{\text{int}})^\circ)\setminus \Delta].
\end{split}
\end{equation*}
The map $F_1((U_2(X)_{\ell}^{\text{int}})^\circ/(F_2(X)_\ell^{\text{int}})^\circ)\setminus F^{\text{int}} \to (F_2(X)_\ell^{\text{int}})^\circ$ is a $(\P^2\setminus \P^1)=\A^2$-bundle, and the map 
$\Sym^2(F^{\text{int}}/(F_2(X)_\ell^{\text{int}})^\circ)\setminus \Delta\to (F_2(X)_{\ell}^{\text{int}})^\circ$ is a $((\Sym^2\P^1)\setminus \P^1)$-bundle.
Hence by Lemma~\ref{lem:fibration-general} and Lemma \ref{lem:bijection-same-class}, 
\[[F_1((U_2(X)_{\ell}^{\text{int}})^\circ/(F_2(X)_\ell^{\text{int}})^\circ)\setminus F^{\text{int}}] = [(F_2(X)_\ell^{\text{int}})^\circ]\L^2 =[\Sym^2(F^{\text{int}}/(F_2(X)_\ell^{\text{int}})^\circ)\setminus \Delta].\]
Therefore
\begin{equation}\label{eq:ell+P}
[V_{\ell+P}]=[W_{\ell+P}]=[(F_2(X)_\ell^{\text{int}})^\circ]\L^2.
\end{equation}

\subsection{\texorpdfstring{$V_{m+P}$}{VmplusP}}\label{section:v-m+P}
If \(N=6,7\) then \(V_{m+P}=\emptyset\) by Lemma~\ref{lem:general-r-plane-contained-in-r+1-plane}. So for the computation of the class of \(V_{m+P}\) we assume \(N\geq 8\).
Let
\[I\coloneqq \{(P,L)\in F_2(X)_\ell\times \Gr(4,N+1)\mid P\subset L\},\]
and define the following disjoint subsets of \(I\):
\begin{equation*}
\begin{split}
I_{\text{$3$-plane}}&\coloneqq \{(P,L)\in I\mid L\subset X\},\\
I_{\text{quad}}&\coloneqq\{(P,L)\in I\mid X\cap L\text{ is a quadric surface of rank}\leq 2\},\\
I_{\text{plane}^+}&\coloneqq\{(P,L)\in I\mid X\cap L\text{ is $P$ with an embedded line}\},\\
I_{m+P, \ell\text{-inc}}&\coloneqq\{(P,L)\in I\mid X\cap L = m+P\text{ for some }m\in V_{\ell+m}\}.
\end{split}
\end{equation*}
By Lemma~\ref{lem:intersect-3plane-containing-2plane-cases}, the map
\[
I\setminus (I_{\text{$3$-plane}}\sqcup I_{\text{quad}}\sqcup I_{\text{plane}^+}\sqcup I_{m+P,\ell\text{-inc}})\to V_{m+P},\qquad (P, L)\mapsto m,
\]
where $X\cap L = m+ P$, is a bijection with inverse $m\mapsto (P, \langle\ell,m\rangle)$, where $X\cap \langle \ell,m\rangle = m+P$;
hence, by Lemma \ref{lem:bijection-same-class},
\[
[V_{m+P}]=[I]-([I_{\text{$3$-plane}}]+ [I_{\text{quad}}]+ [I_{\text{plane}^+}]+ [I_{m+P,\ell\text{-inc}}]).
\]
The projection $I\to F_2(X)_\ell$ is a $\P^{N-3}$-bundle, and hence, by Lemma~\ref{lem:r-planes-in-hyperbolic-reduction-loci}\eqref{item:image-of-E-in-rel-Fr},
\[
[I]=[F_2(X)_\ell][\P^{N-3}]=[Y_\ell][\P^{N-3}].
\]
The projection $I_{\text{$3$-plane}}\to \Gr(4,N+1)$ induces a $\P^1$-bundle $I_{\text{$3$-plane}}\to F_3(X)_\ell$, so
\[
[I_{\text{$3$-plane}}]=[F_3(X)_\ell][\P^1]=[F_1(Y_\ell)][\P^1]
\]
by Lemma~\ref{lem:r-planes-in-hyperbolic-reduction-loci}\eqref{item:relative-m-planes-in-Qr}.
It remains to compute the classes of $I_{\text{quad}}$, $I_{\text{plane}^+}$, and $I_{m+P,\ell\text{-inc}}$.

As for $I_{\text{quad}}\sqcup I_{\text{plane}^+}$, the base change $\id_{F_2(X)_\ell}\times \varphi \colon F_2(X)_\ell\times \mathcal{Q}\to F_2(X)_\ell\times \P^1$ has a nondegenerate $2$-section $s\colon F_2(X)_\ell\times\P^1\to F_2(X)_\ell\times F_2(\mathcal{Q}/\P^1)$ given by $(P, t)\mapsto (P,(P,t))$.
By Lemma~\ref{lem:hyperbolic-reduction-properties-not-pencil}\eqref{item:hyperbolic-reduction-not-pencil-embedding}, the corresponding hyperbolic reduction satisfies
\[(F_2(X)_\ell\times\mathcal Q)^{(2)}_s \xrightarrow{\sim} \{(L,P,t) \in F_3(\mathcal Q/\P^1) \times F_2(X)_\ell \mid P \subset L \subset \mathcal Q_t\},\]
and in particular the fiber of \((F_2(X)_\ell\times\mathcal Q)^{(2)}_s\) over \(P\in F_2(X)_\ell\) is \(\mathcal Q^{(2)}_P\).
Let $\mathcal{P},\mathcal{Q}',\mathcal{R}, \mathcal{P}'$ be the schemes over $F_2(X)_\ell\times \P^1$ whose fibers over $P\in F_2(X)_\ell$ are respectively $\mathcal{P}^{(2)}_{P},\mathcal{Q}_{Y_P}, Y_P\times \P^1, \mathcal{P}_{Y_P}$, as defined in \eqref{eqns:P(r)-and-Q(r)} and~\eqref{eqns:Z-and-Y}.
We have
\[
\mathcal{R}\subset \mathcal{Q}'\subset (F_2(X)_\ell\times \mathcal Q)^{(2)}_s\subset \mathcal{P},\qquad  \mathcal{Q}'\subset \mathcal{P}'\subset \mathcal{P}, \qquad (F_2(X)_\ell\times \mathcal Q)^{(2)}_s\cap \mathcal{P}'=\mathcal{Q}'.
\]
By Lemma~\ref{lem:planes-in-hyperbolic-reduction-loci}\eqref{item:image-of-Q^2-loci-in-rel-F3},
the map
\[
(F_2(X)_\ell\times \mathcal{Q})^{(2)}_s\setminus \mathcal{R}\to I_{\text{quad}},\qquad (L,P, t)\mapsto (P,L)
\]
is a bijection with inverse $(P,L)\mapsto (L,P,t)$, where $t\in \P^1$ is the unique (by Lemma~\ref{lem:intersect-3plane-containing-2plane-cases}) point such that $L\subset \mathcal{Q}_t$.
By Lemma~\ref{lem:planes-in-hyperbolic-reduction-loci}\eqref{item:image-of-P^2-in-rel-F3},
the map
\[
\mathcal{P}\setminus ((F_2(X)_\ell\times \mathcal Q)^{(2)}_s\cup \mathcal{P'})\to I_{\text{plane}^+}, \qquad (L,P, t)\mapsto (P,L)
\]
is a bijection with inverse $(P,L)\mapsto (L,P,t)$, where $t\in \P^1$ is the unique point such that \(\mathcal Q_t\cap L = 2P\). 
Hence, by Lemma \ref{lem:bijection-same-class},
\[
[I_{\text{quad}}]=[(F_2(X)_\ell\times \mathcal{Q})^{(2)}_s]-[\mathcal{R}],\qquad [I_{\text{plane}^+}]=[\mathcal{P}]- [(F_2(X)_\ell\times \mathcal Q)^{(2)}_s\cup \mathcal{P'}],
\]
and thus
\[
[I_{\text{quad}}\sqcup I_{\text{plane}^+}]=[\mathcal{P}]-[\mathcal{P}']+[\mathcal{Q}'\setminus \mathcal{R}].
\]
The projections $\mathcal{P}\to F_2(X)_\ell\times \P^1$ and $\mathcal{P}'\to F_2(X)_\ell\times \P^1$ are $\P^{N-6}$- and $\P^{N-9}$-bundles, respectively, where for the latter one uses that Lemma~\ref{lem:plane-line-tangent-isom} induces an isomorphism $\mathcal{P}'\cong\P(T_{Y_\ell})\times \P^1$ over $F_2(X)_\ell\times \P^1\cong Y_\ell\times \P^1$ and that $\P(T_{Y_\ell})\to Y_\ell$ is a $\P^{N-9}$-bundle since \(Y_\ell\) is smooth of dimension \(N-8\) by Lemma~\ref{lem:general-l-hyperbolic-red}\eqref{item:general-l-Y-smooth-ci}.
%the map $R\to F_3(X)_\ell\times \P^1$ given by $(L,P,t)\mapsto (L,t)$ is a $\P^1$-bundle,
Hence, by Lemma~\ref{lem:fibration-general} and Lemma~\ref{lem:bijection-same-class},
\begin{equation*}
\begin{split}
[\mathcal{P}]&=[F_2(X)_\ell\times \P^1][\P^{N-6}]=[Y_\ell][\P^1][\P^{N-6}],\\
[\mathcal{P}']&=[F_2(X)_\ell\times \P^1][\P^{N-9}]=[Y_\ell][\P^1][\P^{N-9}].
%\\ &[\mathcal{R}]=[F_3(X)_\ell\times \P^1][\P^1]=[F_1(Y_\ell)][\P^1]^2.
\end{split}
\end{equation*}
Moreover, by~Lemma~\ref{lem:lines-in-hyperbolic-reduction-loci}\eqref{item:rel-lines-in-Q^1-in-rel-F3} and~Lemma~\ref{lem:planes-in-hyperbolic-reduction-loci}\eqref{item:image-of-Q^2-loci-in-rel-F3}, the map
\[
\mathcal{Q}'\setminus \mathcal{R}\to F_1(\mathcal{Q}_{Y_\ell}/\P^1)_{Y_\ell\times \P^1}^{\text{tan}}\setminus F_1((Y_\ell\times \P^1)/\P^1),\qquad (L,P,t)\mapsto (L,t)
\]
is a bijection with inverse \((L,t) \mapsto (L,P,t)\) where \(X\cap L=2P\), so by Lemma~\ref{lem:bijection-same-class},
\[
[\mathcal{Q}'\setminus \mathcal{R}]=[F_1(\mathcal{Q}_{Y_\ell}/\P^1)^{\text{tan}}_{Y_\ell\times \P^1}]-[F_1(Y_\ell)][\P^1].
\]
To compute the class of $F_1(\mathcal{Q}_{Y_\ell}/\P^1)_{Y_\ell\times \P^1}^{\text{tan}}$, we have a diagram
\[
\begin{tikzcd}
  &\left\{(n,t,y)\in F_1(\mathcal{Q}_{Y_\ell}/\P^1)\times Y_\ell \mid \text{$n$ is tangent to $Y_\ell$ at $y$}\right\}\arrow[ld] \arrow[rd] &   \\
F_1(\mathcal{Q}_{Y_\ell}/\P^1)^{\text{tan}}_{Y_\ell\times \P^1}  &                         & \P(T_{Y_\ell}),
\end{tikzcd}
\]
where the first projection is a bijection over $F_1(\mathcal{Q}_{Y_\ell}/\P^1)^{\text{tan}}_{Y_\ell\times \P^1}\setminus F_1((Y_\ell\times \P^1)/\P^1)$ and a $\P^1$-bundle over $F_1((Y_\ell\times \P^1)/\P^1)$.
Let \(U_1(Y_\ell)\to F_1(Y_\ell)\) be the universal family; then \(U_1(Y_\ell)\) embeds naturally into the projectivized tangent bundle $\P(T_{Y_\ell})$ of \(Y_\ell\).
The second projection is a bijection over $\P(T_{Y_\ell})\setminus U_1(Y_\ell)$ because letting $(n,y)\in \P(T_{Y_{\ell}})\setminus U_1(Y_\ell)$, we have $n\not\subset Y_\ell$, and for every $t\in \P^1$ the intersection $(\mathcal{Q}_{Y_\ell})_t\cap n$ contains the length $2$ zero-dimensional scheme $2y$, hence there is a unique $t'\in \P^1$ such that $n\subset (\mathcal{Q}_{Y_\ell})_{t'}$.
In addition, the second projection is a $\P^1$-bundle over $U_1(Y_\ell)$. 
Thus, by Lemma~\ref{lem:fibration-general} and Lemma \ref{lem:bijection-same-class},
\[
[F_1(\mathcal{Q}_{Y_\ell}/\P^1)^{\text{tan}}_{Y_\ell\times \P^1}]=-[F_1((Y_\ell\times \P^1)/\P^1)]\L+[\P(T_{Y_\ell})]+[U_1(Y_\ell)]\L,
\]
and using that $U_1(Y_\ell)\to F_1(Y_\ell)$ is a $\P^1$-bundle and $\P(T_{Y_\ell})\to Y_\ell$ is a $\P^{N-9}$-bundle by Lemma~\ref{lem:general-l-hyperbolic-red}\eqref{item:general-l-Y-smooth-ci},
\[
[F_1(\mathcal{Q}_{Y_\ell}/\P^1)^{\text{tan}}_{Y_\ell\times \P^1}]=[Y_\ell][\P^{N-9}].
\]
Hence
\[
[\mathcal{Q}'\setminus \mathcal{R}]=[Y_\ell][\P^{N-9}]-[F_1(Y_\ell)][\P^1],
\]
and therefore, combining the above results, we obtain
\[
[I_{\text{quad}}\sqcup I_{\text{plane}^+}]=[Y_\ell]([\P^1][\P^{N-6}]-[\P^{N-9}]\L)-[F_1(Y_\ell)][\P^1].
\]

As for $I_{m+P,\ell\text{-inc}}$, the map
\[
I_{m+P,\ell\text{-inc}}\to F_2(X)_\ell\times V_{\ell+m},\qquad (P,L)\mapsto (P, m),
\]
where $X\cap L = m+P$, is injective with image \(\{(P,m)\in F_2(X)_\ell \times V_{\ell+m} \mid X \cap \langle P,m\rangle=m+P\}\), and the map
\[
\{(P_1\cup P_2,m)\in S_3(X)_{\ell,\text{rk}\leq 2}^\circ\times F_1(X)\mid \ell\subset P_1, m\subset P_2, m\not\subset P_1, \text{ and }\ell\cap m\neq\emptyset\}\to F_2(X)_\ell\times V_{\ell+m} 
\]
given by $(P_1\cup P_2,m)\mapsto (P_1,m)$
is also injective (recall \(S_3(X)_{\ell,\text{rk}\leq 2}^\circ\) was defined in Section~\ref{section:quad}), and its image is \(\{(P,m)\in F_2(X)_\ell \times V_{\ell+m} \mid X \cap \langle P,m\rangle\) is a rank 2 quadric surface \(\}\). Note that any \((P,m)\in F_2(X)_\ell\times V_{\ell+m}\) satisfies \(\langle P,m\rangle\cong\P^3\), and the intersection \(X\cap\langle P,m\rangle\) is either equal to \(m+P\) or a rank 2 quadric surface by Lemma~\ref{lem:intersect-3plane-containing-2plane-cases}. So the image of the first map is the complement of the image of the second.
The domain of the second map is a $(\P^1\setminus \text{point})=\A^1$-bundle over $S_3(X)_{\ell,\text{rk}\leq 2}^\circ$, so by Lemma~\ref{lem:fibration-general}, Lemma \ref{lem:bijection-same-class}, \eqref{eq:v-w-ell-m}, and the computation \([S_3(X)_{\ell,\text{rk}\leq 2}^\circ] = [F_1(\mathcal{Q}^{(1)}_\ell/\P^1)_{\P^1\times Y_\ell}^{\text{int}}]-[F_1(\mathcal{Q}_{Y_\ell}/\P^1)] = [Y_\ell][\mathcal Q^{(2)}]-[Y_\ell][(\mathcal Q_{Y_\ell})^{(0)}]\) in Section~\ref{section:quad},
\[
[I_{m+P,\ell\text{-inc}}]=[F_2(X)_\ell\times V_{\ell+m}]-[S_3(X)_{\ell,\text{rk}\leq 2}^\circ]\L = [Y_\ell]([\mathcal{Q}^{(1)}_\ell]-[\mathcal{Q}^{(2)}]\L-([\mathcal{Q}_{Y_\ell}]-[(\mathcal{Q}_{Y_\ell})^{(0)}]\L)).
\]
Since $[\mathcal{Q}^{(1)}_\ell]=[\P^1](1+\L^{N-5})+[\mathcal{Q}^{(2)}]\L$ and $[\mathcal{Q}_{Y_\ell}]=[\P^1](1+\L^{N-7})+[(\mathcal{Q}_{Y_\ell})^{(0)}]\L$ by Corollary \ref{cor:formula-class-of-hyperbolic-reduction} and Proposition \ref{prop:KS-hyperbolic-reduction}\eqref{item:KS-hyperbolic-reduction-formula},
\[
[I_{m+P,\ell\text{-inc}}]=[Y_\ell][\P^1](\L^{N-5}-\L^{N-7}).
\]

Combining this with the equality \([V_{m+P}] = [I]-[I_{\text{$3$-plane}}]-[I_{\text{quad}}\sqcup I_{\text{plane}^+}]-[I_{m+P,\ell\text{-inc}}]\) and the above computations of \([I]\), \([I_{\text{$3$-plane}}]\), and \([I_{\text{quad}}\sqcup I_{\text{plane}^+}]\), we conclude
\begin{equation}\label{eq:v-m+P}
[V_{m+P}]=[Y_\ell][\P^1](\L^{N-4}-\L^{N-5}).
\end{equation}
This equality also holds for \(N=6,7\) because both sides are 0 by Lemma~\ref{lem:general-r-plane-contained-in-r+1-plane} and Lemma~\ref{lem:general-l-hyperbolic-red}\eqref{item:general-l-Y-smooth-ci}.

\subsection{\texorpdfstring{$V_{\ell+m+2m'}$}{Vellplusmplus2m'}}\label{section:v-ell+m+2m'}

Before computing the class of \(V_{\ell+m+2m'}\), recall that, by Lemma~\ref{lem:line-disjoint-from-l-cases}, for \(m\in F_1(X)\) disjoint from \(\ell\), we have \(X\cap\langle\ell,m\rangle=\ell+m+2m'\) for some \(m'\in F_1(X)\) if and only if \(\langle\ell,m'\rangle\not\subset X\), \(\langle m,m'\rangle\not\subset X\), and there is a unique \(t\in\P^1\) such that \(\mathcal Q_t\cap\langle\ell,m\rangle=\langle\ell,m'\rangle+\langle m,m'\rangle\).

Now we compute the class of \(V_{\ell+m+2m'}\).
The map $\id_{V_{\ell+m}}\times \varphi\colon V_{\ell+m}\times \mathcal{Q}\to V_{\ell+m}\times \P^1$ has a nondegenerate $1$-section $\sigma\colon V_{\ell+m}\times \P^1\to V_{\ell+m}\times F_1(\mathcal{Q}/\P^1)$ given by $(m',t)\mapsto (m',(m',t))$.
Define the morphism $s_{\text{inc}}\colon V_{\ell+m}\to V_{\ell+m}\times \P^1$ by $m'\mapsto (m',t)$, where $t\in \P^1$ is the unique point such that $\langle \ell,m'\rangle \subset \mathcal{Q}_t$,
and let $\mathcal{Q}_{\ell\text{-inc}}\to V_{\ell+m}$ be the pullback of the hyperbolic reduction $(V_{\ell+m}\times \mathcal{Q})^{(1)}_\sigma\to V_{\ell+m}\times \P^1$ along $s_{\text{inc}}\colon V_{\ell+m}\to V_{\ell+m}\times \P^1$.
Then $\mathcal{Q}_{\ell\text{-inc}}\to V_{\ell+m}$ is a quadric fibration of relative dimension $N-5$, and by Lemma~\ref{lem:hyperbolic-reduction-properties-not-pencil}\eqref{item:hyperbolic-reduction-not-pencil-embedding},
\[
\mathcal{Q}_{\ell\text{-inc}}\xrightarrow{\sim}\{(m',P,t)\in V_{\ell+m}\times F_2(\mathcal{Q}/\P^1)\mid m'\subset P\subset \mathcal{Q}_t \text{ and }\langle\ell,m'\rangle \subset \mathcal{Q}_t\}.
\]
By identifying the domain and codomain of the above isomorphism, let
\begin{equation*}
\begin{split}
\Sigma_1&\coloneqq\{(m',P,t)\in \mathcal{Q}_{\ell\text{-inc}} \mid \langle \ell,P\rangle \subset \mathcal{Q}_t\},\\
\Sigma_2&\coloneqq \{(m',P,t)\in \mathcal{Q}_{\ell\text{-inc}}\setminus \Sigma_1\mid X\cap P=m+m'\text{ for some }m\in F_1(X)\setminus \{\ell, m'\}\text{ and }\ell\cap m\cap m'\neq \emptyset\},\\
\Sigma_3&\coloneqq \{(m',P,t)\in \mathcal{Q}_{\ell\text{-inc}}\setminus \Sigma_1\mid X\cap P=2m'\},\\
\Sigma_4&\coloneqq\{(m',P,t)\in \mathcal{Q}_{\ell\text{-inc}}\setminus\Sigma_1 \mid P\subset X\}.
\end{split}
\end{equation*}
Then $(m',P,t)\in\mathcal{Q}_{\ell\text{-inc}}\setminus (\Sigma_1\sqcup \Sigma_2\sqcup \Sigma_3\sqcup \Sigma_4)$ if and only if $X\cap P=m+m'$ for some $m\in F_1(X)\setminus \{\ell, m'\}$ with \(\ell\cap m\cap m'=\emptyset\). Furthermore, in this case we have $\langle\ell,m,m'\rangle=\langle\ell,P\rangle\cong\P^3$ by Lemma~\ref{lem:lines-meeting-l-span-P2} and so $\ell\cap m=\emptyset$; $\mathcal{Q}_t\cap \langle \ell, m,m'\rangle=\langle \ell, m'\rangle +\langle m,m'\rangle$ (because \(\langle \ell, m'\rangle +\langle m,m'\rangle \subset \mathcal{Q}_t\cap \langle \ell, m,m'\rangle\), so this intersection must be equal to this rank 2 quadric surface or to \(\langle \ell, m,m'\rangle = \langle\ell,P\rangle\), and we have \(\langle\ell,P\rangle\not\subset\mathcal Q_t\) by assumption);
and the intersection \(X\cap\langle\ell,m,m'\rangle = X\cap\langle\ell,m\rangle\) is \(\ell+m+2m'\) by Lemma~\ref{lem:line-disjoint-from-l-cases} (using that \(\langle\ell,m'\rangle\not\subset X\) and \(P=\langle m,m'\rangle\not\subset X\)). Then,
the map
\[
\mathcal{Q}_{\ell\text{-inc}}\setminus (\Sigma_1\sqcup\Sigma_2\sqcup \Sigma_3\sqcup \Sigma_4)\to V_{\ell+m+2m'},\qquad (m',P,t)\mapsto m
\]
is a bijection with inverse $m\mapsto (m',\langle m,m'\rangle,t)$, where $t\in \P^1$ is the unique (by Lemma~\ref{lem:r-planes-in-hyperbolic-reduction-loci}\eqref{item:image-of-Z-in-rel-Fr}) point such that $\langle m,m'\rangle \subset \mathcal{Q}_t$.
Hence, by Lemma \ref{lem:bijection-same-class},
\[
[V_{\ell+m+2m'}]=[\mathcal{Q}_{\ell\text{-inc}}\setminus \Sigma_1]-[\Sigma_2\sqcup \Sigma_3\sqcup \Sigma_4].
\]

Next, we compute the class of \(\mathcal{Q}_{\ell\text{-inc}}\setminus \Sigma_1\).
The quadric fibration $Q_{\ell\text{-inc}}\to V_{\ell+m}$ has a section $s$ given by $m'\mapsto (m',\langle\ell,m'\rangle,t)$, where $t\in \P^1$ is the unique point such that $\langle \ell,m'\rangle \subset \mathcal{Q}_t$,
and the image of $s$ is contained in $\Sigma_1$.
By identifying the domain and codomain of the isomorphism $V_{\ell+m}\cong \mathcal{Q}^{(1)}_\ell\setminus \mathcal{Q}_{Y_\ell}$ of Lemma \ref{lem:r-planes-in-hyperbolic-reduction-loci}\eqref{item:image-of-Z-in-rel-Fr}, 
for $m'\in V_{\ell+m}$, the image $s(m')$ is a singular point of $\mathcal{Q}_{\ell\text{-inc},m'}$ if and only if $m'$ is the singular point of $(\mathcal{Q}^{(1)}_\ell)_{\varphi^{(1)}(m')}$.
Indeed, letting $t = \phi^{(1)}(m')\in \P^1$ be the unique point such that $\langle \ell,m'\rangle \subset \mathcal{Q}_t$, we have isomorphisms
\[
\mathcal{Q}_{\ell\text{-inc},m'}\xrightarrow{\sim} \{P\in F_2(\mathcal{Q}_t)\mid m'\subset P\subset \mathcal{Q}_t\},\qquad (\mathcal{Q}^{(1)}_\ell)_{\varphi^{(1)}(m')}\xrightarrow{\sim} \{P\in F_2(\mathcal{Q}_t)\mid \ell\subset P\subset \mathcal{Q}_t\},
\]
where the points $s(m')\in \mathcal{Q}_{\ell\text{-inc},m'}$ and $m' \in (\mathcal{Q}^{(1)}_\ell)_{\varphi^{(1)}(m')}$ both correspond to the 2-plane $\langle \ell,m'\rangle\in F_2(\mathcal{Q}_t)$ via Lemma~\ref{lem:hyperbolic-reduction-properties-not-pencil}\eqref{item:hyperbolic-reduction-not-pencil-embedding}.
By the Witt extension theorem, there is a linear automorphism of $\mathcal{Q}_t$ which switches $\ell$ and $m'$ and maps $\langle\ell,m'\rangle$ to itself, and this induces an isomorphism $\mathcal{Q}_{\ell\text{-inc},m'}\xrightarrow{\sim}(\mathcal{Q}^{(1)}_\ell)_{\varphi^{(1)}(m')}$ which maps $s(m')\in \mathcal{Q}_{\ell\text{-inc},m'}$ to $m' \in (\mathcal{Q}^{(1)}_\ell)_{\varphi^{(1)}(m')}$.
By Proposition~\ref{prop:KS-hyperbolic-reduction}\eqref{item:KS-hyperbolic-red-degeneracy}, $\varphi^{(1)}\colon \mathcal{Q}^{(1)}_\ell\to \P^1$ has $N+1$ singular fibers, which are of corank $1$ (i.e., rank $N-4$), and by Lemma \ref{lem:general-l-vertices-avoid-Z}, the singular points $m_1,\cdots, m_{N+1}$ of the singular fibers of $\varphi^{(1)}$ are disjoint from $\mathcal{Q}_{Y_\ell}$. 
Hence $s$ defines a nondegenerate section of $\mathcal{Q}_{\ell\text{-inc}}\setminus (\mathcal{Q}_{\ell\text{-inc},m_1}\sqcup\cdots \sqcup \mathcal{Q}_{\ell\text{-inc},m_{N+1}})\to V_{\ell+m}\setminus \{m_1,\cdots, m_{N+1}\}$, which we also denote by $s$.
Then the corresponding hyperbolic reduction can be identified with
\[(\mathcal{Q}_{\ell\text{-inc}}\setminus (\mathcal{Q}_{\ell\text{-inc},m_1}\sqcup\cdots \sqcup \mathcal{Q}_{\ell\text{-inc},m_{N+1}}))^{(0)}_s \xrightarrow{\sim}\left\{(n,m')\in F_1(\mathcal{Q}_{\ell\text{-inc}}/V_{\ell+m}) \,\middle\vert\, \substack{ (m',\langle\ell,m'\rangle,t)\in n \subset \mathcal{Q}_{\ell\text{-inc},m'} \\ \text{ and }m'\neq m_i}\right\}\]
\[
\xrightarrow{\sim}\{(m',L,t) \in V_{\ell+m}\times F_3(\mathcal{Q}/\P^1)\mid \langle\ell,m'\rangle\subset L\subset\mathcal Q_t\text{ and }m'\neq m_i\},\]
where we identify \(F_1(\mathcal{Q}_{\ell\text{-inc}}/V_{\ell}+m)\) with its image in \(V_{\ell+m}\times F_3(\mathcal Q/\P^1)\) using Lemma~\ref{lem:hyperbolic-reduction-properties-not-pencil}\eqref{item:hyperbolic-reduction-not-pencil-embedding}.
The map
\[
\Sigma_1\setminus (s\cup (\mathcal{Q}_{\ell\text{-inc},m_1}\sqcup\cdots \sqcup \mathcal{Q}_{\ell\text{-inc},m_{N+1}}))\to (\mathcal{Q}_{\ell\text{-inc}}\setminus (\mathcal{Q}_{\ell\text{-inc},m_1}\sqcup\cdots \sqcup \mathcal{Q}_{\ell\text{-inc},m_{N+1}}))^{(0)}_s,\quad (m',P,t) \mapsto (m',\langle \ell, P\rangle,t)
\]
is an $\A^1$-bundle (whose fiber over \((m',L,t)\) is \(F_2(L)_{m'} \setminus\{\langle\ell,m'\rangle\}\)), and hence by Lemma~\ref{lem:fibration-general} and Lemma \ref{lem:bijection-same-class}, 
\[
[\Sigma_1\setminus (s\cup (\mathcal{Q}_{\ell\text{-inc},m_1}\sqcup\cdots \sqcup \mathcal{Q}_{\ell\text{-inc},m_{N+1}}))] = [(\mathcal{Q}_{\ell\text{-inc}}\setminus (\mathcal{Q}_{\ell\text{-inc},m_1}\sqcup\cdots \sqcup \mathcal{Q}_{\ell\text{-inc},m_{N+1}}))^{(0)}_s] \L,
\]
and since $\Im(s)\cap (\Sigma_1\setminus (\mathcal{Q}_{\ell\text{-inc},m_1}\sqcup\cdots \sqcup \mathcal{Q}_{\ell\text{-inc},m_{N+1}}))\cong V_{\ell+m}\setminus \{m_1,\cdots, m_{N+1}\}$,
\[
[\Sigma_1\setminus (\mathcal{Q}_{\ell\text{-inc},m_1}\sqcup\cdots \sqcup \mathcal{Q}_{\ell\text{-inc},m_{N+1}})] = [(\mathcal{Q}_{\ell\text{-inc}}\setminus (\mathcal{Q}_{\ell\text{-inc},m_1}\sqcup\cdots \sqcup \mathcal{Q}_{\ell\text{-inc},m_{N+1}}))^{(0)}_s] \L + [V_{\ell+m}\setminus \{m_1,\cdots, m_{N+1}\}].
\]
Moreover, $\mathcal{Q}_{\ell\text{-inc},m_1},\cdots,\mathcal{Q}_{\ell\text{-inc},m_{N+1}} \subset \Sigma_1$.
Indeed, for each $m_i$, the fiber $\mathcal{Q}_{\ell\text{-inc},m_i}$ is an $(N-5)$-dimensional corank $1$ (i.e., rank $N-4$) quadric with vertex $s(m_i)$, and thus any point $(m_i,P,t)\in \mathcal{Q}_{\ell\text{-inc},m_i}\setminus s(m_i)$ lies on the line connecting the point and $s(m_i)$, which corresponds by Lemma~\ref{lem:hyperbolic-reduction-properties-not-pencil}\eqref{item:hyperbolic-reduction-not-pencil-embedding} to the $3$-plane $\langle\ell, P\rangle$ on $\mathcal{Q}_t$; hence, \((m_i,P,t)\in\Sigma_1\). Since the point $s(m_i)=(m_i, \langle \ell,m_i\rangle ,t) $ is also in $\Sigma_1$, this shows $\mathcal{Q}_{\ell\text{-inc},m_i} \subset \Sigma_1$.
Therefore
\begin{equation*}
\begin{split}
[\mathcal{Q}_{\ell\text{-inc}}\setminus \Sigma_1]&=[\mathcal{Q}_{\ell\text{-inc}}\setminus (\mathcal{Q}_{\ell\text{-inc},m_1}\sqcup\cdots \sqcup \mathcal{Q}_{\ell\text{-inc},m_{N+1}})]-[\Sigma_1\setminus (\mathcal{Q}_{\ell\text{-inc},m_1}\sqcup\cdots \sqcup \mathcal{Q}_{\ell\text{-inc},m_{N+1}})]\\
&=[\mathcal{Q}_{\ell\text{-inc}}\setminus (\mathcal{Q}_{\ell\text{-inc},m_1}\sqcup\cdots \sqcup \mathcal{Q}_{\ell\text{-inc},m_{N+1}})]-[(\mathcal{Q}_{\ell\text{-inc}}\setminus (\mathcal{Q}_{\ell\text{-inc},m_1}\sqcup\cdots \sqcup \mathcal{Q}_{\ell\text{-inc},m_{N+1}}))^{(0)}_s] \L\\
&\qquad -[V_{\ell+m}\setminus\{m_1,\cdots,m_{N+1}\}]\\
&=[V_{\ell+m}\setminus\{m_1,\cdots,m_{N+1}\}](1+\L^{N-5})-[V_{\ell+m}\setminus\{m_1,\cdots,m_{N+1}\}]\\
&=([V_{\ell+m}]-(N+1))\L^{N-5}
\end{split}
\end{equation*}
where we have used Proposition \ref{prop:KS-hyperbolic-reduction}\eqref{item:KS-hyperbolic-reduction-formula} for the third equality.
So by~\eqref{eq:v-w-ell-m}, we conclude
\begin{equation}\label{eq:v-ell+m+2m'}
[V_{\ell+m+2m'}]=([\mathcal{Q}^{(1)}_\ell]-[\mathcal{Q}_{Y_\ell}]-(N+1))\L^{N-5}-[\Sigma_2\sqcup \Sigma_3\sqcup\Sigma_4].
\end{equation}
We will not compute the explicit class of \([\Sigma_2\sqcup \Sigma_3\sqcup\Sigma_4]\) here; instead, this term will appear again in Section~\ref{section:w-concurrent}, so that it cancels out in the proofs of Theorems~\ref{thm:main-odd} and~\ref{thm:main-even}.

\subsection{\texorpdfstring{$W_{2\ell+m_0+m_1}$}{W2ellplusm0plusm1}}\label{section:w-2ell+m_0+m_1}
Recall that \(\mathcal P^{(1)}_\ell \subset \P^1\times\P^{N-2}\) with projection \(\mathcal P^{(1)}_\ell\to\P^1\) is the \(\P^{N-4}\)-bundle defined in Section~\ref{sec:hyperbolic-reductions-pencil} with \(\mathcal Q^{(1)}_\ell \subset\mathcal P^{(1)}_\ell\) as a divisor of relative degree 2 over \(\P^1\), and \(\mathcal{Q}_{Y_\ell} \subset \mathcal Q^{(1)}_\ell\) is defined by linear equations in the \(\P^{N-2}\) variables.
Define the following subsets of \(F_1(\mathcal{P}^{(1)}_\ell/\P^1)\):
\begin{equation*}
    \begin{split}
        F_1(\mathcal{P}^{(1)}_\ell/\P^1)^{\text{tan}}_{\mathcal{Q}^{(1)}_\ell} &\coloneqq \{(n,t)\in F_1(\mathcal{P}^{(1)}_\ell/\P^1) \mid n \text{ is tangent to }\mathcal{Q}^{(1)}_\ell\}, \\
        F_1(\mathcal{P}^{(1)}/\P^1)^{\text{int}}_{\mathcal{Q}_{Y_\ell}} & \coloneqq \{(n,t)\in F_1(\mathcal{P}^{(1)}_\ell/\P^1) \mid \mathcal{Q}_{Y_\ell} \cap n \neq\emptyset \}, \\
        F_1(\mathcal{P}^{(1)}/\P^1)^{\text{tan}}_{\mathcal{Q}_{Y_\ell}} & \coloneqq \{(n,t)\in F_1(\mathcal{P}^{(1)}_\ell/\P^1) \mid n \text{ is tangent to \(\mathcal{Q}^{(1)}_\ell\), and } \mathcal{Q}_{Y_\ell} \cap n \neq\emptyset\}.
    \end{split}
\end{equation*}
Since \(\mathcal Q^{(1)}_\ell \subset\mathcal P^{(1)}_\ell\) is a divisor of relative degree 2 over \(\P^1\), any \((n,t)\in F_1(\mathcal{P}^{(1)}_\ell/\P^1) \setminus F_1(\mathcal{P}^{(1)}_\ell/\P^1)^{\text{tan}}_{\mathcal{Q}^{(1)}_\ell}\) meets \(\mathcal{Q}^{(1)}_\ell\) in two distinct points. So the map
\[
F_1(\mathcal{P}^{(1)}_\ell/\P^1)\setminus (F_1(\mathcal{P}^{(1)}_\ell/\P^1)^{\text{tan}}_{\mathcal{Q}^{(1)}_\ell}\cup F_1(\mathcal{P}^{(1)}/\P^1)^{\text{int}}_{\mathcal{Q}_{Y_\ell}})\to W_{2\ell+m_0+m_1},\qquad (n,t)\mapsto (m_0,m_1),
\]
where $(m_0, m_1)$ corresponds to the pair of distinct points of $\mathcal{Q}^{(1)}\cap n$,
is a bijection with inverse $(m_0,m_1)\mapsto (\langle m_0,m_1\rangle,t)$, where
$t\in \P^1$ is the unique point such that $\mathcal{Q}_t\cap\langle \ell,m_0,m_1\rangle= \langle \ell,m_0\rangle + \langle \ell,m_1\rangle$, which exists by Corollary~\ref{cor:lines-meeting-l-span-P3-disjoint} and Lemma~\ref{lem:lines-meeting-l-span-P3-common-point}\eqref{item:case-sym-Q1-meeting-computation-general-rank-3-multiplicity-2-root}\ref{item:w-2ell+m_0+m_1-mult-2-root}.
Hence, by Lemma \ref{lem:bijection-same-class}, 
\[
[W_{2\ell+m_0+m_1}]=[F_1(\mathcal{P}^{(1)}_\ell/\P^1)\setminus (F_1(\mathcal{P}^{(1)}_\ell/\P^1)^{\text{tan}}_{\mathcal{Q}^{(1)}_\ell}\cup F_1(\mathcal{P}^{(1)}/\P^1)^{\text{int}}_{\mathcal{Q}_{Y_\ell}})].
\]
Thus, it remains to compute the classes of \(F_1(\mathcal{P}^{(1)}_\ell/\P^1)\), \(F_1(\mathcal{P}^{(1)}_\ell/\P^1)^{\text{tan}}_{\mathcal{Q}^{(1)}_\ell}\), \(F_1(\mathcal{P}^{(1)}_\ell/\P^1)^{\text{int}}_{\mathcal{Q}_{Y_\ell}}\), and the intersection \(F_1(\mathcal{P}^{(1)}_\ell/\P^1)^{\text{tan}}_{\mathcal{Q}^{(1)}_\ell}\cap F_1(\mathcal{P}^{(1)}_\ell/\P^1)^{\text{int}}_{\mathcal{Q}_{Y_\ell}}=F_1(\mathcal{P}^{(1)}/\P^1)_{\mathcal{Q}_{Y_\ell}}^\text{tan}\).

Since $\mathcal{P}^{(1)}_\ell\to \P^1$ is a $\P^{N-4}$-bundle, $F_1(\mathcal{P}^{(1)}_\ell/\P^1)\to \P^1$ is a $\Gr(2, N-3)$-bundle. So by Lemma~\ref{lem:fibration-general},
\[
[F_1(\mathcal{P}^{(1)}_\ell/\P^1)]=[\P^1][\Gr(2,N-3)].
\]
As for $F_1(\mathcal{P}^{(1)}_\ell/\P^1)^{\text{tan}}_{\mathcal{Q}^{(1)}_\ell}$, we have a diagram
\[
\begin{tikzcd}
  &\left\{(n,t,x)\in F_1(\mathcal P^{(1)}_\ell/\P^1)^{\text{tan}}_{\mathcal{Q}^{(1)}_\ell} \times_{\P^1} \mathcal{Q}^{(1)}_\ell\mid x\in n \right\}\arrow[ld] \arrow[rd] &   \\
F_1(\mathcal P^{(1)}_\ell/\P^1)^{\text{tan}}_{\mathcal{Q}^{(1)}_\ell}  &                         & \mathcal{Q}^{(1)}_\ell,
\end{tikzcd}
\]
where the first projection is a bijection over $F_1(\mathcal{P}^{(1)}/\P^1)^{\text{tan}}_{\mathcal{Q}^{(1)}_\ell}\setminus F_1(\mathcal{Q}^{(1)}_\ell/\P^1)$ and a $\P^1$-bundle over $F_1(\mathcal{Q}^{(1)}_\ell/\P^1)$.
Since \(\mathcal Q^{(1)}\to\P^1\) is a fibration of quadric \((N-5)\)-folds with \(N+1\) singular fibers that each have corank 1 (Proposition~\ref{prop:KS-hyperbolic-reduction}\eqref{item:KS-hyperbolic-red-degeneracy}), the second projection is a $\P^{N-6}$-bundle over the complement of the cone points of the $N+1$ singular fibers of $\varphi^{(1)}$, and the fiber over each cone point is $\P^{N-5}$.
Hence, by Lemma~\ref{lem:fibration-general} and Lemma~\ref{lem:bijection-same-class},
\[
[F_1(\mathcal{P}^{(1)}_\ell/\P^1)^{\text{tan}}_{\mathcal{Q}^{(1)}_\ell}]=-[F_1(\mathcal{Q}^{(1)}_\ell/\P^1)]\L+[\mathcal{Q}_\ell^{(1)}][\P^{N-6}] + (N+1)\L^{N-5}.
\]
As for $F_1(\mathcal{P}^{(1)}_\ell/\P^1)^{\text{int}}_{\mathcal{Q}_{Y_\ell}}$, we have a diagram
\[
\begin{tikzcd}
  &\{(n,t,x)\in F_1(\mathcal{P}^{(1)}_\ell/\P^1)^{\text{int}}_{\mathcal{Q}_{Y_\ell}} \times_{\P^1} \mathcal{Q}_{Y_\ell} \mid x\in n\} \arrow[ld, "pr_1"'] \arrow[rd] &   \\
F_1(\mathcal{P}^{(1)}_\ell/\P^1)^{\text{int}}_{\mathcal{Q}_{Y_\ell}}  &  & \mathcal{Q}_{Y_\ell}.
\end{tikzcd}
\]
Recall that \(\mathcal P_{Y_\ell} \subset \mathcal P^{(1)}_\ell \subset \P^1\times\P^{N-2}\) is defined by linear equations in \(\P^{N-2}\) (see \eqref{eqns:Z-and-Y}).
If \(N=6\), then by Lemma~\ref{lem:general-l-hyperbolic-red}, \(\mathcal P_{Y_\ell}\cong\mathbb P^1\) so \(F_1(\mathcal{P}_{Y_\ell}/\P^1)=\emptyset\), and \(\mathcal Q_{Y_\ell}\) is a (reduced) point lying over a unique \(t\in\P^1\). Then the relative lines in \(\mathcal P^{(1)}_\ell\) that meet \(\mathcal Q_{Y_\ell}\) are the lines in the fiber \((\mathcal P^{(1)}_\ell)_t\cong\mathbb P^2\) containing this point, and so \(F_1(\mathcal{P}^{(1)}_\ell/\P^1)^{\text{int}}_{\mathcal{Q}_{Y_\ell}}\cong\P^1\).
If \(N\geq 7\), there is an isomorphism \(\mathcal P_{Y_\ell} \cong \P^1\times\P^{N-6}\) over \(\P^1\) by Lemma~\ref{lem:general-l-hyperbolic-red}\eqref{item:general-l-Z-dimension}, and \(\mathcal Q_{Y_\ell}\subset\mathcal P_{Y_\ell}\) is a divisor of relative degree 2 over \(\P^1\).
Therefore, for a line \((n,t)\in F_1(\mathcal P^{(1)}_\ell/\P^1)\), \(n\subset\mathcal P_{Y_\ell}\) if and only if \(n\) either meets \(\mathcal Q_{Y_\ell}\) in a zero-dimensional scheme of length $\geq 2$ or $n\subset\mathcal Q_{Y_\ell}$. If \((n,t)\in F_1(\mathcal P^{(1)}_\ell/\P^1)\) is not contained in \(\mathcal P_{Y_\ell}\), then \(n\) meets \(\mathcal Q_{Y_\ell}\) in at most one point.
So, in the above diagram,
the first projection is a bijection over $F_1(\mathcal{P}^{(1)}_\ell/\P^1)^{\text{int}}_{\mathcal{Q}_{Y_\ell}}\setminus F_1(\mathcal{P}_{Y_\ell}/\P^1)$, and the restriction of the second projection to
\(pr_1^{-1}(F_1(\mathcal{P}_{Y_\ell}/\P^1))\)
is a $\P^{N-7}$-bundle.
The second projection is a $\P^{N-5}$-bundle.
Hence, by Lemma~\ref{lem:fibration-general} and Lemma \ref{lem:bijection-same-class},
\[
[F_1(\mathcal{P}^{(1)}_\ell/\P^1)^{\text{int}}_{\mathcal{Q}_{Y_\ell}}]=[\mathcal{Q}_{Y_\ell}]([\P^{N-5}]-[\P^{N-7}])+[F_1(\mathcal{P}_{Y_\ell}/\P^1)]=[\mathcal{Q}_{Y_\ell}]([\P^{N-5}]-[\P^{N-7}])+[\P^1][\Gr(2,N-5)],
\]
and this formula holds for any \(N\geq 6\).

Finally, as for $F_1(\mathcal{P}^{(1)}_\ell/\P^1)^{\text{tan}}_{\mathcal{Q}_{Y_\ell}}$, we have a diagram
\[
\begin{tikzcd}
  &\left\{(n,t,x)\in F_1(\mathcal P^{(1)}_\ell/\P^1)_{\mathcal{Q}_{Y_\ell}}^{\text{tan}} \times_{\P^1} \mathcal{Q}_{Y_\ell} \mid x\in n \right\} \arrow[ld] \arrow[rd] &   \\
F_1(\mathcal P^{(1)}_\ell/\P^1)_{\mathcal{Q}_{Y_\ell}}^{\text{tan}}  &  & \mathcal{Q}_{Y_\ell},
\end{tikzcd}
\]
where the first projection is a bijection over $F_1(\mathcal{P}^{(1)}_\ell/\P^1)_{\mathcal{Q}_{Y_\ell}}^{\text{tan}}\setminus F_1(\mathcal{Q}_{Y_\ell}/\P^1)$ and a $\P^1$-bundle over $F_1(\mathcal{Q}_{Y_\ell}/\P^1)$; the second projection is a $\P^{N-6}$-bundle because none of the cone points of the $N+1$ singular fibers of $\varphi^{(1)}$ are contained in $\mathcal{Q}_{Y_\ell}$ by Lemma~\ref{lem:general-l-vertices-avoid-Z}.
Hence, by Lemma~\ref{lem:fibration-general} and Lemma~\ref{lem:bijection-same-class}, 
\[
[F_1(\mathcal{P}^{(1)}_\ell/\P^1)_{\mathcal{Q}_{Y_\ell}}^{\text{tan}}
]=-[F_1(\mathcal{Q}_{Y_\ell}/\P^1)]\L+[\mathcal{Q}_{Y_\ell}][\P^{N-6}].\]

We conclude
\begin{equation}\label{eq:w-2ell+m_0+m_1}
\begin{split}
[W_{2\ell+m_0+m_1}]=&[\P^1]([\Gr(2,N-3)]-[\Gr(2,N-5)]) + ([F_1(\mathcal{Q}^{(1)}_\ell/\P^1)]-[F_1(\mathcal{Q}_{Y_\ell}/\P^1)])\L\\
&-[\mathcal{Q}^{(1)}_\ell][\P^{N-6}]+ [\mathcal{Q}_{Y_\ell}](-\L^{N-5}+[\P^{N-7}])-(N+1)\L^{N-5}.
\end{split}
\end{equation}

\subsection{\texorpdfstring{$W_{\ell+m_0+2m_1}$}{Wellplusm0plus2m1} and \texorpdfstring{$W_{\mathrlap{\times}+}$}{W4concurrent}}\label{section:w-concurrent}
We first explain the idea behind the computation.
First, let \(x\in X\) be any point, and let $T_{x}X\subset \P^{N}$ be the tangent space to $X$ at $x$. The intersection \(X\cap T_{x}X\) is the cone with vertex \(x\) over a subvariety \(Y_x\) of \(T_{x}X/x\cong \P^{N-3}\) defined by two quadratic equations (which do not both vanish identically because \(T_{x}X\not\subset X\) by Lemma~\ref{lem:F_r(X)}\eqref{item:lem-F_r(X)-dim}). Then points of \(Y_x\) correspond to lines on \(X\) passing through \(x\).
Furthermore, \(Y_x\) is the same as the intersection defined in~\eqref{eqns:Z-and-Y} for the hyperbolic reduction \(\mathcal Q^{(0)}_x\), so if \(x\in X\) is general then \(Y_x\subset\P^{N-3}\) is a smooth complete intersection of two quadrics by Lemma~\ref{lem:general-l-hyperbolic-red}\eqref{item:general-l-Y-smooth-ci}.

Now let $(m_0,m_1) \in W_{\mathrlap{\times}+}$, and let $x= \ell\cap m_0\cap m_1$.
Write \(X\cap\langle\ell,m_0,m_1\rangle=\ell+m_0+m_1+m\). Then the lines $\ell,m_0,m_1,m$ through \(x\) yield points $p_\ell, p_m,p_{m_0},p_{m_1}\in Y_x$ such that $\langle p_\ell,p_m,p_{m_0},p_{m_1}\rangle \cong \P^2$ and $Y_x\cap \langle p_\ell,p_m,p_{m_0},p_{m_1}\rangle=p_\ell+p_m+p_{m_0}+p_{m_1}$.
Conversely, to get $(m_0,m_1)$ as above, one needs to choose $x\in \ell$ and two points $p_0,p_1\in Y_x$, and this is generically equivalent to choosing $x\in \ell$, $q\in Y_x$, and a plane in $\P^{N-3}=T_{x}X/x$ containing $\langle p_\ell,q\rangle$.
Similarly, for every $(m_0,m_1)\in W_{\ell+m_0+2m_1}$, the lines $\ell, m_0, m_1$ intersect at a single point $x$, and they correspond to $p_\ell, p_{m_0}, p_{m_1}\in Y_x$ such that $\langle p_\ell,p_{m_0},p_{m_1}\rangle \cong\P^2$ and $Y_x\cap \langle p_\ell, p_{m_0}, p_{m_1}\rangle = p_\ell + p_{m_0} + 2p_{m_1}$. 

Now we compute the class of \(W_{\ell+m_0+2m_1} \sqcup W_{\mathrlap{\times}+}\).
Let \(U_1(X)\subset F_1(X)\times X\) be the universal family of lines on \(X\), and let \(M\coloneqq U_1(X)\times_X\ell\) be the preimage of \(\ell\) under the second projection. Then the fiber of the second projection
$M\to \ell$ over $x\in \ell$ is $Y_x$.
For general $x\in \ell$, the fiber \(Y_x\) is smooth by Lemma~\ref{lem:general-l-hyperbolic-red}\eqref{item:general-l-Y-smooth-ci}, generality of \(\ell\), and Lemma~\ref{lem:general-r-plane-contained-in-r+1-plane}.
Let $F_M\to M$ be the morphism whose fiber over $(q,x)\in M$ equals the set of planes in $T_{x}X/x \cong\P^{N-3}$ containing $\langle p_\ell,q\rangle$.
Let $M'\subset M$ be the closed subset whose fiber over $x\in \ell$ is
the union of $\{p_\ell\}$ and
the closed subset of $Y_x$ swept out by lines passing through $p_\ell$.
Note that, for general \(x\in\ell\), $Y_x$ contains lines through $p_\ell$ if and only if \(N-3\geq 5\) (i.e., \(N\geq 8\)) by Lemma~\ref{lem:general-r-plane-contained-in-r+1-plane} and smoothness of \(Y_x\).

Next, we show that, for any \(x\in\ell\), the point \(p_\ell\) is in the smooth locus of \(Y_x\), and \(Y_x\subset T_{x}X/x\cong\P^{N-3}\) is a complete intersection of two quadrics.
For smoothness at \(p_\ell\), choose coordinates on \(\P^N\) so that \(\ell=\{x_2=\cdots=x_N=0\}\) and \(x=[1:0:\cdots:0]\), and write \(X\) as in~\eqref{eqn:X}. The four linear forms \(l_{ij}\) are independent by Lemma~\ref{lem:general-l-hyperbolic-red}\eqref{item:general-l-Z-dimension}. The tangent space to \(X\) at \(x\) is \(T_x X = \{l_{00}=l_{10}=0\}\subset\P^N\), so \(T_x X/x = \{l_{00}=l_{10}=0\}\subset\P^{N-1}_{[x_1:\cdots:x_N]}\). In \(\P^{N-1}_{[x_1:\cdots:x_N]}\) we have \(p_\ell = [1:0:\cdots:0]\) and
\[Y_x = \{l_{00}=l_{10}=x_1 l_{01}+q_0 = x_1 l_{11}+q_1=0\} \subset\mathbb P^{N-1}_{[x_1:\cdots:x_N]}.\]
Then a direct computation shows that \(Y_x\) is smooth at \(p_\ell\), and \(Y_x\) has dimension \(N-5\) at \(p_\ell\).
Next, assume by contradiction that \(Y_x\subset T_x X/x\) is not a complete intersection. The above shows that \(Y_x\) is not a quadric in \(T_x X/x\), so \(Y_x\) contains a hyperplane \(H\subset T_x X/x\) by Lemma~\ref{lem:singular-intersection}\eqref{Y-possibilities}. Since \(X\cap T_xX\) is the cone with vertex \(x\) over \(Y_x\), this intersection contains the cone over \(H\cong\mathbb P^{N-4}\), so \(X\) contains an \((N-3)\)-plane. But this contradicts Lemma~\ref{lem:F_r(X)}\eqref{item:lem-F_r(X)-dim} because \(N-3>\lfloor\frac{N}{2}\rfloor-1\). Therefore \(Y_x\subset T_x X/x\) is a complete intersection.

Now let \(F_{M,1}\coloneqq F_M\times_{M}M'\), and define the following subsets of \(F_M\setminus F_{M,1}\):
\begin{equation*}
\begin{split}
F_{M,2}&\coloneqq \left\{(P,q,x)\in F_M\setminus F_{M,1}\,\middle\vert\, \substack{ Y_x\cap P\text{ is one of }2p_\ell+p+q\text{ for some }p\in Y_x, \\ q+n \text{ for some } n\in F_1(Y_x)_{p_\ell}, \text{ or a conic in }P }\right\},\\
F_{M,3}&\coloneqq\{(P,q,x)\in F_M\setminus F_{M,1}\mid Y_x\cap P = p_\ell + n \text{ for some }n\in F_1(Y_x) \text{ (not containing \(p_\ell\)), and \(q\in n\)}\},\\
F_{M,4}&\coloneqq\{(P,q,x)\in F_M\setminus F_{M,1}\mid Y_x\cap P = p_\ell + 2p + q\text{ for some }p\in Y_x\setminus \{p_\ell\} \text{ (possibly $p=q$)} \}.
\end{split}
\end{equation*}
By definition of \(F_{M,1}\), for \((P,q,x)\in F_{M,2} \sqcup F_{M,3} \sqcup F_{M,4}\) we have \(q\neq p_\ell\).

For a plane \(P\subset T_{x}X/x\), Lemma~\ref{lem:singular-intersection}\eqref{Y-possibilities} shows that
the intersection \(Y_x\cap P\) is
one of \(P\), a conic in \(P\), the union of a line and a point (possibly embedded in the line), or a zero-dimensional scheme of length four.
So \((P,q,x)\in F_M\setminus (F_{M,1}\sqcup F_{M,2}\sqcup F_{M,3}\sqcup F_{M,4})\) if and only if \(Y_x\cap P = p_\ell+q+p_0+p_1\) where \(q,p_0,p_1\in Y_x\setminus\{p_\ell\}\) are points with \(p_0\neq p_1\). Since \(Y_x \cap \langle p_\ell,p_i\rangle = p_\ell + p_i\), the line \(m_i\) corresponding to \(p_i\) is in \(V_{\ell+m}\).
The map
\[
F_M\setminus (F_{M,1}\sqcup F_{M,2}\sqcup F_{M,3}\sqcup F_{M,4})\to W_{\ell+m_0+2m_1}\sqcup W_{\mathrlap{\times}+},\qquad (P,q,x)\mapsto (m_0,m_1),
\]
where $Y_x\cap P=p_\ell + q + p_0 + p_1$ and $m_0,m_1$ correspond to $p_0, p_1$, is a bijection; its inverse is given by $(m_0,m_1)\mapsto (P,q,x)$, where $x=\ell\cap m_0\cap m_1$, $P=\langle \ell,m_0,m_1\rangle /x$, and $Y_x\cap P= p_\ell + p_{m_0} + p_{m_1} + q$.
Hence, by Lemma \ref{lem:bijection-same-class},
\[
[W_{\ell+m_0+2m_1}\sqcup W_{\mathrlap{\times}+}]=[F_M\setminus F_{M,1}]-([F_{M,2}]+[F_{M,3}]+[F_{M,4}]).
\]
It remains to compute the classes of $F_{M}\setminus F_{M,1}$, $F_{M,2}$, $F_{M,3}$, and $F_{M,4}$.

As for $F_{M}\setminus F_{M,1}$, the morphism $F_{M}\setminus F_{M,1}\to M\setminus M'$ is a $\P^{N-5}$-bundle.
In addition, the map
\[
M\setminus M'\to V_{\ell+m}, \qquad (q,x) \mapsto \langle x, q\rangle
\] 
is a bijection with inverse $m\mapsto (m/x,\ell\cap m)$.
Hence by Lemma~\ref{lem:fibration-general}, Lemma~\ref{lem:bijection-same-class}, and~\eqref{eq:v-w-ell-m},
\[
[F_M\setminus F_{M,1}]=[V_{\ell+m}][\P^{N-5}]=([\mathcal{Q}^{(1)}_\ell]-[\mathcal{Q}_{Y_\ell}])[\P^{N-5}].
\]
As for $F_{M,2}$, let $\widetilde{M}\to M$ be the blow-up along the section of $M\to \ell$ given by $x\mapsto p_\ell$, and let $E_M$ be the exceptional divisor. Since \(p_\ell\) is a smooth point of the complete intersection \(Y_x\subset\P^{N-3}\), the composition $E_M\to M\to \ell$ is a $\P^{N-6}$-bundle whose fiber over $x\in \ell$ is $(T_{p_\ell}Y_x)/p_\ell$.
For \((v,q,x) \in E_M\times_\ell (M\setminus M')\), the span \(\langle p_\ell, v,q\rangle\) is a 2-plane since \(\langle p_\ell,q\rangle \not\subset Y_x\), and the intersection \(Y_x\cap\langle p_\ell, v,q\rangle\) is equal to one of the following: a conic (necessarily reduced) containing \(p_\ell,q\) with tangent direction \(v\) at \(p_\ell\); the union of the line \(\langle p_\ell,v\rangle\) and \(\{q\}\); or \(2p_\ell+p+q\) for some point \(p\in Y_x\) (possibly \(p_\ell\) or \(q\)). So the map
\[
E_M\times_\ell (M\setminus M')\to F_{M,2},\qquad (v,q,x)\mapsto (\langle p_\ell, v,q\rangle,q,x)
\]
is well-defined, and it is a bijection with inverse $(P,q,x)\mapsto (v,q,x)$, where $v$ is the unique point in $((T_{p_\ell}Y_x)\cap P)/p_\ell$. Indeed, \(T_{Y_x,p_\ell}\cap P= T_{p_\ell}(Y_x\cap P)\) has dimension at least 1 for each of the cases in \(F_{M,2}\), and the dimension is exactly 1 because the assumption that \(q\notin M'\) implies \(q\notin T_{p_\ell} Y_x\) (since \(Y_x\cap T_{p_\ell} Y_x\) is a cone with vertex \(p_\ell\) by Lemma~\ref{lem:singular-intersection}\eqref{Y-tangentspace-intersection}) and hence \(P\not\subset T_{p_\ell} Y_x\).
Hence, by Lemma~\ref{lem:fibration-general} and Lemma~\ref{lem:bijection-same-class},
\[
[F_{M,2}]=[M\setminus M'][\P^{N-6}]=([\mathcal{Q}^{(1)}_\ell]-[\mathcal{Q}_{Y_\ell}])[\P^{N-6}].
\]
To address $F_{M,3}$ and $F_{M,4}$, we relate them to the loci \(\Sigma_2,\Sigma_3,\Sigma_4 \subset \mathcal{Q}_{\ell\text{-inc}}\) defined in Section~\ref{section:v-ell+m+2m'}. For $F_{M,3}$, the map
\[
\Sigma_4\to F_{M,3},\qquad (m',P,t)\mapsto (\langle\ell,P\rangle/(\ell\cap P), p_{m'}, \ell\cap P)
\]
is a bijection with inverse $(P,q,x)\mapsto (\langle x,q\rangle, \langle x,n\rangle,t)$, where $Y_x\cap P=p_\ell+n$ and $t\in \P^1$ is the unique point such that $\langle \ell, x,q\rangle =\langle\ell,q\rangle \subset \mathcal{Q}_t$.

For $F_{M,4}$, first note that if \((m',P,t)\in\Sigma_2\), then we necessarily have \(\langle\ell,m\rangle\not\subset \mathcal Q_t\) (otherwise the intersection \(\mathcal Q_t\cap\langle\ell,P\rangle\), which is defined by a quadratic equation in \(\langle\ell,P\rangle\cong\P^3\), contains the three distinct 2-planes \(\langle\ell,m'\rangle\), \(P=\langle m',m\rangle\), and \(\langle\ell,m\rangle\) and hence must be equal to \(\langle\ell,P\rangle\), which contradicts the definition of \(\Sigma_2\)). Since \(\langle\ell,m'\rangle\not\subset X\) and \(P=\langle m,m'\rangle\not\subset X\), Lemma~\ref{lem:lines-meeting-l-span-P3-common-point} then implies that \(X\cap\langle\ell,P\rangle=X\cap\langle\ell,m',m\rangle\) is equal to \(\ell+2m'+m\).
On the other hand, any \((m',P,t)\in\Sigma_3\) satisfies \(X\cap\langle\ell,P\rangle=\ell+3m'\) (by Lemma~\ref{lem:intersect-3plane-containing-2plane-meeting-in-2m'-cases}, since \(\langle\ell,m'\rangle\subset\mathcal Q_t\) and \(\langle\ell,m'\rangle\not\subset X\)).
So the map
\[
\Sigma_2\sqcup \Sigma_3\to F_{M,4}, \qquad (m',P,t)\mapsto (\langle \ell,P\rangle/(\ell\cap m\cap m'), p_m, \ell\cap m\cap m'),
\]
where $X\cap P=m+m'$ (possibly $m=m'$) is well defined; furthermore, it is a bijection with inverse \((P,q,x) \mapsto (m_p, \langle m_p,m_q\rangle,t)\) where \(Y_x\cap P=p_\ell+2p+q\), \(m_p\) (resp. \(m_q\)) is the line corresponding to \(p\) (resp. \(q\)), and \(t\in\P^1\) is the unique point with \(\langle\ell,m_p\rangle\subset\mathcal Q_t\) (which exists because \(\langle p_\ell,p\rangle\not\subset Y_x\), so \(m_p\in V_{\ell+m}\)); we have \(\langle m_p,m_q\rangle\subset\mathcal Q_t\) by Lemmas~\ref{lem:lines-meeting-l-span-P3-common-point}, \ref{lem:intersect-3plane-containing-2plane-meeting-in-2m'-cases}, and~\ref{lem:intersect-3m'+l-then-2plane}. Here, if \(p=q\), then \(\langle m_p,m_q\rangle\) is the unique 2-plane such that \(X\cap L \cap \langle m_p,m_q\rangle\) is nonreduced along the line \(m_p\) (Lemmas~\ref{lem:intersect-3plane-containing-2plane-meeting-in-2m'-cases} and~\ref{lem:intersect-3m'+l-then-2plane}), where \(L\subset\P^N\) is the 3-plane corresponding to \(P\subset T_{x}X/x\).
Hence by Lemma \ref{lem:bijection-same-class},
\[
[F_{M,3}]=[\Sigma_4],\qquad [F_{M,4}]=[\Sigma_2\sqcup \Sigma_3].
\]

We conclude
\begin{equation}\label{eqn:W_5}
[W_{\ell+m_0+2m_1}\sqcup W_{\mathrlap{\times}+}]=([\mathcal{Q}^{(1)}_\ell]-[\mathcal{Q}_{Y_\ell}])\L^{N-5} - [\Sigma_2\sqcup \Sigma_3\sqcup \Sigma_4].
\end{equation}

\subsection{Summary}

\begin{prop}\label{prop:excep-loci}
In the above notations, we have the following equalities. For \(N=6,7\), the terms involving \([Y_\ell]\) are omitted since \(Y_\ell=\emptyset\); the displayed formulas below are written uniformly for \(N\geq 8\).
\begin{equation*}
\begin{split}
& \eqref{eq:v-plane}\quad [V_{\text{plane}}]=[Y_\ell][\P^1]\L,\\
& \eqref{eq:v-w-ell-m}\quad [V_{\ell+m}]=[W_{\ell+m_0}]=[\mathcal{Q}^{(1)}_\ell]-[\mathcal{Q}_{Y_\ell}],\\
& \eqref{eq:w}\quad [W]=[\Sym^2\mathcal{Q}_\ell^{(1)}]-[\Sym^2\mathcal{Q}_{Y_\ell}]+[\mathcal{Q}_{Y_\ell}]^2-[\mathcal{Q}^{(1)}][\mathcal{Q}_{Y_\ell}],\\
& \eqref{eq:v-3-plane}\quad [V_{\text{$3$-plane}}]=[F_1(Y_\ell)]\L^4,\\
& \eqref{eq:v-w-quad}\quad [{V_{\text{quad,rk}\geq3}}]=([F_1(\mathcal{Q}^{(1)}_\ell/\P^1)]-[F_1(\mathcal{Q}_{Y_\ell}/\P^1)])\L - [Y_\ell]([\mathcal{Q}^{(1)}_\ell]-[\mathcal{Q}_{Y_\ell}]+[\P^1](\L^{N-7}-\L^{N-5})) \\ & \qquad\qquad\qquad\qquad -[S_3(X)_{\ell,\text{rk}=3}]\L,\\
& \eqref{eq:v-w-quad}\quad [{W_{\text{quad,rk}\geq3}}]=([F_1(\mathcal{Q}^{(1)}_\ell/\P^1)]-[F_1(\mathcal{Q}_{Y_\ell}/\P^1)])\L^2 - [Y_\ell]([\mathcal{Q}^{(1)}_\ell]-[\mathcal{Q}_{Y_\ell}]+[\P^1](\L^{N-7}-\L^{N-5}))\L \\ & \qquad\qquad\qquad\qquad -[S_3(X)_{\ell,\text{rk}=3}]\L,\\
& \eqref{eq:v-w-quad}\quad [V_{\text{quad,rk}\leq 2}]=[Y_\ell]([\mathcal{Q}^{(1)}_\ell]-[\mathcal{Q}_{Y_\ell}]+[\P^1](\L^{N-7}-\L^{N-5}))\L,\\
& \eqref{eq:v-w-quad}\quad [W_{\text{quad,rk}\leq 2}]=[Y_\ell]([\mathcal{Q}^{(1)}_\ell]-[\mathcal{Q}_{Y_\ell}]+[\P^1](\L^{N-7}-\L^{N-5}))(\L-1),\\
&\eqref{eq:ell+P}\;\; [V_{\ell+P}]=[W_{\ell+P}]=[(F_2(X)^{\text{int}}_\ell)^\circ]\L^2,\\
&\eqref{eq:v-m+P}\;\; [V_{m+P}]=[Y_\ell][\P^1](\L^{N-4}-\L^{N-5}),\\
&\eqref{eq:v-ell+m+2m'}\;\; [V_{\ell+m+2m'}]=([\mathcal{Q}^{(1)}_\ell]-[\mathcal{Q}_{Y_\ell}]-(N+1))\L^{N-5}-[\Sigma_2\sqcup \Sigma_3\sqcup \Sigma_4],\\
&\eqref{eq:w-2ell+m_0+m_1}\;\; [W_{2\ell+m_0+m_1}]=[\P^1]([\Gr(2,N-3)]-[\Gr(2,N-5)]) + ([F_1(\mathcal{Q}^{(1)}_\ell/\P^1)]-[F_1(\mathcal{Q}_{Y_\ell}/\P^1)])\L\\
&\qquad\qquad\qquad\qquad\quad -[\mathcal{Q}^{(1)}_\ell][\P^{N-6}]+ [\mathcal{Q}_{Y_\ell}](-\L^{N-5}+[\P^{N-7}])-(N+1)\L^{N-5}, \\
& \eqref{eqn:W_5}\;\; [W_{\ell+m_0+2m_1}\sqcup W_{\mathrlap{\times}+}]=([\mathcal{Q}^{(1)}_\ell]-[\mathcal{Q}_{Y_\ell}])\L^{N-5} - [\Sigma_2\sqcup \Sigma_3\sqcup \Sigma_4].
\end{split}
\end{equation*}
\end{prop}

\begin{rem}
Among the list, $V_{\ell+m+2m'}, W_{2\ell+m_0+m_1}, W_{\ell+m_0+2m_1}\sqcup W_{\mathrlap{\times}+}$ are divisors, and the other exceptional loci are of codimension $>1$. 
This is consistent with the fact that $\rho(F_1(X))-\rho(\Sym^2\mathcal{Q}^{(1)})=-1$. 
Some of the loci of codimension $>1$ seem to form pairs, and one can speculate that they correspond to flopped-flopping or flipped-flipping loci of $F_1(X)\dashrightarrow \Sym^2\mathcal{Q}^{(1)}$. 
\end{rem}

\section{Proof of main theorems and corollaries}\label{sec:summing-together}

In this section, we prove Theorems~\ref{thm:main-odd} and~\ref{thm:main-even}.
As an intermediate step, we express the class of $F_1(X)$ in terms of the classes of $\mathcal{Q}^{(1)}_\ell, F_1(\mathcal{Q}^{(1)}_\ell/\P^1), Y_\ell, F_1(\mathcal{Q}_{Y_\ell}/\P^1)$.
The formula is independent of the parity of $N$.

\begin{thm}\label{thm:towards-thm}
Over an algebraically closed field $k$ of characteristic $\neq 2$, let $N\geq 6$, $X\subset \P^N$ be a smooth complete intersection of two quadrics, and $\ell\in F_1(X)$ be a general line.
Then the following equalities hold in $\widetilde{K}_0(\mathrm{Var}/k)$:

If \(N=6,7\), then
\begin{equation*}
\begin{split}
[F_1(X)]&=[\Sym^2\mathcal{Q}^{(1)}_\ell]-[F_1(\mathcal{Q}^{(1)}_\ell/\P^1)]\L^2 -[\P^1][\P^{N-5}]\L^{N-5}-[\Sym^2\P^{N-6}]+1.
\end{split}
\end{equation*}

If \(N\geq 8\), then
\begin{equation*}
\begin{split}
[F_1(X)]&=[\Sym^2\mathcal{Q}^{(1)}_\ell]+[F_1(Y_\ell)]\L^4-[\Sym^2 Y_\ell]\L^2
+[Y_\ell]([\P^1]-[\P^{N-8}])\L
+(-[F_1(\mathcal{Q}^{(1)}_\ell/\P^1)]+[F_1(\mathcal{Q}_{Y_\ell}/\P^1)])\L^2\\
&\quad -[\P^1][\P^{N-5}]\L^{N-5}-[\Sym^2\P^{N-6}]+1.
\end{split}
\end{equation*}
\end{thm}
\begin{proof}
By Proposition \ref{prop:v-w-isom} and Proposition \ref{prop:excep-loci},
\begin{equation}\label{eq:towards-thm-1}
\begin{split}
[F_1(X)]&=[V_{\text{plane}}] + ([V_{\ell+m}]-[W_{\ell+m_0}]) + [V_{\text{$3$-plane}}]+ ([V_{\text{quad,rk}\geq 3}]-[W_{\text{quad,rk}\geq 3}])+([V_{\text{quad,rk}\leq 2}]-[W_{\text{quad,rk}\leq 2}])\\
&\quad+([V_{\ell+P}]-[W_{\ell+P}]) + [V_{m+P}] + [V_{\ell+m+2m'}]-[W_{2\ell+m_0+m_1}]-[W_{\ell+m_0+2m_1}\sqcup W_{\mathrlap{\times}+}]+[W] +1\\
&=[\Sym^2\mathcal{Q}^{(1)}_\ell]+[F_1(Y_\ell)]\L^4+(-[F_1(\mathcal{Q}^{(1)}_\ell/\P^1)]+[F_1(\mathcal{Q}_{Y_\ell}/\P^1)])\L^2\\
&\quad -[\Sym^2 \mathcal{Q}_{Y_\ell}] +[\mathcal{Q}_{Y_\ell}]^2-[\mathcal{Q}^{(1)}_\ell][\mathcal{Q}_{Y_\ell}] + [\mathcal{Q}^{(1)}_\ell][\P^{N-6}]-[\mathcal{Q}_{Y_\ell}](-\L^{N-5}+[\P^{N-7}])\\
&\quad +[Y_\ell]\L([\mathcal{Q}^{(1)}_\ell]-[\mathcal{Q}_{Y_\ell}]+[\P^1](\L^{N-7}-\L^{N-6}+1)) -[\P^1]([\Gr(2,N-3)]-[\Gr(2,N-5)]) +1.
\end{split}
\end{equation}
By Lemma~\ref{lem:general-l-hyperbolic-red}\eqref{item:general-l-Y-smooth-ci}, $\mathcal{Q}_{Y_\ell}$ is the blow-up of $\P^{N-6}$ along $Y_\ell$, so 
\begin{equation}\label{eq:towards-thm-2}
[\mathcal{Q}_{Y_\ell}]=[\P^{N-6}]+[Y_\ell]\L,
\end{equation}
and accordingly,
\begin{equation}\label{eq:towards-thm-3}
\begin{split}
[\mathcal{Q}_{Y_\ell}]^2 &=[\P^{N-6}]^2 + 2[Y_\ell]\L[\P^{N-6}] + [Y_\ell]^2 \L^2,\\
[\Sym^2\mathcal{Q}_{Y_\ell}] &=[\Sym^2\P^{N-6}]+[Y_\ell]\L[\P^{N-6}] + [\Sym^2 Y_\ell]\L^2.
\end{split}
\end{equation}
In addition, by \cite[Proposition 2.1]{Martin16} (which is valid over any field),
\begin{equation}\label{eq:towards-thm-4}
[\Gr(2,N-3)]=[\Gr(2,N-5)] + \L^{N-6}[\P^{N-6}] + \L^{N-5}[\P^{N-5}].
\end{equation}
It is straightforward to deduce from \eqref{eq:towards-thm-1}, \eqref{eq:towards-thm-2}, \eqref{eq:towards-thm-3}, \eqref{eq:towards-thm-4} the desired formula.
For \(N=6,7\) we have \(Y_\ell=F_1(\mathcal{Q}_{Y_\ell}/\P^1)=\emptyset\) by Lemma~\ref{lem:general-l-hyperbolic-red}, and all terms involving these classes vanish.
\end{proof}

Now we are ready to prove Theorem~\ref{thm:main-odd}.

\begin{proof}[Proof of Theorem~\ref{thm:main-odd}]
    If \(N=2g+1\), then \([\mathcal Q^{(1)}_\ell] = [C]\L^{g-2} + [\P^1][\P^{g-3}](1+\L^{g-1})\) and
    \begin{equation*}
        \begin{split}
            [\Sym^2\mathcal Q^{(1)}_\ell] &= [\Sym^2 C]\L^{2g-4} + [C][\P^1][\P^{g-3}](1+\L^{g-1})\L^{g-2} \\ & \quad + 
        [\Sym^2(\P^1 \times \P^{g-3})](1+\L^{2(g-1)}) + [\P^1]^2[\P^{g-3}]^2 \L^{g-1}
        \end{split}
    \end{equation*}
    by Corollary~\ref{cor:class-of-Qr-formula} and Lemma~\ref{lem:K_0-sym-formulas}\eqref{item:formula-sym},
    so Theorem~\ref{thm:towards-thm} implies
    \begin{equation}\begin{split}\label{eq:N=2g+1-expression}
    [F_1(X)] &= [\Sym^2 C]\L^{2g-4} + [C][\P^1][\P^{g-3}](1+\L^{g-1})\L^{g-2} \\
    & \quad +[F_1(Y_\ell)]\L^4+(-[F_1(\mathcal{Q}^{(1)}_\ell/\P^1)]+[F_1(\mathcal{Q}_{Y_\ell}/\P^1)])\L^2 -[\Sym^2 Y_\ell]\L^2 +[Y_\ell]([\P^1]-[\P^{2g-7}])\L \\
    &\quad + [\Sym^2(\P^1 \times \P^{g-3})](1+\L^{2(g-1)}) + [\P^1]^2[\P^{g-3}]^2 \L^{g-1} -[\P^{2g-4}][\P^1]\L^{2g-4} - [\Sym^2 \P^{2g-5}] +1 .
    \end{split}\end{equation}
    To prove the formula in Theorem~\ref{thm:main-odd}, we induct on \(g\). There are three base cases: \(g=3,4,5\).

    If \(g=3\), then \(\mathcal Q_{Y_\ell}\cong\P^1\), \(Y_{\ell}=F_1(Y_{\ell})=F_1(\mathcal{Q}_{Y_\ell}/\P^1)=\emptyset\), and \([F_1(\mathcal{Q}^{(1)}_\ell/\P^1)]=[C][\P^1]\) by Corollary~\ref{cor:relative-F_1}, so~\eqref{eq:N=2g+1-expression} and Lemma~\ref{lem:formulas-Gr-quadric-Sym-Pn}\eqref{item:class-of-Sym-Pn} yield
    \begin{equation*}
        \begin{split}
            [F_1(X)] &= [\Sym^2 C]\L^2 + [C]\L(\L^3+1) + \L^6 + 1 \\
            &= \textstyle [\Sym^2 C] \L^2 + [C] \L \left( \sum_{i=0}^3 \L^i - \L - \L^2 \right) \\ & \qquad + \textstyle \left(\sum_{i=0}^4 \L^i\right) \left(\sum_{i=0}^1\L^{2i}\right)- (\L + \L^5) - (\L^2 + \L^3 + \L^4 + \L^2 + \L^3 + \L^4) .
        \end{split}
    \end{equation*}

    If \(g=4\), then \(Y_\ell\cong C_\ell\) is the genus 1 curve from Definition~\ref{defn:C_Y-curve} by Lemma~\ref{lem:general-l-hyperbolic-red}\eqref{item:general-l-Y-smooth-ci} and \cite[Proposition 4.2]{Reid-thesis}, and \([\Sym^2 C_\ell] = [C_\ell][\P^1]\) by Riemann--Roch. We have \([F_1(\mathcal Q^{(1)}_\ell/\P^1)] = [\P^3] ( [C]\L + \L^3+1)\) and \([F_1(\mathcal Q_{Y_\ell}/\P^1)] = [C_\ell][\P^1]\) by Corollary~\ref{cor:relative-F_1}, so~\eqref{eq:N=2g+1-expression} and Lemma~\ref{lem:formulas-Gr-quadric-Sym-Pn}\eqref{item:class-of-Sym-Pn} yield
    \([F_1(X)] = [\Sym^2 C]\L^4 + [C]\L^2 (\L^5+\L^4+\L+1) + (\L^2 - \L + 1) (\L^2 + \L + 1)^2 (\L^4 - \L^2 + 1)\).
    \begin{equation*}
        \begin{split}
            [F_1(X)] &= [\Sym^2 C]\L^4 + [C]\L^2 (\L^5+\L^4+\L+1) + (\L^2 - \L + 1) (\L^2 + \L + 1)^2 (\L^4 - \L^2 + 1) \\
            &= \textstyle [\Sym^2 C] \L^4 + [C] \L^2 \left( \sum_{i=0}^5 \L^i - \L^2 - \L^3 \right) \\
            & \qquad \textstyle  + \left(\sum_{i=0}^6 \L^i\right) \left(\sum_{i=0}^2\L^{2i}\right)- (\L^2 + \L^6) \sum_{i=0}^2\L^i - (\L^3 + \L^4 + \L^5 + \L^5 + \L^6 + \L^7) .
        \end{split}
    \end{equation*}

    If \(g=5\), then \(Y_\ell\) is a smooth complete intersection of two quadrics in \(\P^5\) by Lemma~\ref{lem:general-l-hyperbolic-red}\eqref{item:general-l-Y-smooth-ci}, so \(F_1(Y_\ell)\) is isomorphic to the Jacobian of the genus 2 curve \(C_\ell\) by Lemma~\ref{lem:F_r(X)}\eqref{item:lem-F_r(X)-maximal-odd} and hence \([F_1(Y_\ell)] = [\Sym^2 C_\ell]-\L\) (see, e.g., \cite[Proposition 4.6]{BGM2023decomposition}).
    By Corollary~\ref{cor:relative-F_1} we have
    \begin{equation*}
        \begin{split}
            [F_1(\mathcal Q^{(1)}_\ell/\P^1)] &= [C][\P^5]\L^2 +[\P^5]([\P^5] - \L^3 - \L^2) \\
            [F_1(\mathcal Q_{Y_\ell}/\P^1)] &= [C_\ell][\P^3]\L + [\P^3]([\P^3] - \L^2 - \L)
        \end{split}
    \end{equation*}
    By Proposition~\ref{prop:r=0-case}, \([Y_\ell] = [C_\ell]\L + [\P^3] - \L[\P^1]\), so \([\Sym^2 Y_\ell] = [\Sym^2 C_\ell]\L^2 + [C_\ell](\L^4+\L)+\L^6+\L^3+1\) by Lemma~\ref{lem:K_0-sym-formulas}\eqref{item:formula-sym}. Then~\eqref{eq:N=2g+1-expression} and Lemma~\ref{lem:formulas-Gr-quadric-Sym-Pn}\eqref{item:class-of-Sym-Pn} yield
    \begin{equation*}
        \begin{split}
            [F_1(X)] &= [\Sym^2 C]\L^6 + [C]\L^3 (\L^7 + \L^6 + \L^5 + \L^2 + \L + 1) \\ & \qquad + (\L^2 + 1) (\L^4 + 1) (\L^8 + \L^7 + \L^6 - \L^4 + \L^2 + \L + 1) \\
            &= \textstyle [\Sym^2 C] \L^6 + [C] \L^3 \left( \sum_{i=0}^7 \L^i - \L^3 - \L^4 \right) \\
            & \qquad \textstyle  + \left(\sum_{i=0}^8 \L^i\right) \left(\sum_{i=0}^3\L^{2i}\right)- (\L^3 + \L^7) \sum_{i=0}^4 \L^i - (\L^4 + \L^5 + \L^6 + \L^8 + \L^{9} + \L^{10}) .
        \end{split}
    \end{equation*}

    Now let \(g\geq 6\) and assume Theorem~\ref{thm:main-odd} holds for \(g-3\). Then \(Y_\ell\) is a smooth complete intersection of two quadrics in \(\P^{2(g-3)+1}\) by Lemma~\ref{lem:general-l-hyperbolic-red}\eqref{item:general-l-Y-smooth-ci}, so the induction hypothesis implies
    \begin{equation*}
        \begin{split}
            [F_1(Y_\ell)] &= \textstyle [\Sym^2 C_\ell] \L^{2g-10} + [C_\ell] \L^{g-5} \left( \sum_{i=0}^{2g-9} \L^i - \L^{g-5} - \L^{g-4} \right) + \left(\sum_{i=0}^{2g-8} \L^i\right) \left(\sum_{i=0}^{g-5}\L^{2i}\right) \\
            & \textstyle \phantom{=} - (\L^{g-5} + \L^{g-1}) \sum_{i=0}^{2g-12}\L^i - (\L^{g-4} + \L^{g-3} + \L^{g-2} + \L^{3g-16} + \L^{3g-15} + \L^{3g-14}) .
        \end{split}
    \end{equation*}
    Corollary~\ref{cor:relative-F_1}, Proposition~\ref{prop:r=0-case}, and Lemma~\ref{lem:K_0-sym-formulas}\eqref{item:formula-sym} imply
    \begin{equation*}
        \begin{split}
            [F_1(\mathcal Q^{(1)}_\ell/\P^1)] &= [C][\P^{2g-5}]\L^{g-3} +[\P^{2g-5}]([\P^{2g-5}] - \L^{g-2} - \L^{g-3})  \\
            [F_1(\mathcal Q_{Y_\ell}/\P^1)] &= [C_\ell][\P^{2g-7}]\L^{g-4} + [\P^{2g-7}]([\P^{2g-7}] - \L^{g-3} - \L^{g-4}) \\
            [\Sym^2 Y_\ell] &= [\Sym^2 C_\ell]\L^{2g-8} + [C_\ell]\L^{g-4}([\P^{2g-7}] - \L^{g-4}[\P^1]) + [\Sym^2 \P^{2g-7}] - [\P^{2g-7}][\P^1]\L^{g-4} + \L^{2g-7} \\
            [Y_\ell] &= [C_\ell]\L^{g-4} + [\P^{2g-7}] - \L^{g-4}[\P^1] \\
        \end{split}
    \end{equation*}
    Plugging these into equation~\eqref{eq:N=2g+1-expression}, we find that
    the coefficient of \([\Sym^2 C_\ell]\) is \(0\), the coefficient of \([C_\ell]\) is \(\L^{g-3}+\frac{\L^{g-3}}{\L-1}+\L^{g-2}-\frac{\L^{g-1}}{\L-1}=0\), the coefficient of \([\Sym^2 C]\) is \(\L^{2g-4}\), the coefficient of \([C]\) is equal to
    \(\L^{g-2} \left( \sum_{i=0}^{2g-3} \L^i - \L^{g-2} - \L^{g-1} \right)\), and the constant term is equal to \(\left(\sum_{i=0}^{2g-2} \L^i\right) \left(\sum_{i=0}^{g-2}\L^{2i}\right)- (\L^{g-2} + \L^{g+2}) \sum_{i=0}^{2g-6}\L^i - (\L^{g-1} + \L^g + \L^{g+1} + \L^{3g-7} + \L^{3g-6} + \L^{3g-5})\). These claims can be verified using the attached \texttt{Mathematica} code.
    Therefore
    \begin{equation*}
        \begin{split}
            [F_1(X)] &= \textstyle [\Sym^2 C] \L^{2g-4} + [C] \L^{g-2} \left( \sum_{i=0}^{2g-3} \L^i - \L^{g-2} - \L^{g-1} \right) + \left(\sum_{i=0}^{2g-2} \L^i\right) \left(\sum_{i=0}^{g-2}\L^{2i}\right) \\
            & \textstyle \phantom{=} - (\L^{g-2} + \L^{g+2}) \sum_{i=0}^{2g-6}\L^i - (\L^{g-1} + \L^g + \L^{g+1} + \L^{3g-7} + \L^{3g-6} + \L^{3g-5})
        \end{split}
    \end{equation*}
    as required, completing the induction. 
    
    Finally, the coefficients in Theorem~\ref{thm:main-odd} are effective and are equal to the \(M_{g,1,i}\) by Lemma~\ref{lem:coeffs-effective-odd} below.
    This finishes the proof of Theorem~\ref{thm:main-odd}.
\end{proof}

\begin{lem}\label{lem:coeffs-effective-odd}
    Let \(g\geq 3\) be an integer, and let \(M_{g,1,i}(q)\) be the polynomials in the variable \(q\) such that \(q\mapsto\L\) gives the expressions in \cite[(4)]{BBFGHLPRAS}.
    The following equalities of polynomials hold in \(\mathbb Z[q]\):
    \begin{equation*}
        \begin{split}
            M_{g,1,2}(q) &= q^{2g-4} ,\\
            M_{g,1,1}(q) &= q^{g-2} \left( \sum_{i=0}^{2g-3} q^i - q^{g-2} - q^{g-1} \right) = \sum_{i=g-2}^{2g-5}q^i+\sum_{i=2g-2}^{3g-5}q^i, \\
            M_{g,1,0}(q) &= \left(\sum_{i=0}^{2g-2} q^i\right) \left(\sum_{i=0}^{g-2}q^{2i}\right)- (q^{g-2} + q^{g+2}) \left(\sum_{i=0}^{2g-6}q^i\right) - (q^{g-1} + q^g + q^{g+1} + q^{3g-7} + q^{3g-6} + q^{3g-5}) \\
            &= \left(\sum_{-2g+3\leq i\leq 2g-3}\left(g-1-\left\lfloor\frac{|i|}{2}\right\rfloor\right)q^{2g-3+i}\right)-\sum_{i=g-2}^{3g-5}(q^{i}+q^{i+1}) .
        \end{split}
    \end{equation*}
    Furthermore, these expressions are all effective, i.e., they are polynomials with nonnegative integer coefficients.
\end{lem}

\begin{proof}
    In each of the first three lines, the first equality holds by expanding the definition of \(M_{g,1,i}(q)\) from \cite[(4)]{BBFGHLPRAS} using that \(\binom{2g-2}{1}_q = \sum_{i=0}^{2g-3}q^i\), \(\binom{2g-1}{2}_q = \left(\sum_{i=0}^{2g-2} q^i\right)\left(\sum_{i=0}^{g-2} q^{2i}\right)\), and \(\binom{2g-5}{1}_q = \sum_{i=0}^{2g-6}q^i\).
    In the second line, the second equality holds by expanding the middle polynomial, and effectivity of the right-hand polynomial is clear.
    Thus, it remains to show the final equality and to show that this polynomial is effective.
    The final equality follows from expanding the right-hand polynomial in the third line and combining the second and third terms, and the polynomial in the fourth line is effective by observing that the coefficient of $q^i$ in the first term is $\geq 2$ for $2\leq i\leq 4g-8$ and $1$ for $i\in \{1,4g-7\}$ while that in the second term is $2$ for $g-1\leq i\leq 3g-5$, $1$ for $i\in\{g-2, 3g-4\}$, and $0$ otherwise.
\end{proof}
\begin{rem}
    Theorem~\ref{theorem:appendix-main} in the appendix for the case of \(M_{g,1,i}\) shows that
    \begin{equation*}
    \begin{aligned}
     M_{g,1,2}(q)
     &=q^{2g-4},\\
    M_{g,1,1}(q)
    &=q^{g-2}(1+q^g)\binom{g-2}{1}_q,\\
    M_{g,1,0}(q)
    &=(1+q^{2g})\binom{g-1}{2}_q+q^g\binom{g-3}{1}_q\binom{g-1}{1}_q,
    \end{aligned}
    \end{equation*}
    giving another proof of effectivity of these polynomials.
\end{rem}

\begin{proof}[Proof of Theorem~\ref{thm:main-even}]
    If \(N=2g\), then \([\mathcal Q^{(1)}] = (2g+1)\L^{g-2} + [\P^1][\P^{g-3}](1+\L^{g-2})\) and
    \begin{equation*}
        \begin{split}
            [\Sym^2\mathcal Q^{(1)}] &= (2g+1)(g+1)\L^{2g-4} + (2g+1)[\P^1][\P^{g-3}](1+\L^{g-2})\L^{g-2} \\
        & \quad + [\Sym^2(\P^1\times\P^{g-3})](1+\L^{2g-4})+[\P^1]^2[\P^{g-3}]^2\L^{g-2}
        \end{split}
    \end{equation*}
    by Corollary~\ref{cor:class-of-Qr-formula} and Lemma~\ref{lem:K_0-sym-formulas}\eqref{item:formula-sym},
    so Theorem~\ref{thm:towards-thm} implies
    \begin{equation}\begin{split}\label{eq:N=2g-expression}
    [F_1(X)]&= (2g+1)(g+1)\L^{2g-4} + (2g+1)[\P^1][\P^{g-3}](1+\L^{g-2})\L^{g-2} \\
    & \quad +[F_1(Y_\ell)]\L^4 +(-[F_1(\mathcal{Q}^{(1)}_\ell/\P^1)]+[F_1(\mathcal{Q}_{Y_\ell}/\P^1)])\L^2 -[\Sym^2 Y_\ell]\L^2 +[Y_\ell]([\P^1]-[\P^{2g-8}])\L\\
    &\quad + [\Sym^2(\P^1\times\P^{g-3})](1+\L^{2g-4})+[\P^1]^2[\P^{g-3}]^2\L^{g-2} -[\P^1][\P^{2g-5}]\L^{2g-5}-[\Sym^2\P^{2g-6}]+1.
    \end{split}\end{equation}
    We induct on \(g\). There are three base cases: \(g=3,4,5\).

    If \(g=3\), then \(\mathcal Q_{Y_\ell}\) is a point by Lemma~\ref{lem:general-l-hyperbolic-red}, \(Y_\ell=F_1(Y_\ell)=F_1(\mathcal Q_{Y_\ell}/\P^1)=\emptyset\), and \([F_1(\mathcal Q^{(1)}_\ell/\P^1)] = 14\) by Corollary~\ref{cor:relative-F_1}, so~\eqref{eq:N=2g-expression} and Lemma~\ref{lem:formulas-Gr-quadric-Sym-Pn}\eqref{item:class-of-Sym-Pn} yield
    \[[F_1(X)] = \L^4 + 8 \L^3 + 30 \L^2 + 8 \L + 1 .\]

    If \(g=4\), then \(Y_\ell\) is a smooth complete intersection of two quadrics in \(\P^2\) by Lemma~\ref{lem:general-l-hyperbolic-red}\eqref{item:general-l-Y-smooth-ci}, so \([Y_\ell]=4\), \([\Sym^2 Y_\ell]=10\), and \(F_1(Y_\ell)=\emptyset\). We have \([F_1(\mathcal Q^{(1)}_\ell/\P^1)] = 9([\P^2]+\L)\L + [\P^2]^2-\L^2\) and \([F_1(\mathcal Q_{Y_\ell}/\P^1)] = 6\) by Corollary~\ref{cor:relative-F_1}, so~\eqref{eq:N=2g-expression} and Lemma~\ref{lem:formulas-Gr-quadric-Sym-Pn}\eqref{item:class-of-Sym-Pn} yield
    \[\L^8 + \L^7 + 11 \L^6 + 11 \L^5 + 48 \L^4 + 11 \L^3 + 11 \L^2 + \L + 1 .\]

    If \(g=5\), then \(Y_\ell\) is a smooth complete intersection of two quadrics in \(\P^4\) by Lemma~\ref{lem:general-l-hyperbolic-red}\eqref{item:general-l-Y-smooth-ci}, so \(F_1(Y_\ell)\) is a reduced finite scheme of length \(16\) by Lemma~\ref{lem:F_r(X)}\eqref{item:lem-F_r(X)-maximal-even} and thus \([F_1(Y_\ell)]=16\). We have
    \begin{align*}
        [F_1(\mathcal Q^{(1)}_\ell/\P^1)] &= 11([\P^4]+\L^2)\L^2 + [\P^4]^2-\L^4 , & [\Sym^2 Y_\ell] &= [\Sym^2\P^2] + 5[\P^2]\L + 15\L^2 , \\
        [F_1(\mathcal Q_{Y_\ell}/\P^1)] &= 5([\P^2]+\L)\L + [\P^2]^2-\L^2 , & [Y_\ell] &= [\P^2] +5\L
    \end{align*}
    by Corollary~\ref{cor:relative-F_1}, Proposition~\ref{prop:r=0-case}, and Lemma~\ref{lem:K_0-sym-formulas}\eqref{item:formula-sym}, so~\eqref{eq:N=2g-expression} and Lemma~\ref{lem:formulas-Gr-quadric-Sym-Pn}\eqref{item:class-of-Sym-Pn} yield
    \[[F_1(X)] = \L^{12} + \L^{11} + 2 \L^{10} + 13 \L^9 + 14 \L^8 + 14 \L^7 + 70 \L^6 + 14 \L^5 + 14 \L^4 + 13 \L^3 + 2 \L^2 + \L + 1.\]

    Now let \(g\geq 6\) and assume Theorem~\ref{thm:main-even} holds for \(g-3\). Then \(Y_\ell\) is a smooth complete intersection of two quadrics in \(\P^{2(g-3)}\) by Lemma~\ref{lem:general-l-hyperbolic-red}\eqref{item:general-l-Y-smooth-ci}, so the induction hypothesis implies
    \[\textstyle [F_1(Y_\ell)] = 2(g-3)(g-1)\L^{2g-10} + (\L+1) \sum_{i=0}^{g-6}(i+1)(\L^{2i} + \L^{4g-21-2i})+(2g-5)(\L^{g-5}+\L^{2g-9})\sum_{i=0}^{g-6}\L^i .\]
    By Corollary~\ref{cor:relative-F_1}, Proposition~\ref{prop:r=0-case}, and Lemma~\ref{lem:K_0-sym-formulas}\eqref{item:formula-sym} we have
    \begin{equation*}\begin{split}
        [F_1(\mathcal Q^{(1)}_\ell/\P^1)] &= (2g+1)([\P^{2g-6}]+\L^{g-3})\L^{g-3} + [\P^{2g-6}]^2-\L^{2g-6} , \\
        [F_1(\mathcal Q_{Y_\ell}/\P^1)] &= (2g-5)([\P^{2g-8}]+\L^{g-4})\L^{g-4} + [\P^{2g-8}]^2-\L^{2g-8} , \\
        [Y_\ell] &= [\P^{2g-8}] +(2g-5)\L^{g-4} , \\
        [\Sym^2 Y_\ell] &= [\Sym^2 \P^{2g-8}] +(2g-5)[\P^{2g-8}]\L^{g-4} + (2g-5)(g-2)\L^{2g-8} .
    \end{split}\end{equation*}
    Plugging these equalities and Lemma~\ref{lem:formulas-Gr-quadric-Sym-Pn}\eqref{item:class-of-Sym-Pn} into equation~\eqref{eq:N=2g-expression} yields
    \[\textstyle [F_1(X)] = 2g(g+2)\L^{2g-4} + (\L+1) \sum_{i=0}^{g-3}(i+1)(\L^{2i} + \L^{4g-9-2i})+(2g+1)(\L^{g-2}+\L^{2g-3})\sum_{i=0}^{g-3}\L^i \]
    as required (this can be verified using the attached \texttt{Mathematica} code). This proves Theorem~\ref{thm:main-even}.
\end{proof}

\subsection{Consequences of Theorems~\ref{thm:main-odd} and~\ref{thm:main-even}}

First we prove a lemma that will be used to compute the Betti and Hodge numbers of \(F_1(X)\). Recall that by Lemma~\ref{lem:F_r(X)}\eqref{item:lem-F_r(X)-dim}, if \(N=2g+1\) is odd, then \(\dim F_1(X)=4g-6\), and if $N=2g$ is even, then $\dim F_1(X)=4g-8$.

\begin{lem}\label{lem:odd-case-N-and-betti-hodge-numbers}
Fix $g\geq 3$, and let $N(g-1,g-j;i)$ be the coefficient of $q^{i}$ in the Laurent polynomial
\[
q^{-(2-j)(2g-3)}(1-q^{4(g-j)})\frac{\prod_{l=3-j}^{2g-j-3}(1-q^{2l})}{\prod_{l=1}^{2g-4} (1-q^{2l})}.
\]
For $0\leq i\leq 8g-12$, let $b_{g,i}^{\text{odd}}=\sum_{j=0}^2 N(g-1,g-j;i-4g+6)\binom{2g}{j}$,
and for $0\leq p,q\leq 4g-6$, let 
\[h^{p,q}_g=
\begin{cases}
N(g-1,g,2p-4g+6) + N(g-1,g-2,2p-4g+6)g^2 & \text{if }p=q,\\
N(g-1,g-1,p+q-4g+6)g & \text{if }|p-q|=1, \\
N(g-1,g-2,p+q-4g+6)\binom{g}{2} & \text{if }|p-q|=2,\\
0&\text{otherwise}.
\end{cases}
\]
Then the following hold:
\begin{enumerate}
\item\label{item:odd-case-N} We have
\begin{equation*}
\begin{split}
N(g-1,g;i)&= \begin{cases}
g-2 & \text{if }i=0,\\
g-1 - \left\lfloor\frac{|i|}{4}\right\rfloor & \text{if }-2(2g-3)\leq i\leq 2(2g-3) \text{ and $i$ is non-zero even},\\
0&\text{otherwise},
\end{cases}\\
N(g-1,g-1;i)&=\begin{cases}
1 & \text{if } -(2g-3)\leq i\leq 2g-3\text{ and $i$ is odd} \\
0 & \text{otherwise},
\end{cases}\\
N(g-1,g-2;i)& = \begin{cases}
1 & \text{if }i=0,\\
0 & \text{otherwise};
\end{cases}
\end{split}
\end{equation*}
\item\label{item:odd-case-poincare-serre-duality-hodge-decomposition} $b_{g,i}^{\text{odd}}=b_{g,8g-12-i}^{\text{odd}}$, $h^{p,q}_g=h^{q,p}_g$, $h^{p,q}_g=h^{4g-6-p, 4g-6-q}_g$, and $b_{g,i}^{\text{odd}}=\sum_{\substack{0\leq p,q\leq 4g-6 \\ p+q=i}}h^{p,q}_g$;
\item\label{item:odd-case-betti-numbers} For $0\leq i\leq 4g-6$, we have
\[
b_{g,i}^{\text{odd}}=\begin{cases}
\left\lfloor\frac{i}{4}\right\rfloor+1 & 0\leq i\leq 4g-8\text{ and $i$ is even}\\
2g & 2g-3\leq i\leq 4g-7 \text{ and $i$ is odd}, \\
\binom{2g}{2}+g-2 & i=4g-6,\\
0&\text{otherwise};
\end{cases}
\]
\item\label{item:odd-case-hodge-numbers} For $0\leq p, q\leq 4g-6$ and $0\leq p+q\leq 4g-6$, we have
\[
h^{p,q}_g=\begin{cases}
    \lfloor\frac{p}{2}\rfloor +1 & p=q\text{ and }0\leq p\leq 2g-4,\\
    g & |p-q|=1\text{ and }g-1\leq \max\{p,q\}\leq 2g-3,\\
    g^2+g-2 & p=q=2g-3,\\
    \binom{g}{2} & |p-q|=2\text{ and }\max\{p,q\}=2g-2,\\
    0&\text{otherwise}.
\end{cases}
\]
\end{enumerate}
\end{lem}
\begin{proof}
\eqref{item:odd-case-N} follows from direct computation. For~\eqref{item:odd-case-poincare-serre-duality-hodge-decomposition}, the equalities $b_{g,i}^{\text{odd}}=b_{g,8g-12-i}^{\text{odd}}$ and $h^{p,q}_g=h^{4g-6-p, 4g-6-q}_g$ follow from the symmetry of $N(g-1,g-j;i)$ with respect to $i=0$, $h^{p,q}_g=h^{q,p}_g$ by definition, and
$b_{g,i}^{\text{odd}}=\sum_{\substack{0\leq p,q\leq 4g-6\\ p+q=i}}h^{p,q}_g$ follows by direct computation.
\eqref{item:odd-case-betti-numbers} and \eqref{item:odd-case-hodge-numbers} are straightforward from \eqref{item:odd-case-N}.
\end{proof}

\begin{lem}[{\cite{CVX20}}]\label{lem:even-case-betti-number-CVX}
    Fix \(g\geq 3\), and let \(M(2,j;i)\) be the coefficient of \(q^{i-j(g-2)}\) in the polynomial \(g_{2-j,2g-2-j}(q)\) defined as the Gaussian binomial coefficient \[g_{d,e}(q) \coloneqq \frac{\prod_{l=e-d+1}^e (1-q^l)}{\prod_{l=1}^d (1-q^l)} = \binom{e}{d}_q.\]
    For \(0\leq i\leq 4g-8\), let \(b_{g,i}^{\text{even}} = \sum_{j=0}^2 M(2,j;i)\binom{2g+1}{j}\). Then the following hold:
    \begin{enumerate}
        \item\label{item:even-case-poincare-duality} \(b_{g,i}^{\text{even}} = b_{g,4g-8-i}^{\text{even}}\); and
        \item\label{item:even-case-betti-numbers} For \(0\leq i\leq 2g-4\), we have \[b_{g,i}^{\text{even}} = \begin{cases} \left\lfloor\frac{i}{2}\right\rfloor + 1 & \text{ if }0\leq i \leq g-3, \\ 2g+2+\left\lfloor\frac{i}{2}\right\rfloor & \text{ if }g-2 \leq i\leq 2g-5, \\ 2g(g+2) & \text{ if }i=2g-4. \end{cases} \]
    \end{enumerate}
\end{lem}
\begin{proof}
    Let \(X'\subset\P^{2g}_{\mathbb C}\) be a smooth complete intersection of two quadrics.
    Then \(\dim F_1(X')=4g-8\) by Lemma~\ref{lem:F_r(X)}\eqref{item:lem-F_r(X)-dim} and \(b_{g,i}^{\text{even}} = \dim_{\mathbb C} H^{2i}(F_1(X'),\mathbb C)\) by \cite[Theorem 1.1]{CVX20} (where \(M(2,j;i)\) is denoted by \(M_2(i,j)\)), so~\eqref{item:even-case-poincare-duality} follows from Poincar\'e duality. (Alternatively, \eqref{item:even-case-poincare-duality} can also be shown by direct computation.) For~\eqref{item:even-case-betti-numbers}, the values of \(b_{g,i}^{\text{even}}\) are computed in \cite[Example 6.5]{CVX20}.
\end{proof}

\begin{cor}\label{cor:complex}
Over the complex numbers, let $N\geq 6$, and let $X\subset \P^N$ be a smooth complete intersection of two quadrics.
\begin{enumerate}
\item\label{item:cor-odd-HS} (\cite[Theorem C]{BBFGHLPRAS} for $k=1$) Assume that $N=2g+1$ for some $g\geq 3$ and let $N(g-1,g-j;i)$, $b_{g,i}^{\text{odd}}$, and $h^{p,q}_g$ be as defined in Lemma~\ref{lem:odd-case-N-and-betti-hodge-numbers}.
For $0\leq i\leq 8g-12$, there exists an isomorphism of $\Q$-Hodge structures
\begin{equation}\label{eq:odd-case-hodge-str-isom}
H^i(F_1(X),\Q)\cong \bigoplus_{\substack{0\leq j\leq 2\\i-j\text{ is even}}} H^j(\Jac C, \Q((j-i)/2))^{\oplus N(g-1,g-j,i-4g+6)}.
\end{equation}
In addition, we have $b_i(F_1(X))=b_{g,i}^{\text{odd}}$ and $h^{p,q}(F_1(X))=h^{p,q}_g$.
\item\label{item:cor-even-HS} (\cite[Proposition E]{BBFGHLPRAS} for $k=1$) Assume that $N=2g$ for some $g\geq 3$ and let $M(2,j;i)$ and $b_{g,i}^{\text{even}}$ be as defined in Lemma~\ref{lem:even-case-betti-number-CVX}.
The cohomology of $F_1(X)$ is of Hodge-Tate type, and for $0\leq i\leq 2g-4$, there exists an isomorphism of $\Q$-Hodge structures
\begin{equation}\label{eq:even-case-hodge-str-isom}
H^{2i}(F_1(X),\Q)\cong \bigoplus_{0\leq j\leq 2}\Q(-i)^{\oplus M(2,j;i)\binom{2g+1}{j}}.
\end{equation}
In addition, we have $b_{2i}(F_1(X))=h^{i,i}(F_1(X))=b_{g,i}^{\text{even}}$.
\item\label{item:cor-cohomology-tf} $H^*(F_1(X),\Z)$ is torsion free.
\end{enumerate}
\end{cor}
\begin{proof}
Let $\mathrm{HS}_\Q$ be the category of polarizable pure $\Q$-Hodge structures.
By Bittner's theorem, there is a well-defined group homomorphism
$K_0(\mathrm{Var}/\C)\to K_0(\mathrm{HS}_\Q)$ by setting $[Z]\mapsto [H^{*}(Z,\Q)]\coloneqq\bigoplus_{i=0}^{2\dim Z}[H^i(Z,\mathbb Q)]$ whenever $Z$ is smooth, connected, and projective,
and it is moreover a $\Z[\L]$-module homomorphism where $\L^n\mapsto [\Q(-n)]$.
Accordingly, Theorems~\ref{thm:main-odd} and~\ref{thm:main-even} express $[H^*(F_1(X),\Q)]$ in $K_0(\mathrm{HS}_\Q)$ as a $\Z$-linear combination of Tate twists of $[H^*(\Sym^2 C,\Q)]$, $[H^*(C,\Q)]$, and $[\Q]$ with nonnegative coefficients.
To prove~\eqref{item:cor-odd-HS} and~\eqref{item:cor-even-HS}, we first show that the expressions match with the sums of the images of the decompositions \eqref{eq:odd-case-hodge-str-isom} and \eqref{eq:even-case-hodge-str-isom} in $K_0(\mathrm{HS}_\Q)$.

If $N=2g+1$ for some $g\geq 3$, 
then Theorem~\ref{thm:main-odd} and Lemma~\ref{lem:coeffs-effective-odd} give the following identity in $K_0(\mathrm{Var}/\C)$:
\begin{equation*}\begin{split}
[F_1(X)]&= \textstyle [\Sym^2 C]\L^{2g-4} + [C]\left(\sum_{i=g-2}^{2g-5}\L^i+\sum_{i=2g-2}^{3g-5}\L^i\right) \\ & \textstyle \qquad\qquad + \left(\sum_{-2g+3\leq i\leq 2g-3}\left(g-1-\left\lfloor\frac{|i|}{2}\right\rfloor\right)\L^{2g-3+i}\right)-\left(\sum_{i=g-2}^{3g-5}(\L^{i}+\L^{i+1})\right),
\end{split}\end{equation*}
which we aim to pass to $K_0(\mathrm{HS}_\Q)$. 
We have the following isomorphisms of $\Q$-Hodge structures
\begin{equation*}
\begin{split}
H^0(\Jac C,\Q)&\cong H^0(C,\Q)\cong H^2(C,\Q(1))\cong H^0(\Sym^2C,\Q)\cong H^4(\Sym^2C,\Q(2))\cong \Q,\\
H^1(\Jac C,\Q)&\cong H^1(C,\Q)\cong H^1(\Sym^2C,\Q)\cong H^3(\Sym^2C,\Q(1)), \\
H^2(\Jac C,\Q)&\cong \textstyle\bigwedge^2H^1(C,\Q),\\
H^2(\Sym^2C,\Q)&\cong H^2(C\times C,\Q)^{S_2}\cong \wedge^2H^1(C,\Q)\oplus \Q(-1)\cong H^2(\Jac C,\Q)\oplus \Q(-1),
\end{split}
\end{equation*}
where we have used Poincar\'e duality for $C$ and $\Sym^2C$.
We thus have the following identities in $K_0(\mathrm{HS}_\Q)$:
\begin{equation*}
\begin{split}
[H^*(C,\Q)]&= [\Q] + [H^1(\Jac C,\Q)] + [\Q(-1)],\\
[H^{*}(\Sym^2 C,\Q)]&=[\Q] + [H^1(\Jac C,\Q)] + ([H^2(\Jac C,\Q)]+[\Q(-1)]) + [H^1(\Jac C,\Q(-1))]+[\Q(-2)].
\end{split}
\end{equation*}
The contributions to $[H^*(F_1(X),\Q)]\in K_0(\mathrm{HS}_\Q)$ of twists of $H^2(\Jac C,\Q), H^1(\Jac C,\Q), H^0(\Jac C,\Q)\cong \Q$ are then respectively as follows:
{\footnotesize\begin{equation*}
\begin{split}
&[H^2(\Jac C,\Q(-2g+4))],\\
&([H^1(\Jac C,\Q(-2g+4))]+[H^1(\Jac C,\Q(-2g+3))]) + \left(\sum_{i=g-2}^{2g-5}[H^1(\Jac C,\Q(-i))]+ \sum_{i=2g-2}^{3g-5}[H^1(\Jac C,\Q(-i))]\right)\\
&\qquad=\sum_{i=g-2}^{3g-5}[H^1(\Jac C,\Q(-i))],\\
&\left([\Q(-2g+4)]+[\Q(-2g+3)]+[\Q(-2g+2)]\right)+ \left(\sum_{i=g-2}^{2g-5}([\Q(-i)]+[\Q(-i-1)])+\sum_{i=2g-2}^{3g-5}([\Q(-i)]+[\Q(-i-1)])\right)\\
&\qquad\phantom{=} +\left(\sum_{-2g+3\leq i\leq 2g-3}\left(g-1-\left\lfloor\frac{|i|}{2}\right\rfloor\right)[\Q(-2g+3-i)]\right)-\left(\sum_{i=g-2}^{3g-5}([\Q(-i)]+[\Q(-i-1)])\right)\\
&\qquad=\left(\sum_{-2g+3\leq i\leq 2g-3}\left(g-1-\left\lfloor\frac{|i|}{2}\right\rfloor\right)[\Q(-2g+3-i)]\right)-[\Q(-2g+3)].
\end{split}
\end{equation*}}\ignorespaces
It is now straightforward to check that $[H^{*}(F_1(X),\Q)]$ agrees with the sum of the images of the decompositions \eqref{eq:odd-case-hodge-str-isom} using Lemma~\ref{lem:odd-case-N-and-betti-hodge-numbers}.

If $N=2g$ for some $g\geq 3$, then the assertion follows from direct computation, by expanding the polynomial in Theorem~\ref{thm:main-even} and comparing coefficients. Indeed,
    \begin{equation*}
        \begin{split}
            &\phantom{=} 2g(g+2)\L^{2g-4} + (\L+1) \sum_{i=0}^{g-3}(i+1)(\L^{2i} + \L^{4g-9-2i})+(2g+1)(\L^{g-2}+\L^{2g-3})\sum_{i=0}^{g-3}\L^i \\
            &= 2g(g+2)\L^{2g-4} + \sum_{i=0}^{g-3}(i+1)\L^{2i} + \sum_{i=0}^{g-3}(i+1) \L^{4g-8-(2i+1)} + \sum_{i=0}^{g-3}(i+1)\L^{2i+1} + \sum_{i=0}^{g-3}(i+1)\L^{4g-8-2i} \\ & \qquad +(2g+1)\sum_{i=g-2}^{2g-5}\L^i + (2g+1)\sum_{i=2g-3}^{3g-6}\L^i \\
            &= 2g(g+2)\L^{2g-4} + \sum_{i=0}^{2 g - 5}\left(\left\lfloor\tfrac{i}{2}\right\rfloor+1\right)\L^i + \sum_{i=0}^{2g-5}\left(\left\lfloor\tfrac{i}{2}\right\rfloor+1\right) \L^{4g-8-i} +(2g+1)\sum_{i=g-2}^{2g-5}\L^i + (2g+1)\sum_{i=2g-3}^{3g-6}\L^i,
        \end{split}
\end{equation*}
so the coefficient of \(\L^i\) and \(\L^{4g-8-i}\) is \(\lfloor\frac{i}{2}\rfloor+1\) if \(0\leq i\leq g-3\), \(2g+2+\lfloor\frac{i}{2}\rfloor\) if \(g-2 \leq i\leq 2g-5\), and \(2g(g+2)\) if \(i=2g-4\). So by Lemma~\ref{lem:even-case-betti-number-CVX}, this polynomial is equal to \(\sum_{i=0}^{4g-8} b_{g,i}^{\text{even}} \L^i\). 
By passing to $K_0(\mathrm{HS}_\Q)$, we get the sum of the images of the decompositions \eqref{eq:even-case-hodge-str-isom}, as desired.

Since $\mathrm{HS}_\Q$ is semi-simple by \cite{deligne1973}, the identities in $K_0(\mathrm{HS}_{\mathbb Q})$ induce actual isomorphisms of $\Q$-Hodge structures at each weight.
Hence we get \eqref{eq:odd-case-hodge-str-isom} and \eqref{eq:even-case-hodge-str-isom}.
The assertions about the Betti and Hodge numbers of $F_1(X)$ are immediate from Lemmas~\ref{lem:odd-case-N-and-betti-hodge-numbers} and \ref{lem:even-case-betti-number-CVX}.
This proves \eqref{item:cor-odd-HS} and \eqref{item:cor-even-HS}.

For \eqref{item:cor-cohomology-tf}, we follow the idea of \cite{shinder-mathoverflow-answer}.
Let $\mathrm{HS}$ be the category of polarizable pure $\Z$-Hodge structures, let $\mathrm{FinAb}$ be the category of finite abelian groups, and let \(K_0(\mathrm{HS})\) and \(K_0(\mathrm{FinAb})\) be the Grothendieck groups with respect to the direct sum operation.
Consider a composition of group homomorphisms
\[
K_0(\mathrm{Var}/\C)\to K_0(\mathrm{HS})\to K_0(\mathrm{FinAb}),
\]
where the first map sends $[Z]$ to $[H^*(Z,\Z)]$ whenever $Z$ is smooth, connected, and projective, and the second map sends the class $[H]$ of an integral Hodge structure to the class $[H_{\operatorname{tors}}]$ of its torsion subgroup (forgetting the weights).
Theorems~\ref{thm:main-odd} and \ref{thm:main-even} then express $[H^*(F_1(X),\Z)_{\operatorname{tors}}]$ in $K_0(\mathrm{FinAb})$ as a $\Z$-linear combination of $[H^*(\Sym^2 C,\Z)_{\operatorname{tors}}]$, $[H^*(C,\Z)_{\operatorname{tors}}]$, and $[\Z_{\operatorname{tors}}]$,
and since the groups $H^*(C,\Z), H^*(\Sym^2 C,\Z), \Z$ are all torsion-free, we get $[H^*(F_1(X),\Z)_{\operatorname{tors}}]=0$ in $K_0(\mathrm{FinAb})$.
Since every finite abelian group uniquely decomposes into indecomposables,
it follows that $H^*(F_1(X),\Z)$ is torsion-free, as desired.
\end{proof}

\begin{cor}\label{cor:finite}
Over a finite field $\F_q$ of characteristic \(\neq 2\), let $N\geq 6$, and let $X\subset \P^N$ be a smooth complete intersection of two quadrics. Then there exists $e_0\geq 1$ such that for every $e\geq 1$ that is divisible by $e_0$, the following holds:
\begin{enumerate}
\item\label{item:cor-finite-odd} If $N=2g+1$ for some $g\geq 3$, then
the number of lines on $X$ defined over $\F_{q^e}$ equals 
{\footnotesize
\begin{equation*}
    \begin{split}
        & \frac{|C(\F_{q^e})|^2+|C(\F_{q^{2e}})|}{2} q^{e(2g-4)} + |C(\F_{q^e})| q^{e(g-2)} \left( \sum_{i=0}^{2g-3} q^{ei} - q^{e(g-2)} - q^{e(g-1)} \right) \\
        & \phantom{=}  + \left(\sum_{i=0}^{2g-2} q^{ei}\right) \left(\sum_{i=0}^{g-2} q^{2ei}\right)- (q^{e(g-2)} + q^{e(g+2)}) \sum_{i=0}^{2g-6} q^{ei} - (q^{e(g-1)} + q^{eg} + q^{e(g+1)} + q^{e(3g-7)} + q^{e(3g-6)} + q^{e(3g-5)}) .
    \end{split}
\end{equation*}
}
\item\label{item:cor-finite-even} If $N=2g$ for some $g\geq 3$, then
the number of lines on $X$ defined over $\F_{q^e}$ equals
\begin{equation*}
\begin{split}
&2g(g+2)q^{e(2g-4)} + (q^e+1)\sum_{i=0}^{g-3}(i+1)(q^{2ei}+q^{e(4g-9-2i)})+(2g+1)(q^{e(g-2)}+q^{e(2g-3)})\sum_{i=0}^{g-3}q^{ei}.
\end{split}
\end{equation*}
\end{enumerate}
\end{cor}
\begin{proof}
By Theorems~\ref{thm:main-odd} and~\ref{thm:main-even}, there exists \(e_0\) such that for every \(e\) divisible by \(e_0\), the class of \(F_1(X_{\mathbb F_{q^e}})\) in \(\widetilde{K}_0(\mathrm{Var}/\F_{q^e})\) is equal to the right-hand expression in the respective theorem statement.
By \cite[Proposition 7.26]{mustata-notes}, the ring homomorphism $K_0(\mathrm{Var}/\F_{q^e})\to \Z, \, [Z]\mapsto |Z(\F_{q^e})|$
factors through $\widetilde{K}_0(\mathrm{Var}/\F_{q^e})$.
If \(N=2g\) is even, then~\eqref{item:cor-finite-even} follows from the equality \(|\mathbb A^1(\mathbb F_{q^e})|=q^e\) and the expression in Theorem~\ref{thm:main-even}. If \(N=2g+1\) is odd, then~\eqref{item:cor-finite-odd} follows from the expression in Theorem~\ref{thm:main-odd} and the equality \(|(\Sym^2 C)(\mathbb F_{q^e})| = \frac{|C(\F_{q^e})|^2+|C(\F_{q^{2e}})|}{2}\).
\end{proof}

\begin{cor}[{\cite[Example 2.8, Lemma 3.17, and Lemma 4.7]{BBFGHLPRAS}}]\label{cor:K_0-cat}
    Over an algebraically closed field \(k\) of characteristic zero, let \(N\geq 6\), and let \(X\subset\P^N\) be a smooth complete intersection of two quadrics. Then the class of the derived category of \(F_1(X)\) in the Grothendieck ring of categories \(K_0(\mathrm{Cat}/k)\) (see \cite[Definition 2.7]{BBFGHLPRAS}) is equal to
    \[[D^b(F_1(X))] = \begin{cases}
        [D^b(\Sym^2 C)] + (2g-4)[D^b(C)] + (2 g-5)(g-1)\cdot 1 & \text{if \(N=2g+1\) is odd,} \\
        8g(g-1) \cdot 1 & \text{if \(N=2g\) is even,}
    \end{cases}\]
    where \(1 = [D^b(\Spec k)]\) and, if \(N=2g+1\) is odd, \(C\) is the genus \(g\) hyperelliptic curve associated to the pencil.
    In particular, this is the image in \(K_0(\mathrm{Cat}/k)\) of the semiorthogonal decomposition for \(D^b(F_1(X))\) conjectured in \cite[Conjectures A and D]{BBFGHLPRAS}.
\end{cor}

\begin{proof}
    By \cite{BondalLarsenLunts}, \([Y] \mapsto [D^b(Y)]\) for smooth projective varieties \(Y\) over \(k\) defines a ring homomorphism \(K_0(\mathrm{Var}/k) \to K_0(\mathrm{Cat}/k)\), and the image of \(\L\) under this homomorphism is \(1\). Since \(\Sym^2 C\), \(C\), and \(F_1(X)\) are smooth and projective (Lemma~\ref{lem:F_r(X)}), the above identities hold by Theorems~\ref{thm:main-odd} and~\ref{thm:main-even}.
    If \(N=2g+1\) is odd, the above class is the image in \(K_0(\mathrm{Cat}/k)\) of \cite[Conjecture A]{BBFGHLPRAS} for \(F_1(X)\).

    If \(N=2g+1\) is odd, then \cite[Conjecture A]{BBFGHLPRAS} predicts a semiorthogonal decomposition of \(D^b(F_1(X))\) consisting of one copy of \(D^b(\Sym^2 C)\), \(2g-4\) copies of \(D^b(C)\), and \((2g-5)(g-1)\) copies of \(D^b(\Spec k)\); the image of this conjectured semiorthogonal decomposition in \(K_0(\mathrm{Cat}/k)\) is the first equation in the corollary statement.
    If \(N=2g\) is even, let \(\widehat{C}\) be the associated stacky curve (Section~\ref{sec:prelim-F_r}).
    Then \cite[Conjecture D]{BBFGHLPRAS} predicts a semiorthogonal decomposition of \(D^b(F_1(X))\) consisting of one copy of \(D^b(\widetilde{\Sym}^2\widehat{C})\), \(2g-5\) copies of \(D^b(\widehat{C})\), and \((2g-5)(g-2)\) copies of \(D^b(\Spec k)\); here \(\widetilde{\Sym}^2\widehat{C}\) is defined in \cite[Section 4.3]{BBFGHLPRAS}. By \cite[Lemma 4.7]{BBFGHLPRAS}, we have
    \([D^b(\widetilde{\Sym}^2\widehat{C})]=(2g^2+5g+5) \cdot 1\) and \([D^b(\widehat{C})]=(2g+3)\cdot 1\) in \(K_0(\mathrm{Cat}/k)\), so
    \[[D^b(\widetilde{\Sym}^2\widehat{C})] + (2g-5)[D^b(\widehat{C})] + (2g-5)(g-2)[D^b(\Spec k)] = 8g(g-1)\cdot 1.\]
\end{proof}

\section{\texorpdfstring{\(F_r(X)\)}{FrX} for arbitrary \texorpdfstring{\(r\)}{r}}\label{sec:artibrary-r}

In this section, we describe a possible strategy to prove the formula in \cite[Conjecture B]{BBFGHLPRAS} and its even analogue for arbitrary \(0\leq r\leq\lfloor\frac{N}{2}\rfloor-2\), which generalizes the approach for \(r=0,1\) that we used in this paper.

The following definition specializes to Definition~\ref{defn:reid-type-r=1} in the case \(r=1\).
\begin{defn}\label{defn:reid-type}
    Over a field \(k\) of characteristic \(\neq 2\), a subscheme $X'\subset \P^{2r+1}$ defined by two quadratic equations is said to be \defi{singular of Reid type} if over an algebraic closure \(\kbar\) of \(k\) there is a choice of coordinates so that \(X'_{\kbar}\) is defined by symmetric matrices of the form \[\begin{pmatrix} 0 & I_{r+1} \\ I_{r+1} & 0 \end{pmatrix}, \quad \begin{pmatrix} 0 & D \\ D & 0 \end{pmatrix}\] where \(D\) is diagonal with distinct diagonal entries.
    By \cite[Lemma 3.4]{Reid-thesis}, if \(X'\) is singular of Reid type, then it contains \(r\)-planes over \(\kbar\), and for each \(r\)-plane \(\rplane\) on \(X'_{\kbar}\) there are exactly \(r+1\) distinct \(r\)-planes \(M_1,\ldots,M_{r+1}\) on \(X'_{\kbar}\) that intersect \(\rplane\) along an \((r-1)\)-plane.
\end{defn}

Using the interpretation of points on \(\mathcal Q^{(r)}_{\rplane}\setminus\mathcal Q_{Y_{\rplane}}\) from Lemma~\ref{lem:r-planes-in-hyperbolic-reduction-loci}\eqref{item:image-of-Z-in-rel-Fr}, the following birational description of \(F_r(X)\) was first proven by Reid for \(N=2g+1\) and \(r=g-1\) \cite{Reid-thesis}, \cite{CTSSD} for \(r=0\), and later proven by \cite{JS24} in general:

\begin{thm}[{\cite[Theorem 1.3 and Section 3.2]{JS24}}]\label{thm:js-bir-description-F_r(X)}
    Let \(X\) be a smooth complete intersection of two quadrics in \(\P^N\) over a field \(k\) of characteristic \(\neq 2\), and let \(0\leq r<\frac{N}{2}-1\). Let \(\rplane\) be an \(r\)-plane on \(X\) defined over \(k\).
    Then there is an open subset of \(F_r(X)\) defined over \(k\) parametrizing the \(r\)-planes \(M\) such that \(X \cap\langle\rplane,M\rangle\) is singular of Reid type, and \(M \mapsto (M_1,\ldots,M_{r+1})\) (where the \(M_i\) are as in Definition~\ref{defn:reid-type}) defines a birational map over \(k\)
    \[\psi_{\rplane}\colon F_r(X)\dashrightarrow \Sym^{r+1}(\mathcal Q^{(r)}_{\rplane}\setminus\mathcal Q_{Y_{\rplane}}).\]
\end{thm}

For \(r=0,1\), we explicitly described the exceptional loci of \(\psi_\Lambda\) in Sections~\ref{sec:r=0}, \ref{sec:computations}, and~\ref{sec:r=1} to obtain formulas for the classes of \(X\) and \(F_1(X)\) in \(\widetilde{K}_0(\mathrm{Var}/k)\).
For \(r\geq 2\), the same strategy should work, but the computations for the classes of the exceptional loci of \(\psi_\Lambda\) may be more complicated.
For instance, even the analogue of Section~\ref{sec:computation-exceptional-cases} for \(r\geq 2\) involves more cases.

As evidence for this possible strategy to compute the class of \(F_r(X)\), we note that
if \(N=2g+1\) is odd, then by Corollary~\ref{cor:class-of-Qr-formula} and Lemma~\ref{lem:K_0-sym-formulas}\eqref{item:formula-sym-times-L} the coefficient of \([\Sym^{r+1}C]\) in \([\Sym^{r+1} \mathcal Q^{(r)}_\Lambda]\) is \(\L^{(r+1)(g-r-1)}\). This agrees with the coefficient of \([\Sym^{r+1} C]\) in \cite[Conjecture B]{BBFGHLPRAS}.

\appendix
\section{An effective expression for the decomposition of Fano schemes of intersections of two quadrics, by Pieter Belmans}\label{appendix}
  The polynomials~$\mathrm{M}_{g,k,i}(t)\in\mathbb Z[t]$ introduced in \cite{BBFGHLPRAS}
  encode the multiplicities of Lefschetz twists of symmetric powers of a hyperelliptic curve
  in a conjectural motivic decomposition of Fano schemes of intersections of two quadrics.
  We show that these polynomials are effective, i.e.,~$\mathrm{M}_{g,k,i}(t)\in\mathbb N[t]$,
  settling the effectivity statement of the refined decomposition conjecture in loc.\ cit.
  We moreover give a manifestly effective closed expression for~$\mathrm{M}_{g,k,i}(t)$.

\subsection{Introduction}
\label{section:introduction}
Consider the polynomial in~$\mathbb Z[t]$ introduced in \cite[Equation~(4)]{BBFGHLPRAS},
given by
\begin{equation}
  \label{equation:M}
  \begin{aligned}
    \mathrm{M}_{g,k,i}(t)
    &\coloneqq
    t^{i(g-k-1)}
    \Biggl(
      \binom{2g-k-i}{k+1-i}_t
      - \left( t^{g-k-1} + t^{g+2k-3i} \right)\binom{2g-k-i-4}{k-i}_t \\
      &\qquad
      - \left( t^{g-k} + t^{g-i} + t^{g+k-2i} + t^{3g-3k-4} + t^{3g-2k-4-i} + t^{3g-k-2i-4} \right)\binom{2g-k-i-4}{k-i-1}_t \\
      &\qquad
      - \left( t^{3(g-k-1)} + t^{3(g-k-1)+1} + t^{3g-2k-i-3} + t^{3g-2k-i-2} \right)\binom{2g-k-i-4}{k-i-2}_t \\
      &\qquad
      - t^{4(g-k)-2}\binom{2g-k-i-4}{k-i-3}_t
    \Biggr),
  \end{aligned}
\end{equation}
where~$g\geq 2$,~$k=0,\ldots,g-2$ and~$i=0,\ldots,k+1$.
Here, we write
\begin{equation}
  \label{equation:gauss}
  \binom{n}{m}_t
  \coloneqq
  \frac{(1-t^n)(1-t^{n-1})\cdots(1-t^{n-m+1})}
  {(1-t)(1-t^2)\cdots(1-t^m)}
\end{equation}
for the Gaussian binomial coefficient in the variable~$t$,
with integers~$0\le m\le n$,
and the convention that~$\binom{n}{m}_t=0$ for~$m<0$ or~$m>n$.
%It is the generating function of the partitions whose Young diagram
%fits in an~$m\times(n-m)$ box, see \cite[Theorem~3.1]{MR1634067}.

Although the four subtractions in \eqref{equation:M} might suggest otherwise,
the polynomial~$\mathrm{M}_{g,k,i}(t)$ was expected to have non-negative coefficients,
as expressed by the first half of \cite[Conjecture~B]{BBFGHLPRAS}.
The following result confirms this expectation.

\begin{thm}
  \label{theorem:appendix-main}
  For all~$g\geq 2$,~$k=0,\ldots,g-2$,
  and~$i=0,\ldots,k+1$
  the polynomial~$\mathrm{M}_{g,k,i}(t)$ defined in \eqref{equation:M} is effective,
  i.e.,~$\mathrm{M}_{g,k,i}(t)\in\mathbb N[t]$.

  Moreover, with the convention that an empty product equals~$1$,
  it is given by the manifestly effective expression
  \begin{equation}
    \label{equation:main}
    \mathrm{M}_{g,k,i}(t)
    =
    t^{i(g-k-1)}\,[x^{k+1-i}]
    \Biggl(
      \prod_{j=0}^{g-k-3}\frac{1}{1-t^jx}\prod_{j=g-k+1}^{2g-2k-1}\frac{1}{1-t^jx}
      +t^{g-k-2}x\prod_{j=0}^{g-k-2}\frac{1}{1-t^jx}\prod_{j=g-k+2}^{2g-2k-1}\frac{1}{1-t^jx}
    \Biggr),
  \end{equation}
  where~$[x^r]$ extracts the coefficient of~$x^r$ from the bracketed power series.
\end{thm}

Here each factor~$\tfrac{1}{1-t^jx}=\sum_{r\ge0}t^{jr}x^r$
has non-negative coefficients,
hence,
the right-hand side of \eqref{equation:main} is a non-negative combination of monomials in~$t$.

\subsection*{Geometric motivation}
Let us briefly describe the geometric context for \cite[Conjecture~B]{BBFGHLPRAS}.
For this geometric motivation, work over an algebraically closed field of characteristic zero.
The variable~$t$ in \eqref{equation:main}
corresponds to the Lefschetz motive~$\mathbb{L}=[\mathbb{A}^1]$,
the class of the affine line in~$K_0(\mathrm{Var})$.
For the smooth intersection of two quadrics~$Q_1\cap Q_2\subset\mathbb{P}^{2g+1}$,
with associated hyperelliptic curve~$C$,
\cite[Conjecture~B]{BBFGHLPRAS} predicts the identity
\begin{equation}
  \label{equation:conjecture}
  [\mathrm{Fano}_k(Q_1\cap Q_2)]
  =
  \sum_{i=0}^{k+1}
  \mathrm{M}_{g,k,i}(\mathbb{L})\,[\Sym^iC]
  \qquad\text{in }K_0(\mathrm{Var}),
\end{equation}
where~$\mathrm{Fano}_k(Q_1\cap Q_2)$ denotes the Fano scheme of~$k$-planes,
and~$\Sym^iC$ the~$i$th symmetric power.
In other words,~$\mathrm{M}_{g,k,i}(\mathbb{L})$ plays the role of a motivic multiplicity,
recording the multiplicity of each Lefschetz twist of~$[\Sym^iC]$ in \eqref{equation:conjecture}.
In \cite{BBFGHLPRAS}, the effectivity was established only after setting~$\mathbb{L}=1$,
where~$\mathrm{M}_{g,k,i}(1)$ is an explicit positive integer.
In Corollary~\ref{corollary:euler} we give a proof of this fact
using our approach to Theorem~\ref{theorem:appendix-main}.

% to be used in the main body of the paper?
%The case~$k=1$ of \cite[Conjecture~B]{BBFGHLPRAS}
%gives
%\begin{equation}
%  \begin{aligned}
%    \mathrm{M}_{g,1,0}(t)
%    &=(1+t^{2g})\binom{g-1}{2}_t+t^g\binom{g-3}{1}_t\binom{g-1}{1}_t,\\
%    \mathrm{M}_{g,1,1}(t)
%    &=t^{g-2}(1+t^g)\binom{g-2}{1}_t,\\
%    \mathrm{M}_{g,1,2}(t)
%    &=t^{2g-4}.
%  \end{aligned}
%\end{equation}

\subsection*{AI disclosure}
The effectivity part of \cite[Conjecture~B]{BBFGHLPRAS}
was submitted as a benchmark problem to IMProofBench~\cite{2509.26076v2}.
The initial proof of Theorem~\ref{theorem:appendix-main} was obtained by GPT-5.4
in March~2026.
In order to turn the one-shot proof from IMProofBench
into something digestible,
this write-up was produced by the author,
with LLMs used for proofreading.

\subsection*{Acknowledgements}
P.B.~was partially supported by NWO (\href{https://doi.org/10.61686/RZKLF82806}{\texttt{doi:10.61686/RZKLF82806}}).
We thank the creators of IMProofBench
for having built this benchmark system
that led to the proof of Theorem~\ref{theorem:appendix-main}.

\subsection{Proof}
\label{section:proof}
We will prove Theorem~\ref{theorem:appendix-main} in four steps:
\begin{enumerate}
  \item First we perform an easy change of variables in Lemma~\ref{lemma:reindexing},
    and encode the coefficients we are interested in
    using a generating function defined in \eqref{equation:Fa}.
  \item Next,
    we recall a standard generating function for certain Gaussian binomial coefficients
    using Pochhammer symbols in Proposition~\ref{proposition:pochhammer},
    which is the main external input for the proof.
    In Lemma~\ref{lemma:rational} we use this to give a new expression for \eqref{equation:Fa}.
  \item In Lemma~\ref{lemma:collapse} and Corollary~\ref{corollary:closed}
    we will further reduce this expression by elementary means.
  \item Finally, in Lemma~\ref{lemma:decomposition}
    we express the generating function in Corollary~\ref{corollary:closed}
    as a sum of monomial multiples of reciprocals of products of Pochhammer symbols,
    making its effectivity explicit.
\end{enumerate}

%\begin{enumerate}
%  \item We split off the monomial prefactor~$t^{i(g-k-1)}$,
%    and reindex by~$a=g-k-1$,~$m=k-i$ (Lemma~\ref{lemma:reindexing});
%    effectivity then amounts to that
%    of a single two-parameter family~$\mathrm{B}_{a,m}(t)$.
%    This is merely book-keeping.
%  \item We recall the generating function of the Gaussian binomial coefficients
%    in Pochhammer form (Proposition~\ref{proposition:pochhammer}); we isolate it as a
%    proposition because it is the conceptual crux of the argument, not because it
%    is deep.
%  \item We package the~$\mathrm{B}_{a,m}(t)$ into a generating function~$\mathrm{F}_a(t,x)$, compute it in closed form, and observe a collapse
%    (Lemma~\ref{lemma:rational}, Lemma~\ref{lemma:collapse}, and Corollary~\ref{corollary:closed}).
%  \item An elementary splitting (Lemma~\ref{lemma:decomposition}) writes~$\mathrm{F}_a(t,x)$ as a sum of two manifestly effective series, from which
%    Theorem~\ref{theorem:appendix-main} is read off.
%\end{enumerate}

\subsection*{Reindexing}
First, a change of variables removes the dependence on~$i$ and turns the five-term sum
into a single two-parameter family.

\begin{lem}
  \label{lemma:reindexing}
  Set~$a\coloneqq g-k-1$
  (and thus~$a\geq 1$)
  and~$m\coloneqq k-i$
  (and thus~$m\in\{-1,0,\ldots,k\}$).
  Then~$\mathrm{M}_{g,k,i}(t)=t^{ia}\,\mathrm{B}_{a,m}(t)$, where
  \begin{equation}
    \label{equation:B}
    \begin{aligned}
      \mathrm{B}_{a,m}(t)
      \coloneqq{}&\binom{2a+m+2}{m+1}_t
      -\bigl(t^a+t^{a+1+3m}\bigr)\binom{2a+m-2}{m}_t \\
      &-\bigl(t^{a+1}+t^{a+1+m}+t^{a+1+2m}+t^{3a-1}+t^{3a-1+m}+t^{3a-1+2m}\bigr)\binom{2a+m-2}{m-1}_t \\
      &-\bigl(t^{3a}+t^{3a+1}+t^{3a+m}+t^{3a+m+1}\bigr)\binom{2a+m-2}{m-2}_t \\
      &-t^{4a+2}\binom{2a+m-2}{m-3}_t.
    \end{aligned}
  \end{equation}
\end{lem}

\begin{proof}
  Substitute~$g=a+k+1$ and~$i=k-m$ into \eqref{equation:M}.
  We get
  \begin{equation}
    2g-k-i=2a+m+2,\quad k+1-i=m+1,\quad 2g-k-i-4=2a+m-2,\quad k-i-j=m-j,
  \end{equation}
  and every monomial exponent becomes an affine function of~$a$ and~$m$.
  Reading off the five terms of \eqref{equation:M}
  gives~\eqref{equation:B}.
\end{proof}

For a fixed~$a\geq 1$,
and allowing~$k$ to vary so that we consider all integers~$m\geq -1$,
we assemble the polynomials~$\mathrm{B}_{a,m}(t)$ of Lemma~\ref{lemma:reindexing}
into the following generating function in a new variable~$x$,
\begin{equation}
  \label{equation:Fa}
  \mathrm{F}_a(t,x)\coloneqq\sum_{m\geq-1}\mathrm{B}_{a,m}(t)\,x^{m+1}\ \in\ \mathbb Z[t][\![x]\!].
\end{equation}
All admissible polynomials~$\mathrm{M}_{g,k,i}(t)$ are effective
if and only if all polynomials~$\mathrm{B}_{a,m}(t)$ with~$a\geq1$ and~$m\geq-1$ are effective,
equivalently if~$\mathrm{F}_a(t,x)\in\mathbb N[t][\![x]\!]$ for every~$a\geq1$.
Our goal is to give an explicit expression for~$\mathrm{F}_a(t,x)$,
which we will accomplish in Corollary~\ref{corollary:closed}.

\subsection*{Packaging Gaussian binomial coefficients in a generating function}
For an integer~$n\ge0$ write the \defi{$t$-Pochhammer symbol}
\begin{equation}
  \label{equation:pochhammer}
  (y;t)_n\coloneqq\prod_{r=0}^{n-1}\bigl(1-t^r y\bigr)
  =(1-y)(1-t y)\cdots(1-t^{n-1}y),
  \qquad (y;t)_0=1.
\end{equation}

Its inverse gives the following generating function of the Gaussian binomial coefficients
along a fixed difference of indices, see \cite[Theorem~3.3, Equation~(3.3.7)]{MR1634067}.
\begin{prop}
  \label{proposition:pochhammer}
  For every~$n\geq0$, as an identity of formal power series in~$y$,
  we have that
  \begin{equation}
    \label{equation:pochhammer-series}
    \sum_{r\geq0}\binom{n+r}{r}_ty^r=\frac{1}{(y;t)_{n+1}}.
  \end{equation}
\end{prop}
This is the tool that allows us to convert \eqref{equation:B}
from a difference of five terms using Gaussian binomials
into a rational generating function.
Note that~\eqref{equation:pochhammer-series} follows by summing
the~$t$-Pascal recursion~$\binom{n}{m}_t=\binom{n-1}{m}_t+t^{n-m}\binom{n-1}{m-1}_t$.
We apply this generating-function identity in Lemma~\ref{lemma:rational};
the subsequent proof of Lemma~\ref{lemma:collapse} consists of clearing denominators and collecting coefficients.

We use this to rewrite \eqref{equation:Fa} as follows.
\begin{lem}
  \label{lemma:rational}
  For every~$a\geq1$,
  \begin{equation}
    \label{equation:rational}
    \begin{aligned}
      \mathrm{F}_a(t,x)={}&\frac{1}{(x;t)_{2a+2}}
      -\frac{t^a x}{(x;t)_{2a-1}}
      -\frac{t^{a+1}x}{(t^3x;t)_{2a-1}} \\
      &-x^2\left(
        \frac{t^{a+1}+t^{3a-1}}{(x;t)_{2a}}
        +\frac{t^{a+2}+t^{3a}}{(t x;t)_{2a}}
        +\frac{t^{a+3}+t^{3a+1}}{(t^2x;t)_{2a}}
      \right) \\
      &-x^3\left(
        \frac{t^{3a}+t^{3a+1}}{(x;t)_{2a+1}}
        +\frac{t^{3a+2}+t^{3a+3}}{(t x;t)_{2a+1}}
      \right) \\
      &-\frac{t^{4a+2}x^4}{(x;t)_{2a+2}}.
    \end{aligned}
  \end{equation}
\end{lem}

\begin{proof}
  Multiply \eqref{equation:B} by~$x^{m+1}$ and sum over~$m\geq-1$.
  By Proposition~\ref{proposition:pochhammer}, the leading term gives~$1/(x;t)_{2a+2}$.
  For the remaining terms, the zero convention means that a summand with lower index~$m-j$ contributes only for~$m\geq j$.
  For~$j=0,1,2,3$ and non-negative integers~$c,d$,
  setting~$r=m-j$ and applying Proposition~\ref{proposition:pochhammer} gives
  \begin{equation}
    \begin{aligned}
      \sum_{m\geq j}t^{c+dm}\binom{2a+m-2}{m-j}_tx^{m+1}
      &=t^{c+dj}x^{j+1}\sum_{r\geq0}\binom{2a+j+r-2}{r}_t(t^dx)^r\\
      &=\frac{t^{c+dj}x^{j+1}}{(t^dx;t)_{2a+j-1}}.
    \end{aligned}
  \end{equation}
  For example, the two parts of the second term of \eqref{equation:B} contribute~$-t^a x/(x;t)_{2a-1}$ and~$-t^{a+1}x/(t^3x;t)_{2a-1}$.
\end{proof}

This allows us to prove that the following product of a Pochhammer symbol
with the generating function~$\mathrm{F}_a(t,x)$ from \eqref{equation:Fa}
is an explicit polynomial.

\begin{lem}
  \label{lemma:collapse}
  For every~$a\geq1$
  we have that
  \begin{equation}
    \label{equation:collapse}
    (x;t)_{2a+2}\,\mathrm{F}_a(t,x)
    =\bigl(1-t^a x\bigr)\bigl(1-t^{a+1}x\bigr)\bigl(1-t^{2a+1}x^2\bigr).
  \end{equation}
\end{lem}

\begin{proof}
  Multiply \eqref{equation:rational} by~$(x;t)_{2a+2}=\prod_{r=0}^{2a+1}(1-t^rx)$.
  Every denominator divides this product,
  so clearing all of them leaves a polynomial of degree~$\le4$ in~$x$.
  We will write~$(x;t)_{2a+2}\,\mathrm{F}_a(t,x)=T_0(t,x)+\cdots+T_5(t,x)$,
  where
  \begin{equation}
    \begin{aligned}
      T_0(t,x)&=1,\\
      T_1(t,x)&=-t^a x(1-t^{2a-1}x)(1-t^{2a}x)(1-t^{2a+1}x),\\
      T_2(t,x)&=-t^{a+1}x(1-x)(1-tx)(1-t^2x),\\
      T_3(t,x)&=-x^2\bigl[(t^{a+1}+t^{3a-1})(1-t^{2a}x)(1-t^{2a+1}x)\\
      &\qquad\quad{}+(t^{a+2}+t^{3a})(1-x)(1-t^{2a+1}x)+(t^{a+3}+t^{3a+1})(1-x)(1-tx)\bigr],\\
      T_4(t,x)&=-x^3\bigl[(t^{3a}+t^{3a+1})(1-t^{2a+1}x)+(t^{3a+2}+t^{3a+3})(1-x)\bigr],\\
      T_5(t,x)&=-t^{4a+2}x^4.
    \end{aligned}
  \end{equation}
  After expanding the products, the coefficient of each power of~$x$ is a sum of signed monomials~$t^{sa+c}$.
  We group these contributions by their \defi{slope}~$s$,
  the coefficient of~$a$ in the exponent.

  The coefficients of~$x^0$ and~$x^1$ are immediate,
  and equal~$1$ and~$-(t^a+t^{a+1})$, respectively.
  For~$x^2$, the contribution~$(t^{a+1}+t^{3a-1})(1+t+t^2)$ from~$T_1+T_2$ cancels that from~$T_3$.

  The coefficients of~$x^3$ and~$x^4$ are more involved,
  as monomials of several slopes contribute and partly cancel.
  For~$x^3$, the slopes~$1$,~$3$, and~$5$ occur before cancellation.
  The contributions of slopes~$1$ and~$5$ cancel, leaving only that of slope~$3$:
  \begin{equation}
    \begin{aligned}
      s=1:&\quad -(t^{a+2}+t^{a+3}+t^{a+4})+t^{a+2}+(t^{a+3}+t^{a+4})=0,\\
      s=5:&\quad -(t^{5a-1}+t^{5a}+t^{5a+1})+(t^{5a-1}+t^{5a})+t^{5a+1}=0,\\
      s=3:&\quad (t^{3a+1}+t^{3a+2})+(t^{3a}+t^{3a+3})+(t^{3a+1}+t^{3a+2})\\
      &\qquad{}-(t^{3a}+t^{3a+1})-(t^{3a+2}+t^{3a+3})=t^{3a+1}+t^{3a+2},
    \end{aligned}
  \end{equation}
  while for~$x^4$ four slopes cancel and only~$s=4$ survives:
  \begin{equation}
    \begin{aligned}
      s=1:&\quad t^{a+4}-t^{a+4}=0,\\
      s=3:&\quad -(t^{3a+2}+t^{3a+3})+(t^{3a+2}+t^{3a+3})=0,\\
      s=5:&\quad -(t^{5a+1}+t^{5a+2})+(t^{5a+1}+t^{5a+2})=0,\\
      s=7:&\quad t^{7a}-t^{7a}=0,\\
      s=4:&\quad -t^{4a+2}.
    \end{aligned}
  \end{equation}
  These five coefficients are exactly those of~$(1-t^a x)(1-t^{a+1}x)(1-t^{2a+1}x^2)$,
  which proves \eqref{equation:collapse}.
\end{proof}

\begin{cor}
  \label{corollary:closed}
  For every~$a\geq1$,
  \begin{equation}
    \label{equation:closed}
    \mathrm{F}_a(t,x)=\frac{1-t^{2a+1}x^2}{(x;t)_a\,(t^{a+2}x;t)_a}.
  \end{equation}
\end{cor}

\begin{proof}
  In \eqref{equation:collapse} the factors~$1-t^a x$ and~$1-t^{a+1}x$ are the~$r=a$
  and~$r=a+1$ factors of~$(x;t)_{2a+2}$; cancelling them leaves~$1-t^{2a+1}x^2$
  over~$\prod_{r=0}^{a-1}(1-t^rx)\cdot\prod_{r=a+2}^{2a+1}(1-t^rx)=(x;t)_a(t^{a+2}x;t)_a$.
\end{proof}

\subsection*{Finishing the proof}
One elementary splitting of the numerator clears the last minus sign and exhibits~$\mathrm{F}_a(t,x)$ as a sum of two effective series.

\begin{lem}
  \label{lemma:decomposition}
  For every~$a\geq1$,
  \begin{equation}
    \label{equation:decomposition}
    \mathrm{F}_a(t,x)
    =\frac{1}{(x;t)_{a-1}\,(t^{a+2}x;t)_a}
    +\frac{t^{a-1}x}{(x;t)_a\,(t^{a+3}x;t)_{a-1}}
    \ \in\ \mathbb N[t][\![x]\!].
  \end{equation}
  In particular~$\mathrm{B}_{a,m}(t)\in\mathbb N[t]$ for all~$a\geq1$,~$m\geq-1$.
\end{lem}

\begin{proof}
  Set~$u=t^{a-1}x$ and~$v=t^{a+2}x$.
  The identity~$1-uv=(1-u)+u(1-v)$ splits the numerator of \eqref{equation:closed}
  into terms divisible by factors already present in the denominator:
  the factor~$1-u$ is the last factor of~$(x;t)_a$ and~$1-v$ is the
  first factor of~$(t^{a+2}x;t)_a$, so cancelling them yields
  \eqref{equation:decomposition}.
  Each remaining factor~$1/(1-t^jx)=\sum_{r\ge0}t^{jr}x^r$ has non-negative coefficients.
  Hence~$\mathrm{B}_{a,m}(t)=[x^{m+1}]\mathrm{F}_a(t,x)\in\mathbb N[t]$.
\end{proof}

We can now prove the main result of this appendix.

\begin{proof}[Proof of Theorem~\ref{theorem:appendix-main}]
  By Lemma~\ref{lemma:decomposition},~$\mathrm{B}_{a,m}(t)\in\mathbb N[t]$ for all~$a\geq1$ and~$m\geq-1$. Hence~$\mathrm{M}_{g,k,i}(t)=t^{ia}\mathrm{B}_{a,m}(t)$ lies in~$\mathbb N[t]$ by
  Lemma~\ref{lemma:reindexing}, which is the asserted effectivity.

  For the explicit formula, expand each factor in \eqref{equation:decomposition}
  as a geometric series~$1/(1-t^jx)=\sum_{r\ge0}t^{jr}x^r$ and extract the
  coefficient of~$x^{k+1-i}$; restoring the prefactor~$t^{i(g-k-1)}$ then gives
  \eqref{equation:main}.
\end{proof}

%The coefficient extraction in \eqref{equation:main} has a clean
%symmetric-function interpretation. Recall that the complete homogeneous symmetric
%polynomial~$\mathrm{h}_r(u_1,\ldots,u_\ell)$ is the sum of all monomials of degree~$r$ in~$u_1,\ldots,u_\ell$ (repetitions allowed), equivalently the coefficient~$[x^r]\prod_{p=1}^{\ell}(1-u_px)^{-1}$.
%
%\begin{cor}
%  \label{corollary:homogeneous}
%  With~$\mathrm{h}_r$ as above,
%  \begin{equation}
%    \label{equation:homogeneous}
%    \begin{aligned}
%      \mathrm{M}_{g,k,i}(t)
%      =t^{i(g-k-1)}\Bigl(
%      &\mathrm{h}_{k+1-i}\bigl(1,t,\ldots,t^{g-k-3},t^{g-k+1},\ldots,t^{2g-2k-1}\bigr)\\
%      &{}+t^{g-k-2}\mathrm{h}_{k-i}\bigl(1,t,\ldots,t^{g-k-2},t^{g-k+2},\ldots,t^{2g-2k-1}\bigr)
%      \Bigr).
%    \end{aligned}
%  \end{equation}
%\end{cor}
%
%\begin{proof}
%  By the description of~$\mathrm{h}_r$ above,~$[x^r]\prod_p 1/(1-t^{j_p}x)
%  =\mathrm{h}_r(t^{j_1},t^{j_2},\ldots)$. Applying this to the two products of
%  \eqref{equation:main} gives \eqref{equation:homogeneous}.
%\end{proof}

Setting~$t=1$ gives the sum of the coefficients of each polynomial~$\mathrm{M}_{g,k,i}(t)$.
The following corollary recovers the formula for these sums established in \cite[Lemma~3.17]{BBFGHLPRAS}.

\begin{cor}
  \label{corollary:euler}
  Setting~$t=1$ in \eqref{equation:main},
  we have~$\mathrm{M}_{g,k,k+1}(1)=1$ and, for~$i=0,\ldots,k$,
  \begin{equation}
    \label{equation:euler}
    \mathrm{M}_{g,k,i}(1)
    =\binom{2g-4-k-i}{k+1-i}+2\binom{2g-4-k-i}{k-i}.
  \end{equation}
\end{cor}

\begin{proof}
  At~$t=1$ each Pochhammer factor degenerates as~$(y;t)_n\mapsto(1-y)^n$,
  so by Corollary~\ref{corollary:closed}
  \begin{equation}
    \mathrm{F}_a(1,x)=\frac{1-x^2}{(1-x)^{2a}}=\frac{1+x}{(1-x)^{2a-1}}.
  \end{equation}
  The constant coefficient is~$1$, giving~$\mathrm{M}_{g,k,k+1}(1)=1$.
  For~$i=0,\ldots,k$, extracting the coefficient of~$x^{k+1-i}$ (with~$a=g-k-1$) gives~$\binom{2g-k-i-3}{k+1-i}+\binom{2g-k-i-4}{k-i}$,
  which equals \eqref{equation:euler} by Pascal's identity.
\end{proof}

\begin{rem}
  \label{remark:pascal}
  The value~$\mathrm{M}_{g,k,i}(1)$ of Corollary~\ref{corollary:euler} is obtained in
  \cite[Lemma~3.17]{BBFGHLPRAS} by applying Pascal's identity repeatedly to
  the leading binomial of \eqref{equation:M},
  but only \emph{after} having set~$t=1$.
  For the polynomial~$\mathrm{M}_{g,k,i}(t)\in\mathbb Z[t]$ itself,
  a direct repetition of this argument does not seem to make effectivity apparent:
  the~$t$-Pascal identity~$\binom{n}{m}_t=\binom{n-1}{m}_t+t^{n-m}\binom{n-1}{m-1}_t$
  introduces additional monomial weights whose cancellations must be controlled.

  The remedy is to package the whole family
  into the generating function \eqref{equation:Fa},
  and pass to its closed form.
  The~$t$-Pascal recursion enters through the generating-function identity of Proposition~\ref{proposition:pochhammer},
  which gives \eqref{equation:rational}.
  Clearing denominators and collecting coefficients then gives \eqref{equation:collapse}.
\end{rem}

\bibliographystyle{alpha}
\bibliography{references.bib}

\end{document}